\documentclass[11pt]{amsart}
 
\usepackage{amsmath}
\usepackage{amsthm}  
\usepackage{verbatim}
\usepackage{enumerate} 
\usepackage{mathtools}  
\usepackage{amssymb}
\usepackage{mathrsfs}
\usepackage{makecell}
\usepackage{dsfont}
\usepackage{overpic}
\usepackage[top=1.5in, bottom=1.5in, left=0.7in, right=0.7in]{geometry}  
\usepackage[colorlinks]{hyperref}
\hypersetup{
	colorlinks,
	citecolor=blue,
	linkcolor=red
}   
\numberwithin{equation}{section}
\usepackage{xypic}
\usepackage{bm}
\usepackage{tikz}
\usepackage{float}
\usetikzlibrary{arrows,positioning,decorations.pathreplacing} 
\tikzset{
    >=stealth',
    punkt/.style={
           rectangle,
           rounded corners,
           draw=black, very thick,
           text width=6.5em,
           minimum height=2em,
           text centered},
    pil/.style={
           ->,
           thick,
           shorten <=2pt,
           shorten >=2pt,}
}

\newtheorem*{theorem*}{\textbf{Theorem}}
\newtheorem{theorem}{\textbf{Theorem}}[section]
\newtheorem{thmx}{Theorem}
 
\newtheorem{corollaryx}[thmx]{\textbf{Corollary}}
\newtheorem{definition}[theorem]{\textbf{Definition}}
\newtheorem{proposition}[theorem]{\textbf{Proposition}}
\newtheorem{lemma}[theorem]{\textbf{Lemma}}
\newtheorem{question}[theorem]{\textbf{Question}}

\newtheorem{axiom}[theorem]{\textbf{Axiom}}
\newtheorem{claim}[theorem]{\textbf{Claim}}
\newtheorem*{claim*}{\textbf{Claim}}

\newtheorem{corollary}[theorem]{\textbf{Corollary}}
\newtheorem{remark}[theorem]{\textbf{Remark}}

\newtheorem{example}[theorem]{\textbf{Example}}
\newtheorem{conjecture}[theorem]{\textbf{Conjecture}}
\newtheorem{definition/proposition}[theorem]{\textbf{Definition/Proposition}}

\def\N{{\mathbb N}}
\def\R{\mathbb{R}}
\def\Z{{\mathbb Z}}

\def\C{{\mathbb C}}
\def\D{{\mathbb D}}

\def\G{{\mathbb G}}

\def\Q{{\mathbb Q}}

\newcommand{\CP}{\mathbb{C}\mathbb{P}}

\def\cA{{\mathcal A}}

\def\cC{{\mathcal C}}

\def\cH{{\mathcal H}}

\def\cL{{\mathcal L}}
\def\cM{{\mathcal M}}

\def\cO{{\mathcal O}}

\def\cR{{\mathcal R}}
\def\cS{{\mathcal S}}

\def\cW{{\mathcal W}}

\def\bH{{\bm H}}

\def\bk{{\bm k}}

\def\rd{{\rm d}}

\DeclareMathOperator{\Ima}{im}
\DeclareMathOperator{\ind}{ind}
\DeclareMathOperator{\Hom}{Hom}
\DeclareMathOperator{\Id}{id}

\DeclareMathOperator{\supp}{supp}

\DeclareMathOperator{\rank}{rank}

\DeclareMathOperator{\codim}{codim}

\DeclareMathOperator{\vdim}{vdim}
\DeclareMathOperator{\vir}{vir}
\DeclareMathOperator{\rel}{rel}
\DeclareMathOperator{\ev}{ev}
\DeclareMathOperator{\ext}{ext}
\DeclareMathOperator{\Aut}{Aut}
\DeclareMathOperator{\cof}{cof}
\DeclareMathOperator{\Aug}{Aug}
\DeclareMathOperator{\RSFT}{RSFT}
\DeclareMathOperator{\SFT}{SFT}
\DeclareMathOperator{\com}{c}
\DeclareMathOperator{\Hcx}{H_{cx}}
\DeclareMathOperator{\Pl}{P}
\DeclareMathOperator{\PT}{APT}
\DeclareMathOperator{\SD}{SD}
\DeclareMathOperator{\U}{U}
\DeclareMathOperator{\AU}{AU}
\DeclareMathOperator{\TT}{T}
\DeclareMathOperator{\CHA}{CHA}
\DeclareMathOperator{\LCH}{LCH}
\DeclareMathOperator{\GW}{GW}
\def\cont{{\mathfrak{Con}}}

\newcommand{\Addresses}{{% additional braces for segregating \footnotesize
		\bigskip
		\footnotesize
		
	     Agustin Moreno, \par\nopagebreak
        \textsc{Institute for Advanced Study, Princeton, U.S. / Heidelberg Universität, Germany}\par\nopagebreak
         \textit{E-mail address:} \href{mailto:agustin.moreno2191@gmail.com}{agustin.moreno2191@gmail.com}
		
		\medskip
		
	    Zhengyi Zhou, \par\nopagebreak
	    \textsc{Morningside Center of Mathematics, Chinese Academy of Sciences}\par\nopagebreak
         \textsc{Academy of Mathematics and Systems Science, Chinese Academy of Sciences, China}\par\nopagebreak
		\textit{E-mail address}: \href{mailto:zhyzhou@amss.ac.cn}{zhyzhou@amss.ac.cn}

}}

\date{}
\title{A landscape of contact manifolds via rational SFT }
\author{Agustin Moreno, Zhengyi Zhou}

\begin{document}
	\maketitle
\begin{abstract}
We define a hierarchy functor from the exact symplectic cobordism category to a totally ordered set from a $BL_\infty$ (Bi-Lie) formalism of the rational symplectic field theory (RSFT). The hierarchy functor consists of three levels of structures, namely \emph{algebraic planar torsion}, \emph{order of semi-dilation} and \emph{planarity}, all taking values in $\N\cup \{\infty\}$, where algebraic planar torsion can be understood as the analogue of Latschev-Wendl's algebraic torsion \cite{LW} in the context of RSFT. The hierarchy functor is well-defined through a partial construction of RSFT and is within the scope of established virtual techniques. We develop computational tools for those functors and prove all three of them are surjective. In particular, the planarity functor is surjective in all dimension $\ge 3$. Then we use the hierarchy functor to study the existence of exact cobordisms. We discuss examples including iterated planar open books, spinal open books, affine varieties with uniruled compactification and links of singularities. 
\end{abstract}
%\tableofcontents
\section{Introduction} 
One central subject in symplectic and contact topology is the study of symplectic cobordisms. Unlike the usual cobordism relation in differential topology, a symplectic cobordism is asymmetric; the collection of such cobordisms endows the collection of contact manifolds with a structure similar to a partial order. The fundamental dichotomy between \emph{overtwisted} contact structures and \emph{tight} contact structures discovered by Eliashberg in dimension $3$ \cite{eliashberg1989classification} and Borman-Eliashberg-Murphy in higher dimensions \cite{borman2015existence} is reflected by the fact that overtwisted contact structures behave like least elements\footnote{A least element in a poset is an element that is smaller than any other elements. A minimal element in a poset is an element such that there is no smaller element. }. Namely, there is always an exact cobordism from an overtwisted contact manifold to any other contact manifold in dimension $3$ \cite{etnyre2002symplectic} and the same holds for higher dimensions when the obvious topological obstructions vanish \cite{eliashberg2015making}. Moreover, the overtwisted contact $3$-manifold behaves like minimal elements in the Weinstein cobordism category, as any contact $3$-manifold that is Weinstein cobordant to an overtwisted contact manifold is overtwisted by \cite{MR3418529}. To explore the realm of the more mysterious class of tight contact structures, the hierarchy imposed by the existence of symplectic cobordisms is a useful guiding principle, as the complexity of contact topology should not decrease in a cobordism. In dimension $3$, a further hierarchy in the world of tight contact manifolds was discovered by Giroux \cite{giroux1994structure} and Wendl \cite{wendl2013hierarchy}. In higher dimensions, the notion of Giroux torsion was generalized by Massot-Niederkr\"uger-Wendl \cite{massot2013weak}. 

On the other hand, since contact manifolds and (exact) symplectic cobordisms form a natural category, which we will refer to as the (exact) symplectic cobordism category $\cont$, one natural approach to study $\cont$ is by understanding functors from $\cont$ to some algebraic category, a.k.a.\ a field theory. \emph{Symplectic field theory} (SFT), as proposed by Eliashberg-Givental-Hofer in \cite{eliashberg2000introduction}, is a very general framework for defining such functors, and many invariants of contact manifolds and symplectic cobordisms can be defined via suitable counts of punctured holomorphic curves which approach Reeb orbits at their punctures.  The formidable algebraic richness of the general theory, together with the serious technical difficulties arising in building its analytical foundations,  conspire to make explicit computations a complicated matter. Therefore, rather than focusing on computing the full SFT invariant, one could focus on extracting simpler invariants from the general theory whose computation is in principle approachable via currently available techniques. An example of this philosophy is the notion of \emph{algebraic torsion} introduced by Latschev-Wendl in \cite{LW}, which associates to every contact manifold a number in $\N\cup\{\infty\}$ and can be viewed as the algebraic interpretation of the geometric concept of \emph{planar torsion} defined by Wendl \cite{wendl2013hierarchy}.

In this paper, we follow the same methodology of Latschev-Wendl to study the structure of $\cont$. Instead of the full SFT, we use the rational SFT  (RSFT), i.e.\ we only consider genus $0$ curves, to construct a functor from $\cont$ to a totally ordered set. Our main theorem is the following. 

\begin{thmx}\label{thm:main}
    We have the following monoidal functors:
    \begin{enumerate}
        \item \emph{Algebraic planar torsion} (Definition \ref{def:APT}) $\PT:\cont \to \N\cup \{\infty\}$, where the monoidal structure on  $(\N\cup \{\infty\},\le)$ is given by $a\otimes b:=\min \{a,b\}$ (Proposition \ref{prop:PT});
        \item \emph{Planarity} (Definition \ref{def:planarity}) $\Pl:\cont \to \N\cup \{\infty\}$, here the monoidal structure on  $(\N\cup \{\infty\},\le)$ is given by $a\otimes b:=\max\{a,b\}$ if $a,b\ne 0$ and $0\otimes a = a\otimes 0=0$ (Proposition \ref{prop:P_order});
        \item\label{SD} \emph{Order of semi-dilation} (\eqref{eqn:SD}) $\SD:\Pl^{-1}(1)\to (\N \cup \{\infty\},\le)$, here the monoidal structure on  $\N\cup \{\infty\}$ is given by $a\otimes b:=\max\{a,b\}$ (Proposition \ref{prop:SD}).
    \end{enumerate}
\end{thmx}
When $\PT$ is finite, it is necessary to have $\Pl=0$, i.e.\ $\PT$ and $\SD$ are refinements of the cases with $\Pl=0,1$. Therefore we can assemble all three functors into a functor $\Hcx$ measuring the contact complexity,
$$\Hcx:\cont \to \cH:=\{\underbrace{0^{\PT}<1^{\PT}<\ldots<\infty^{\PT}}_{0^{\Pl}}< \underbrace{0^{\SD}<1^{\SD}<\ldots<\infty^{\SD}}_{1^{\Pl}}< 2^{\Pl}<\ldots<\infty^{\Pl} \}\footnote{There is a (unique) morphism $a\to b$ iff $a\le b$, $a,b\in \cH$.}.$$
Here $k^{\PT}$ stands for $\Pl=0$ and $\PT=k$, $k^{\SD}$ stands for $\Pl=1$ and $\SD=k$, $k^{\Pl}$ stands for $\Pl=k$. $\PT$ can be viewed as the analogue of algebraic torsion in the context of RSFT. In particular, finiteness of $\PT(Y)$ implies that $Y$ has no strong filling just like algebraic torsion. However, $\Hcx$ goes well beyond non-fillable contact manifolds, i.e.\ $\SD,\Pl$ provide measurements for fillable contact manifolds. Roughly speaking, $\PT$ looks for rational curves without negative punctures and $\Pl$ looks for rational curves with a point constraint in symplectizations. And $\SD$ is defined using the $\Q[U]$-module structure on linearized contact homology introduced by Bourgeois-Oancea \cite{bourgeois2009exact}. $\PT$ measures the obstruction to augmentations of RSFT, while $\SD,\Pl$ can be phrased in the linearized theory, hence require the existence of augmentations. To make  $\SD,\Pl$ independent of the augmentation,  we need to define $\SD$ and $\Pl$ via traversing the set of all possible augmentations of the RSFT.

Of course, Theorem \ref{thm:main} as stated could be trivial, as the true content of the claim is contained in the algebraic construction. The following results endow the functors with geometric content.
\begin{thmx}\label{thm:RSFT}
	The functors above have the following properties.
	\begin{enumerate}
		\item\label{RSFT1} If $Y$ has planar $k$-torsion \cite{wendl2013hierarchy}, then $\PT(Y)\le k$ (\cite[Theorem 6]{LW}, Theorem \ref{thm:planar_torsion_1}). 
		\item\label{RSFT5} If $Y$ is overtwisted then $\PT(Y)=0$ (\cite{bourgeois2010contact}). 
		\item\label{Girouxtorsion} If $Y$ has (higher dimensional) Giroux torsion \cite{massot2013weak}, then $\PT(Y)\le 1$ (\cite[Theorem 1.7]{MorPhD}, Theorem \ref{thm:planar_torsion_3}).
		\item If $\PT(Y)<\infty$, then $Y$ is not strongly fillable (Corollary \ref{cor:nofilling}). If $Y$ admits an exact filling then $\Pl(Y)\ge 1$ (Proposition \ref{prop:aug}).
		\item\label{RSFT2} If $Y$ is an iterated planar open book \cite{acu2017weinstein} where the initial page has $k$-punctures, then $\Pl(Y)\le k$ (Theorem \ref{thm:upper}).
		\item\label{RSFT3} If $Y$ has an exact filling that is not $k$-uniruled \cite{mclean2014symplectic}, then $\Pl(Y)\ge k+1$ (Theorem \ref{thm:unirule}).
		\item\label{RSFT4} $\PT,\SD,\Pl$ are all surjective (\cite[Theorem 4]{LW}, Theorem \ref{thm:planar_torsion_2}, Theorem \ref{thm:brieskorn}, Corollary \ref{cor:product}). In particular, $\Pl$ is surjective in all odd dimension $\ge 3$ (Corollary \ref{cor:product}).
	\end{enumerate} 
\end{thmx}

\subsection{Rational SFT}
The original algebraic formalism of SFT in \cite{eliashberg2000introduction} packaged the full SFT into a super Weyl algebra with a distinguished odd degree Hamiltonian $\bH$ such that $\bH \star \bH=0$.  Cieliebak-Latschev reformulated the algebra into a $BV_\infty$ algebra \cite{cieliebak2009role}, which was used in the definition of algebraic torsion \cite{LW}. The $BV_\infty$ algebra structure was further refined to an $IBL_\infty$ (\emph{\textbf{I}nvolutive \textbf{B}i-\textbf{L}ie infinity}) algebra by Cieliebak-Fukaya-Latschev \cite{cieliebak2015homological}, which roughly speaking, is precisely the boundary combinatorics for the SFT compactification \cite{bourgeois2003compactness}. For rational SFT, the original algebraic formalism was a Poisson algebra with a distinguished odd degree Hamiltonian $\bm{h}$ such that $\{\bm{h},\bm{h}\}=0$. Analogous reformulations of the algebraic structure of RSFT can be found in Hutchings' ``$q$-variable only RSFT'' \cite{rational}, and an $L_\infty$ formalism of RSFT by Siegel \cite{siegel2019higher}. In this paper, we introduce a notion of $BL_\infty$ (\emph{\textbf{B}i-\textbf{L}ie infinity}) algebra to describe RSFT, which precisely describes the boundary combinatorics for rational curves in the SFT compactification and is a specialization of the $IBL_\infty$ formalism. By building functors from the category of $BL_\infty$ algebras to totally ordered sets, we can build the hierarchy functor in Theorem \ref{thm:main} by a composition
\begin{equation}\label{eqn:compo}
\Hcx:\cont \stackrel{\RSFT}{\longrightarrow} BL_\infty \text{(with additional structures up to homotopy)} \longrightarrow \cH.
\end{equation}

On the other hand, the general holomorphic curve theory in manifolds with contact boundaries faces serious analytical challenges, which makes a complete construction of the first functor in \eqref{eqn:compo} a difficult task. To obtain a construction of SFT/RSFT, one needs to deploy more powerful virtual techniques, e.g.\ either \emph{polyfold} approaches \cite{SFT,hofer2017polyfold} by Fish-Hofer and Hofer-Wysocki-Zehnder, \emph{implicit atlases and virtual fundamental cycles} by Pardon \cite{pardon2016algebraic,pardon2019contact}, or \emph{Kuranishi} approaches by Ishikawa \cite{ishikawa2018construction}. However, for the purpose of defining $\Hcx$, it is sufficient to build $\RSFT$ partially. In particular, we do not need to discuss compositions and homotopies for morphisms of $BL_\infty$ algebras as $\cH$ is a totally ordered set, where there is no ambiguity for compositions and homotopy equivalences. This greatly simplifies our demands for virtual machinery, as homotopies in SFT is a subtle subject. Moreover, the combinatorics for a $BL_\infty$ algebra is ``tree-like", which is very similar to the combinatorics for contact homology. As a consequence, we can use Pardon's construction of contact homology \cite{pardon2019contact} to provide all the analytic foundation of the functor $\Hcx$. In particular, Theorem \ref{thm:main} is well-posed without any hidden hypotheses on virtual machinery (except for \eqref{SD} of Theorem \ref{thm:main}, for which we give a sketch and details will appear in a future work). Moreover, it is expected that any other virtual technique will suffice for Theorem \ref{thm:main}. We will also explain how to obtain another construction of $\Hcx$ from a small part of the polyfold construction of SFT \cite{SFT}.

In general, a full computation of RSFT and SFT is very difficult, as we need to understand many moduli spaces. On the other hand, the hierarchy functor $\Hcx$  extracts partial information from $BL_\infty$ algebras, so that only partial knowledge of the moduli spaces are needed. In particular, $\Hcx$ is relatively computable. It is a nontrivial question whether $\Hcx$ is independent of the choice of virtual technique. However, since every virtual technique has the property that we can count a compactified moduli space geometrically if it is cut out transversely in the classical sense, Theorem \ref{thm:RSFT} does not depend on the choice of virtual technique.

\subsection{Foundational claims.}\label{ss:FC} For the sake of clarity, let us here make precise what our claims are, pertaining to the foundations of RSFT. In particular, to avoid misunderstandings, we will clarify which parts of the necessary foundational work are carried out in this paper, and which ones are not. The following claims are proven in Theorem \ref{thm:BL}.

\begin{itemize}
    \item For a fixed choice of geometric data on a symplectization, there is a non-empty set of perturbation data in the context of Pardon’s VFC package, which allow to define (virtual) counts of rigid rational curves which satisfy a quadratic relation. RSFT can then be implemented as an instance of what we call a $\emph{BL}_\infty$-algebra, where the differential is given by these counts.
    \item For a fixed choice of geometric data on an exact cobordism (compatible with choices at the ends) there is a non-empty set of perturbation data (compatible with the perturbation data at the ends), which allow to define virtual counts of rigid rational curves in the cobordism satisfying a compatibility relation with the previously defined counts at the ends. This gives a morphism from the $BL_\infty$ algebra of the convex end, to that of the concave end, i.e.\ this establishes functoriality of the theory.
    \item In both cases above, there are versions of virtual counts of rational curves with a point constraint, again satisfying compatibility relations with the above differentials and morphisms.
    \item We make no claims about homotopies between morphisms obtained from different choices of geometric data and/or perturbation data. In particular, we do not claim to prove independence of the resulting \emph{full} algebraic structures from the auxiliary choices (contact form, almost complex structure, perturbations).
    \item However, our invariants algebraic planar torsion and planarity, which are defined purely at the algebraic level, i.e.\ for $BL_\infty$ algebras, are well-defined for RSFT. In other words, invariance under homotopies is not needed for their implementation. These invariants are moreover independent on the choice of virtual perturbation scheme for applications in this paper, provided very mild requirements are satisfied by the scheme (namely, that virtual counts coincide with geometric counts if transversality holds). All the results pertaining to these invariants are therefore rigorous.
    \item A full implementation of the order of semi-dilation requires a rigorous implementation of the $U$-map in linearized contact homology (cf.\ \cite{bourgeois2009exact}), e.g.\ within the framework of Pardon's VFC package, or other virtual technique. We give a heuristic approach at the end of Section \ref{sec:U-map}, which we defer to later work, and precisely state what is needed to define the order of semi-dilation (see Claim \ref{claim:u}). All results pertaining to this invariant are therefore conditional on the unproven Claim \ref{claim:u}. On the other hand, implementation of homotopies is also not necessary for this invariant to be well-defined.
\end{itemize}

%\begin{remark}
%The relation between $\PT$ and algebraic torsion $\mathrm{AT}$ \cite{LW} is not direct. In fact, they are both implied by a stronger notion of torsion, which is implied by planar $k$-torsion \cite{wendl2013hierarchy}, defined through an alternative representation of the $IBL_\infty$ algebra of the full SFT. There is a grid of torsions serving as obstructions to strong fillings, where algebraic planar torsion and algebraic torsion are two axes. The details, which were included in the previous version this paper, can be now found in \cite{supp}. The only common ground is that $0$-algebraic planar torsion is equivalent to $0$-algebraic torsion and \emph{algebraically overtwistedness} \cite{bourgeois2010towards}, which is implied by overtwistedness \cite{bourgeois2010contact,MR2230587}. Moreover, there are $5$-dimensional examples with underlying smooth manifold $Y=S^*X\times \Sigma_g$ (where $S^*X$ is the unit cotangent bundle of a hyperbolic surface $X$ and $\Sigma_g$ is the orientable surface of genus $g\geq 1$), considered originally in \cite{MorPhD}, which we conjecturally expect to have $\PT(Y)>1$ but $\mathrm{AT}(Y)=1$; see \cite[Sec.\ 6.5]{MorPhD}. This would provide concrete examples on which the two notions of torsion strictly differ. 
%\end{remark}

\subsection{Applications} Since $\Hcx$ is a measurement of the complexity of contact topology, the main application of $\Hcx$ is obstructing the existence of exact cobordisms. The following theorem answers a conjecture of Wendl \cite{wendl2013hierarchy} affirmatively, although the invariant we use is $\Pl$ instead of algebraic torsion (as opposed to the original conjecture).
\begin{thmx}[Theorem \ref{thm:main} + Corollary \ref{cor:product}]\label{thm:seq}
	For any dimension $\ge 3$, there exists an infinite sequence of contact manifolds $Y_1,Y_2\ldots,$ such that there is an exact cobordism from $Y_i$ to $Y_{i+1}$, but there is no exact cobordism from $Y_{i+1}$ to $Y_i$. 
\end{thmx}
The above result was obtained in dimension $3$ in \cite{LW}. In fact, there are many examples of $Y_i$, the simplest example being the boundary of the product of $n$ copies of $S_k$--$2$-spheres with $k$ disks removed, as we will show in \S \ref{s6} that $\Pl(\partial(S_k)^n)=k$ for $n\ge 2$. There are many more examples for Theorem \ref{thm:seq} to hold; see e.g. Theorem \ref{thm:SK} below.

Following the definition of $\Pl$, it is easy to see that if $Y$ admits a contact structure without Reeb orbits, then $\Pl(Y)=\infty$. Therefore, as a corollary, we have the following.
\begin{corollaryx}\label{cor:Weinstein}
	If $\Pl(Y)< \infty$, then the Weinstein conjecture holds for $Y$.
\end{corollaryx}
In other words, counterexamples to the Weinstein conjecture (if any) should be looked for in the highest complexity level $\infty^P$. In particular, the combination of \eqref{RSFT2} in Theorem \ref{thm:RSFT} and Corollary \ref{cor:Weinstein} yields a proof of the Weinstein conjecture for iterated planar open books, which was previously obtained for dimension $3$ in \cite{abbas2005weinstein} and higher dimensions in \cite{acu2017weinstein,acu2018planarity}. In some sense, the proof of \eqref{RSFT2} of Theorem \ref{thm:RSFT} endows the ruling holomorphic curve in \cite{abbas2005weinstein,acu2017weinstein,acu2018planarity} with a homological meaning, i.e.\ the ruling curve defines a map that is visible on homology; in particular, such curve can not be eliminated by perturbing the contact form. On the other hand, not every contact manifold with finite planarity is iterated planar: for example $\Pl(T^3,\xi_{std})=2$ by Corollary \ref{cor:product}, while $(T^3,\xi_{std})$ is not supported by a planar open book by \cite{etnyre2004planar} (it is, however, supported by a planar \emph{spinal} open book \cite{wendl2010strongly,LVHMW}). By functoriality, if there is an exact cobordism from $Y$ to $Y'$ with $\Pl(Y')<\infty$, then the Weinstein conjecture holds for $Y$. 

The study of planar open books in dimension $3$ has a very long history, since they enjoy nice properties like equivalence of weak fillability and Weinstein fillability \cite{niederkruger2011weak,wendl2010strongly}. We refer readers to the introduction of \cite{acu2020introduction} for a comprehensive summary on the subject. Obstructions to planar open book structures were obtained in \cite{etnyre2004planar,ozsvath2005planar}. In higher dimensions, obstructions to supporting an iterated planar open book were found in \cite{acu2018planarity}. By \eqref{RSFT2} of Theorem \ref{thm:RSFT}, infinite planarity is an obstruction to an iterated planar structure. In particular, we answer \cite[Question 1.14]{acu2020generalizations} negatively by the following general result.
\begin{corollaryx}[Corollary \ref{cor:noIP}] 
	In all dimensions $\ge 5$, consider $(Y,J)$ an almost contact manifold which has an exactly fillable contact representative $(Y,\xi)$. Then there is a contact structure $\xi'$ in the homotopy class of $J$, such that $(Y,\xi')$ is not iterated planar.
\end{corollaryx}
In particular, since every simply connected almost contact $5$-manifold is almost Weinstein fillable \cite{geiges1991contact}, there is a contact structure in each homotopy class of almost complex structures that is not iterated planar for every simply connected $5$-manifold.

\subsection{Examples}  In addition to Theorem \ref{thm:RSFT}, there are many situations where we can compute or estimate $\Hcx$. By \eqref{RSFT3} of Theorem \ref{thm:RSFT}, it is natural to look at affine varieties with a uniruled projective compactification. One special case is affine varieties with a $\CP^n$ compactification. 

\begin{thmx}[Theorem \ref{thm:CP}]\label{thm:proj}
	Let $D$ be $k$ generic hyperplanes in $\CP^n$ for $n\ge 2$, then we have the following.
\begin{enumerate}
	\item $\Pl(\partial D^{\com})\ge k+1-n$ for $k>n+1$, where $\partial D^{\com}$ is the contact boundary of the affine variety $D^{\com}:=\CP^n\backslash D$ (see \S \ref{ss:affine}).
	\item $\Pl(\partial D^{\com})= k+1-n$ for  $n+1<k<\frac{3n-1}{2}$ and $n$ odd.
	\item $\Pl(\partial D^{\com})=2$ for $k=n+1$. 
	\item $\Hcx(\partial D^{\com})=0^{\SD}$ for $k\le n$.
\end{enumerate}
\end{thmx}	
The condition on $n$ being odd (also for Theorem \ref{thm:gen}, \ref{thm:SD} below) is not essential. We use it to obtain automatic closedness of a chain in the computation of planarity for any augmentation. In Remark \ref{rmk:n=3}, we explain how one can drop this condition using polyfold techniques in \cite{zhou2020quotient}. On the other hand, the role of $k<\frac{3n-1}{2}$ is more mysterious. Although it is unlikely to be optimal, whether an upper bound is necessary is unclear. An extreme case is when $n=1$, then $\partial D^{\com}$ is a disjoint union of circles, then we have $\Pl(\partial D^{\com})=\infty$ and unlike the situation in Theorem \ref{thm:proj} there is no obstruction to exact cobordisms for different $k$ when $n=1$. One difficulty of computing $\Pl$ and obtaining cobordism obstructions is that we need to carry out computation for all hypothetical ``fillings", i.e.\ augmentations. Indeed, different choices of augmentations will affect the computation dramatically. For example, there exists affine varieties with a $\CP^n$ compactification whose contact boundary has infinite planarity, cf.\ Theorem \ref{thm:infinity}. However, if we use the augmentation from the affine variety, then the planarity is finite. 

\begin{remark}
Based on the notion of asymptotically dynamically convex manifolds introduced by Lazarev \cite{MR4081058}, \cite{zhou2019symplectic,zhou2019symplecticI} exploited the uniqueness of $\Z$ graded (dga) augmentations to the contact homology to obtain some cobordism obstructions. However, to maintain the effectiveness of the $\Z$-grading, one needs to make additional topological assumptions (vanishing of first Chern class, injectivity of the fundamental groups etc.) for the argument in \cite{zhou2019symplectic,zhou2019symplecticI} to work. Dropping the $\Z$-grading condition will almost certainly result in multiple augmentations (if an augmentation exists). Moreover, having a unique $\Z$-graded $BL_\infty$ augmentation requires an index assumption much more restrictive than asymptotically dynamical convexity, which, in the context of flexibly fillable contact manifolds with vanishing first Chern class, requires that the flexible filling has only $0,1,2,3$ handles. The strategy in this paper is very different. Instead of using the uniqueness of certain augmentations, we search for examples and  structures on the RSFT that are independent of all augmentations (which are not unique). Showing such independence is the main challenge in the proof of Theorem \ref{thm:proj}. 
\end{remark}

On the other hand, $D_{k}^{\com}$ embeds exactly into $D_{k+1}^{\com}$, which follows from a general construction, as follows. Let $\cL$ be a very ample line bundle over a smooth projective variety $X$. Then for any nonzero holomorphic section $s\in H^0(\cL)$, $X\backslash s^{-1}(0)$ is an affine variety whose contact boundary is denoted by $Y_s$. The projective space $\mathbb{P}H^0(\cL)$ should be stratified by the singularity type of $s^{-1}(0)$\footnote{It is quite a nontrivial task to make this stratification precise, as in general we do not have a classification of the possible singularities of the divisor.}, with the top stratum corresponding to the case where $s^{-1}(0)$ is smooth with multiplicity $1$. We say that there is a morphism from stratum $A$ to stratum $B$, if we can change $s^{-1}(0)$ from $A$ to $B$ by an arbitrarily small perturbation of the section $s$, i.e.\ $A$ is contained in the closure of $B$. Moreover, one obtains an exact cobordism from $Y_s$ to $Y_{s^\prime}$, where $s^\prime$ is the perturbed section. Then we have a natural functor from the category of strata to $\cont$. As a concrete example, consider $\cL=\cO(2)$ on $\CP^2$, then the category of stratification is the graph $A_1\to A_2\to A_3$, where $A_1,A_2,A_3$ correspond to a double line, two generic lines and a smooth quadratic curve as the divisor, respectively. The corresponding affine varieties as exact domains are $\C^2,\C\times T^*S^1,T^*{\mathbb{RP}^2}$, which clearly have the exact embedding relations as claimed. 

In view of this, when we view one of the $k$ hyperplanes in $D_k$ as having multiplicity $2$ (a ``double'' hyperplane), we can get an exact cobordism from $\partial D_{k}^{\com}$ to $\partial D_{k+1}^{\com}$, by perturbing the double hyperplane to two distinct hyperplanes. Then Theorem \ref{thm:proj} asserts that a reversed exact cobordism can not be found if $n\le k<\frac{3n-1}{2}$ and $n$ is odd. Note that the natural inclusion $D^{\com}_{k+1}\subset D^{\com}_{k}$ is symplectic, hence we always have a strong cobordism from $\partial D^{\com}_{k+1}$ to $\partial D^{\com}_{k}$, which shows the essential difference between these two notions of cobordisms and the obstruction from $\Pl$ is not topological. When $k\le n$, $D_{k}^{\com}$ is in fact subcritical and they can be embedded exactly into each other regardless of $k$. 

As a concept closely related to $\cont$, we introduce $\cont_*$ as the under category of $\cont$ under $\emptyset$, i.e.\ the objects of $\cont_*$ are pairs of contact manifolds with exact fillings and morphisms are exact embeddings. Then $\SD,\Pl$ can be defined on $\cont_*$\footnote{The functorial property of $\Pl$ requires a full construction of RSFT, including compositions and homotopies. In particular, it makes more demands for virtual constructions than explained in this paper. On the other hand, the functorial property of $\SD$ follows from the Viterbo transfer map of $S^1$-equivariant symplectic cohomology and the isomoprhism between the positive $S^1$-equivariant symplectic cohomology and the linearized contact homology \cite{bourgeois2009exact}.} using the augmentation from the given exact filling. Moreover, we recall another functor $\U$, called the \emph{order of uniruledness}, which is defined to be the minimal $k$ such that an exact domain $W$ is $k$-uniruled in the sense of McLean \cite{mclean2014symplectic}. That $\U$ is a functor from $\cont_*\to \N_+\cup \{\infty\}$  was proven in \cite{mclean2014symplectic}. By \eqref{RSFT3} of Theorem \ref{thm:RSFT}, $\Pl(Y)$ is bounded below by $\U(W)$ for an exact filling $W$ of $Y$. The functor $\U$ measures the complexity of exact domains and serves as an exact embedding obstruction. An interesting aspect of $\U$ is that the well-definedness and basic properties of $\U$ do not depend on any Floer theory. As a byproduct of the proof of Theorem \ref{thm:proj}, we have following for any $n\ge 1$.
\begin{thmx}[Theorem \ref{thm:embedding}]\label{thm:unirle}
	Let $D_k$ denote the divisor of $k$ generic hyperplanes in $\CP^n$ for $n\ge 1$ and $D_k^{\com}$ denote the complement affine variety. Then $\U(D_k^{\com})=\max\{1,k+1-n\}$. In particular, $D_{k+1}^{\com}$ can not be embedded into $D_{k}^{\com}$ exactly for $k\ge n$.
\end{thmx}

\begin{remark}
	The same embedding question is studied independently by Ganatra and Siegel \cite{LC}, where more general normal crossing divisors in $\CP^n$ are studied. The planarity for exact domains mentioned above is equivalent to $\G\langle p \rangle$ in \cite{LC}. The authors of \cite{LC} also consider holomorphic curves with local tangent constraints to define functors  $\G\langle T^m p \rangle$ on $\cont_*$. In view of the local tangent constraints, one can define an analogous order of uniruledness with local tangent conditions, the well-definedness and functorial property of such invariants follows from the same argument of \cite{mclean2014symplectic}. Such functor can also serve as an embedding obstruction as $\U$ in Theorem \ref{thm:unirle}. It is an interesting question on whether those geometric invariants are the same as the algebraic invariants (defined via RSFT in \cite{LC}), which is the case for Theorem \ref{thm:unirle}.
\end{remark}

In view of \eqref{RSFT3} of Theorem \ref{thm:RSFT}, one can also consider affine varieties with uniruled compactification, in particular those affine varieties  with Fano hypersurfaces as compactification. In general, we have the following.

\begin{thmx}[Theorem \ref{thm:generalization}]\label{thm:gen}
	Let $X$ be a smooth degree $m$ hypersurface in $\CP^{n+1}$ for $2\le m< \frac{n+1}{2}\le n$ and $D$ be $k\ge n$ generic hyperplanes, i.e.\ $D=(H_1\cup \ldots \cup H_k)\cap X$ for $H_i$ is a hyperplane in $\CP^{n+1}$ in generic position with each other and $X$, then $\Pl(\partial D^{\com})=k+m-n$ for $n$ odd and $k+m<\frac{3n+1}{2}$.
\end{thmx}

The following results provide affine variety examples with nontrivial $\SD$.
\begin{thmx}[Theorem \ref{thm:P1}]\label{thm:SD}
	Assume $D_s$ is a smooth degree $2\le k<n$ hypersurface in $\CP^n$ for $n\ge 3$ odd, then $(k-1)^{\SD}\le \Hcx(\partial D_s^{\com})\le (2k-2)^{\SD}$. When $n$ is even and $2\le k <\frac{n+1}{2}$, then  we have $\Hcx(\partial D_s^{\com})\le (2k-2)^{\SD}$.
\end{thmx}

As explained before, the difficultly of computing $\Pl$ and $\SD$ is from enumerating through all possible augmentations. The strategy of proving Theorem \ref{thm:proj}, \ref{thm:gen}, \ref{thm:SD} is finding the curve responsible for $\Pl,\SD$ with low energy, so that there are no room for the dependence on augmentations. This is the reason why the results require that the hyperspaces to have low degrees. In particular, a degree $m\le n$ smooth hypersurface in $\CP^{n+1}$ is uniruled by degree $1$ curves. Our proof also uses somewhere injectivity of degree $1$ curves to obtain transversality in various places, hence this low degree condition is needed for technical reasons as well. It is interesting to look at the case of degree $n+1$ smooth hypersurfaces in $\CP^{n+1}$, which is uniruled but not by degree $1$ curves. A more systematic way to study $\Hcx$ is deriving formulas for RSFT of affine varieties with normal crossing divisor complement using log/relative Gromov-Witten invariants similar to the formula for symplectic cohomology in \cite{diogo2019symplectic}.

Another rich class of contact manifolds comes from links of isolated singularities. They provide examples with every order of semi-dilation based on computations in \cite{zhou2019symplectic}.
\begin{thmx}[Theorem \ref{thm:brieskorn}]\label{thm:brisk}
	We use $LB(k,n)$ to denote the contact link of the Brieskorn singularity $x_0^{k}+\ldots+x_n^k=0$,  then $\Hcx(LB(k,n))$ is
	\begin{enumerate}
		\item $(k-1)^{\SD}$ if $k<n$ and is $\ge (k-1)^{\SD}$ if $k=n$;
		\item $>1^{\Pl}$ if $k=n+1$;
		\item $\infty^{\Pl}$, if $k>n+1$.
	\end{enumerate}	
\end{thmx}

Another type of singularity is the quotient singularity, whose contact links are not exact fillable in many cases \cite{zhou2020mathbb}. In fact, the symplectic aspect of the proof in \cite{zhou2020mathbb} can be restated as a computation of $\Hcx$ as follows.
\begin{thmx}[Theorem \ref{thm:quotient}]
	Let $Y$ be the quotient $(S^{2n-1}/\Z_k,\xi_{std})$ by the diagonal action of $e^{\frac{2\pi i}{k}}$ for $n\ge 2$.
    \begin{enumerate}
		\item If $n>k$, we have $\Hcx(Y)=0^{\SD}$.
	    \item If $n\le k$, we have $0^{\SD} \le \Hcx(Y) \le (n-1)^{\SD}$. When $n=k$, we have $\Hcx(Y)\ge 1^{\SD}$.
    \end{enumerate}
\end{thmx}
The second case of the above theorem is another situation where the computation depends on the augmentation. Roughly speaking, $\Hcx(Y)=0^{\SD}$ means that any exact filling of $Y$ has vanishing symplectic cohomology. And if there is a (possibly strong) filling with vanishing symplectic cohomology, then the order of semi-dilation using the induced augmentation from the filling is $0$.\footnote{Assuming the filling is monotone, so that we can evaluate $T=1$ in the Novikov coefficient to get back to $\Q$-coefficient.} The natural pre-quantization bundle filling provides augmentations such that the symplectic cohomology vanishes \cite{ritter2014floer}. On the other hand, there are other augmentations with positive orders of semi-dilation. For example the exact filling $T^*S^2$ of $(\mathbb{RP}^3,\xi_{std})$ has order of semi-dilation $1$, such phenomenon was also explained in \cite[Remark 2.16]{zhou2020mathbb}. 

\begin{thmx}\label{thm:SK}
	Let $V$ be an exact domain with $c_1(V)=0$ and $S_k$ be the $k$-punctured sphere. Then 
	\begin{enumerate}
		\item\label{product1} $\Pl(\partial(V\times S_k))\le k$ (Theorem \ref{thm:SOBD}).
		\item\label{product2} If $V$ is an affine variety that is not $(k-1)$-uniruled, then $\Pl(\partial(V\times S_k))= k$ (Corollary \ref{cor:product}).
		\item\label{product3} $\Hcx(\partial (V\times \D))=0^{\SD}$ (Theorem \ref{thm:product}).
	\end{enumerate}
\end{thmx}
In particular, \eqref{product2} in Theorem \ref{thm:SK} provides many examples to Theorem \ref{thm:seq} and \eqref{product3} is a reformulation of the symplectic step in \cite{zhou2020filings} to obtain uniqueness results on fillings of  $\partial(V\times \D)$.

\subsection*{Organization of the paper}
We introduce the concept of $BL_\infty$ algebra in \S \ref{s2} and then define algebraic planar torsion as well as planarity at the level of algebra. In \S \ref{s3}, we implement Pardon's VFC \cite{pardon2019contact} to define $\PT$ and $\Pl$. We recall in \S \ref{s4} the $\Q[U]$ module structure on linearized contact homology following \cite{bourgeois2009exact} to define $\SD$ and finish the proof of Theorem \ref{thm:main}. We give a lower bound for $\Pl$ in \S \ref{s5} and an upper bound for $\Pl$ in  \S \ref{s6}. We discuss examples, applications, and finish the proof of Theorem \ref{thm:RSFT} in \S \ref{s7}. 
\subsection*{Acknowledgments}
We would like express our gratitude to Helmut Hofer, Mark McLean, John Pardon and Chris Wendl for helpful conversations.
We are very grateful to the anonymous referee for the significant effort put into assessing our paper and many suggestions that significantly improve the quality of the work. A.M.\ acknowledges the support by the Swedish Research Council under grant no.\ 2016-06596, while the author was in residence at Institut Mittag-Leffler in Djursholm, Sweden. He also acknowledges the support of the National Science Foundation under Grant No.\ DMS-1926686, by the Sonderforschungsbereich TRR 191 Symplectic Structures in Geometry, Algebra and Dynamics, funded by the DFG (Projektnummer 281071066 – TRR 191), and also by the DFG under Germany's Excellence Strategy EXC 2181/1 - 390900948 (the Heidelberg STRUCTURES Excellence Cluster). Z.Z.\ is supported by National Key R\&D Program of China under Grant No.\ 2023YFA1010500, National Science Foundation under Grant No.\ DMS-1926686, the National Natural Science Foundation of China under Grant No.\ 12288201 and 12231010.
\section{$L_\infty$ algebras and $BL_\infty$ algebras}\label{s2}
In this section, we recall the basics of $L_\infty$ algebras and introduce $BL_\infty$ algebras, which serve as the underlying algebraic structures for rational symplectic field theory. The algebraic formalism here is essentially the  $q$-variable only reformulation in \cite{rational} and the $L_\infty$ algebra formalism on contact homology algebra in \cite{siegel2019higher}, but we make the compatibility of the algebraic structure on the contact homology algebra with the $L_\infty$ structure more precise and define such an object as a $BL_\infty$ algebra, which is a specialization of the $IBL_\infty$ algebra in \cite{cieliebak2015homological} and the homotopic version of bi-Lie algebras, with (co)-curvature. The algebraic relations in $BL_\infty$ algebra are precisely the boundary combinatorics of the moduli spaces of rational curves in the SFT compactification. We then introduce algebraic planar torsion and planarity at the algebraic level.

\subsection{$L_\infty$ algebras}\label{SS:L}
Throughout this section, we assume $\bk$ is a field with characteristic $0$ for simplicity, although the discussion works for any commutative ring.  Let $V$ be a $\Z_2$-graded $\bk$-vector space. Then we have the $\Z_2$-graded symmetric algebra $S V:=\bigoplus_{k\ge 0} S^k V$ and the non-unital symmetric algebra $\overline{S} V =\bigoplus_{k\ge 1} S^k V$, where $S^kV=\otimes^k V/Sym_k$ in the graded sense.  In particular, we have
$$a b=(-1)^{|a||b|} b a$$
for homogeneous elements $a,b$ in $SV,\overline{S}V$. Therefore $S^kV$ is spanned by vectors of the form $v_1 \ldots  v_k$ for $v_i\in V$. However, to introduce the $L_\infty$ algebra, we will view $SV,\overline{S}V$ as coalgebras by the following co-product operation:
$$\Delta(v_1 \ldots  v_k)=\sum_{i=1}^{k-1}\sum_{\sigma \in Sh(i,k-i)}(-1)^\diamond(v_{\sigma(1)} \ldots  v_{\sigma(i)})\otimes (v_{\sigma(i+1)} \ldots  v_{\sigma(k)}),$$
where $Sh(i,k-i)$ is the subset of permutations $\sigma$ such that $\sigma(1)<\ldots<\sigma(i)$ and $\sigma(i+1)<\ldots < \sigma(k)$ and 
$$\diamond=\sum_{\substack {1\le i < j \le k\\ \sigma(i)>\sigma(j)}}|v_i||v_j|.$$
Then both $S V$ and $\overline{S}V$ satisfy the coassociativity property $(\Id\otimes \Delta)\circ \Delta= (\Delta \otimes \Id)\circ \Delta$, and the cocommutativity property $\mathbf{R}\circ \Delta=\Delta$, where $\mathbf{R}: S V \otimes S V \rightarrow S V \otimes S V$ is given by $\mathbf{R}(x \otimes y)=(-1)^{|x||y|}y \otimes x$ for homogeneous elements $x,y$. A coderivation of the coalgebra $(\overline{S}V,\Delta)$ is a $\bk$-linear map $\delta: \overline{S}V \rightarrow \overline{S}V$ satisfying the coLeibniz rule $\Delta \circ \delta= (\delta \otimes \Id)\circ \Delta+ (\Id \otimes \delta)\circ \Delta $. Here we use the Koszul-Quillen sign convention that $(f\otimes g)(x\otimes y) = (-1)^{|x||g|}f(x)\otimes g(y)$, for $x,y \in V, W$ and $f,g \in V^{\vee},W^{\vee}$.

\begin{definition}\label{def:L_alg}
	We use $sV$ to denote $V[1]$. An $L_\infty$ algebra structure on $V$ is a degree $1$ coderivation $\widehat{\ell}$ on $\overline{S}sV$ satisfying $\widehat{\ell}^2=0$.
\end{definition}

Note that we have a well-defined degree $-1$ map $s:V\to sV$. The coderivation property of $\widehat{\ell}$ implies that it is determined by maps $\ell^k:S^ksV \to sV$ defined by the composition $S^k sV \hookrightarrow \overline{S}sV \overset{\widehat{\ell}}{\rightarrow} \overline{S}sV \rightarrow sV$, where the first map is the natural inclusion and the last map is the natural projection, satisfying the quadratic relation
\begin{equation}\label{eqn:L_infinity}
\sum_{k=1}^n \sum_{\sigma \in Sh(k,n-k)}(-1)^{\diamond'} \ell^{n-k+1}(\ell^k(sv_{\sigma(1)} \ldots  sv_{\sigma(k)}) sv_{\sigma(k+1)} \ldots  sv_{\sigma(n)})=0,
\end{equation}
where 
$$\diamond'=\sum_{\substack {1\le i < j \le k\\ \sigma(i)>\sigma(j)}}(|v_i|-1)(|v_j|-1).$$
$(\overline{S}sV,\widehat{\ell})$ is called the reduced bar complex. The word length filtration $\overline{B}^1sV \subset \overline{B}^2sV\subset \ldots \subset \overline{S}sV$ is compatible with the differential, where $\overline{B}^ksV:=\bigoplus_{j=1}^ kS^j sV$.

\begin{definition}\label{def:L_mor}
	An $L_\infty$ homomorphism from $(V,\ell)$ to $(V',\ell')$ is a degree $0$ coalgebra map $\widehat{\phi}:\overline{S}sV\to \overline{S}sV'$ such that $\widehat{\phi}\circ \widehat{\ell} = \widehat{\ell}'\circ \widehat{\phi}$.
\end{definition}
Given $\phi_i:S^{k_i} sV \to sV', 1\le i \le n$ and $m=\sum_{i=1}^n k_i$, we define $\phi_{1} \ldots  \phi_n:S^{m}sV \to S^n sV'$ by sending $sv_1 \ldots  sv_{m}$ to 
$$\pi \left(\sum_{\sigma} \frac{(-1)^{\diamond'}}{k_1!\ldots k_n!} (\phi_{1}\otimes \ldots \otimes \phi_n)((sv_{\sigma(1)} \ldots  sv_{\sigma(k_1)})\otimes \ldots \otimes (sv_{\sigma(m-k_n+1)} \ldots  sv_{\sigma(m)}))\right).\footnote{The factorial here is a consequence of the redundant summing over all permutations. However, the construction works any coefficient ring, this is easier to see using the tree description in \S \ref{ss:tree}.}$$
Here $\pi$ is the natural map $\otimes^k sV \to S^k sV$. By the coalgebra property, if $\widehat{\phi}$ is an $L_\infty$ morphism, we know that $\widehat{\phi}$ is determined by $\{\phi^k:S^ksV \to \overline{S} sV \stackrel{\widehat{\phi}}{\to} \overline{S}sV'\to sV'\}_{k\ge 1}$. More explicitly, $\widehat{\phi}$ is defined by the following formula,
$$\widehat{\phi}(sv_1\ldots sv_n)=\sum_{\substack{k\ge 1\\i_1+\ldots+i_k=n}}\frac{1}{k!}(\phi^{i_1}\ldots\phi^{i_k})(sv_{1}\ldots sv_{n}).$$
The relation $\widehat{\phi}\circ \widehat{\ell}=\widehat{\ell'}\circ \widehat{\phi}$ can be written as
\begin{align*} \sum_{p+q=n+1}\sum_{\sigma\in Sh(q,n-q)} (-1)^{\diamond'}\phi^{p}(\ell^q(sv_{\sigma(1)} &  \ldots sv_{\sigma(q)})  sv_{\sigma(q+1)} \ldots  sv_{\sigma(n)})= \\ &
\sum_{\substack{k\ge 1\\i_1+\ldots+i_k=n}} \frac{1}{k!}\ell^k(\phi^{i_1}\ldots\phi^{i_k})(sv_{1}\ldots sv_{n}).
\end{align*}
In particular, $\widehat{\phi}$ preserves the word length filtration. The composition of $L_\infty$ homomorphism is the naive composition $\widehat{\phi}\circ \widehat{\psi}$, which is clearly a coalgebra chain map. Unwrapping the definition, we have
$$(\phi\circ \psi)^n=\sum_{\substack{k\ge 1\\i_1+\ldots+i_k=n}}\frac{1}{k!}\phi^k((\psi^{i_1}\ldots\psi^{i_k})(sv_{1}\ldots sv_{n})).$$

\subsection{$BL_\infty$ algebras}\label{ss:BL}
In this section, we define the $BL_\infty$ (bi-Lie-infinity) algebra structure on a $\Z_2$ graded vector space $V$, which will govern the rational symplectic field theory. Let $EV$ denote $\overline{S}SV$. Given a linear operator $p^{k,l}:S^kV \to S^l V$ for $k\ge 1, l\ge 0$, we will define a map $\widehat{p}^{k,l}:S^k SV \to SV$. To emphasize the differences between products on two symmetric algebras, we use $\odot$ for the product on the outside symmetric product $\overline{S}$ and $\ast$ for the product on the inside symmetric product $S$ when it can not be abbreviated. We will first describe the definition using formulas and then introduce a graph description, which is very convenient to describe $BL_\infty$ algebras as well as various related structures and also governs all the signs and coefficients. Let $w_1,\ldots, w_k\in SV$, then $\widehat{p}^{k,l}$ is defined by the following properties.
\begin{enumerate}
	\item $\widehat{p}^{k,l}|_{\odot^k V\subset S^kSV}$ is defined by $p^{k,l}$.
	\item If $w_i\in \bk$, then $\widehat{p}^{k,l}(w_1\odot \ldots \odot w_k)=0$.
	\item\label{leb} $\widehat{p}^{k,l}$ satisfies the Leibniz rule in each argument, i.e.\ we have 
\begin{equation}\label{eq:p_hat}
 \widehat{p}^{k,l}(w_1\odot \ldots \odot w_k)=\sum_{j=1}^m(-1)^{\square} v_1 \ldots  v_{j-1} \widehat{p}^{k,l}(w_1\odot \ldots \odot v_j \odot \ldots \odot w_k) v_{j+1} \ldots  v_m.
 \end{equation}
	Here $w_i=v_1 \ldots  v_m$ and 
\begin{equation}\label{eqn:sign}
    \square = \sum_{s=1}^{i-1}|w_s|\cdot \sum_{s=1}^{j-1}|v_s|+\sum_{s=1}^{j-1}|v_s||p^{k,l}|+\sum_{s=i+1}^n |w_s|\cdot \sum_{s=j+1}^{m}|v_s|.
\end{equation}
\end{enumerate}
It is clear from the definition that $\widehat{p}^{k,l}$ is determined uniquely by the above three conditions. More explicitly, $\widehat{p}^{k,l}$ is defined by the following:
\begin{equation}\label{eq:p_hat_1}
    w_1\odot \ldots \odot w_k \mapsto \sum_{\substack{(i_1,\ldots,i_k) \\ 1\le i_j\le n_j}} (-1)^{\bigcirc} p^{k,l}(v^1_{i_1}\odot\ldots \odot v^k_{i_k}) \check{w}_1 \ldots \check{w}_k,
\end{equation}
where $w_j=v^j_1 \ldots  v^j_{n_j}$, $\check{w}_j=v^j_1 \ldots  \check{v}^j_{i_j} \ldots  v^j_{n_j}$ and $w_1 \ldots  w_k=(-1)^\bigcirc v^1_{i_1} \ldots  v^k_{i_k} \check{w}_1 \ldots  \check{w}_k$. Then we define $\widehat{p}^k:S^k S V \to S V$ by $\bigoplus_{l\ge 0} \widehat{p}^{k,l}$. To assure it is well-defined, we need to assume for any $v_1,\ldots,v_k\in V$, there are at most finitely many $l$ such that $p^{k,l}(v_1\odot\ldots\odot v_k)\ne 0$. Then we can define $\widehat{p}:EV \to EV$ by
\begin{equation}\label{eq:p_hat2}
    w_1\odot \ldots \odot w_n \mapsto \sum_{k=1}^n\sum_{\sigma \in Sh(k,n-k)}(-1)^{\diamond} \widehat{p}^k(w_{\sigma(1)}\odot \ldots \odot w_{\sigma(k)})\odot w_{\sigma(k+1)}\odot \ldots \odot w_{\sigma(n)},
\end{equation}
i.e.\ following the same rule of $\widehat{\ell}$ from $\ell^k$.  
\begin{definition}\label{def:BL}
$(V,\{p^{k,l}\}_{k\ge 1, l \ge 0})$ is a $BL_\infty$ algebra if $\widehat{p}\circ \widehat{p}=0$ and $|\widehat{p}|=1$.
\end{definition}
To explain the terminology, assume $p^{1,0}=0,p^{2,0}=0$. Then $p^{1,1}$ defines a differential on $V$, such that $p^{2,1}$ defines a Lie bracket on the homology of $(V, p^{1,1})$ and $p^{1,2}$ defines a Lie cobracket on the homology. The compatibility is that $p^{1,2}\circ p^{2,1}=0$ on the homology level. The main difference with the $IBL_\infty$ algebra \cite[Definition 2.3]{cieliebak2015homological} is that we will not consider the compatibility condition on $p^{2,1}\circ p^{1,2}=0$, which will increase genus\footnote{The other difference is that the $IBL_\infty$ algebra in \cite{cieliebak2015homological} describes the algebra for linearized SFT, where $p^{k,0,g}=0$ for any number of positive punctures $k$ and genus $g$.}. A direct consequence of the definition is that $(SV, \widehat{p}^1)$ is a chain complex and the $\widehat{p}^k$ define an $L_\infty$ structure on $(SV)[-1]$. As noted in \cite[Remark 3.12]{siegel2019higher}, $SV$ carries a natural commutative algebra structure, the Leibniz rule in the definition of $\widehat{p}^{k,l}$ implies the $L_\infty$ structure is compatible with the algebra structure, and $(SV)[-1]$ should be some version of a $G_\infty$ algebra.  Definition \ref{def:BL} can be viewed as one method of making the compatibility precise.

\begin{remark}[Shift v.s.\ no shift]
    The degree shift in \S \ref{SS:L} is the classical sign convention introduced by Stasheff \cite[Page 133]{MR1183483}, as $L_\infty$ algebra is a higher generalization of Lie algebra, where the Lie bracket is skew-symmetric and has degree $0$. The $IBL_\infty$ formalism in \cite{cieliebak2015homological} kept such tradition of shifting degrees by $1$, as a result, in the SFT context, generators are graded by the SFT degree shifted by $1$ \cite[\S 7]{cieliebak2015homological} to cancel the shift in the definition. Here, we choose to drop the degree shift in Definition \ref{def:BL}, so that our generators will be graded by the SFT degree in the context of SFT, since operations from counting holomorphic curves are naturally super-symmetric w.r.t.\ the SFT degrees due to the orientation scheme in SFT. As a consequence, $((SV)[-1],\{\widehat{p}^k\})$ and $(V[-1],\{p^{k,1}\})$ (assuming all other $p^{k,l}$ is zero in the latter case) are $L_\infty$ algebras in the sense of Definition \ref{def:L_alg}. If we use a degree shifted version of Definition \ref{def:BL}, $(V,\{p^{k,1}\})$ (under the same vanishing assumption) is an $L_\infty$ algebra, while neither $SV$ nor $SsV$ are $L_\infty$ algebras, i.e.\ extra shift is inevitable. All shifts can go away if one is willing to adopt a version of $L_\infty$ algebra without degree shift.  
\end{remark}

\begin{remark}
	$BL_\infty$ algebra is not a ``direct" specialization of the $IBL_\infty$ algebra as introduced in \cite{cieliebak2015homological}. However, there is an equivalent reformulation of the $IBL_\infty$ relations\footnote{In the special case of setting $\tau=1$ in \cite[Definition 2.3]{cieliebak2015homological}}, from which one can see that an $IBL_\infty$ algebra contains a $BL_\infty$ algebra, as well as algebras with any genus upper bound, see \cite[\S 5.2, Proposition 5.10, Corollary 5.12]{supp} for details.
\end{remark}

\subsection{The rules for tree calculus}\label{ss:tree}
A useful way to explain the combinatorics of operations is the following description using graphs, which appeared in \cite[\S 3.4.2]{siegel2019higher}. The combinatorics is also relevant in the virtual technique setup, see \S \ref{s:virtual}. The main advantage of such graphical language is freeing us from the book-keeping of signs and explicit components of compositions, e.g.\ in \eqref{eq:p_hat}, which are governed by graphs. 

Let $w\in S^kV$, we can represent $w$ by an element $\overline{w}$ in $\otimes^k V$, that is $\overline{w}=\sum_{i=1}^N c_iv^i_1\otimes \ldots \otimes v^i_k$ for $c_i\in \bk$ and $v_*^*\in V$, such that $\pi(\overline{w})=w$ for $\pi:\otimes^kV \to S^kV$. We represent it by a rooted tree with $k$ leaves (represented by $\bullet$) labeled by $\overline{w}$. The leaves are ordered from left to right to indicate the $k$ copies of $V$ in $\otimes^k V$. When $\overline{\omega} = v_1\otimes \ldots \otimes v_k$, we may label the leaves by $v_1,\ldots,v_k$ to mean the same thing. We can view a general labeled tree as a formal linear combination of such trees with leaves labeled. 
    	\begin{center}
        \begin{tikzpicture}
        \node at (0,0) [circle,fill,inner sep=1.5pt] {};
		\node at (1,0) [circle,fill,inner sep=1.5pt] {};
		\node at (2,0) [circle,fill,inner sep=1.5pt] {};
		\draw (0,0) to (1,1) to (1,0);
		\draw (1,1) to (2,0);
        \node at (2,1) {$\overline{w}\in \otimes^3 V$};
        \end{tikzpicture}
        \qquad
		\begin{tikzpicture}
		\node at (0,0) [circle,fill,inner sep=1.5pt] {};
        \node at (0.3,0) {$v_1$};
		\node at (1,0) [circle,fill,inner sep=1.5pt] {};
        \node at (1.3,0) {$v_2$};
		\node at (2,0) [circle,fill,inner sep=1.5pt] {};
        \node at (2.3,0) {$v_3$};
		\draw (0,0) to (1,1) to (1,0);
		\draw (1,1) to (2,0);
		\end{tikzpicture}
	\end{center}
Now let $s\in S^kSV$, we can represent $s$ by $\overline{s}\in \boxtimes^k TV$, where $TV=\oplus_{k\in \N}(\otimes^k V)$. Here we use $\boxtimes$ to differentiate it from the inner tensor $\otimes$. We write
$$\overline{s}=\sum_{i=1}^N c_i\overline{w}^i_1 \boxtimes \ldots  \boxtimes \overline{w}^i_k, \quad c_i\in \bk, \overline{w}_*^*\in \otimes^{m^*_*}V.$$
We represent $\overline{w}^i_1 \boxtimes\ldots  \boxtimes \overline{w}^i_k$ by an ordered forest of labeled trees as follows. Then $\overline{s}$ is a formal linear combination of such forests.
\begin{figure}[H]
    \begin{center}
		\begin{tikzpicture}
		\node at (0,0) [circle,fill,inner sep=1.5pt] {};
		\node at (1,0) [circle,fill,inner sep=1.5pt] {};
		\node at (2,0) [circle,fill,inner sep=1.5pt] {};
		\node at (3,0) [circle,fill,inner sep=1.5pt] {};
		\node at (4,0) [circle,fill,inner sep=1.5pt] {};
		\node at (5,0) [circle,fill,inner sep=1.5pt] {};
		\node at (6,0) [circle,fill,inner sep=1.5pt] {};
		\node at (7,0) [circle,fill,inner sep=1.5pt] {};
		
		\draw (0,0) to (1,1) to (1,0);
		\draw (1,1) to (2,0);
		\draw (3,0) to (4,1) to (4,0);
		\draw (4,1) to (5,0);
		\draw (6,0) to (6.5,1) to (7,0);

       \node at (2,1) {$\overline{w}_1\in \otimes^3V$};
       \node at (5,1) {$\overline{w}_2\in \otimes^3V$};
       \node at (7.5, 1) {$\overline{w}_3\in \otimes^2V$};
	\end{tikzpicture}
	\end{center}
    \caption{forest of labeled trees}
    \label{fig:forest}
\end{figure}
We represent the operation $p^{k,l}:S^kV\to S^lV$ by a graph with $k+l+1$ vertices,  $k$ top input vertices, $l$ bottom output vertices and one middle vertex $\tikz\draw[black,fill=white] (0,-1) circle (0.4em);$ labeled by $p^{k,l}$ representing the operation type. 
    \begin{center}
        \begin{tikzpicture}
        \node at (2,0) [circle,fill,inner sep=1.5pt] {};
		\node at (3,0) [circle,fill,inner sep=1.5pt] {};
        \draw (2,0) to (2.5,-1) to (3,0);
		\draw (2,-2) to (2.5,-1) to (2.5,-2);
		\draw (2.5,-1) to (3,-2);
		\node at (2.5,-1) [circle, fill=white, draw, outer sep=0pt, inner sep=3 pt] {};
        \node at (2,-2) [circle,fill,inner sep=1.5pt] {};
	    \node at (3,-2) [circle,fill,inner sep=1.5pt] {};
	    \node at (2.5,-2) [circle,fill,inner sep=1.5pt] {};
        \node at (3,-1) {$p^{2,3}$};
        \end{tikzpicture}
    \end{center}
So far the discussion is completely formal without any actual content, the real content is in the following interpretation of a glued graph, whose definition will be clear from one example.
\begin{figure}[H]
	\begin{center}
		\begin{tikzpicture}
		\node at (0,0) [circle,fill,inner sep=1.5pt] {};
		\node at (1,0) [circle,fill,inner sep=1.5pt] {};
		\node at (2,0) [circle,fill,inner sep=1.5pt] {};
		\node at (3,0) [circle,fill,inner sep=1.5pt] {};
		\node at (4,0) [circle,fill,inner sep=1.5pt] {};
		\node at (5,0) [circle,fill,inner sep=1.5pt] {};
		\node at (6,0) [circle,fill,inner sep=1.5pt] {};
		\node at (7,0) [circle,fill,inner sep=1.5pt] {};
		
		\draw (0,0) to (1,1) to (1,0);
		\draw (1,1) to (2,0);
		\draw (3,0) to (4,1) to (4,0);
		\draw (4,1) to (5,0);
		\draw (6,0) to (6.5,1) to (7,0);

		\draw (2,0) to (2.5,-1) to (3,0);
		\draw (2,-2) to (2.5,-1) to (2.5,-2);
		\draw (2.5,-1) to (3,-2);
		\node at (2.5,-1) [circle, fill=white, draw, outer sep=0pt, inner sep=3 pt] {};
		
		\node at (0,-2) [circle,fill,inner sep=1.5pt] {};
		\node at (1,-2) [circle,fill,inner sep=1.5pt] {};
		\node at (2,-2) [circle,fill,inner sep=1.5pt] {};
		\node at (3,-2) [circle,fill,inner sep=1.5pt] {};
		\node at (4,-2) [circle,fill,inner sep=1.5pt] {};
		\node at (5,-2) [circle,fill,inner sep=1.5pt] {};
		\node at (6,-2) [circle,fill,inner sep=1.5pt] {};
		\node at (7,-2) [circle,fill,inner sep=1.5pt] {};
		\node at (2.5,-2) [circle,fill,inner sep=1.5pt] {};
		
		\draw[dashed] (0,0) to (0,-2);
		\draw[dashed] (1,0) to (1,-2);
		\draw[dashed] (4,0) to (4,-2);
		\draw[dashed] (5,0) to (5,-2);
		\draw[dashed] (6,0) to (6,-2);
		\draw[dashed] (7,0) to (7,-2);
		\node at (0.4,0) {$v_1$};
		\node at (1.4,0) {$v_2$};
		\node at (2.4,0) {$v_3$};
		\node at (3.4,0) {$v_4$};
		\node at (4.4,0) {$v_5$};
		\node at (5.4,0) {$v_6$};
		\node at (6.4,0) {$v_7$};
		\node at (7.4,0) {$v_8$};
		\end{tikzpicture}
	\end{center}
    \caption{Gluing forests $\Leftrightarrow$ applying operations}
    \label{fig:gluing}
\end{figure}
The above glued graph represents a forest: we first fix a representative $\overline{p}$ of $p^{2,3}(v_3v_4)$ in $\otimes^3V$. The glued forest in Figure \ref{fig:gluing} represents $\pm (v_1\otimes v_2 \otimes \overline{p}\otimes v_5\otimes v_6)\boxtimes (v_7\otimes v_8)$. In the gluing, we do not create cycles in the glued graph, each dashed line represents the identity map,  and each connected component represents a tree in the output. \textbf{The trees and forests are considered as abstract trees and forests, the inclusion into the plane/space is not part of the information. Drawing the input element as a forest of ordered trees with ordered leaves means that we are choosing representatives from the tensor product, not the symmetric product.} Finally, when we draw the glued graph as above, i.e.\ choosing an order of the trees and leaves (hence edges will cross over each other if we draw it on a plane),  will determine a representative in the tensor product, i.e.\ we view different orders as equivalent up to the obvious sign change. For example, the following is an equivalent gluing as that in Figure \ref{fig:gluing} but with an extra sign when viewing it in the tensor product. The extra sign is $(-1)^{|v_5||b_3|}$ if the representative $\overline{p}$ is $b_1\otimes b_2\otimes b_3$.
\begin{figure}[H]
	\begin{center}
		\begin{tikzpicture}
		\node at (0,0) [circle,fill,inner sep=1.5pt] {};
		\node at (1,0) [circle,fill,inner sep=1.5pt] {};
		\node at (2,0) [circle,fill,inner sep=1.5pt] {};
		\node at (3,0) [circle,fill,inner sep=1.5pt] {};
		\node at (4,0) [circle,fill,inner sep=1.5pt] {};
		\node at (5,0) [circle,fill,inner sep=1.5pt] {};
		\node at (6,0) [circle,fill,inner sep=1.5pt] {};
		\node at (7,0) [circle,fill,inner sep=1.5pt] {};
		
		\draw (0,0) to (1,1) to (1,0);
		\draw (1,1) to (2,0);
		\draw (3,0) to (4,1) to (4,0);
		\draw (4,1) to (5,0);
		\draw (6,0) to (6.5,1) to (7,0);

		\draw (2,0) to (2.5,-1) to (3,0);
		\draw (2,-2) to (2.5,-1) to (2.5,-2);
		\draw (2.5,-1) to (4.5,-2);
		\node at (2.5,-1) [circle, fill=white, draw, outer sep=0pt, inner sep=3 pt] {};
		
		\node at (0,-2) [circle,fill,inner sep=1.5pt] {};
		\node at (1,-2) [circle,fill,inner sep=1.5pt] {};
		\node at (2,-2) [circle,fill,inner sep=1.5pt] {};
		\node at (4.5,-2) [circle,fill,inner sep=1.5pt] {};
		\node at (4,-2) [circle,fill,inner sep=1.5pt] {};
		\node at (5,-2) [circle,fill,inner sep=1.5pt] {};
		\node at (6,-2) [circle,fill,inner sep=1.5pt] {};
		\node at (7,-2) [circle,fill,inner sep=1.5pt] {};
		\node at (2.5,-2) [circle,fill,inner sep=1.5pt] {};
		
		\draw[dashed] (0,0) to (0,-2);
		\draw[dashed] (1,0) to (1,-2);
		\draw[dashed] (4,0) to (4,-2);
		\draw[dashed] (5,0) to (5,-2);
		\draw[dashed] (6,0) to (6,-2);
		\draw[dashed] (7,0) to (7,-2);
		\node at (0.4,0) {$v_1$};
		\node at (1.4,0) {$v_2$};
		\node at (2.4,0) {$v_3$};
		\node at (3.4,0) {$v_4$};
		\node at (4.4,0) {$v_5$};
		\node at (5.4,0) {$v_6$};
		\node at (6.4,0) {$v_7$};
		\node at (7.4,0) {$v_8$};
		\end{tikzpicture}
	\end{center}
    \caption{If we switch the output order, it still represents the same glued forest as Figure \ref{fig:gluing}.}
    \label{fig:switch}
\end{figure}
The sign is determined similarly to \eqref{eqn:sign}, in this case of Figure \ref{fig:gluing}, the sign is $(-1)^{(|v_1|+|v_2|)|p^{2,3}|}$. In a formal description,  we apply order changes to the input forest (edges can cross when it is drawn in a row), then we glue $p^{k,l}$, such that there is no new edge crossing (this corresponds to that $p^{k,l}$ acts on  $k$ consecutive leaves), finally, we change the output order back to the chosen one (e.g.\ an order prescribed in Figure \ref{fig:switch}), the final sign is given by the product of the sign changes of the two order changes and the sign of the composition using the Koszul-Quillen convention. 
\begin{example}\label{ex:sign}
In the following, we work out an explicit example which involves the features explained above. We consider the following glued graph representing a component of $p^{2,3}$ acting on $v_1v_2v_3\odot v_4v_5v_6\odot v_7v_8$.
\begin{figure}[H]
\begin{overpic}[scale=0.8]
{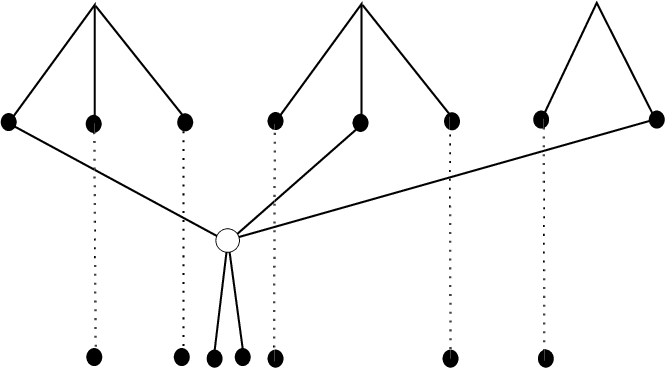}
\put (-4,36) {$v_1$}
\put (8,36) {$v_2$}
\put (22,36) {$v_3$}
\put (35,36) {$v_4$}
\put (48,36) {$v_5$}
\put (62,36) {$v_6$}
\put (75,36) {$v_7$}
\put (93,36) {$v_8$}
\put (33,23) {$p^{2,3}$}
\end{overpic}
\end{figure}
\noindent
We first switch the input leaves of the input forest, so that the insertion of $p^{2,3}$ will not create crossings at the input edges of $p^{2,3}$ as follows, this will pick up a sign $(-1)^{|v_5|(|v_2|+|v_3|+|v_4|)+|v_8|(|v_2|+|v_3|+|v_4|+|v_6|+|v_7|)}$.
\begin{figure}[H]
\begin{overpic}[scale=0.8]
{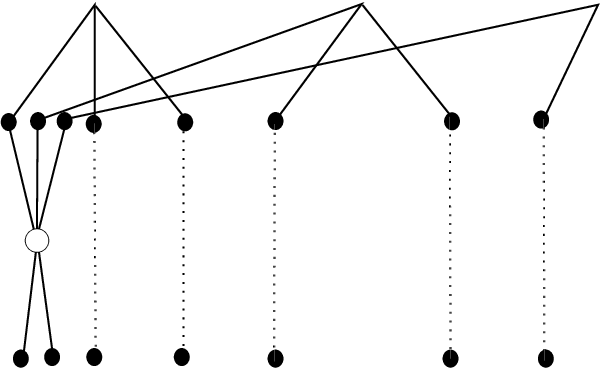}
\put (-4,36) {$v_1$}
\put (4,36) {$v_5$}
\put (9,36) {$v_8$}
\put (18,36) {$v_2$}
\put (30,36) {$v_3$}
\put (46,36) {$v_4$}
\put (75,36) {$v_6$}
\put (93,36) {$v_7$}
\put (-4,21) {$p^{2,3}$}
\end{overpic}
\end{figure}
\noindent
Now the composition with $p^{2,3}$ has no sign from the  Koszul-Quillen convention, the output is a forest, which is a single tree here, representing 
\begin{equation}\label{eqn:comp1}
    (-1)^{|v_5|(|v_2|+|v_3|+|v_4|)+|v_8|(|v_2|+|v_3|+|v_4|+|v_6|+|v_7|)}p^{2,3}(v_1v_5v_8)v_2v_3v_4v_6v_7.
\end{equation}
Alternatively, we can do the following switch of the input leaves, resulting in a sign change by $$(-1)^{|v_1|(|v_2|+|v_3|+|v_4|)+|v_8|(|v_6|+|v_7|)}.$$
\begin{figure}[H]
\begin{overpic}[scale=0.8]
{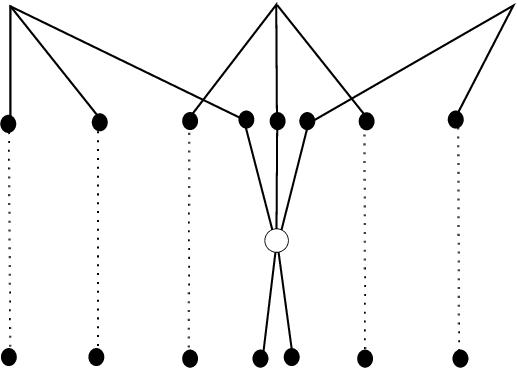}
\put (-2,43) {$v_2$}
\put (53,43) {$v_5$}
\put (60,43) {$v_8$}
\put (18,43) {$v_3$}
\put (32,43) {$v_4$}
\put (43,43) {$v_1$}
\put (70,43) {$v_6$}
\put (90,43) {$v_7$}
\put (56,21) {$p^{2,3}$}
\end{overpic}
\end{figure}
\noindent
Now the composition will pick up a sign $(-1)^{|p^{2,3}|(|v_2|+|v_3|+|v_4|)}$ by the Koszul-Quillen convention, so the output is 
$$(-1)^{|v_1|(|v_2|+|v_3|+|v_4|)+|v_8|(|v_6|+|v_7|)+|p^{2,3}|(|v_2|+|v_3|+|v_4|)}v_2v_3v_4 p^{2,3}(v_1v_5v_8)v_6v_7,$$
which is 
$$(-1)^{|v_1|(|v_2|+|v_3|+|v_4|)+|v_8|(|v_6|+|v_7|)+|p^{2,3}|(|v_2|+|v_3|+|v_4|)+(|p^{2,3}|+|v_1|+|v_5|+|v_8|)(|v_2|+|v_3|+|v_4|)}p^{2,3}(v_1v_5v_8)v_2v_3v_4v_6v_7.$$
This expression is exactly \eqref{eqn:comp1}, as there are sign cancelations. And indeed they are the same glued tree/forest, as the different ordering is just choosing a different representative in the tensor product. There is more freedom in terms of choosing representatives, e.g.\ switching the leaves in the input for $p^{k,l}$, and switching as in Figure \eqref{fig:switch}. They all result in choosing different representatives in the same equivalence class of the same glued forests. We can also have part of the forest not interacting with $p^{k,l}$ (but after choosing an order, it can have crossings with the part interacting with $p^{k,l}$), but this will not change the discussion. In other words, since everything is graded and super-commutative, the application of the Koszul-Quillen convention guarantees everything is well-defined as equivalence classes in the symmetric product.
\end{example}

Writing the forest using a glued graph as in Figure \ref{fig:gluing} contains slightly more refined information than just labeling the forest as in Figure \ref{fig:forest}, namely we keep track of which leaves are from $p^{k,l}$ in a representative. From the discussion in Eample \ref{ex:sign}, the following observation is tautological.
\begin{proposition}
    The output of a glued forest is well-defined in $EV$.
\end{proposition}
To enumerate through all admissible gluings, each output leaf and tree are considered as different.  However, we do not differentiate the input leaves of $p^{k,l}$. Therefore when we glue a $p^{k,l}$ component, we pick $k$ trees (this is $Sh(k,n-k)$ in \eqref{eq:p_hat2} ) from the forest and then one leaf from each chosen tree (that is $1\le i_j\le n_j$ in \eqref{eq:p_hat_1}) to glue to $p^{k,l}$. For example, in the situation of Figure \ref{fig:gluing}, we have $3*3+3*2+3*2=21$ direct ways to glue $p^{2,3}$. The ambiguity from choosing a representative of the input is then eliminated from summing over all possible gluing by the following tautological observation. 
\begin{proposition}
    When summed over all possible gluing of one $p^{k,l}$, the output is independent of the choice of representatives of the input forest.
\end{proposition}
Combining the above two propositions, we see that gluing forests corresponds to operations on $EV$.  Indeed, the language of trees and forests is just packaging the signs and components in \eqref{eq:p_hat}, \eqref{eq:p_hat_1} and \eqref{eq:p_hat2} by providing a geometric intuition. The translation into forests makes it easier to understand algebraic relations, for example, many relations come from interpreting the same glued forests in two different ways. The following is the dictionary of the algebraic formulae in \S \ref{ss:BL} in terms of forests, where the signs can be compared directly from the rule convention before Example \ref{ex:sign} and signs in \S \ref{ss:BL}.
\begin{center}
\begin{tabular}{c|c}
  $\widehat{p}^{k,l}$ on $\odot^kV$ & \makecell{(unique) gluing of $p^{k,l}$ to a  \\ forest of $k$ trees of single leaf} \\
  \hline
  $\widehat{p}^{k,l}$ in \eqref{eq:p_hat_1}  & \makecell{sum of gluing of $p^{k,l}$ to a \\ forest of $k$ trees to get a tree}    \\
  \hline
  $\widehat{p}^k$  & \makecell{sum of gluing of $p^{k,*}$ to a \\ forest of $k$ trees to get a tree}  \\
  \hline
   $\widehat{p}$ in  \eqref{eq:p_hat2}  & \makecell{sum of gluing of $p^{*,*}$ to a \\ forest to get a forest} \\
   \hline
   $\widehat{p}^{k,l}(w_1\odot \ldots \odot w_k)=0$ if $w_i\in \bk$ & \makecell{no way to glue $p^{k,l}$ to a forest \\with a tree without leaf to get a  tree}\\
\end{tabular}
\end{center}
We use $p^{k,l}_2:S^kV \to S^lV$ for $k\ge 1, l\ge 0$ to denote the sum of all \emph{connected} graphs with two levels of $\tikz\draw[black,fill=white] (0,-1) circle (0.4em);$ vertices, $k$ input vertices and $l$ output vertices as follows. 
	\begin{center}
		\begin{tikzpicture}
		\node at (0,0) [circle,fill,inner sep=1.5pt] {};
		\node at (1,0) [circle,fill,inner sep=1.5pt] {};
		\node at (2,0) [circle,fill,inner sep=1.5pt] {};
		\draw (0,0) to (0.5,-1) to (1,0);
		\draw (1,-4) to (1.5,-3) to (2,-4);
        \draw (0,-2) to (0.5,-1) to (1.5,-3) to (2,-2);
        \draw[dashed](0,-2) to (0,-4);
        \draw[dashed] (2,0) to (2,-2);
		\node at (0.5,-1) [circle, fill=white, draw, outer sep=0pt, inner sep=3 pt] {};
		\node at (1.5,-3) [circle, fill=white, draw, outer sep=0pt, inner sep=3 pt] {};
		\node at (0,-2) [circle,fill,inner sep=1.5pt] {};
		\node at (1,-2) [circle,fill,inner sep=1.5pt] {};
		\node at (2,-2) [circle,fill,inner sep=1.5pt] {};
        \node at (0,-4) [circle,fill,inner sep=1.5pt] {};
		\node at (1,-4) [circle,fill,inner sep=1.5pt] {};
		\node at (2,-4) [circle,fill,inner sep=1.5pt] {};
		\end{tikzpicture}
	\end{center}
In terms of a formula, we have that $p^{k,l}_2= \pi_{1,l}\circ \widehat{p}^2|_{\odot^k V}$, where $\pi_{1,l}$ denotes the projection $EV\to S^1SV \to S^lV$. This follows from that elements in $\odot^k V$ are represented by forests consisting of trees of a single leaf, to get a single tree after applying $\widehat{p}^2$, we must have the two $p^{*,*}$ components connected to each other directly. Note that, in the applications we have in mind, i.e.\ rational SFT, $p^{k,l}_2$ can be viewed as the codimension $1$ boundary of the rational SFT moduli space. The following proposition shows that the $BL_\infty$ algebra structure captures exactly such combinatorics.
\begin{proposition}\label{prop:2level}
	$\{p^{k,l}\}_{k\ge 1,\l\ge 0}$ forms a $BL_\infty$ algebra if and only if  $p^{k,l}_2=0$ for $k\ge 1, l\ge 0$. 
\end{proposition}
\begin{proof}
	Since $p^{k,l}_2= \pi_{1,l}\circ \widehat{p}^2|_{\odot^k V}$, if $\{p^{k,l}\}_{k\ge 1,\l\ge 0}$ forms a $BL_\infty$ algebra then $p^{k,l}_2=0$ for $k\ge 1, l\ge 0$. Now assume  $p^{k,l}_2=0$ for $k\ge 1, l\ge 0$. In the glued forests representing $\widehat{p}^2$ acting on an element in $EV$, there are two cases: (1) the two $p^{*,*}$-components are glued to each other directly, those are zero because  $p^{k,l}_2=0$; (2) the two $p^{*,*}$ components are not glued to each other (but they could be in the same tree after gluing), by switching the levels of those two $p^{*,*}$-components, we see that they pair up and cancel with each other as $|p^{*,*}|=1$.
	\begin{center}
		\begin{tikzpicture}
		\node at (0,0) [circle,fill,inner sep=1.5pt] {};
		\node at (1,0) [circle,fill,inner sep=1.5pt] {};
		\node at (2,0) [circle,fill,inner sep=1.5pt] {};
        \node at (3,0) [circle,fill,inner sep=1.5pt] {};
		\draw (0,0) to (0.5,-1) to (1,0);
        \draw (0,-2) to (0.5,-1) to (1,-2);
		\draw (2,-4) to (2.5,-3) to (3,-4);
  	\draw (2,-2) to (2.5,-3) to (3,-2);
        \draw[dashed](0,-2) to (0,-4);
        \draw[dashed](1,-2) to (1,-4);
        \node at (1.5,-2) {$\ldots$};
        \node at (-.5,-2) {$\ldots$};
        \node at (3.5,-2) {$\ldots$};
        \draw[dashed](3,-2) to (3,0);
        \draw[dashed] (2,0) to (2,-2);
		\node at (0.5,-1) [circle, fill=white, draw, outer sep=0pt, inner sep=3 pt] {};
		\node at (2.5,-3) [circle, fill=white, draw, outer sep=0pt, inner sep=3 pt] {};
		\node at (0,-2) [circle,fill,inner sep=1.5pt] {};
		\node at (1,-2) [circle,fill,inner sep=1.5pt] {};
		\node at (2,-2) [circle,fill,inner sep=1.5pt] {};
        \node at (3,-2) [circle,fill,inner sep=1.5pt] {};
        \node at (0,-4) [circle,fill,inner sep=1.5pt] {};
		\node at (1,-4) [circle,fill,inner sep=1.5pt] {};
		\node at (2,-4) [circle,fill,inner sep=1.5pt] {};
        \node at (3,-4) [circle,fill,inner sep=1.5pt] {};
		\end{tikzpicture}
        \hspace{1cm}
  	\begin{tikzpicture}
		\node at (0,0) [circle,fill,inner sep=1.5pt] {};
		\node at (1,0) [circle,fill,inner sep=1.5pt] {};
		\node at (2,0) [circle,fill,inner sep=1.5pt] {};
        \node at (3,0) [circle,fill,inner sep=1.5pt] {};
		\draw (2,0) to (2.5,-1) to (3,0);
        \draw (2,-2) to (2.5,-1) to (3,-2);
		\draw (0,-4) to (.5,-3) to (1,-4);
  	\draw (0,-2) to (.5,-3) to (1,-2);
        \draw[dashed](0,0) to (0,-2);
        \draw[dashed](1,0) to (1,-2);
        \node at (1.5,-2) {$\ldots$};
        \node at (-.5,-2) {$\ldots$};
        \node at (3.5,-2) {$\ldots$};
        \draw[dashed](3,-2) to (3,-4);
        \draw[dashed] (2,-4) to (2,-2);
		\node at (2.5,-1) [circle, fill=white, draw, outer sep=0pt, inner sep=3 pt] {};
		\node at (0.5,-3) [circle, fill=white, draw, outer sep=0pt, inner sep=3 pt] {};
		\node at (0,-2) [circle,fill,inner sep=1.5pt] {};
		\node at (1,-2) [circle,fill,inner sep=1.5pt] {};
		\node at (2,-2) [circle,fill,inner sep=1.5pt] {};
        \node at (3,-2) [circle,fill,inner sep=1.5pt] {};
        \node at (0,-4) [circle,fill,inner sep=1.5pt] {};
		\node at (1,-4) [circle,fill,inner sep=1.5pt] {};
		\node at (2,-4) [circle,fill,inner sep=1.5pt] {};
        \node at (3,-4) [circle,fill,inner sep=1.5pt] {};
		\end{tikzpicture}
	\end{center}
\end{proof}

\subsection{$BL_\infty$ morphisms}
In the following, we define morphisms between $BL_{\infty}$ algebras. Given a family of operators $\{\phi^{k,l}:S^k V\to S^l V'\}_{k\ge 1,l\ge 0}$ of degree $0\in \Z_2$,  such that for any $v_1 \ldots  v_k\in S^kV$, there are at most finitely many $l$, such that $\phi^{k,l}(v_1 \ldots  v_k)\ne 0$. To explain the map $\widehat{\phi}:EV\to EV'$, we will use the description of graphs. To represent $\phi^{k,l}$, we use a graph similar to the one representing $p^{k,l}$ but replace $\tikz\draw[black,fill=white] (0,-1) circle (0.4em);$ by $\tikz\draw[black,fill=black] (0,-1) circle (0.4em);$ to indicate that they are maps of different roles.

Then $\widehat{\phi}$ is represented by the sum of all possible gluing of a whole layer or $\phi^{k,l}$, such that no cycles are created and every leaf of the input forest is glued.  Unlike the definition of $\widehat{p}$ that we need to glue exactly one $p^{k,l}$ graph, it is possible that we do not glue in any $\phi^{k,l}$ graphs.  This is the case when the input in $\odot^m\bk$, i.e.\ the input is represented by trees without leaves. In particular, $\widehat{\phi}$ is identity in such case, i.e.\ $\widehat{\phi}(1\odot \ldots \odot 1) = 1\odot \ldots \odot 1$. All the rules, like orders, signs, and the well-definedness on $EV$ are similar to the $\widehat{p}$ case. 
\begin{figure}[H]
	\begin{center}
		\begin{tikzpicture}
		\node at (0,0) [circle,fill,inner sep=1.5pt] {};
		\node at (1,0) [circle,fill,inner sep=1.5pt] {};
		\node at (2,0) [circle,fill,inner sep=1.5pt] {};
		\node at (3,0) [circle,fill,inner sep=1.5pt] {};
		\node at (4,0) [circle,fill,inner sep=1.5pt] {};
		\node at (5,0) [circle,fill,inner sep=1.5pt] {};
		\node at (6,0) [circle,fill,inner sep=1.5pt] {};
		\node at (7,0) [circle,fill,inner sep=1.5pt] {};
		
		\draw (0,0) to (1,1) to (1,0);
		\draw (1,1) to (2,0);
		\draw (3,0) to (4,1) to (4,0);
		\draw (4,1) to (5,0);
		\draw (6,0) to (6.5,1) to (7,0);

		\draw (2,0) to (2.5,-1) to (3,0);
		\draw (2,-2) to (2.5,-1) to (2.5,-2);
		\draw (2.5,-1) to (3,-2);
		\node at (2.5,-1) [circle, fill, draw, outer sep=0pt, inner sep=3 pt] {};
		
		\node at (0,-2) [circle,fill,inner sep=1.5pt] {};
		\node at (1,-2) [circle,fill,inner sep=1.5pt] {};
		\node at (2,-2) [circle,fill,inner sep=1.5pt] {};
		\node at (3,-2) [circle,fill,inner sep=1.5pt] {};
		\node at (4,-2) [circle,fill,inner sep=1.5pt] {};
		\node at (2.5,-2) [circle,fill,inner sep=1.5pt] {};
	
		\draw (0,0) to (0,-2);
		\draw (1,0) to (1,-2);
		\draw (4,0) to (4,-2);
		\draw (5,0) to (5.5,-1);
		\draw (5.5,-1) to (6,0);
		\draw (7,0) to (7,-1);
		
		\node at (0,-1) [circle, fill, draw, outer sep=0pt, inner sep=3 pt] {};
		\node at (1,-1) [circle, fill, draw, outer sep=0pt, inner sep=3 pt] {};
		\node at (4,-1) [circle, fill, draw, outer sep=0pt, inner sep=3 pt] {};
		\node at (5.5,-1) [circle, fill, draw, outer sep=0pt, inner sep=3 pt] {};
		\node at (7,-1) [circle, fill, draw, outer sep=0pt, inner sep=3 pt] {};
		\end{tikzpicture}
	\end{center}
	\caption{A component of $\widehat{\phi}$ }
\end{figure}

In terms of formulae, we first define $\widehat{\phi}^k:S^k S V \to S V^\prime$.
It is determined by the following.
\begin{enumerate}
	\item $\widehat{\phi}^{k+1}(w_1\odot\ldots \odot w_{k}\odot 1)=0$ for $k\ge 1$ and $\widehat{\phi}^1(1)=1$.
	\item $\widehat{\phi}^k:\odot^k V \subset S^k SV \to S V^\prime$ is defined by $\sum_{l\ge 0} \phi^{k,l}$.
	\item Let $\{i_j\}_{1\le j \le k}$ be a sequence of positive integers. We define $N:=\sum_{j=1}^{k}i_j$ and $N_i:=\sum_{j=1}^{i}i_j$.  Let $w_i=v_{N_{i-1}+1} \ldots  v_{N_i}$. The following sum is over all partitions $J_1\sqcup\ldots \sqcup J_b=\{1,\ldots, N\}$, such that the graph with $k+b+N$ vertices $A_1,\ldots, A_k, B_1,\ldots,B_b, v_1,\ldots, v_N$ with $A_i$ connected to $v_{N_{i-1}+1}, \ldots, v_{N_i}$ and $B_i$  connected to $v_j$ iff $j\in J_i$, is connected and has no cycles.
	$$\widehat{\phi}^k(w_1\odot\ldots \odot w_k)=\sum_{\substack{\text{admissible partitions }\\ J_1\sqcup\ldots \sqcup J_b}}\frac{(-1)^\bigcirc}{b!}\sum_{l_1,\ldots,l_b=0}^{\infty} \phi^{|J_1|,l_1}(v^{J_1}) * \ldots * \phi^{|J_b|,l_b}(v^{J_b}),$$
where $w_1 \ldots  w_k=(-1)^{\bigcirc}v^{J_1} \ldots  v^{J_b}$. There is no extra sign as we assume $\phi^{k,l}$ has degree $0\in \Z_2$. The reduction by $b!$ is because a different order of the partition does not count as a different gluing, but the order will affect the sign when viewing it as elements in the tensor product. The appearance of such factor is precisely the reason why we want to use tree descriptions, as algebraic formulae (when not phrased in an optimal form) might give us the \emph{wrong} impression that such structure can only defined over a coefficient ring where $b!\ne 0$. The number of such partitions divided by $b!$ is exactly the number of ways of gluing a layer of $\sum_{l=0}^{\infty}\phi^{*,l}$.
\end{enumerate}
\begin{figure}[H]
	\begin{center}
		\begin{tikzpicture}
		\node at (0,0) [circle,fill,inner sep=1.5pt] {};
		\node at (1,0) [circle,fill,inner sep=1.5pt] {};
		\node at (2,0) [circle,fill,inner sep=1.5pt] {};
		\node at (3,0) [circle,fill,inner sep=1.5pt] {};
		\node at (4,0) [circle,fill,inner sep=1.5pt] {};
		\node at (5,0) [circle,fill,inner sep=1.5pt] {};
		\node at (6,0) [circle,fill,inner sep=1.5pt] {};
		\node at (7,0) [circle,fill,inner sep=1.5pt] {};
		\node at (0.3,0) {$v_1$};
		\node at (1.3,0) {$v_2$};
		\node at (2.3,0) {$v_3$};
		\node at (3.3,0) {$v_4$};
		\node at (4.3,0) {$v_5$};
		\node at (5.3,0) {$v_6$};
		\node at (6.3,0) {$v_7$};
		\node at (7.3,0) {$v_8$};
		\node at (1.3,1) {$A_1$};
		\node at (4.3,1) {$A_2$};
		\node at (6.8,1) {$A_3$};
		\node at (0.4,-1) {$B_1$};
		\node at (1.4,-1) {$B_2$};
		\node at (2.9,-1) {$B_3$};
		\node at (4.4,-1) {$B_4$};
		\node at (5.9,-1) {$B_5$};
		\node at (7.4,-1) {$B_6$};

		\draw (0,0) to (1,1) to (1,0);
		\draw (1,1) to (2,0);
		\draw (3,0) to (4,1) to (4,0);
		\draw (4,1) to (5,0);
		\draw (6,0) to (6.5,1) to (7,0);

		\draw (2,0) to (2.5,-1) to (3,0);
		\node at (2.5,-1) [circle, fill, draw, outer sep=0pt, inner sep=3 pt] {};

		\draw (0,0) to (0,-1);
		\draw (1,0) to (1,-1);
		\draw (4,0) to (4,-1);
		\draw (5,0) to (5.5,-1);
		\draw (5.5,-1) to (6,0);
		\draw (7,0) to (7,-1);
		
		\node at (0,-1) [circle, fill, draw, outer sep=0pt, inner sep=3 pt] {};
		\node at (1,-1) [circle, fill, draw, outer sep=0pt, inner sep=3 pt] {};
		\node at (4,-1) [circle, fill, draw, outer sep=0pt, inner sep=3 pt] {};
		\node at (5.5,-1) [circle, fill, draw, outer sep=0pt, inner sep=3 pt] {};
		\node at (7,-1) [circle, fill, draw, outer sep=0pt, inner sep=3 pt] {};

		\end{tikzpicture}
	\end{center}
	\caption{An admissible partition}
\end{figure}

Then we define $\widehat{\phi}$ from $\widehat{\phi}^k$ just like the $L_\infty$ morphism $\widehat{\phi}$ built from $\phi^k$. Here is dictionary between algebraic description and tree description. 

\begin{center}
\begin{tabular}{c|c}
  $\widehat{\phi}^{k+1}(w_1\odot\ldots \odot w_{k}\odot 1)=0$ & \makecell{no way to glue a forest with at least two trees\\ with one of them having no leave, to a single tree } \\
  \hline
  $\widehat{\phi}^1(1)=1$  & \makecell{ trivial/empty gluing to a leafless tree}    \\
  \hline
  $\widehat{\phi}^k$ on $\odot^kV$ & \makecell{sum of gluings of $\phi^{k,*}$ to a  forest of \\ $k$ trees of single leaf, one gluing for each $*\in \N$}  \\
  \hline
  $\widehat{\phi}^k$    & \makecell{sum of gluing of $\phi^{k,*}$ to a \\ forest of $k$ trees to get a tree} \\
   \hline
   $\widehat{\phi}$  & \makecell{sum of gluing of $\phi^{*,*}$ to a \\ forest  to get a forest} \\
\end{tabular}
\end{center}

\begin{definition}\label{def:morphism}
 $\{\phi^{k,l}\}_{k\ge 1, l\ge 0}$ is a $BL_{\infty}$ morphism from $(V,p)$ to $(V',p')$ if $\widehat{\phi}\circ \widehat{p}=\widehat{p}'\circ \widehat{\phi}$ and $|\widehat{\phi}|=0$.
\end{definition}

The composition of $\phi:V\to V'$ and $\psi:V'\to V''$ is defined by 
$$(\psi\circ \phi)^{k,l}=\pi_{1,l}\circ \widehat{\psi}\circ \widehat{\phi}|_{\odot^kV}.$$
The more explicit algebraic description is as follows. Let $I=\{1,\ldots,k\}$ and $I_1\sqcup \ldots \sqcup I_a$ be any partition of $I$ (any partition is admissible as the input is a forest of single-leaf trees), then 
$$(\psi\circ \phi)^{k,l}(v_1 \ldots  v_k)= \pi_l\left(\sum_{\substack{\text{partitions }\\ I_1\sqcup\ldots \sqcup I_a}} \widehat{\psi}^a\left(\frac{(-1)^\bigcirc}{a!}\sum_{l_1,\ldots,l_a=0}^\infty \phi^{|I_1|,l_1}(v^{I_1})\odot \ldots \odot\phi^{|I_a|,l_a}(v^{I_a})\right)\right),$$
where $v_1 \ldots  v_k=(-1)^{\bigcirc}v^{I_1} \ldots  v^{I_a}$ and $\pi_l$ is the projection $S V'' \to S^l V''$. 
It is clear that the graph representing $\widehat{\psi}\circ \widehat{\phi}$ has no cycle, and that $(\psi\circ \phi)^{k,l}$ is represented by connected graphs without cycles glued from one level from $\phi$ and one level from $\psi$. It is clear from the graph description that $\widehat{\psi\circ \phi}=\widehat{\psi}\circ \widehat{\phi}$. It is simply two ways to interpret the same gluing of two layers.

Similar to the definition of $p^{k,2}$, we define $(\phi \circ p)^{k,l}:=\pi_{1,l}\circ \widehat{\phi}\circ \widehat{p}|_{\odot^k V}$, where $\pi_{1,l}$ denotes the projection $EV\to S^1SV \to S^lV$, and we define $(p \circ \phi)^{k,l}:=\pi_{1,l}\circ \widehat{p}\circ \widehat{\phi}|_{\odot^k V}$. Therefore, $(\phi \circ p)^{k,l}$ is the sum of all \emph{connected} graphs with first a level of a single $\tikz\draw[black,fill=white] (0,-1) circle (0.4em);$ vertex followed by a level of $\tikz\draw[black,fill=black] (0,-1) circle (0.4em);$ vertices, similarly for $(p \circ \phi)^{k,l}$. Similar to Proposition \ref{prop:2level}, we have the following, which allows us to only consider degeneration from connected curves in an analytical setup. 
\begin{proposition}\label{prop:mor2level}
$\{\phi^{k,l}\}_{k\ge 1,\l\ge 0}$ forms a $BL_\infty$ morphism from $\{V,\{p^{k,l}\}_{k\ge 1,\l\ge 0}\}$ to $\{U,\{q^{k,l}\}_{k\ge 1,\l\ge 0}\}$
if and only if  $(\phi \circ p)^{k,l} = (q \circ \phi)^{k,l}$ for $k\ge 1, l\ge 0$.     
\end{proposition}
To motivate the proof below, we first work out a simple example by applying $\widehat{\phi} \circ \widehat{p}$ and $\widehat{q}\circ \widehat{\phi}$ to $v_1v_2\odot u\in S^2V\odot S^1V$ for $v_1,v_2,u\in V$.
\begin{eqnarray}
    \widehat{\phi} \circ \widehat{p}(v_1v_2\odot u) & = & (-1)^{|v_1|} \left(\sum_{l=0}^\infty \phi^{1,l}(v_1) \sum_{l=0}^\infty(\phi\circ p)^{2,l}(v_2\odot u)\right)+(-1)^{|v_2||u|}\left( \sum_{l=0}^\infty(\phi\circ p)^{2,l}(v_1\odot u)\sum_{l=0}^\infty \phi^{1,l}(v_2)\right) \nonumber\\
    & & + \left(\sum_{l=0}^\infty (\phi\circ p)^{1,l}(v_1) \sum_{l=0}^\infty \phi^{1,l}(v_2)\right) \odot \sum_{l=0}^\infty \phi^{1,l}(u) \nonumber \\
    & & +  (-1)^{|v_1|}\left(\sum_{l=0}^\infty \phi^{1,l}(v_2) \sum_{l=0}^\infty (\phi\circ p)^{1,l}(v_2)\right) \odot \sum_{l=0}^\infty \phi^{1,l}(u)\nonumber\\
    & & (-1)^{|v_1|+|v_2|}\left(\sum_{l=0}^\infty \phi^{1,l}(v_1) \sum_{l=0}^\infty \phi^{1,l}(v_2)\right) \odot \sum_{l=0}^\infty (\phi\circ p)^{1,l}(u) \label{ex:phi_p}.
\end{eqnarray}
On the other hand, we have 
\begin{eqnarray}
     \widehat{q}\circ \widehat{\phi} (v_1v_2\odot u) & = &  (-1)^{|v_1|} \left(\sum_{l=0}^\infty \phi^{1,l}(v_1) \sum_{l=0}^\infty(q\circ \phi)^{2,l}(v_2\odot u)\right)+(-1)^{|v_2||u|}\left( \sum_{l=0}^\infty(q\circ \phi)^{2,l}(v_1\odot u)\sum_{l=0}^\infty \phi^{1,l}(v_2)\right) \nonumber \\
     & &  + \left(\sum_{l=0}^\infty (q\circ \phi)^{1,l}(v_1) \sum_{l=0}^\infty \phi^{1,l}(v_2)\right) \odot \sum_{l=0}^\infty \phi^{1,l}(u) \nonumber\\
    & & +  (-1)^{|v_1|}\left(\sum_{l=0}^\infty \phi^{1,l}(v_2) \sum_{l=0}^\infty (q\circ \phi)^{1,l}(v_2)\right) \odot \sum_{l=0}^\infty \phi^{1,l}(u) \nonumber\\
    & & (-1)^{|v_1|+|v_2|}\left(\sum_{l=0}^\infty \phi^{1,l}(v_1) \sum_{l=0}^\infty \phi^{1,l}(v_2)\right) \odot \sum_{l=0}^\infty (q\circ \phi)^{1,l}(u) \label{ex:q_phi}.
\end{eqnarray}
Therefore,  we have that $(\phi \circ p)^{k,l} = (q \circ \phi)^{k,l}$ for $k=1,2, l\ge 0$ implies that  $\widehat{\phi} \circ \widehat{p}(v_1v_2\odot u)= \widehat{q}\circ \widehat{\phi} (v_1v_2\odot u)$. Here we emphasize that $\phi^{1,l}(v_1/v_2/u)$ in \eqref{ex:phi_p} is represented by a dashed line (identity map) in the upper level glued with some $\phi^{1,l}$ in  $\widehat{\phi} \circ \widehat{p}$, while $\phi^{1,l}(v_1/v_2/u)$ in \eqref{ex:q_phi} is represented by some $\phi^{1,l}$ glued with a level of dashed lines in $\widehat{q}\circ \widehat{\phi}$. In terms of forests, they are identified through shifting the isolated $\phi^{1,l}$ components (in the sense of not being connected to the tree representing $(\phi \circ p)^{k,l}$) to the upper level and replacing  $(\phi \circ p)^{k,l}$ by $(q\circ \phi)^{k,l}$.

\begin{proof}[Proof of Proposition \ref{prop:mor2level}]
    If $\{\phi^{k,l}\}_{k\ge 1,\l\ge 0}$ forms a $BL_\infty$ morphism, then
    $$(\phi \circ p)^{k,l} - (q \circ \phi)^{k,l}= \pi_{1,l}\circ \left.\left(\widehat{\phi}\circ \widehat{p}-\widehat{q}\circ \widehat{\phi}\right)\right|_{\odot^kV}=0.$$
    Reversely, we assume that $(\phi \circ p)^{k,l} = (q \circ \phi)^{k,l}$. In the glued forests representing $\widehat{\phi}\circ \widehat{p}$, by looking at the levels containing $p^{*,*}$ and $\phi^{*,*}$, there is exactly one connected graph representing a component in $(\phi \circ p)^{k,l}$, and many $\phi^{*,*}$ components in the second level. Fixing the $(\phi \circ p)^{k,l}$ component, by shifting the remaining isolated $\phi^{*,*}$ components up to the upper level and replacing the $(\phi \circ p)^{k,l}$ component by $(q\circ \phi)^{k,l}$, we find a correspondence between components in $\widehat{\phi}\circ \widehat{p}$ involving $(\phi \circ p)^{k,l}$ and components in $\widehat{q}\circ \widehat{\phi}$ involving $(q\circ \phi)^{k,l}$ as in the concrete example before the proof (or Figure \ref{fig:pair}). They evaluate to the same value as $(\phi \circ p)^{k,l} = (q \circ \phi)^{k,l}$ and there is no extra sign from this shifting (i.e.\ changing the order of applying the operations) since $|\phi^{*,*}|=0$.
\end{proof}

\begin{figure}[H]
\begin{overpic}[scale=0.8]
{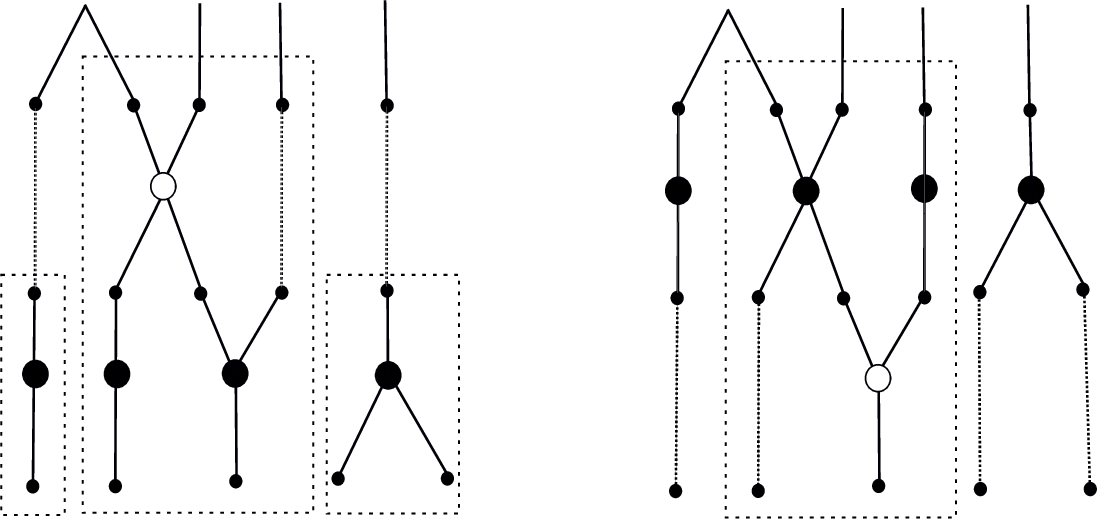}
\put (7,-2) {a component of $(\phi \circ p)^{3,2}$}
\put (65,-2) {a component of $(q \circ \phi)^{3,2}$}
\put (30,24) {an isolated $\phi^{1,2}$}
\put (-5,24) {an isolated $\phi^{1,1}$}
\end{overpic}
\caption{Components of $\widehat{\phi}\circ \widehat{p}$ and  $\widehat{q}\circ \widehat{\phi}$ and the pairing by shifting up isolated $\phi^{*,*}$ components and replacing the $(\phi \circ p)^{k,l}$ by $(q\circ \phi)^{k,l}$ }\label{fig:pair}
\end{figure}

\begin{remark}
	There are different notions of homotopies between $BL_\infty$ morphisms if we wish to define notions of homotopy equivalences of $BL_\infty$ algebras. In practice, we can not associate a canonical $BL_\infty$ algebra to a contact manifold but one depends on various choices and is only well-defined up to homotopy. However, for the purpose of this paper, as we are constructing functors from $\cont$ to a totally ordered set, homotopy invariance is not needed.  The homotopy in SFT is one of most subtle aspects of the theory both algebraically and analytically. Nevertheless, we have the following brief remarks on homotopy.
	\begin{enumerate}
		\item One can define a notion of homotopy, which is a homotopy on the bar/cobar complex. That is one can define a map by counting rigid but \emph{disconnected} curves in a one-parameter family. One advantage of such definition is that it is easier to construct as we will neglect the structures from each connected component.  Any homological structure on the level of bar/cobar complex will be an invariant. For example, the contact homology in \cite{pardon2019contact} used this notion of homotopy. However, homotopic augmentations in this sense do not give rise to homotopic linearized theory.
		\item Another notion of homotopy is defined through the notion of path objects, e.g.\ \cite[Definition 4.1]{cieliebak2015homological}, see also \cite[Definition 2.9]{siegel2019higher} for the homotopy in the $L_\infty$ context with a specific path object model. This definition is the right one to discuss linearized theory but is more involved to get in the construction of SFT. In particular, homotopic augmentations give rise to homotopic linearized theories with such notion of homotopy. Such homotopy is expected to be derived from the homotopy used in \cite[\S 2.4]{eliashberg2000introduction}. However, from the curve counting point of view, such construction is more subtle.   
	\end{enumerate}
\end{remark}

\subsection{Augmentations}
When $V=\{0\}$, it has a unique trivial $BL_\infty$ algebra structure by $p^{k,l}=0$. We use $\mathbf{0}$ to denote this trivial $BL_\infty$ algebra. Note that $\mathbf{0}$ is the initial object in the category of $BL_\infty$ algebras, with $\mathbf{0}\to V$ defined by $\phi^{k,l}=0$. 
\begin{definition}
A $BL_\infty$ augmentation is a $BL_\infty$ morphism $\epsilon:V \to \mathbf{0}$, i.e.\ a family of operators $\epsilon^k:S^kV \to \bk$ satisfying Definition \ref{def:morphism}.
\end{definition}

For a $BL_\infty$ algebra $V$, we define $E^kV=\overline{B}^kSV$, which is a filtration on $EV$ compatible with the differential $\widehat{p}$. Note that $E\mathbf{0}=\bk \oplus S^2\bk \oplus +\ldots$ with $\widehat{p}=0$, we have $H_*(E\mathbf{0})=E\mathbf{0}$. Similarly we have $H_*(E^k\mathbf{0})=E^k\mathbf{0}$ for all $k\ge 1$. We define $1_\mathbf{0}$ be the generator in $E^1\mathbf{0}$, then $1_{\mathbf{0}}\ne 0\in H_*(E^k\mathbf{0})$ for all $k\ge 1$. Then we define $1_V\in H_*(E^kV)$ to be the image of $1_{\mathbf{0}}$ under the chain map $E^k\mathbf{0}\to E^kV$ induced by the trivial $BL_\infty$ morphism $\mathbf{0}\to V$.

\begin{proposition}\label{prop:ob}
 If there exists $k\ge 1$, such that $1_V\in H_*(E^kV)$ is zero, then $V$ has no $BL_\infty$ augmentation.
\end{proposition}
\begin{proof}
If there is an augmentation $\epsilon:V\to \mathbf{0}$, then the sequences of $BL_{\infty}$ morphisms $\mathbf{0}\to V \stackrel{\epsilon}{\to} \mathbf{0}$ induce a chain morphism $E^k\mathbf{0}\mapsto E^k V \mapsto E^k\mathbf{0}$. It is direct to check the composition is identity by definition. If $1_V\in H^*(E^k V)$ is zero, then we have a contradiction since $1_{\mathbf{0}}\ne 0 \in H^*(E^k\mathbf{0})$.
\end{proof}

\begin{definition}
We define the torsion of a $BL_\infty$ algebra $V$ to be 
$$\TT(V):= \min\{k-1| 1_V=0 \in H^*(E^kV),k\ge 1\}.$$
Here the minimum of an empty set is defined to be $\infty$.
\end{definition}
By definition, we have that $\TT(V)=0$ iff $1_V\in H^*(SV,\widehat{p}^1)$ is zero. Since $H^*(SV,\widehat{p}^1)$ is an algebra with $1_V$ a unit, we have $H^*(SV,\widehat{p}^1)=0$.  In the context of SFT, $\TT(V)=0$ iff the contact homology vanishes, i.e.\ algebraically overtwisted \cite{bourgeois2010towards}.

Since a $BL_\infty$ morphism preserves the word filtration on the bar complex, we know that if there is a $BL_\infty$ morphism from $V$ to $V'$ then $\TT(V)\ge \TT(V')$. Therefore we have the following obvious property, which is crucial for the invariant property for our applications in symplectic topology.
\begin{proposition}\label{prop:torsion}	
	If there are $BL_\infty$ morphisms between $V,V'$ in both directions, then we have $\TT(V)=\TT(V')$.
\end{proposition}

Given a $BL_\infty$ augmentation $\epsilon$, we can linearize w.r.t. $\epsilon$ by the following procedure. More precisely, there is a change of coordinate to kill off all constant terms $p^{k,0}$. We define $F^{1,1}_{\epsilon}=\Id_V$ and $F^{k,0}_{\epsilon}=\epsilon^{k}$ and all other $F^{k,l}_{\epsilon}=0$. Then following the recipe of constructing $\widehat{\phi}$ from $\phi^{k,l}$, we can define $\widehat{F}_{\epsilon}$ on $EV$. Then $\widehat{F}_{\epsilon}$ preserves the word length filtration and on the diagonal $\pi_k\circ\widehat{F}_{\epsilon}|_{S^k S V}$ is $\odot^k \widehat{F}^1_{\epsilon}$, where $\widehat{F}^1_{\epsilon}$ is an algebra isomorphism determined by $\widehat{F}^1_{\epsilon}(x)=x+\epsilon^{1}(x)$ and $\pi_k$ is the projection $EV\to S^kSV$. Indeed the inverse is given by the following proposition. 

\begin{proposition}\label{prop:inverse}
	Let $\widehat{F}_{-\epsilon}$ denote the map on $EV$ defined by $F^{1,1}_{-\epsilon}=\Id_V$ and $F^{k,0}_{-\epsilon}=-\epsilon^{k,0}$ and all other $F^{k,l}_{-\epsilon}=0$, then $\widehat{F}_{-\epsilon}$ is the inverse of $\widehat{F}_{\epsilon}$.
\end{proposition}
\begin{proof}
In the gluing of forests representing $\widehat{F}_{-\epsilon}\circ \widehat{F}_{\epsilon}(x)$, the $\epsilon$ and $-\epsilon$ components are not connected directly as they have no outputs. Therefore we can shift up a $-\epsilon$ in the second layer to $\epsilon$, which is still an admissible gluing, or reversely, e.g.\ the figure below. 
		\begin{center}
			\begin{tikzpicture}
			\node at (0,0) {$\ldots$};
			\node at (1,0) {$\ldots$};
			\node at (4,0) {$\ldots$};
			\node at (5,0) {$\ldots$};
			\node at (2,0) [circle,fill,inner sep=1.5pt] {};
			\node at (3,0) [circle,fill,inner sep=1.5pt] {};
			\draw (0,0.1) to (1,1) to (1,0.1);
			\draw (1,1) to (2,0);
			\draw (3,0) to (4,1);
			\draw (4,0.1) to (4,1) to (5,0.1);
			
			\node at (2.5,-1) [circle, fill, draw, outer sep=0pt, inner sep=3 pt] {};
			\node at (2.7,-1.3) {$\epsilon^2$};
			\draw (2,0) to (2.5,-1) to (3,0);
			
			\node at (0,-2) {$\ldots$};
			\node at (1,-2) {$\ldots$};
			\node at (4,-2) {$\ldots$};
			\node at (5,-2) {$\ldots$};
			
		    \node at (0,-4) {$\ldots$};
			\node at (1,-4) {$\ldots$};
			\node at (4,-4) {$\ldots$};
			\node at (5,-4) {$\ldots$};
			
			\draw (0,-0.1) to (0,-1.9);
			\draw (1,-0.1) to (1,-1.9);
			\draw (4,-0.1) to (4,-1.9);
			\draw (5,-0.1) to (5,-1.9);
		    \draw (0,-2.1) to (0,-3.9);
			\draw (1,-2.1) to (1,-3.9);
			\draw (4,-2.1) to (4,-3.9);
			\draw (5,-2.1) to (5,-3.9);
			
			\node at (-0.5,0.5) {$\ldots$};
			\node at (5.5,0.5) {$\ldots$};
			\end{tikzpicture}
			\begin{tikzpicture}
			\node at (0,0) {$\ldots$};
			\node at (1,0) {$\ldots$};
			\node at (4,0) {$\ldots$};
			\node at (5,0) {$\ldots$};
			\node at (2,0) [circle,fill,inner sep=1.5pt] {};
			\node at (3,0) [circle,fill,inner sep=1.5pt] {};
			\draw (0,0.1) to (1,1) to (1,0.1);
			\draw (1,1) to (2,0);
			\draw (3,0) to (4,1);
			\draw (4,0.1) to (4,1) to (5,0.1);
			
			\node at (2,-1) [circle, fill, draw, outer sep=0pt, inner sep=3 pt] {};
			\node at (3,-1) [circle, fill, draw, outer sep=0pt, inner sep=3 pt] {};
			\node at (2,-2) [circle,fill,inner sep=1.5pt] {};
			\node at (3,-2) [circle,fill,inner sep=1.5pt] {};
			\draw (2,0) to (2,-2);
			\draw (3,0) to (3,-2);
			\node at (1.7,-1) {$\Id$};
			\node at (3.3,-1) {$\Id$};
			
			\node at (2.5,-3) [circle, fill, draw, outer sep=0pt, inner sep=3 pt] {};
			\node at (2.7,-3.3) {$-\epsilon^2$};
			\draw (2,-2) to (2.5,-3) to (3,-2);
			
			\node at (0,-2) {$\ldots$};
			\node at (1,-2) {$\ldots$};
			\node at (4,-2) {$\ldots$};
			\node at (5,-2) {$\ldots$};
			
			\node at (0,-4) {$\ldots$};
			\node at (1,-4) {$\ldots$};
			\node at (4,-4) {$\ldots$};
			\node at (5,-4) {$\ldots$};
			
			\draw (0,-0.1) to (0,-1.9);
			\draw (1,-0.1) to (1,-1.9);
			\draw (4,-0.1) to (4,-1.9);
			\draw (5,-0.1) to (5,-1.9);
			\draw (0,-2.1) to (0,-3.9);
			\draw (1,-2.1) to (1,-3.9);
			\draw (4,-2.1) to (4,-3.9);
			\draw (5,-2.1) to (5,-3.9);
			
			\node at (-0.5,0.5) {$\ldots$};
			\node at (5.5,0.5) {$\ldots$};
			\end{tikzpicture}
        \end{center}
    We can group all gluings into those that can be related by these moves. Then there is one group containing the unique gluing containing $\Id$ only. All other groups have $2^N$ gluings, where $N>0$ is the number of $\pm \epsilon$ involved. It is clear that we pair up gluings in a group such that they cancel with each other by a move at a chosen position, e.g. the figure above. Therefore we have $\widehat{F}_{-\epsilon}\circ \widehat{F}_{\epsilon}=\Id$, similarly for $\widehat{F}_{\epsilon}\circ \widehat{F}_{-\epsilon}=\Id$.
\end{proof}

We use $\widehat{F}_{\epsilon}$ as a change of coordinate on $EV$ and consider $\widehat{p}_{\epsilon}:=\widehat{F}_{\epsilon}\circ \widehat{p}\circ \widehat{F}_{-\epsilon}:EV \to EV$, then $\widehat{p}_{\epsilon}^2=0$. We can define 
$$p^{k,l}_{\epsilon}:= \pi_{1,l} \circ \widehat{F}_{\epsilon}\circ \widehat{p}\circ \widehat{F}_{-\epsilon}|_{\odot^k V},$$
where $\pi_{1,l}$ is the projection $EV \to S^1S^lV$. We claim that $p_{\epsilon}^{k,l}=\pi_{1,l}\circ \widehat{F}_{\epsilon}\circ \widehat{p}|_{\odot^k V}$. To see this, note that $\widehat{F}_{-\epsilon}$ restricted on $\odot^k V$ is the identity map plus extra terms landing in $(\odot^m V) \odot (\odot^l \bm{k})$ for $l+m\le k,l>0$. When we apply $\widehat{p}$ on the extra term, we must have the output is in $\Ima \widehat{p}|_{\odot^m V}\odot (\odot^l \bm{k})$. Then image after applied $\pi_{1,l}\circ \widehat{F}_{\epsilon}$ must be zero since it contains at least two $\odot$ components. In other words, those extra terms represent disconnected graphs in the description of $\widehat{F}_{\epsilon}\circ \widehat{p}\circ \widehat{F}_{-\epsilon}|_{\odot^k V}$, hence is projected to zero by $\pi_{1,l}$. Then $p^{k,l}_{\epsilon}$ is represented by the sum of following connected graphs without cycle with $k$ inputs, $l$ outputs, one $\bigcirc$ component and possibly several components from $\epsilon$.
\begin{figure}[H]
	\begin{center}
		\begin{tikzpicture}
		\node at (0,0) [circle,fill,inner sep=1.5pt] {};
		\node at (1,0) [circle,fill,inner sep=1.5pt] {};
		\node at (2,0) [circle,fill,inner sep=1.5pt] {};
		\node at (3,0) [circle,fill,inner sep=1.5pt] {};
		\node at (0,-2) [circle,fill,inner sep=1.5pt] {};
		\node at (1,-2) [circle,fill,inner sep=1.5pt] {};
		\node at (1.5,-2) [circle,fill,inner sep=1.5pt] {};
		\node at (2,-2) [circle,fill,inner sep=1.5pt] {};
		\node at (3,-2) [circle,fill,inner sep=1.5pt] {};
		\draw[dashed] (0,0) to (0,-2);
		\draw[dashed] (3,0) to (3,-2);
		\draw (1,0) to (1.5,-1) to (1,-2) to (0.5,-3) to (0,-2);
		\draw (2,0) to (1.5,-1) to (2,-2) to (2.5,-3) to (3,-2);
		\draw (1.5,-1) to (1.5,-2);
		\node at (0.5,-3) [circle, fill, draw, outer sep=0pt, inner sep=3 pt] {};
		\node at (2.5,-3) [circle, fill, draw, outer sep=0pt, inner sep=3 pt] {};
		\node at (1.5,-1) [circle, fill=white, draw, outer sep=0pt, inner sep=3 pt] {};
		\node at (0.5,-3.4) {$\epsilon^2$};
		\node at (2.5,-3.4) {$\epsilon^2$};
		\node at (2, -1) {$p^{2,3}$};
		\end{tikzpicture}
	\end{center}	
    \caption{A component of $p^{4,1}_{\epsilon}$}
\end{figure}

\begin{proposition}\label{prop:linear}
	$\widehat{p}_\epsilon$ is determined by $p^{k,l}_{\epsilon}$ following the same recipe for $\widehat{p}$ from $p^{k,l}$. Moreover, we have $p^{k,0}_{\epsilon}=0$ for all $k$.
\end{proposition}
\begin{proof}
	In the graph representation of $\widehat{F}_{\epsilon}\circ \widehat{p}\circ \widehat{F}_{-\epsilon}$ on $S^{i_1}V\odot \ldots \odot S^{i_k}V$, there is exactly one component containing a $p^{k,l}_{\epsilon}$ as a subgraph. All the other $\pm \epsilon$ components (i.e.\ those not in $p^{k,l}_{\epsilon}$) are not connected to the $p^{*,*}$ component directly (they could be in the same tree), we call them isolated $\pm \epsilon$ components. Similar to the proof of Proposition \ref{prop:inverse}, we get collections of gluings that can be related by moving up/down the isolated $\pm \epsilon$ components. Hence those with isolated $\pm$ components in  $\widehat{F}_{\epsilon}\circ \widehat{p}\circ \widehat{F}_{-\epsilon}$ sum to zero, that is  $\widehat{F}_{\epsilon}\circ \widehat{p}\circ \widehat{F}_{-\epsilon}=\widehat{p}_{\epsilon}$ is determined by gluing only $p^{k,l}_{\epsilon}$. 
 
    Finally, because $\epsilon$ is an augmentation, we have $\pi^{1,0}\circ \widehat{F}_{\epsilon}\circ \widehat{p}=\widehat{\epsilon}\circ \widehat{p}=0$. Therefore $p^{k,0}_{\epsilon}=\pi^{1,0}\circ \widehat{F}_{\epsilon}\circ \widehat{p}|_{\odot^kV}$ for all $k$.
\end{proof}

As a corollary of Proposition \ref{prop:linear}, $\ell^k_{\epsilon}:=p^{k,1}_{\epsilon}$ defines an $L_\infty$ structure on $V[-1]$. Next we introduce the structure which will be relevant to the definition of planarity. Let $p_\bullet^{k,l}:S^kV \to S^lV, k\ge 1, l\ge 0$ be a family of linear maps. We can define $\widehat{p}^{k,l}_\bullet$ and $\widehat{p}_\bullet$ just like $\widehat{p}^{k,l}$ and $\widehat{p}$ (see \eqref{eq:p_hat}), i.e\ by component-wise Leibniz rule and coLeibniz rule with the modification that $|p_{\bullet}^{k,l}|$ is not necessarily $1\in \Z_2$. 

\begin{definition}\label{def:pointedmap}
	We say $\{p_\bullet^{k,l}\}$ is a pointed map, if $\widehat{p}_\bullet\circ \widehat{p}=(-1)^{|\widehat{p}_{\bullet}|}\widehat{p}\circ \widehat{p}_\bullet$.
\end{definition}

\begin{proposition}\label{prop:pointed_2level}
    $\{p_\bullet^{k,l}\}$ is a pointed map if and only if 
    $$\pi_{1,l}\circ \left.\left( \widehat{p}_\bullet\circ \widehat{p}-(-1)^{|\widehat{p}_{\bullet}|}\widehat{p}\circ \widehat{p}_\bullet\right)\right|_{\odot^kV} =0, \quad \forall k\ge 1, l\ge 0.$$
\end{proposition}
\begin{proof}
    The proof is identical to the proof of Proposition \ref{prop:2level}, the extra sign is consistent with the fact that $|p^{*,*}|=1$ in view of the paring from switching levels.
\end{proof}
The relation $\pi_{1,l}\circ \left( \widehat{p}_\bullet\circ \widehat{p}-(-1)^{|\widehat{p}_{\bullet}|}\widehat{p}\circ \widehat{p}_\bullet\right)\left|_{\odot^kV} \right.=0$ only involves degeneration from connected curves in an analytical setup. In applications, $p_\bullet^{k,l}$ will come from counting holomorphic curves with one interior marked point subject to a constraint from $H_*(Y)$. The degree of $p_\bullet$ is the same as the degree of the constraint. Typically we will only consider a point constraint, then the degree is $0$. Note that it does not define $BL_\infty$ morphisms as the combinatorics for packaging $\widehat{p}_\bullet$ is different from $\widehat{\phi}$. Nevertheless, $\widehat{p}_\bullet$ still defines a morphism on the bar complex and preserves the word length filtration.

Then by the same argument in Proposition \ref{prop:linear}, we can define $\widehat{p}_{\bullet,\epsilon}:=\widehat{F}_{\epsilon}\circ \widehat{p}_{\bullet}\circ \widehat{F}_{-\epsilon}$ and $\widehat{p}_{\bullet,\epsilon}$ is determined by $p^{k,l}_{\bullet,\epsilon}$, which is defined similarly to $p^{k,l}_{\epsilon}$. Note that we also have $\widehat{p}_{\epsilon}\circ \widehat{p}_{\bullet,\epsilon}=\widehat{p}_{\bullet,\epsilon}\circ \widehat{p}_{\epsilon}$. However, it is not necessarily true that $p^{k,0}_{\bullet,\epsilon}=0$. In fact, the failure of this property on the homological level will be another hierarchy that we are interested in.  We define $\ell_{\bullet,\epsilon}^{k,0}$ by $p^{k,0}_{\bullet,\epsilon}$. Then $\widehat{\ell}_{\bullet,\epsilon}:=\sum_{k\ge 0} \ell_{\bullet,\epsilon}^{k,0}$ defines a chain morphism $(\overline{S} V, \widehat{\ell}_{\epsilon}) \to \bk$. That $\widehat{\ell}_{\bullet,\epsilon}\circ \widehat{\ell}_{\epsilon}=0$ follows from $0=\pi_{\bk}\circ \widehat{p}_{\epsilon}\circ  \widehat{p}_{\bullet,\epsilon}=\pi_{\bk}\circ \widehat{p}_{\bullet,\epsilon}\circ \widehat{p}_{\epsilon}$ restricted to $\overline{S}V=\overline{S}S^1V\subset EV$ and $\pi_{\bk}$ is the projection from $EV$ to $\bk\subset S^1SV\subset EV$.  

\begin{definition}\label{def:order}
		Given a $BL_{\infty}$ augmentation and a pointed map $p_{\bullet}$, the $(\epsilon,p_{\bullet})$ order of $V$ is defined to be 
		$$O(V,\epsilon,p_{\bullet}):=\min \left\{k\left|1\in \Ima \widehat{\ell}_{\bullet,\epsilon}|_{H_*(\overline{B}^k V, \widehat{\ell}_{\epsilon})}\right. \right\},$$
		where the minimum of an empty set is defined to be $\infty$.
\end{definition}

Next, we need to compare the construction under $BL_{\infty}$ morphisms. Given a $BL_\infty$ morphism $\phi:(V,p) \to (V',q)$ and a family of morphisms $\phi^{k,l}_{\bullet}:S^k V \to S^l V^\prime$, then we can define $\widehat{\phi}_{\bullet}:EV\to EV'$ by the same rule of $\widehat{\phi}$ with exactly one $\phi^{k,l}_{\bullet}$ component and all the others are $\phi^{k,l}$ components. Since we must have one $\phi^{k,l}_{\bullet}$ component, we have $\widehat{\phi}_{\bullet}(1\odot \ldots \odot 1) = 0$. 

In terms of formulae, we define $\widehat{\phi}^k_{\bullet}:S^k S V \to S V^\prime$, which corresponds to the tree component of $\widehat{\phi}$ containing $\phi^{*,*}_{\bullet}$. It is determined by the following.
\begin{enumerate}
	\item $\widehat{\phi}_{\bullet}^{k}(w_1\odot\ldots \odot w_{k-1}\odot 1)=0$ for $k\ge 0$.
	\item $\widehat{\phi}^k_{\bullet}:\odot^k V \subset S^k SV \to S V^\prime$ is defined by $\sum_{l\ge 0} \phi^{k,l}_{\bullet}$.
	\item Let $\{i_j\}_{1\le j \le k}$ be a sequence of positive integers. We define $N:=\sum_{j=1}^{k}i_j$ and $N_i:=\sum_{j=1}^{i}i_j$.  Let $w_i=v_{N_{i-1}+1} \ldots  v_{N_i}$. The following sum is over all partitions $J_1\sqcup\ldots \sqcup J_b=\{1,\ldots, N\}$, such that the graph with $k+b+N$ vertices $A_1,\ldots, A_k, B_1,\ldots,B_b, v_1,\ldots, v_N$ with $A_i$ connected to $v_{N_{i-1}+1}, \ldots, v_{N_i}$ and $B_i$  connected to $v_j$ iff $j\in J_i$, is connected and has no cycles:
	$$\widehat{\phi}^k_{\bullet}(w_1\odot\ldots \odot w_k)=\sum_{\substack{\text{admissible partitions }\\ J_1\sqcup\ldots \sqcup J_b}}\frac{(-1)^\bigcirc}{(b-1)!}\sum_{l_1,\ldots,l_b=0}^{\infty} \phi^{|J_1|,l_1}_{\bullet}(v^{J_1}) * \phi^{|J_2|,l_2}(v^{J_2})* \ldots * \phi^{|J_b|,l_b}(v^{J_b}),$$
where $w_1 \ldots  w_k=(-1)^{\bigcirc}v^{J_1} \ldots  v^{J_b}$. There is no extra sign as we assume $\phi^{k,l}$ has degree $0\in \Z_2$. The reduction by $(b-1)!$ is because a different order of the partition for the $\phi^{*,*}$ components does not count as a different gluing. As a consequence, $\widehat{\phi}^k_{\bullet}$ is the sum of glued trees from a forest of $k$ tress glued with one $\phi^{*,*}_{\bullet}$ components and multiple $\phi^{*,*}$ components.
\end{enumerate}
Then we define $\widehat{\phi}_{\bullet}:EV\to EV'$ similarly to $\widehat{\phi}$, but with exactly one $\widehat{\phi}^{*}_{\bullet}$ term. In terms of formulae, we have 
\begin{eqnarray*}
\widehat{\phi}_{\bullet}(w_1\odot \ldots \odot w_n) & = & \sum_{\substack{k\ge 1\\ i_1+\ldots+i_k=n}} \frac{1}{(k-1)!}(\widehat{\phi}^{i_1}_{\bullet} \widehat{\phi}^{i_2} \ldots \widehat{\phi}^{i_k})(w_1\odot\ldots \odot w_n)
\end{eqnarray*}
where $(\widehat{\phi}^{i_1}_{\bullet} \widehat{\phi}^{i_2} \ldots \widehat{\phi}^{i_k})$ is defined after Definition \ref{def:L_mor} with $sV$ replaced by $SV$ here.

\begin{definition}\label{def:compatible}
Assume $p_{\bullet},q_{\bullet}$ are two pointed maps of $(V,p),(V',q)$ respectively, of the same degree, and $\phi$ is a $BL_\infty$ morphism from $(V,p)$ to $(V',q)$. We say $p_{\bullet},q_{\bullet},\phi$ are compatible, if there is a family of $\phi^{k,l}_{\bullet}$ such that $\widehat{q}_\bullet\circ \widehat{\phi}-(-1)^{|\widehat{q}_{\bullet}|}\widehat{\phi}\circ \widehat{p}_{\bullet}=\widehat{q}\circ \widehat{\phi}_{\bullet}-(-1)^{|\widehat{\phi}_{\bullet}|}\widehat{\phi}_{\bullet}\circ \widehat{p}$ and $|\widehat{\phi}_{\bullet}|=|\widehat{p}_{\bullet}|+1$.
\end{definition}

\begin{proposition}\label{prop:pqphi_2level}
    $p_{\bullet},q_{\bullet},\phi$ are compatible using $\phi_{\bullet}$ if and only if 
    $$\pi_{1,l}\left.\left( \widehat{q}_\bullet\circ \widehat{\phi}-(-1)^{|\widehat{q}_{\bullet}|}\widehat{\phi}\circ \widehat{p}_{\bullet}-\widehat{q}\circ \widehat{\phi}_{\bullet}+(-1)^{|\widehat{\phi}_{\bullet}|}\widehat{\phi}_{\bullet}\circ \widehat{p}\right)\right|_{\odot^kV}=0, \quad\forall k\ge 1, l\ge 0.$$
\end{proposition}
\begin{proof}
    The proof is similar to the proof of Proposition \ref{prop:mor2level} with one difference: in the middle two levels presenting 
    $$ \widehat{q}_\bullet\circ \widehat{\phi}-(-1)^{|\widehat{q}_{\bullet}|}\widehat{\phi}\circ \widehat{p}_{\bullet}-\widehat{q}\circ \widehat{\phi}_{\bullet}+(-1)^{|\widehat{\phi}_{\bullet}|}\widehat{\phi}_{\bullet}\circ \widehat{p},$$
    the $p^{*,*}$, $q^{*,*}$ components may not glue to the $\phi^{*,*}_{\bullet}$ components. However, those pieces correspond to  $(\phi \circ p)^{k,l}$ and $(q \circ \phi)^{k,l}$ along with one $\phi^{*,*}_{\bullet}$ and multiple $\phi^{*,*}$ components. Since $(\phi \circ p)^{k,l}=(q \circ \phi)^{k,l}$ from Proposition \ref{prop:mor2level}, those extra terms sum to zero as $|p^{*,*}|=1$.
\end{proof}

In practice, $\phi^{k,l}_{\bullet}$ is defined by counting connected rational holomorphic curves in the cobordism with a marked point passing through a cobordism between the constraints in the definition of $p_{\bullet},q_{\bullet}$. In our typical case of point constraint, the cobordism will be a path connecting the point constraints, where we have $|\widehat{\phi}_{\bullet}|=1$. In principle, we can consider the category consisting of pairs $(p,p_\bullet)$ with morphisms given by pairs $(\phi,\phi_\bullet)$ with a suitable definition of composition. Then the definition of orders is functorial. For our purpose, we only need the following property without the precise definition of a composition\footnote{It is easy to spell out the composition using the graph description.}.
\begin{proposition}\label{prop:order}
	Assume $\phi$ is a $BL_{\infty}$-morphism from $(V,p)$ to $(V',q)$ with pointed maps $p_{\bullet},q_{\bullet}$ of degree $0$ respectively, such that $p_{\bullet},q_{\bullet},\phi$ are compatible. Then for any $BL_{\infty}$ augmentation $\epsilon$ of $V'$, we have $O(V,\epsilon\circ \phi,p_{\bullet})\ge O(V',\epsilon,q_{\bullet})$.
\end{proposition}
\begin{proof}
	By the definition of compatibility, we compose $\widehat{F}_{\epsilon}$ in the front and $\widehat{F}_{-\epsilon\circ \phi}$ in the bottom to have
	\begin{equation}\label{eqn:homotopy}
	    \widehat{q}_{\bullet,\epsilon}\circ \widehat{F}_{\epsilon}\circ \widehat{\phi}\circ \widehat{F}_{-\epsilon\circ \phi}-\widehat{F}_{\epsilon}\circ \widehat{\phi}\circ \widehat{F}_{-\epsilon \circ \phi}\circ \widehat{p}_{\bullet,\epsilon\circ \phi}=\widehat{F}_{\epsilon}\circ \widehat{\phi}_{\bullet}\circ \widehat{F}_{-\epsilon \circ \phi}\circ \widehat{p}_{\epsilon\circ \phi}+\widehat{q}_{\epsilon}\circ \widehat{F}_{\epsilon}\circ \widehat{\phi}_{\bullet}\circ \widehat{F}_{-\epsilon\circ \phi}.
	\end{equation}

	We have a diagram (not commutative) 
	$$
		\xymatrix{
		\overline{S} V \ar[rr]^{\widehat{\ell}_{\bullet,\epsilon\circ \phi}} \ar[d]^{\widehat{\phi}^1_{\epsilon}} & &\bk \ar[d]^{\Id} \\
		\overline{S} V' \ar[rr]^{\widehat{\ell}_{\bullet,\epsilon}} & & \bk
	}	
	$$
	where $\widehat{\phi}^1_{\epsilon}$ is determined by the $L_\infty$ morphism $\phi^{k,1}_{\epsilon}$ since $\phi^{k,0}_{\epsilon}=0$. In other words, $\widehat{\phi}^1_{\epsilon}|_{\odot^k V} = \sum_{l\ge 1}\pi_{l,1}\circ \widehat{\phi}_{\epsilon}|_{\odot^k V}$, where $\pi_{l,1}$ is the projection from $EV$ to $\odot^l V$. We claim that the diagram is commutative up to homotopy $\widehat{\phi}^0_{\bullet,\epsilon}:\overline{S}V\to \bk$ defined by $\sum_{k\ge 1} \phi^{k,0}_{\bullet,\epsilon}$, where $\phi^{k,0}_{\bullet,\epsilon}=\pi_{\bk}\circ \widehat{F}_{\epsilon}\circ \widehat{\phi}_{\bullet}\circ \widehat{F}_{-\epsilon\circ \phi}|_{\odot^kV}$.  Indeed, this homotopy relation is exactly \eqref{eqn:homotopy} restricted to $\overline{S}V=\oplus_{k=1}^\infty \odot^k V$ and then projected to $\bk$. It is clear that $\widehat{\phi}^1_{\epsilon},\widehat{\phi}^0_{\bullet,\epsilon}$ preserve the length filtration,  therefore  $1\in \Ima \widehat{\ell}_{\bullet,\epsilon\circ \phi}|_{H_*(\overline{B}^kV,\widehat{\ell}_{\epsilon\circ \phi})}$ implies that $1\in \Ima \widehat{\ell}_{\bullet,\epsilon}|_{H_*(\overline{B}^kV',\widehat{\ell}_{\epsilon})}$. Hence $O(V,\epsilon\circ \phi,p_{\bullet})\ge O(V',\epsilon,q_{\bullet})$.
\end{proof}	
\begin{remark}
    There are various generalizations of $O(V,\epsilon,p_{\bullet})$, some of the associated spectral invariants are Siegel's higher symplectic capacities with multiple point constraints \cite{siegel2019higher}, see \cite[\S 5.1]{supp} for details of the construction and relation.
\end{remark}

\section{Rational symplectic field theory}\label{s3}
In this section, we explain the construction of rational symplectic field theory (RSFT) as $BL_\infty$ algebras. RSFT was original packaged into a Poisson algebra with a distinguished odd degree class $\bm{h}$ such that $\{\bm{h},\bm{h}\}=0$ in \cite[\S 2.1]{eliashberg2000introduction}. However for the purpose of building hierarchy functors from contact manifolds, it is useful to reformulate RSFT as $BL_\infty$ algebras. It is important to note that we will use same moduli spaces of holomorphic curves as the original RSFT but reinterpret the relations as other algebraic structures. 
\subsection{Notations on symplectic topology}
We first briefly recall the basics of symplectic and contact topology. A (co-oriented) contact manifold $(Y,\xi)$ is a $2n-1$ dimensional manifold with a (co-oriented) hyperplane distribution $\xi$ such that there is a one form $\alpha$ with $\xi=\ker \alpha$, $\alpha \wedge (\rd \alpha)^{n-1}\ne0$ and $\alpha$ induces the co-orientation on $\xi$. Such one form $\alpha$ is called a contact form and we will call $(Y,\alpha)$ a strict contact manifold. The manifold $Y$ will always assumed to be closed. Given a contact form $\alpha$, the Reeb vector field $R_{\alpha}$ is characterized by $\alpha(R_{\alpha})=1,\iota_{R_\alpha}\rd \alpha=0$. We say a contact form $\alpha$ is non-degenerate iff all Reeb orbits are non-degenerate. Any contact form can be perturbed into a non-degenerate contact form and in particular, every contact manifold admits non-degenerate contact forms. Throughout this paper $(Y,\alpha)$ is always assumed to be a strict contact manifold with a non-degenerate contact form unless specified otherwise. 

\begin{definition}
	A compact symplectic manifold $(X,\omega)$ with $\partial W=Y_-\sqcup Y_+$ is
	\begin{enumerate}
		\item a strong cobordism from $(Y_-,\xi_-)$ to $(Y_+,\xi_+)$ iff $\omega=\rd \lambda_{\pm}$ near $Y_{\pm}$ with $\xi_{\pm}=\ker \lambda_{\pm}$ such that, if we define $V_{\pm}$ by $\iota_{V_{\pm}}\omega=\lambda_{\pm}$, then $V_+$ points out along $Y_+$ and $V_-$ points in along $Y_-$;
		\item an exact cobordism from $(Y_-,\xi_-)$ to $(Y_+,\xi_+)$  if moreover $\omega=\rd \lambda$ on $X$. The vector field $V$ defined by $\iota_V \omega=\lambda$ is called the Liouville vector field;
		\item a Weinstein cobordism from $(Y_-,\xi_-)$ to $(Y_+,\xi_+)$  if moreover the Liouville vector field is gradient like for some Morse function $\phi$ with $Y_{\pm}$ as the regular level sets of maximum/minimum value for $\phi$.
	\end{enumerate}
\end{definition}

We say a cobordism $(X,\omega)$ from $(Y_-,\alpha_-)$ to $(Y,\alpha_+)$ is strict iff $\lambda_{\pm}|_{Y_{\pm}}=\alpha_{\pm}$. It is clear from definition that we can glue strict cobordisms to get a strict cobordism.  In general, given two exact cobordisms $W_1,W_2$ from $Y_1,Y_2$ to $Y_2,Y_3$ respectively, the composition $W_2\circ W_1$ from $Y_1$ to $Y_3$ is not uniquely defined, but up to homotopies of Liouville structures \cite[\S 11.2]{cieliebak2012stein}, it is well-defined. The central geometric object of our interests is the following cobordism category.

\begin{definition}
	The exact cobordism category of contact manifolds $\cont$ is defined to be the category whose objects are contact manifolds and morphisms are exact cobordisms up to homotopy. The composition is given by gluing cobordisms. We will use $\cont^{2k-1}$ to denote the subcategory of $2k-1$ dimensional contact manifolds. Similarly, we use $\cont_W$ to denote the Weinstein cobordism category and $\cont_S$ to denote the strong cobordism category. 
\end{definition}
All of the categories above have monoidal structures given by the disjoint union. It is clear that we have natural functors $\cont_{W}\to \cont \to \cont_S$, which are identities on the object level.

\begin{remark}
	There is a forgetful functor from $\cont$ to the cobordism category of almost contact manifolds, where the cobordisms are almost symplectic cobordisms. In the case of $\cont_W$, there is a forgetful map to the almost Weinstein cobordism category of almost contact manifolds. These are purely topological objects, the latter was studied thoroughly in \cite{bowden2014topology,bowden2015topology}.
\end{remark}

Roughly speaking, the principle in the symplectic cobordism category is that the complexity of contact geometry increases in the direction of the cobordism. In view of this, we can introduce the following category, which only remembers if there exists a cobordism.
\begin{definition}
	We define $\cont_{\le }$ to be the category of contact manifolds, such that there is at most one arrow between two contact manifolds and the arrow exists iff there is an exact cobordism. Similarly, we can define $\cont_{\le ,W}$ and $\cont_{\le,S}$.
\end{definition}

It is a natural question to ask whether $\cont_{\le }$ is a poset. It is clear that we only need to prove that $Y_1\le Y_2, Y_2\le Y_1$ implies that $Y_1=Y_2$. Unfortunately, this is not the case, as we may take $Y_1,Y_2$ as two different $3$-dimensional overtwisted contact manifolds \cite{etnyre2002symplectic} or suitable flexibly fillable contact manifolds. One extreme case is that $Y_1$ can be different from $Y_2$ (as \emph{smooth} manifolds) even if the cobordisms are inverse to each other \cite{courte2014contact}. However, we can mod out this ambiguity to get a poset. It is clear that the existence of (exact) cobordisms between $Y_1,Y_2$ in both directions defines an equivalence relation. We denote by $[Y]$ the corresponding equivalence class of the contact manifold $Y$.

\begin{definition}
	We define $\overline{\cont}_{\le}$ to be the poset, such that the objects are the equivalence classes of contact manifolds with respect to the above equivalence relation, and there is a morphism $[Y_1]\le [Y_2]$ iff there is an exact cobordism from $Y_1$ to $Y_2$. Similarly, we can define the posets $\overline{\cont}_{\le,W}$ and $\overline{\cont}_{\le,S}$.
\end{definition}
Under this condition, all three dimensional overtwisted contact manifolds become the same least object in $\overline{\cont}_{\le}$ \cite{etnyre2002symplectic}. In higher dimensions, overtwisted contact manifolds are least objects up to topological constrains \cite{eliashberg2015making}. It is clear that we have functors $\cont \to \cont_{\le} \to \overline{\cont}_{\le}$. The theme of this paper is constructing functors from $\cont$ to some totally ordered set. Since it always descends to $\overline{\cont}_{\le}$, results in this paper can be understood as some structures on the poset $\overline{\cont}_{\le}$. It is also a natural question on whether $\overline{\cont}_{\le}$ is a totally ordered set\footnote{For the total order, we will not consider $\emptyset$ as an object in $\overline{\cont}_{\le }$, as overtwisted contact manifolds and $\emptyset$ are not comparable in $\overline{\cont}_{\le }$, by well-known results.}, which is addressed in the negative \cite{gironella2021exact} using exact orbifolds.

An exact (Weinstein, strong) cobordism from $\emptyset$ to $Y$ is called an exact (Weinstein, strong) filling of $Y$. We also introduce a category $\cont_*$ as the under category under the empty set.
\begin{definition}
	The objects of $\cont_*$ are pairs $(Y,W)$, where $W$ is an exact filling of $Y$ up to homotopy. A morphism from $(Y_1,W_1)$ to $(Y_2,W_2)$ is an exact cobordism $X$ from $Y_1$ to $Y_2$ such that $X\circ W_1 =W_2$ up to homotopy, or equivalently an exact embedding of $W_1$ into $W_2$ up to homotopy. 
\end{definition}

\begin{example}
The	symplectic cohomology is a functor from $\cont_*$ to the category of BV algebras, where the functoriality follows from the Viterbo transfer map. The $S^1$-equivariant symplectic cochain complex is also a functor $\cont_*$ to the homotopy category of $S^1$-cochain complexes. The order of dilation and the order of semi-dilation in \cite[Corollary D]{zhou2019symplectic} are functors from $\cont_*$ to $\N\cup \{\infty\}$.
\end{example}

\subsection{Geometric setups for holomorphic curves}
As usual, an almost complex structure $J$ on the symplectization $(\R_s \times Y, \rd(e^s \alpha))$ is said to be admissible iff 
\begin{enumerate}
	\item $J$ is invariant under the $s$-translation and restricts to a tame almost complex structure on $(\xi = \ker \alpha,\rd \alpha)$,
	\item $J$ sends $\partial_s$ to the Reeb vector $R_\alpha$.
\end{enumerate}
Let $(W,\lambda)$ be an exact filling and $(X,\lambda)$ an exact cobordism. An almost complex structure $J$ on completions $(\widehat{W},\widehat{\lambda})$ or  $(\widehat{X},\widehat{\lambda})$ is admissible iff
\begin{enumerate}
	\item $J$ is tame for $\rd\widehat{\lambda}$,
	\item $J$ is admissible on cylindrical ends.
\end{enumerate}
Occasionally, we will also consider strong fillings $(W,\omega)$, where the definition of admissible almost complex structure on $\widehat{W}$ is similar. For each Reeb orbit $\gamma$, we can fix a basepoint $b_{\gamma}$ on the image. Now fix an admissible $J$, and consider two collections of Reeb orbits $\gamma_1^+, \ldots, \gamma_{s^+}^+$ and $\gamma_1^{-}, \ldots, \gamma_{s^-}^-$, possibly with duplicates. A pseudoholomorphic map in the symplectization $\R \times Y$ or completions $\widehat{W},\widehat{X}$ with positive asymptotics $\gamma_1^+,\ldots,\gamma^+_{s^+}$ and negative asymptotics $\gamma_1^-,\ldots, \gamma_{s^-}^-$ consists of:
\begin{enumerate}
	\item a sphere $\Sigma$, with a complex structure denoted by $j$,
	\item a collection of pairwise distinct points $z_1^+, \ldots, z^+_{s^+}, z^-_{1}, \ldots, z^-_{s^-} \in \Sigma$, each equipped with an asymptotic marker, i.e.\ a direction in the tangent sphere bundle $S_{z^{\pm}_i}\Sigma$,
	\item a map $\dot{\Sigma} \to \R \times Y, \widehat{W},\widehat{X}$ satisfying $\rd u \circ j = J \circ \rd u$, where $\dot{\Sigma}$ denotes the punctured Riemann surface $\Sigma \backslash \{z_1^+, \ldots, z^+_{s^+},z^-_{1},\ldots, z^-_{s^-}\}$,
	\item for each $z_i^+$ with corresponding polar coordinates $(r,\theta)$ around $z_i^+$ such that the asymptotic marker corresponds to $\theta=0$, we have 
	\begin{equation}\label{eqn:asymp1}
	\lim_{r\to 0}(\pi_{\R} \circ u)(re^{i\theta})=+\infty, \quad u_{z_i^+}(\theta):=\displaystyle \lim_{r\to 0}(\pi_Y \circ u)(re^{i\theta}) = \gamma^+_i\left(\frac{1}{2\pi}T_i^+\theta\right)
	\end{equation}
	where $T_i^+$ is the period of the parameterized orbit $\gamma^+_i$ and $\gamma^+_i(0)=b_{\gamma^+_i}$,
	\item  for each $z_i^-$ with corresponding polar coordinates $(r,\theta)$ compatible with asymptotic marker, we have 
	\begin{equation}\label{eqn:asymp2}
	\lim_{r\to 0}(\pi_{\R} \circ u)(re^{i\theta})=-\infty, \quad u_{z_i^-}(\theta):=\displaystyle \lim_{r\to 0}(\pi_Y \circ u)(re^{-i\theta}) = \gamma^-_i\left(-\frac{1}{2\pi}T_i^-\theta\right)
	\end{equation}
	 where $T_i^-$ is the period of  the parameterized orbit $\gamma^-_i$ and $\gamma^-_i(0)=b_{\gamma^+_i}$.
\end{enumerate}
A holomorphic curve is an equivalence of holomorphic maps modulo biholomorphisms of $\Sigma$ commuting with all the data. Throughout this paper, we will work with $\Z_2$-grading unless specified otherwise. Let $\Gamma^+=\{\gamma_1^+,\ldots,\gamma_{s^+}^+ \}$, $\Gamma^-=\{\gamma_1^-,\ldots,\gamma_{s^-}^- \}$ be two ordered sets of Reeb orbits possibly with duplicates. Choosing trivializations of $\xi$ over orbits in $\Gamma^+,\Gamma^-$, we can assign the Conley-Zehnder index $\mu_{CZ}(\gamma^{\pm}_i)$ to each orbit. With such trivialization, we have a relative first Chern class
$$c_1: H_2(Y,\Gamma^+\cup \Gamma^-;\Z)\to \Z,$$
similarly for $W$ and $X$. Let $A$ be a relative homology class representing the curve $u$, then the Fredholm index of the Cauchy Riemann operator at $u$ minus the dimension of the automorphism group (i.e.\ biholomorphisms of $\Sigma$ commuting with all the data) is the following,
$$\ind(u)=(n-3)(2-s^+-s^-)+\sum_{i=1}^{s^+}\mu_{CZ}(\gamma_i^+) - \sum_{i=1}^{s^-}\mu_{CZ}(\gamma_i^-) + 2c_1(A).$$

In this paper, we will consider the following moduli spaces. 
\begin{enumerate}
	\item $\cM_{Y,A}(\Gamma^+,\Gamma^-)$ is the moduli space of rational holomorphic curves in the symplectization, modulo automorphism and the $\R$ translation. The expected dimension is $\ind(u)-1$.
	\item $\cM_{W,A}(\Gamma^+,\emptyset)$ and $\cM_{X,A}(\Gamma^+,\Gamma^-)$ are the moduli spaces of rational holomorphic curves in the filling, respectively cobordism, modulo automorphism. The expected dimension is $\ind(u)$.
	\item $\cM_{Y,A,o}(\Gamma^+,\Gamma^-)$ is the moduli space of rational holomorphic curves with one interior marked point in the symplectization modulo automorphisms. Here the marked point is required to be mapped to $(0,o)\in \R \times Y$ for a point $o\in Y$. The expected dimension is $\ind(u)+2-2n$.
	\item $\cM_{W,A,o}(\Gamma^+,\emptyset)$ is the moduli space of rational holomorphic curves with one interior marked point in the filling modulo automorphism. And the marked point is required to be mapped to $o\in W$. The expected dimension is $\ind(u)+2-2n$.
	\item \label{path} $\cM_{X,A,\gamma}(\Gamma^+,\Gamma^-)$ is the moduli space of rational curves with one interior marked point in the exact cobordism $X$ modulo automorphism. The marked point is required to go through a path $\widehat{\gamma}$, which is the completion of a path $\gamma$ from a point in $Y_+$ to a point in $Y_-$, i.e.\ extension by constant maps in each slice in the cylindrical ends. The expected dimension is $\ind(u)+3-2n$.
\end{enumerate}
Another fact that is important for our later proof is that 
\begin{equation}\label{eqn:positive}
\int_{\Gamma^+} \alpha - \int_{\Gamma^-} \alpha \ge 0,
\end{equation}
whenever $\cM_{Y,A}(\Gamma^+,\Gamma^-)$ or $\cM_{Y,A,o}(\Gamma^+,\Gamma^-)$ are not empty. All of the moduli spaces above have a SFT building compactification by \cite{bourgeois2003compactness} denoted by $\overline{\cM}$. The orientation convention follows \cite{bourgeois2004coherent}, and we need to require that all asymptotic Reeb orbits are good \cite[Definition 11.6]{wendl2016lectures} to orient $\overline{\cM}$. One property of this convention is that if we switch two orbits $\gamma_1,\gamma_2$  that are next to each other in $\Gamma^+$ or $\Gamma^-$, the induced orientation is changed by $(-1)^{|\mu_{CZ}(\gamma_1)+n-3|\cdot|\mu_{CZ}(\gamma_2)+n-3|}$ \cite[\S 11.2]{wendl2016lectures}. In the following, we will count zero dimensional moduli spaces to define coefficients in the structural maps. First of all, this requires an orientation, hence we can only count when all asymptotic Reeb orbits are good\footnote{Alternatively, the count is evaluated in the fixed space of an orientation line with a group action, the appearance of a bad orbit is exactly when the group action is not trivial, see \cite[Remark 4.15]{pardon2019contact} for details.}. Next we need transversality, where the count is a honest count of orbifold points, or a virtual machinery, where the count is a count of weighted orbifold points in perturbed moduli spaces \cite{SFT,ishikawa2018construction} or an algebraic count after fixing some auxiliary data \cite{pardon2019contact}. For simplicity, we will just use $\#\overline{\cM}$ to denote the count.

\subsection{Contact homology algebra}\label{ss:CHA}
We will first recall the definition of the contact homology algebra. Let $V_{\alpha}$ denote the free $\Q$-module generated by formal variables $q_\gamma$ for each good orbit $\gamma$ of $(Y,\alpha)$. We grade $q_\gamma$ by $\mu_{CZ}(\gamma)+n-3$, which should be understood as a well-defined $\Z_2$ grading in general. The contact homology algebra $\CHA(Y)$ is the free symmetric algebra $S V_{\alpha}$. The differential is defined as follows.
\begin{equation}\label{eqn:partial}
\partial_l(q_{\gamma}) = \sum_{|[\Gamma]|=l} \#\overline{\cM}_{Y,A}(\{\gamma \},\Gamma) \frac{1}{\mu_{\Gamma}\kappa_{\Gamma}}q^{\Gamma}.
\end{equation}
Where the number $\#\overline{\cM}$ should be understood as a virtual counting once the virtual machinery is chosen, the same applies to the discussion on RSFT in the next subsection. The sum is over all multisets $[\Gamma]$, i.e.\ sets with duplicates, of size $l$. And $\Gamma$ is an ordered representation of $[\Gamma]$, e.g.\ $\Gamma=\{\underbrace{\eta_1,\ldots,\eta_1}_{i_1}, \ldots, \underbrace{\eta_m,\ldots,\eta_m}_{i_m}\}$ is an ordered set of good orbits with $\eta_i\ne \eta_j$ for $i\ne j$ and $\sum i_j = l$. We write $\mu_{\Gamma}=i_1!\ldots i_m!$ and $\kappa_{\Gamma}=\kappa^{i_1}_{\eta_1}\ldots \kappa^{i_m}_{\eta_m}$ is the product of multiplicities, and $q^{\Gamma}=q_{\eta_1} \ldots  q_{\eta_m}$. We mod out $\mu_{\Gamma}$ as we should count holomorphic curves with unordered punctures and mod out $\kappa_{\Gamma}$ to compensate that we have $\kappa_{\gamma}$ different ways to glue when we have a breaking at $\gamma$. The orientation property of $\overline{\cM}_{Y,A}(\{\gamma\},\Gamma)$ implies that \eqref{eqn:partial} is independent of the representative $\Gamma$. The differential on a single generators is defined by 
$$\partial (q_\gamma) = \sum_{l=0}^{\infty}\partial_l(q_\gamma),$$
which is always a finite sum by \eqref{eqn:positive}. Then the differential on $\CHA(Y)$ is defined by the Leibniz rule
$$\partial(q_{\gamma_1} \ldots  q_{\gamma_l})=\sum_{j=1}^l (-1)^{|q_{\gamma_1}|+\ldots + |q_{\gamma_{j-1}}|}q_{\gamma_1} \ldots  q_{\gamma_{j-1}} \partial(q_{\gamma_j}) q_{\gamma_{j+1}}\ldots  q_{\gamma_l}.$$
The relation $\partial^2=0$ follows from the boundary configuration of $\overline{\cM}_{Y,A}(\{\gamma\},\Gamma)$ with virtual dimension $1$ appropriately interpreted in the chosen virtual machinery. 

Given an exact cobordism $(X,\lambda)$ from $Y_-$ to $Y_+$, then we have an algebra map $\phi$ from $\CHA(Y_+)$ to $\CHA(Y_-)$, which on generator is defined by
$$\phi(q_{\gamma})=\sum_{l=0}^{\infty} \sum_{|[\Gamma]|=l}\# \overline{\cM}_{X,A}(\{\gamma\},\Gamma)\frac{1}{\mu_{\Gamma}\kappa_{\Gamma}}q^{\Gamma},$$
where $\Gamma$ is a collection of good orbits of $Y_-$. The boundary configuration of $\overline{\cM}_{X,A}(\{\gamma\},\Gamma)$ with virtual dimension $1$ gives the relation $\partial\circ \phi = \phi \circ \partial$. Then we have a functor from $\cont$ to the category of supercommutative differential graded algebras.
\begin{theorem}[{\cite[(1.22)]{pardon2019contact}}]
	The homology $H_*(\CHA(Y))$ above realized in VFC gives a monoidal functor from $\cont$ to the category of (super)commutative algebras.
\end{theorem}

\begin{remark}
    Using semi-global Kuranishi charts, Bao and Honda \cite{BH} gave an alternative definition of contact homology enjoying the same invariance and functorial properties.
\end{remark}

\subsection{Rational SFT as $BL_\infty$ algebras}
To assign $BL_\infty$ algebras to strict contact manifolds, we need to consider moduli spaces with multiple inputs and multiple outputs. In the following, we give an informal description of the $BL_\infty$ structure arising from counting holomorphic curves neglecting any transversality issues. We use $\overline{\cM}_{Y,A}(\Gamma^+,\Gamma^-)$ to denote the compactified moduli space of rational curves in class $A$ with positive asymptotics $\Gamma^+$ and negative asymptotics $\Gamma^-$ in the symplectization $\R \times Y$. Then we can define $p^{k,l}$ by\footnote{A previous version of this paper made a mistake with an extra $1/\mu_{\Gamma_+}$ in the coefficient.}
\begin{equation}\label{eqn:p}
p^{k,l}(q^{\Gamma^+})=\sum_{|[\Gamma^-]|=l} \# \overline{\cM}_{Y,A}(\Gamma^+,\Gamma^-)\frac{1}{\mu_{\Gamma^-}\kappa_{\Gamma^-}}q^{\Gamma^-}.
\end{equation}
Here $[\Gamma^-]$ is a multiset with $\Gamma^-$ an ordered set representative and $|\Gamma^+|=k$. In particular, the orientation property of $\overline{\cM}_{Y,A}(\Gamma^+,\Gamma^-)$ implies that $p^{k,l}$ is a map from $S^kV_\alpha$ to $S^lV_\alpha$. The count $\#\overline{\cM}$ is a virtual count, which will be made precise after the virtual setup in \S \ref{s:virtual}. For simplicity, we pretend the moduli spaces are cut out transversely and $\#\overline{\cM}$ is the geometric count of oriented orbifold points. When transversality and gluing holds, the boundary of the $1$-dimensional moduli spaces $\overline{\cM}_{Y,A}(\Gamma^+,\Gamma^-)$ yields that $\{p^{k,l}\}$ is a $BL_\infty$ algebra $\RSFT(Y)$ (showed schematically by the pictures); see Theorem \ref{thm:BL} for details.
\begin{figure}[H]
	\begin{center}
		\begin{tikzpicture}[scale=0.6]
		\draw (0,0) to [out=270, in=180] (1,-0.5) to [out=0, in=270] (2,0) to (2,-4) to [out=270, in=0] (1,-4.5) to [out=180,in=270] (0,-4) to (0,0);
		\draw[dashed] (0,0) to [out=90, in=180] (1,0.5) to [out=0,in=90] (2,0);
		\draw[dashed] (0,-4) to [out=90, in=180] (1,-3.5) to [out=0,in=90] (2,-4);
		
		\draw (0,0) to (0,4) to [out=270, in=180] (1,3.5) to [out=0,in=270] (2,4) to [out=270,in=180] (3,3) to [out=0,in=270] (4,4) to [out=270,in=180] (5,3.5) to [out=0,in=270] (6,4) to (6,0) to [out=270, in=0] (5,-0.5) to [out=180, in=270] (4,0) to [out=90, in=0] (3,1) to [out=180, in=90] (2,0);
		\draw (0,4) to [out=90,in=180] (1,4.5) to [out=0,in=90] (2,4);
		\draw (4,4) to [out=90,in=180] (5,4.5) to [out=0,in=90] (6,4);
		\draw[dashed] (4,0) to [out=90, in=180] (5,0.5) to [out=0,in=90] (6,0);

		\draw (4,0) to (4,-4) to [out=270,in=180] (5,-4.5) to [out=0,in=270] (6,-4) to [out=90,in=180] (7,-3) to [out=0,in=90] (8,-4) to [out=270,in=180] (9,-4.5) to [out=0,in=270] (10,-4) to (10,0) to [out=270,in=0] (9,-0.5) to [out=180,in=270] (8,0) to [out=270, in=0] (7,-1) to [out=180, in=270] (6,0); 
		\draw[dashed] (4,-4) to [out=90, in=180] (5,-3.5) to [out=0,in=90] (6,-4);
		\draw[dashed] (8,-4) to [out=90, in=180] (9,-3.5) to [out=0,in=90] (10,-4);
		
		\draw (8,0) to (8,4) to [out=270, in=180] (9,3.5) to [out=0,in=270] (10,4) to (10,0);
		\draw[dashed] (8,0) to [out=90, in=180] (9,0.5) to [out=0,in=90] (10,0);
		\draw (8,4) to [out=90, in=180] (9,4.5) to [out=0,in=90] (10,4);
		
		\node at (1,-2) {T};
		\node at (9,2) {T};
		\node at (3,2) {$p^{2,2}$};
		\node at (7,-2) {$p^{2,2}$};
		\end{tikzpicture}
	\end{center}
	\caption{$\widehat{p}\circ \widehat{p}=0$, where $T$ stands for a trivial cylinder.}
\end{figure}
\begin{remark}\label{rmk:coeff}
In the original formalism of the full SFT by Eliashberg, Givental and Hofer \cite[\S 2.2.3]{SFT}, the Hamiltonian $\bH$ is defined as 
$$\bH=\sum_{A,[\Gamma_+],[\Gamma_-]} \frac{\hbar^{g-1}}{\mu_{\Gamma_+}\mu_{\Gamma_-}\kappa_{\Gamma_+}\kappa_{\Gamma_-}}\#\overline{\cM}_{Y,g,A}(\Gamma_+,\Gamma_-)q^{\Gamma_-}p^{\Gamma_+}$$
in the Weyl algebra $\cW$ of power series in the variables $\hbar, p_{\gamma}$ with coefficients polynomial in the variables $q_{\gamma}$. $\overline{\cM}_{Y,g,A}(\Gamma_+,\Gamma_-)$ is the genus $g$ analogue of $\overline{\cM}_{Y,A}(\Gamma_+,\Gamma_-)$. $\cW$ is equipped with the associative product $*$ in which all variables super-commute according to their gradings except for the variables $p_{\gamma}$, $q_{\gamma}$ corresponding to the same Reeb orbit $\gamma$, for which we have
$$p_{\gamma}*q_{\gamma}- (-1)^{|p_{\gamma}||q_{\gamma}|}q_{\gamma} * p_{\gamma}= \kappa_{\gamma}\hbar.$$
In the $BV_\infty$ reformalism of the full SFT introduced by Cieliebak and Latschev \cite[\S 6]{cieliebak2009role}, the differential $D_{SFT}$ on $SV_{\alpha}[[\hbar]]$ is defined as 
$$ D_{\SFT} (s) = \sum_{A,[\Gamma_+],[\Gamma_-]}   \frac{\hbar^{g-1}}{\mu_{\Gamma_+}\mu_{\Gamma_-}\kappa_{\Gamma_+}\kappa_{\Gamma_-}}\#\overline{\cM}_{Y,g,A}(\Gamma_+,\Gamma_-)q^{\Gamma_-} \prod_{\gamma_i\in \Gamma_+}(\kappa_{\gamma_i}\hbar \frac{\partial}{\partial q_{\gamma_i}}) s.$$
In view of the relation between the Weyl algebra formalism and the $IBL_\infty$ formalism by Cieliebak, Fukaya and Latschev following \cite[(7.4)]{cieliebak2015homological}, the operation $p^{k,l,g}$ is defined by
\begin{eqnarray*}
    p^{k,l,g}(q^{\Gamma_+}) & = &\sum_{A,[\Gamma_-]} \frac{1}{\mu_{\Gamma_+}\mu_{\Gamma_-}\kappa_{\Gamma_+}\kappa_{\Gamma_-}}\#\overline{\cM}_{Y,g,A}(\Gamma_+,\Gamma_-)q^{\Gamma_-} \prod_{\gamma_i\in \Gamma_+}(\kappa_{\gamma_i}\hbar \frac{\partial}{\partial q_{\gamma_i}})) q^{\Gamma^+}\\
    & = & \sum_{A,[\Gamma_-]}\frac{1}{\mu_{\Gamma_-}\kappa_{\Gamma_-}}\#\overline{\cM}_{Y,g,A}(\Gamma_+,\Gamma_-)q^{\Gamma_-}.
\end{eqnarray*}
In view of \cite[Corollary 5.12]{supp}, since $IBL_\infty$ algebras restrict to $BL_\infty$ algebras, our rational SFT formalism in \eqref{eqn:p} has consistent coefficients with \cite{cieliebak2009role,cieliebak2015homological,SFT}. Heuristically, $\mu_{\Gamma_-}\kappa_{\Gamma_-}$ is the order of the ``isotropy group" of the output orbit set $\Gamma_-$.
\end{remark}

Similarly for a strict exact cobordism $X$ from $Y_-$ to $Y_+$, by considering the moduli spaces $\overline{\cM}_{X,A}(\Gamma^+,\Gamma^-)$ of rational curves in $X$, we can define a $BL_{\infty}$ morphism from $\RSFT(Y_+)$ to $\RSFT(Y_-)$ by the following,
\begin{equation}\label{eqn:phi}
\phi^{k,l}(q^{\Gamma^+})=\sum_{|[\Gamma^-]|=l}\#\overline{\cM}_{X,A}(\Gamma^+,\Gamma^-)\frac{1}{\mu_{\Gamma^-}\kappa_{\Gamma^-}}q^{\Gamma^-},
\end{equation}
where $|\Gamma^+|=k$. Then the boundary of the $1$-dimensional moduli spaces $\overline{\cM}_{X,A}(\Gamma^+,\Gamma^-)$ yields that $\{\phi^{k,l}\}$ is a $BL_\infty$ morphism $\RSFT(Y_+)\to \RSFT(Y_-)$. In the following figure we indicate the cobordism level with `C'.
\begin{figure}[H]
	\begin{center}
		\begin{tikzpicture}[scale=0.6]
		\draw (0,0) to [out=270, in=180] (1,-0.5) to [out=0, in=270] (2,0) to (2,-4) to [out=270, in=0] (1,-4.5) to [out=180,in=270] (0,-4) to (0,0);
		\draw[dashed] (0,0) to [out=90, in=180] (1,0.5) to [out=0,in=90] (2,0);
		\draw[dashed] (0,-4) to [out=90, in=180] (1,-3.5) to [out=0,in=90] (2,-4);
		
		\draw (0,0) to (0,4) to [out=270, in=180] (1,3.5) to [out=0,in=270] (2,4) to [out=270,in=180] (3,3) to [out=0,in=270] (4,4) to [out=270,in=180] (5,3.5) to [out=0,in=270] (6,4) to (6,0) to [out=270, in=0] (5,-0.5) to [out=180, in=270] (4,0) to [out=90, in=0] (3,1) to [out=180, in=90] (2,0);
		\draw (0,4) to [out=90,in=180] (1,4.5) to [out=0,in=90] (2,4);
		\draw (4,4) to [out=90,in=180] (5,4.5) to [out=0,in=90] (6,4);
		\draw[dashed] (4,0) to [out=90, in=180] (5,0.5) to [out=0,in=90] (6,0);

		\draw (4,0) to (4,-4) to [out=270,in=180] (5,-4.5) to [out=0,in=270] (6,-4) to [out=90,in=180] (7,-3) to [out=0,in=90] (8,-4) to [out=270,in=180] (9,-4.5) to [out=0,in=270] (10,-4) to (10,0) to [out=270,in=0] (9,-0.5) to [out=180,in=270] (8,0) to [out=270, in=0] (7,-1) to [out=180, in=270] (6,0); 
		\draw[dashed] (4,-4) to [out=90, in=180] (5,-3.5) to [out=0,in=90] (6,-4);
		\draw[dashed] (8,-4) to [out=90, in=180] (9,-3.5) to [out=0,in=90] (10,-4);
		
		\draw (8,0) to (8,4) to [out=270, in=180] (9,3.5) to [out=0,in=270] (10,4) to (10,0);
		\draw[dashed] (8,0) to [out=90, in=180] (9,0.5) to [out=0,in=90] (10,0);
		\draw (8,4) to [out=90, in=180] (9,4.5) to [out=0,in=90] (10,4);
		
		\node at (1,-2) {T};
		\node at (9,2) {$\phi^{1,1}$};
		\node at (3,2) {$\phi^{2,2}$};
		\node at (7,-2) {$p^{2,2}$};
		
		\draw[dashed] (-1,0) to (11,0);
		\draw[dashed] (-1,4) to (11,4);
		\draw[<->] (-0.5,0) to (-0.5,4);
		\node at (-1,2) {C};
		\end{tikzpicture}
		\begin{tikzpicture}[scale=0.6]
		\draw (0,0) to [out=270, in=180] (1,-0.5) to [out=0, in=270] (2,0) to (2,-4) to [out=270, in=0] (1,-4.5) to [out=180,in=270] (0,-4) to (0,0);
		\draw[dashed] (0,0) to [out=90, in=180] (1,0.5) to [out=0,in=90] (2,0);
		\draw[dashed] (0,-4) to [out=90, in=180] (1,-3.5) to [out=0,in=90] (2,-4);
		
		\draw (0,0) to (0,4) to [out=270, in=180] (1,3.5) to [out=0,in=270] (2,4) to [out=270,in=180] (3,3) to [out=0,in=270] (4,4) to [out=270,in=180] (5,3.5) to [out=0,in=270] (6,4) to (6,0) to [out=270, in=0] (5,-0.5) to [out=180, in=270] (4,0) to [out=90, in=0] (3,1) to [out=180, in=90] (2,0);
		\draw (0,4) to [out=90,in=180] (1,4.5) to [out=0,in=90] (2,4);
		\draw (4,4) to [out=90,in=180] (5,4.5) to [out=0,in=90] (6,4);
		\draw[dashed] (4,0) to [out=90, in=180] (5,0.5) to [out=0,in=90] (6,0);

		\draw (4,0) to (4,-4) to [out=270,in=180] (5,-4.5) to [out=0,in=270] (6,-4) to [out=90,in=180] (7,-3) to [out=0,in=90] (8,-4) to [out=270,in=180] (9,-4.5) to [out=0,in=270] (10,-4) to (10,0) to [out=270,in=0] (9,-0.5) to [out=180,in=270] (8,0) to [out=270, in=0] (7,-1) to [out=180, in=270] (6,0); 
		\draw[dashed] (4,-4) to [out=90, in=180] (5,-3.5) to [out=0,in=90] (6,-4);
		\draw[dashed] (8,-4) to [out=90, in=180] (9,-3.5) to [out=0,in=90] (10,-4);
		
		\draw (8,0) to (8,4) to [out=270, in=180] (9,3.5) to [out=0,in=270] (10,4) to (10,0);
		\draw[dashed] (8,0) to [out=90, in=180] (9,0.5) to [out=0,in=90] (10,0);
		\draw (8,4) to [out=90, in=180] (9,4.5) to [out=0,in=90] (10,4);
		
		\node at (1,-2) {$\phi^{1,1}$};
		\node at (9,2) {T};
		\node at (3,2) {$p^{2,2}$};
		\node at (7,-2) {$\phi^{2,2}$};
		
		\draw[dashed] (-1,0) to (11,0);
		\draw[dashed] (-1,-4) to (11,-4);
		\draw[<->] (-0.5,0) to (-0.5,-4);
		\node at (-1,-2) {C};
		\end{tikzpicture}
	\end{center}
    \caption{$\widehat{\phi}\circ \widehat{p}=\widehat{p}\circ \widehat{\phi}$}
\end{figure}

If we fix a point $o$ in $Y$, by considering moduli spaces $\overline{\cM}_{Y,A,o}(\Gamma^+,\Gamma^-)$, we can define a pointed morphism $p_{\bullet}$ by
\begin{equation}\label{eqn:p.}
p^{k,l}_{\bullet}(q^{\Gamma^+})=\sum_{|[\Gamma^-]|=l} \# \overline{\cM}_{Y,A,o}(\Gamma^+,\Gamma^-)\frac{1}{\mu_{\Gamma^-}\kappa_{\Gamma^-}}q^{\Gamma^-}.
\end{equation}
Then the boundary of the $1$-dimensional moduli spaces $\overline{\cM}_{Y,A,o}(\Gamma^+,\Gamma^-)$ yields that $\{p^{k,l}_{\bullet}\}$ is a pointed morphism of degree $0$. Note that $\overline{\cM}_{Y,A,o}(\Gamma^+,\Gamma^-)$ consists of holomorphic curves with a point constraint in the symplectization with a $s$-independent $J$, therefore in the level containing an element of $\overline{\cM}_{Y,A,o}(\Gamma^+,\Gamma^-)$ in a rigid breaking, there is only one nontrivial component.
\begin{figure}[H]
	\begin{center}
		\begin{tikzpicture}[scale=0.6]
		\draw (0,0) to [out=270, in=180] (1,-0.5) to [out=0, in=270] (2,0) to (2,-4) to [out=270, in=0] (1,-4.5) to [out=180,in=270] (0,-4) to (0,0);
		\draw[dashed] (0,0) to [out=90, in=180] (1,0.5) to [out=0,in=90] (2,0);
		\draw[dashed] (0,-4) to [out=90, in=180] (1,-3.5) to [out=0,in=90] (2,-4);
		
		\draw (0,0) to (0,4) to [out=270, in=180] (1,3.5) to [out=0,in=270] (2,4) to [out=270,in=180] (3,3) to [out=0,in=270] (4,4) to [out=270,in=180] (5,3.5) to [out=0,in=270] (6,4) to (6,0) to [out=270, in=0] (5,-0.5) to [out=180, in=270] (4,0) to [out=90, in=0] (3,1) to [out=180, in=90] (2,0);
		\draw (0,4) to [out=90,in=180] (1,4.5) to [out=0,in=90] (2,4);
		\draw (4,4) to [out=90,in=180] (5,4.5) to [out=0,in=90] (6,4);
		\draw[dashed] (4,0) to [out=90, in=180] (5,0.5) to [out=0,in=90] (6,0);

		\draw (4,0) to (4,-4) to [out=270,in=180] (5,-4.5) to [out=0,in=270] (6,-4) to [out=90,in=180] (7,-3) to [out=0,in=90] (8,-4) to [out=270,in=180] (9,-4.5) to [out=0,in=270] (10,-4) to (10,0) to [out=270,in=0] (9,-0.5) to [out=180,in=270] (8,0) to [out=270, in=0] (7,-1) to [out=180, in=270] (6,0); 
		\draw[dashed] (4,-4) to [out=90, in=180] (5,-3.5) to [out=0,in=90] (6,-4);
		\draw[dashed] (8,-4) to [out=90, in=180] (9,-3.5) to [out=0,in=90] (10,-4);
		
		\draw (8,0) to (8,4) to [out=270, in=180] (9,3.5) to [out=0,in=270] (10,4) to (10,0);
		\draw[dashed] (8,0) to [out=90, in=180] (9,0.5) to [out=0,in=90] (10,0);
		\draw (8,4) to [out=90, in=180] (9,4.5) to [out=0,in=90] (10,4);
		
		\node at (1,-2) {T};
		\node at (9,2) {T};
		\node at (3,2) {$p^{2,2}_\bullet$};
		\node at (7,-2) {$p^{2,2}$};
		
		\node at (2,3) [circle,fill,inner sep=1.5pt] {};
		\end{tikzpicture}
		\hspace{1cm}
		\begin{tikzpicture}[scale=0.6]
		\draw (0,0) to [out=270, in=180] (1,-0.5) to [out=0, in=270] (2,0) to (2,-4) to [out=270, in=0] (1,-4.5) to [out=180,in=270] (0,-4) to (0,0);
		\draw[dashed] (0,0) to [out=90, in=180] (1,0.5) to [out=0,in=90] (2,0);
		\draw[dashed] (0,-4) to [out=90, in=180] (1,-3.5) to [out=0,in=90] (2,-4);
		
		\draw (0,0) to (0,4) to [out=270, in=180] (1,3.5) to [out=0,in=270] (2,4) to [out=270,in=180] (3,3) to [out=0,in=270] (4,4) to [out=270,in=180] (5,3.5) to [out=0,in=270] (6,4) to (6,0) to [out=270, in=0] (5,-0.5) to [out=180, in=270] (4,0) to [out=90, in=0] (3,1) to [out=180, in=90] (2,0);
		\draw (0,4) to [out=90,in=180] (1,4.5) to [out=0,in=90] (2,4);
		\draw (4,4) to [out=90,in=180] (5,4.5) to [out=0,in=90] (6,4);
		\draw[dashed] (4,0) to [out=90, in=180] (5,0.5) to [out=0,in=90] (6,0);

		\draw (4,0) to (4,-4) to [out=270,in=180] (5,-4.5) to [out=0,in=270] (6,-4) to [out=90,in=180] (7,-3) to [out=0,in=90] (8,-4) to [out=270,in=180] (9,-4.5) to [out=0,in=270] (10,-4) to (10,0) to [out=270,in=0] (9,-0.5) to [out=180,in=270] (8,0) to [out=270, in=0] (7,-1) to [out=180, in=270] (6,0); 
		\draw[dashed] (4,-4) to [out=90, in=180] (5,-3.5) to [out=0,in=90] (6,-4);
		\draw[dashed] (8,-4) to [out=90, in=180] (9,-3.5) to [out=0,in=90] (10,-4);
		
		\draw (8,0) to (8,4) to [out=270, in=180] (9,3.5) to [out=0,in=270] (10,4) to (10,0);
		\draw[dashed] (8,0) to [out=90, in=180] (9,0.5) to [out=0,in=90] (10,0);
		\draw (8,4) to [out=90, in=180] (9,4.5) to [out=0,in=90] (10,4);
		
		\node at (1,-2) {T};
		\node at (9,2) {T};
		\node at (3,2) {$p^{2,2}$};
		\node at (7,-2) {$p_\bullet^{2,2}$};
		\node at (5,-3) [circle,fill,inner sep=1.5pt] {};
		\end{tikzpicture}
	\end{center}
	\caption{$\widehat{p}_\bullet\circ \widehat{p}=\widehat{p}\circ \widehat{p}_\bullet$}
\end{figure}
For a strict exact cobordism $X$ from $Y_-$ to $Y_+$, if we choose a path $\gamma$ from $o_-\in Y_-$ to $o_+\in Y_+$. then we can complete the path $\gamma$ to a proper $\widehat{\gamma}$ path in $\widehat{X}$ by constants in the cylindrical ends. Then we claim that the pointed morphisms $p_{\bullet},q_{\bullet}$ determined by $o_-,o_+$ and the $BL_\infty$ morphism $\phi$ are compatible, with $\phi_{\bullet}$ given by
\begin{equation}\label{eqn:phi.}
\phi^{k,l}_{\bullet}(q^{\Gamma^+})=\sum_{|[\Gamma^-]|=l}\#\overline{\cM}_{X,A,\gamma}(\Gamma^+,\Gamma^-)\frac{1}{\mu_{\Gamma^-}\kappa_{\Gamma^-}}q^{\Gamma^-}.
\end{equation}

\begin{figure}[H]
	\begin{center}
		\begin{tikzpicture}[scale=0.6]
		\draw (0,0) to [out=270, in=180] (1,-0.5) to [out=0, in=270] (2,0) to (2,-4) to [out=270, in=0] (1,-4.5) to [out=180,in=270] (0,-4) to (0,0);
		\draw[dashed] (0,0) to [out=90, in=180] (1,0.5) to [out=0,in=90] (2,0);
		\draw[dashed] (0,-4) to [out=90, in=180] (1,-3.5) to [out=0,in=90] (2,-4);
		
		\draw (0,0) to (0,4) to [out=270, in=180] (1,3.5) to [out=0,in=270] (2,4) to [out=270,in=180] (3,3) to [out=0,in=270] (4,4) to [out=270,in=180] (5,3.5) to [out=0,in=270] (6,4) to (6,0) to [out=270, in=0] (5,-0.5) to [out=180, in=270] (4,0) to [out=90, in=0] (3,1) to [out=180, in=90] (2,0);
		\draw (0,4) to [out=90,in=180] (1,4.5) to [out=0,in=90] (2,4);
		\draw (4,4) to [out=90,in=180] (5,4.5) to [out=0,in=90] (6,4);
		\draw[dashed] (4,0) to [out=90, in=180] (5,0.5) to [out=0,in=90] (6,0);

		\draw (4,0) to (4,-4) to [out=270,in=180] (5,-4.5) to [out=0,in=270] (6,-4) to [out=90,in=180] (7,-3) to [out=0,in=90] (8,-4) to [out=270,in=180] (9,-4.5) to [out=0,in=270] (10,-4) to (10,0) to [out=270,in=0] (9,-0.5) to [out=180,in=270] (8,0) to [out=270, in=0] (7,-1) to [out=180, in=270] (6,0); 
		\draw[dashed] (4,-4) to [out=90, in=180] (5,-3.5) to [out=0,in=90] (6,-4);
		\draw[dashed] (8,-4) to [out=90, in=180] (9,-3.5) to [out=0,in=90] (10,-4);
		
		\draw (8,0) to (8,4) to [out=270, in=180] (9,3.5) to [out=0,in=270] (10,4) to (10,0);
		\draw[dashed] (8,0) to [out=90, in=180] (9,0.5) to [out=0,in=90] (10,0);
		\draw (8,4) to [out=90, in=180] (9,4.5) to [out=0,in=90] (10,4);
		
		\node at (1,-2) {T};
		\node at (9,2) {$\phi^{1,1}$};
		\node at (3,2) {$\phi^{2,2}$};
		\node at (7,-2) {$q_\bullet^{2,2}$};
		\node at (5,-1) [circle,fill,inner sep=1.5pt] {};
		
		\draw[dashed] (-1,0) to (11,0);
		\draw[dashed] (-1,4) to (11,4);
		\draw[<->] (-0.5,0) to (-0.5,4);
		\node at (-1,2) {C};
		\end{tikzpicture}
		
		\medskip
		
		\begin{tikzpicture}[scale=0.6]
		\draw (0,0) to [out=270, in=180] (1,-0.5) to [out=0, in=270] (2,0) to (2,-4) to [out=270, in=0] (1,-4.5) to [out=180,in=270] (0,-4) to (0,0);
		\draw[dashed] (0,0) to [out=90, in=180] (1,0.5) to [out=0,in=90] (2,0);
		\draw[dashed] (0,-4) to [out=90, in=180] (1,-3.5) to [out=0,in=90] (2,-4);
		
		\draw (0,0) to (0,4) to [out=270, in=180] (1,3.5) to [out=0,in=270] (2,4) to [out=270,in=180] (3,3) to [out=0,in=270] (4,4) to [out=270,in=180] (5,3.5) to [out=0,in=270] (6,4) to (6,0) to [out=270, in=0] (5,-0.5) to [out=180, in=270] (4,0) to [out=90, in=0] (3,1) to [out=180, in=90] (2,0);
		\draw (0,4) to [out=90,in=180] (1,4.5) to [out=0,in=90] (2,4);
		\draw (4,4) to [out=90,in=180] (5,4.5) to [out=0,in=90] (6,4);
		\draw[dashed] (4,0) to [out=90, in=180] (5,0.5) to [out=0,in=90] (6,0);

		\draw (4,0) to (4,-4) to [out=270,in=180] (5,-4.5) to [out=0,in=270] (6,-4) to [out=90,in=180] (7,-3) to [out=0,in=90] (8,-4) to [out=270,in=180] (9,-4.5) to [out=0,in=270] (10,-4) to (10,0) to [out=270,in=0] (9,-0.5) to [out=180,in=270] (8,0) to [out=270, in=0] (7,-1) to [out=180, in=270] (6,0); 
		\draw[dashed] (4,-4) to [out=90, in=180] (5,-3.5) to [out=0,in=90] (6,-4);
		\draw[dashed] (8,-4) to [out=90, in=180] (9,-3.5) to [out=0,in=90] (10,-4);
		
		\draw (8,0) to (8,4) to [out=270, in=180] (9,3.5) to [out=0,in=270] (10,4) to (10,0);
		\draw[dashed] (8,0) to [out=90, in=180] (9,0.5) to [out=0,in=90] (10,0);
		\draw (8,4) to [out=90, in=180] (9,4.5) to [out=0,in=90] (10,4);
		
		\node at (1,-2) {$\phi^{1,1}$};
		\node at (9,2) {T};
		\node at (3,2) {$p_\bullet^{2,2}$};
		\node at (7,-2) {$\phi^{2,2}$};
		\node at (5,1) [circle,fill,inner sep=1.5pt] {};
		
		\draw[dashed] (-1,0) to (11,0);
		\draw[dashed] (-1,-4) to (11,-4);
		\draw[<->] (-0.5,0) to (-0.5,-4);
		\node at (-1,-2) {C};
		\end{tikzpicture}
		\offinterlineskip
		\begin{tikzpicture}[scale=0.6]
		\draw (0,0) to [out=270, in=180] (1,-0.5) to [out=0, in=270] (2,0) to (2,-4) to [out=270, in=0] (1,-4.5) to [out=180,in=270] (0,-4) to (0,0);
		\draw[dashed] (0,0) to [out=90, in=180] (1,0.5) to [out=0,in=90] (2,0);
		\draw[dashed] (0,-4) to [out=90, in=180] (1,-3.5) to [out=0,in=90] (2,-4);
		
		\draw (0,0) to (0,4) to [out=270, in=180] (1,3.5) to [out=0,in=270] (2,4) to [out=270,in=180] (3,3) to [out=0,in=270] (4,4) to [out=270,in=180] (5,3.5) to [out=0,in=270] (6,4) to (6,0) to [out=270, in=0] (5,-0.5) to [out=180, in=270] (4,0) to [out=90, in=0] (3,1) to [out=180, in=90] (2,0);
		\draw (0,4) to [out=90,in=180] (1,4.5) to [out=0,in=90] (2,4);
		\draw (4,4) to [out=90,in=180] (5,4.5) to [out=0,in=90] (6,4);
		\draw[dashed] (4,0) to [out=90, in=180] (5,0.5) to [out=0,in=90] (6,0);

		\draw (4,0) to (4,-4) to [out=270,in=180] (5,-4.5) to [out=0,in=270] (6,-4) to [out=90,in=180] (7,-3) to [out=0,in=90] (8,-4) to [out=270,in=180] (9,-4.5) to [out=0,in=270] (10,-4) to (10,0) to [out=270,in=0] (9,-0.5) to [out=180,in=270] (8,0) to [out=270, in=0] (7,-1) to [out=180, in=270] (6,0); 
		\draw[dashed] (4,-4) to [out=90, in=180] (5,-3.5) to [out=0,in=90] (6,-4);
		\draw[dashed] (8,-4) to [out=90, in=180] (9,-3.5) to [out=0,in=90] (10,-4);
		
		\draw (8,0) to (8,4) to [out=270, in=180] (9,3.5) to [out=0,in=270] (10,4) to (10,0);
		\draw[dashed] (8,0) to [out=90, in=180] (9,0.5) to [out=0,in=90] (10,0);
		\draw (8,4) to [out=90, in=180] (9,4.5) to [out=0,in=90] (10,4);
		
		\node at (1,-2) {T};
		\node at (9,2) {$\phi^{1,1}$};
		\node at (3,2) {$\phi_\bullet^{2,2}$};
		\node at (7,-2) {$q^{2,2}$};
		\node at (5,1) [circle,fill,inner sep=1.5pt] {};
		
		\draw[dashed] (-1,0) to (11,0);
		\draw[dashed] (-1,4) to (11,4);
		\draw[<->] (-0.5,0) to (-0.5,4);
		\node at (-1,2) {C};
		\end{tikzpicture}
		\begin{tikzpicture}[scale=0.6]
		\draw (0,0) to [out=270, in=180] (1,-0.5) to [out=0, in=270] (2,0) to (2,-4) to [out=270, in=0] (1,-4.5) to [out=180,in=270] (0,-4) to (0,0);
		\draw[dashed] (0,0) to [out=90, in=180] (1,0.5) to [out=0,in=90] (2,0);
		\draw[dashed] (0,-4) to [out=90, in=180] (1,-3.5) to [out=0,in=90] (2,-4);
		
		\draw (0,0) to (0,4) to [out=270, in=180] (1,3.5) to [out=0,in=270] (2,4) to [out=270,in=180] (3,3) to [out=0,in=270] (4,4) to [out=270,in=180] (5,3.5) to [out=0,in=270] (6,4) to (6,0) to [out=270, in=0] (5,-0.5) to [out=180, in=270] (4,0) to [out=90, in=0] (3,1) to [out=180, in=90] (2,0);
		\draw (0,4) to [out=90,in=180] (1,4.5) to [out=0,in=90] (2,4);
		\draw (4,4) to [out=90,in=180] (5,4.5) to [out=0,in=90] (6,4);
		\draw[dashed] (4,0) to [out=90, in=180] (5,0.5) to [out=0,in=90] (6,0);

		\draw (4,0) to (4,-4) to [out=270,in=180] (5,-4.5) to [out=0,in=270] (6,-4) to [out=90,in=180] (7,-3) to [out=0,in=90] (8,-4) to [out=270,in=180] (9,-4.5) to [out=0,in=270] (10,-4) to (10,0) to [out=270,in=0] (9,-0.5) to [out=180,in=270] (8,0) to [out=270, in=0] (7,-1) to [out=180, in=270] (6,0); 
		\draw[dashed] (4,-4) to [out=90, in=180] (5,-3.5) to [out=0,in=90] (6,-4);
		\draw[dashed] (8,-4) to [out=90, in=180] (9,-3.5) to [out=0,in=90] (10,-4);
		
		\draw (8,0) to (8,4) to [out=270, in=180] (9,3.5) to [out=0,in=270] (10,4) to (10,0);
		\draw[dashed] (8,0) to [out=90, in=180] (9,0.5) to [out=0,in=90] (10,0);
		\draw (8,4) to [out=90, in=180] (9,4.5) to [out=0,in=90] (10,4);
		
		\node at (1,-2) {$\phi_\bullet^{1,1}$};
		\node at (9,2) {T};
		\node at (3,2) {$p^{2,2}$};
		\node at (7,-2) {$\phi^{2,2}$};
		\node at (1,-3) [circle,fill,inner sep=1.5pt] {};
		
		\draw[dashed] (-1,0) to (11,0);
		\draw[dashed] (-1,-4) to (11,-4);
		\draw[<->] (-0.5,0) to (-0.5,-4);
		\node at (-1,-2) {C};
		\end{tikzpicture}
	\end{center}
    \caption{$\widehat{q}_\bullet\circ \widehat{\phi}-\widehat{\phi}\circ \widehat{p}_{\bullet}=\widehat{\phi}_{\bullet}\circ \widehat{p}+\widehat{q}\circ \widehat{\phi}_{\bullet}$}
\end{figure}
In order to turn the above informal discussion into a rigorous construction, we need to make sense of $\#\overline{\cM}$ such that they have the desired relations. The main theorem of this section is that after fixing auxiliary choices depending on the choice of virtual machinery, we almost have a functor from strict contact cobordism category (with auxiliary choices) to the category of $BL_\infty$ algebras\footnote{The composition is not discussed, nor is needed for our applications.}.
\begin{theorem}\label{thm:BL}
Let $(Y,\alpha)$ be a strict contact manifold with a non-degenerate contact form, then we have the following.
\begin{enumerate}
	\item\label{BL:1} There exists a non-empty set of auxiliary data $\Theta$, such that for each $\theta \in \Theta$ we have a $BL_\infty$ algebra $p_{\theta}$ (Definition \ref{def:BL}) on $V_{\alpha}$.
	\item\label{BL:2} For any point $o\in Y$, there exists a set of auxiliary data $\Theta_o$ with a surjective map $\Theta_o\to \Theta$, such that for any $\theta_o\in \Theta_o$, we have a pointed map $p_{\bullet,\theta_o}$ (Definition \ref{def:pointedmap}) for $p_{\theta}$, where $\theta$ is the image of $\theta_o$ in $\Theta_o\to \Theta$.
	\item\label{BL:3} Assume there is a strict exact cobordism $X$ from $(Y',\alpha')$ to $(Y,\alpha)$. Let $\Theta,\Theta'$ be the sets of auxiliary data for $\alpha,\alpha'$, then there exists a set of auxiliary data $\Xi$ with a surjective map $\Xi\to \Theta \times \Theta'$, such that for $\xi\in \Xi$, there is a $BL_{\infty}$ morphism $\phi_{\xi}$ (Definition \ref{def:morphism}) from $(V_\alpha,p_{\theta})$ to $(V_{\alpha'},p_{\theta'})$, where $(\theta,\theta')$ is the image of $\xi$ under $\Xi\to \Theta \times \Theta'$.
	\item\label{BL:4} Assume in addition that we fix a point $o'\in Y'$ that is in the same component of $o$ in $X$. Then for any compatible auxiliary data $\theta,\theta',\theta_o,\theta_{o'},\xi$, we have $p_{\bullet,\theta_o},p_{\bullet,\theta_{o'}},\phi_{\xi}$ are compatible (Definition \ref{def:compatible}).
	\item\label{BL:5} For compatible auxiliary data $\theta,\theta_o$, there exists compatible auxiliary data $k\theta,k\theta_o$ for $(Y,k\alpha)$ for $k\in \R_+$, such that $p_{k\theta},p_{\bullet, k\theta_o}$ are identified with $p_{\theta},p_{\bullet,\theta_o}$ by the canonical identification between $V_{\alpha}$ and $V_{k\alpha}$.
\end{enumerate}
\end{theorem}
To make sense of $\#\overline{\cM}$ we need to fix a choice of virtual machinery, and the meaning of auxiliary data also depends on the choice. If one adopts the perturbative scheme in \cite{SFT,ishikawa2018construction}, Theorem \ref{thm:BL} is a special case of their main constructions. On the other hand, since we only consider rational curves, the combinatorics is not essentially different from the construction of differentials and morphisms in \cite{pardon2019contact}. In particular, Pardon's VFC works in a verbatim account. We will explain the VFC construction to prove Theorem \ref{thm:BL} and discuss other virtual techniques in \S \ref{s:virtual}.

\subsection{Augmentations and linearized theories}
We start this section with the following definition.
\begin{definition}\label{def:APT}
	For a strict contact manifold $(Y,\alpha)$, we fix an auxiliary choice $\theta \in \Theta$, then we define algebraic planar torsion $\PT(Y,\alpha,\theta)$ to be the torsion of the $BL_{\infty}$ algebra $(V_\alpha,p_{\theta})$ over $\Q$.
\end{definition}

As a consequence of Proposition \ref{prop:torsion} and Theorem \ref{thm:BL}, we have $\PT(Y,\alpha,\theta)$ is an invariant for $Y$ in the following sense.

\begin{proposition}\label{prop:PT}
$\PT(Y,\alpha,\theta)$ is independent of $\alpha,\theta$, hence can be abbreviated as $\PT(Y)$. Moreover, $\PT:\cont \to \N\cup \{\infty\}$ is a monoidal functor, where the monoidal structure on $\N \cup \{\infty\}$ is given by $a\otimes b =\min\{a,b\}$.
\end{proposition}
\begin{proof}
	By \eqref{BL:5} of Theorem \ref{thm:BL}, we have $(V_{\alpha},p_{\theta})=(V_{k\alpha},p_{k\theta})$ for any $k\in \R_+$. Let $\alpha'$ be another contact form, and $\theta'$ corresponding auxiliary data. Then there exists $k_1,k_2$, such that there are strict cobordisms from $(Y,k_1\alpha),(Y,\alpha')$ to $(Y,\alpha'),(Y, k_2\alpha)$ respectively. Then by \eqref{BL:3} of Theorem \ref{thm:BL} and Proposition \ref{prop:torsion}, we have $\PT(Y,\alpha,\theta)=\PT(Y,\alpha',\theta')$. For $(Y_1,\alpha_1,\theta_1),(Y_2,\alpha_2,\theta_2)$, the $BL_\infty$ algebra for the disjoint union $(Y_1\sqcup Y_2,\alpha_1\sqcup \alpha_2,\theta_1\sqcup \theta_2)$ is given by $(V_{\alpha_1}\oplus V_{\alpha_2}, \{p^{k,l}_{\theta_1}\oplus p^{k,l}_{\theta_2}\})$, i.e.\ there are no mixed structural maps. Then it is clear that $\PT(V_{\alpha_1}\oplus V_{\alpha_2}, \{p^{k,l}_{\theta_1}\oplus p^{k,l}_{\theta_2}\})=\min \{\PT(V_{\alpha_1},p_{\theta_1}),\PT(V_{\alpha_2},p_{\theta_2})\}$. That $\PT$ is a functor follows from \eqref{BL:3} of Theorem \ref{thm:BL}. 
\end{proof}
When $\PT(Y)=0$, it is equivalent to $H_*(\CHA(Y))=0$, which is also known as algebraically overtwisted \cite{bourgeois2010towards} or $0$-algebraic torsion \cite{LW}, and is implied by overtwistedness \cite{bourgeois2010contact,MR2230587}. 
Since finite order of torsion is an obstruction to augmentations, finite algebraic planar torsion is an obstruction to symplectic fillings in view of the following.

\begin{proposition}\label{prop:aug}
Let $(Y,\alpha)$ be a strict contact manifold with an auxiliary data $\theta$, if $(W,\rd\lambda)$ is a strict exact filling, then there is a $BL_{\infty}$ augmentation of $(V_\alpha,p_{\theta})$ over $\Q$.
\end{proposition}
\begin{proof}
	We define
	$$\epsilon^k(q^{\Gamma^+}):=\sum_A\#\overline{\cM}_{W,A}(\Gamma^+,\emptyset),$$
	for $|\Gamma^+|=k$. Then by the third claim of Theorem \ref{thm:BL}, $\{\epsilon^k\}_{k \ge 1}$ defines a $BL_{\infty}$ augmentation.
\end{proof}

\begin{corollary}\label{cor:nofilling}
	If $\PT(Y)<\infty$, then $Y$ has no strong filling.
\end{corollary}
\begin{proof}
    Proposition \ref{prop:aug} implies that $Y$ has no exact filling. For the case of strong filling, we consider $BL_\infty$ algebras over $V_{\alpha}\otimes_{\Q}\Lambda$, where $\Lambda$ is Novikov field
        $$\Lambda=\left\{ \sum_{i=1}^\infty a_i T^{\lambda_i}\left| a_i\in \Q, \displaystyle \lim_{i\to \infty} \lambda_i = +\infty \right.\right\}.$$
	The structural maps in \eqref{eqn:p} need an additional factor of $T^{\int_u \rd \alpha}$ for $u\in  \overline{\cM}_{Y,A}(\Gamma^+,\Gamma^-)$, we use $p_{\Lambda}^{k,l}$ to denote such structural maps. If $\PT(Y)<\infty$, we have $x\in E^kV$, such that $\widehat{p}(x) = 1$. Let $F$ denote the map $V_{\alpha}\to V_{\alpha}\otimes_{\Q} \Lambda$ sending $q_{\gamma}$ to $T^{-\int \gamma^*\alpha}q_{\gamma}$. Then we have an induced map $E^kF:E^kV\to E^k(V\otimes_{\Q}\Lambda)$. Since $\int u^*\alpha = \sum_{\gamma \in \Gamma_+} \int \gamma^*\alpha -\sum_{\gamma \in \Gamma_-} \int \gamma^*\alpha$, it is straightforward to check that $\widehat{p}_{\Lambda}(E^kF(x))=1$. Hence $(V_{\alpha}\otimes_{\Q}\Lambda, p^{k,l}_{\Lambda})$ has no augmentations. On the other hand, given a strict strong filling $(W,\omega)$, by a similar argument of the filling case of Theorem \ref{thm:BL}, 
    $$\epsilon^k(q^{\Gamma^+})=\sum_A\#\overline{\cM}_{W,A}(\Gamma^+,\emptyset)T^{\int u^*\overline{\omega}}$$
    defines an augmentation of $(V_{\alpha}\otimes_{\Q}\Lambda, p^{k,l}_{\Lambda})$, 
	where $|\Gamma^+|=k$ and $\overline{\omega}$ is the $2$-form on $\widehat{W}$ that is $\omega$ on $W$ and $\rd \alpha$ on the cylindrical end. As a consequence, there is no strong filling. 
\end{proof}

\begin{remark}
    $\Lambda$ is equipped with a decreasing filtration  $\Lambda_{\rho_1}\subset \Lambda_{\rho_2}$ if $\rho_1\ge \rho_2$, where $\Lambda_{\rho}$ consists of those elements with $\lambda_i\ge  \rho$. $\Lambda$ is complete w.r.t.\ such filtration. We view elements in $V_{\alpha}$ has filtration $0$, then $V\otimes_{\Q}\Lambda$, $SV\otimes_{\Q}\Lambda$ and $EV\otimes_{\Q}\Lambda$ all have induced filtrations, and we use  $\overline{V\otimes_{\Q}\Lambda}$, $\overline{SV\otimes_{\Q}\Lambda}$ and $\overline{EV\otimes_{\Q}\Lambda}$ to denote the completions. In the context of SFT, due to the feature of the compactness result, the structural maps $p^{k,l},\phi^{k,l}$ may only be well-defined on the completion. It is necessary to use the completion to describe Maurer-Cartan elements \cite{cieliebak2009role}, and functoriality for strong cobordisms \cite{supp}, as well as the SFT for a stable Hamiltonian structure \cite{LW}. The situation in Corollary \ref{cor:nofilling} is a rather special case of strong cobordisms where a naive version of Novikov coefficient without completion is sufficient.
\end{remark}

Roughly speaking the algebraic planar torsion looks for rigid curves with multiple positive punctures and no negative puncture. One particular situation, where we can infer information of algebraic planar torsion, is the planar torsion introduced by Wendl \cite[Definition 2.13]{wendl2013hierarchy}, which generalizes overtwisted contact structures and the Giroux torsion in dimension $3$. The following two results were essentially proven in \cite{LW}.
\begin{theorem}\label{thm:planar_torsion_1}
	If $Y$ is a $3$-dimensional contact manifold with planar torsion of order $k$, then $\PT(Y)\le k$. 
\end{theorem}
\begin{proof}
	This follows from the same argument of \cite[Theorem 6]{LW} based on a precise description of low energy curves in \cite[Proposition 3.6]{LW}, see also \cite{wendl2013hierarchy}. In fact, we do not need the genus $>0$ assertion from the fifth property of \cite[Proposition 3.6]{LW}, as we do not consider higher genus curves.
\end{proof}

\begin{theorem}\label{thm:planar_torsion_2}
	For any $k\in \N$, there exists a $3$-dimensional contact manifold $Y$ with $\PT(Y)=k$.
\end{theorem}
\begin{proof}
	This follows from the same argument of \cite[Theorem 4]{LW}. In fact, we only need the genus $0$ part of \cite[Lemma 4.15]{LW} to get a lower bound.
\end{proof}

\begin{remark}
	As follows from \cite[Corollary 1]{LW}, there are examples of $3$-dimensional contact manifolds $Y_i$ with planar torsion of order $k$, such that there is an exact cobordism from $Y_i$ to $Y_{i+1}$ but no exact cobordism from $Y_{i+1}$ to $Y_i$. On the other hand, there is always a strong cobordism from $Y_{i+1}$ to $Y_i$ by \cite[Theorem 1]{wendl2013non}. We will see similar phenomena in higher dimensions in \S \ref{s6} and \S \ref{s7}.
\end{remark}

\begin{remark}\label{rmk:AT}
It is an interesting question to understand the relations between algebraic planar torsion and algebraic torsion. The $BV_\infty$ reformulation of SFT is recalled in Remark \ref{rmk:coeff}. Following \cite[Definition 1.1]{LW}, $Y$ has algebraic torsion $k$ if $k$ is the smallest number such that  $\hbar^k$ is $0$ in the homology of $(SV_\alpha [[\hbar]],D_{\SFT})$. 

Let us consider the simplest case with an algebraic planar $1$-torsion, in which there are two generators $q_1,q_2$ such that $p^{2,0}(q_1q_2)=1$, $p^{2,l}(q_1q_2)=0$ for all $l>0$, and $p^{1,l}(q_{i})=0$ for all $l\ge 0$ and $i=1,2$. The natural candidate for algebraic torsion is $q_1 q_2$, and we compute
\begin{eqnarray*}
D_{\SFT}q_1 q_2
& = & \hbar+\sum_{g=1}^\infty \sum_{A,[\Gamma]} \frac{\hbar^{g+1}}{\mu_{\Gamma}\kappa_{\Gamma}}\#\overline{\cM}_{Y,g,A}(\{\gamma_1,\gamma_2\},\Gamma)q^{\Gamma} \\
& & +\sum_{g=1}^\infty \sum_{A,[\Gamma]} \frac{\hbar^{g}}{\mu_{\Gamma}\kappa_{\Gamma}}(\#\overline{\cM}_{Y,g,A}(\{\gamma_1\},\Gamma)q^{\Gamma}q_{\gamma_2}+(-1)^{|q_{\gamma_1}|}\#\overline{\cM}_{Y,g,A}(\{\gamma_2\},\Gamma)q_{\gamma_1}q^{\Gamma})
\end{eqnarray*}
Since we have no knowledge of $\overline{\cM}_{Y,g,A}$ for $g>0$ in RSFT, one should not expect that $q_1 q_2$ is a primitive of $\hbar$ under $D_{\SFT}$.  We note here that the above consideration is a very special case and in general algebraic planar $k$-torsion is not equivalent to $\hbar^k$ being the image of the genus $0$ term of $D_{\SFT}$ unless $k=0$. In fact, algebraic torsion and algebraic planar torsion can be viewed as two ``independent" axes in a grid of torsions, which make different requirements on holomorphic curves, see \cite[\S 5.2]{supp} for detail.  
\end{remark}

%\begin{remark}
%One can define $BL_\infty$ algebras over group rings as in \cite{LW} for stable Hamiltonian fillings and weak symplectic fillings. Then the finiteness of algebraic planar torsion in this setup is an obstruction to stable/weak fillings. One example with finite algebraic planar torsion in the group ring setup is those with the fully separating planar $k$-torsions \cite[Definition 1.3]{wendl2013hierarchy}, where the finiteness follows from the same proof of \cite[Theorem 6 (2)]{LW}.
%\end{remark}

The notion of Giroux torsion was generalized to higher dimensions in \cite{massot2013weak}; the following theorem is a reformulation of \cite[Theorem 1.7]{MorPhD}.
\begin{theorem}\label{thm:planar_torsion_3}
	If $Y$ has Giroux torsion, then $\PT(Y)\le 1$.
\end{theorem}

Now we assume $(V_{\alpha},p_{\theta})$ does have a $BL_\infty$ augmentation $\epsilon$ over $\Q$, then $\PT(Y)$ is $\infty$. In view of \S \ref{s2} and Theorem \ref{thm:BL}, a point $o\in Y$ and an auxiliary data $\theta_o$ give rise to a pointed morphism $p_{\bullet,\theta_o}$. Hence we can define the order $O(V_{\alpha},\epsilon,p_{\bullet,\theta_o})$. In the following, we use $\Aug_{\Q}(V_\alpha)$ to denote the set of $BL_\infty$ augmentations of $V_\alpha$ over $\Q$. 

\begin{definition}\label{def:planarity}
	For a strict contact manifold $(Y,\alpha)$ with auxiliary data $\theta$, we define
	$$O(Y,\alpha,\theta):=\max\left\{O(V_{\alpha},\epsilon,p_{\bullet,\theta_o})\left|\forall \epsilon \in \Aug_{\Q}(V_\alpha), o \in Y, \theta_o\in \Theta_o\right.\right\},$$
	where the maximum of an empty set is defined to be zero. 
\end{definition}

\begin{figure}[H]
	\begin{center}
		\begin{tikzpicture}[scale=0.6]
		\draw (0,0) to (0,4) to [out=270,in=180] (1,3.5) to [out=0,in=270] (2,4) to [out=270,in=180] (3,3) to [out=0,in=270] (4,4) to [out=270,in=180] (5,3.5) to [out=0,in=270] (6,4) to [out=270,in=180] (7,3) to [out=0,in=270] (8,4) to [out=270, in=180] (9,3.5) to [out=0,in=270] (10,4) to (10,0) to [out=270, in=0] (9,-0.5) to [out=180,in=270] (8,0) to [out=90,in=0] (7,1) to [out=180,in=90] (6,0) to [out=270,in=0] (5,-0.5) to [out=180, in=270] (4,0) to [out=90,in=0] (3,1) to [out=180,in=90] (2,0) to [out=270,in=0] (1,-0.5) to [out=180, in=270] (0,0);
		\draw[dashed] (0,0) to [out=90,in=180] (1,0.5) to [out=0, in =90] (2,0);
		\draw[dashed] (4,0) to [out=90,in=180] (5,0.5) to [out=0, in =90] (6,0);
		\draw[dashed] (8,0) to [out=90,in=180] (9,0.5) to [out=0, in =90] (10,0);
		\draw (0,4) to [out=90,in=180] (1,4.5) to [out=0,in=90] (2,4);
		\draw (4,4) to [out=90,in=180] (5,4.5) to [out=0,in=90] (6,4);
		\draw (8,4) to [out=90,in=180] (9,4.5) to [out=0,in=90] (10,4);
		
		\draw (12,0) to (12,4) to [out=270,in=180] (13,3.5) to [out=0,in=270] (14,4) to (14,0);
		\draw (12,4) to [out=90,in=180] (13,4.5) to [out=0,in=90] (14,4);
		\draw (0,0) to [out=270, in=180] (1,-2.5) to [out=0,in=270] (2,0);
		\draw (8,0) to [out=270,in=180] (11,-2.5) to [out=0, in=270] (14,0) to [out=270,in=0] (13,-0.5) to [out=180,in=270] (12,0) to [out=270, in=0] (11,-1) to [out=180,in=270] (10,0) to [out=270, in=0] (9,-0.5) to [out=180,in=270] (8,0);
		
		\node at (5,2) {$p^{3,3}$};
		\node at (13,2) {T};
		\node at (1,-1.25) {$\epsilon^1$};
		\node at (11,-1.5) {$\epsilon^2$};
		\end{tikzpicture}
	\end{center}
	\caption{A component of $\ell^4_{\epsilon}$}
\end{figure}

\begin{proposition}\label{prop:P}
	$O(Y,\alpha,\theta)$ is independent of $\alpha$ and $\theta$.
\end{proposition}
\begin{proof}
	We first show that if there is a strict exact cobordism $X$ from $(Y_-,\alpha_-)$ to $(Y_+,\alpha_+)$, then $O(Y_+,\alpha_+,\theta_+)\ge O(Y_-,\alpha_-,\theta_-)$ for any $\theta_+,\theta_-$. For any $o_-\in Y_-$, there exists a point $o_+\in Y_+$, such that there is a path in $X$ connecting $o_+$ and $o_-$.
	Then by \eqref{BL:4} of Theorem \ref{thm:BL} and Proposition \ref{prop:order},  for any augmentation $\epsilon$ to $V_{\alpha_-}$ and auxiliary data $\theta_{o_-}$,   there exists an auxiliary data $\xi\in \Xi$ and $\theta_{o_+}$, such that $O(V_{\alpha_+},\epsilon\circ \phi_{\xi},p_{\bullet,\theta_{o_+}})\ge O(V_{\alpha_-},\epsilon,p_{\bullet,\theta_{o_-}})$. Hence $O(Y,\alpha_+,\theta_+)\ge O(Y,\alpha_-,\theta_-)$. Then by the same argument in Proposition \ref{prop:PT}, $O(Y,\alpha,\theta)$ is independent of $\alpha$ and $\theta$.
\end{proof}

\begin{definition}
    The planairty $\Pl(Y)$ of a contact manifold $Y$ is defined to be $O(Y,\alpha,\theta)$.
\end{definition}

\begin{proposition}\label{prop:P_order}
 $\Pl:\cont \to \N\cup \{\infty\}$ is a monoidal functor, where the monoidal structure on $\N\cup \{\infty\}$ is given by $0\otimes a = 0,\forall a$ and $a\otimes b =\max\{a,b\},\forall a,b\ge 1$.
\end{proposition}
\begin{proof}
    That $\Pl$ is a functor is proven in Proposition \ref{prop:P}. To prove the monoidal structure, we first note that $(V_1\oplus V_2,p_1\oplus p_2)$ has an augmentation iff $V_1,V_2$ both have augmentations, since the natural inclusion $V_1\to V_1\oplus V_2$ defines a $BL_{\infty}$ morphism. This verifies the case for $0\otimes a$. When both $Y_1$ and $Y_2$ have augmentations, it follows from definition that $\Pl(Y_1\sqcup Y_2)=\max\{\Pl(Y_1),\Pl(Y_2)\}$. 
\end{proof}
Since finite algebraic planar torsion is an obstruction to $BL_\infty$ augmentations, we have that $\PT(Y)\le \infty$ implies that $\Pl(Y)=0$. Since $\Pl(Y)=0$ corresponds precisely to those contact manifolds without augmentations, algebraic planar torsion is the inner hierarchy inside $\Pl(Y)=0$. However it is still possible (at least on the algebraic level, e.g.\ by exploiting the non-closedness of $\Q$ algebraically) that $\Pl(Y)=0$ but $\PT(Y)=\infty$, i.e.\ there is no augmentation nor finite algebraic planar torsion.

\subsection{Implementation of virtual techniques}\label{s:virtual}
In the following, we will explain how to get the algebraic count of moduli spaces in Theorem \ref{thm:BL} using virtual techniques. Any choice of virtual machinery should give a construction of $\Pl$ and $\PT$ with the claimed properties, although it is not clear whether different virtual techniques give rise to the same $\Pl$ and $\PT$. However, the geometric results, examples and applications in this paper, do not depend on the choice, as we have the following axiom for virtual machinery, which holds for any of the virtual techniques mentioned in this paper.
\begin{axiom}\label{axiom}
	A virtual implementation of a holomorphic curve theory has the property that the virtual count of a compactified moduli space equals the geometric count, when transversality holds for that moduli space. 
\end{axiom}	
In the following, we will finish the proof of Theorem \ref{thm:BL} by implementing Pardon's implicit atlas and virtual fundamental cycles \cite{pardon2016algebraic, pardon2019contact}. The construction is essentially the constructions of contact homology algebra and morphisms in \cite{pardon2019contact}. As explained in \cite[\S 1.8]{pardon2019contact}, the only difference that one needs to pay attention to is the underlying combinatorics for holomorphic curves. One needs to show that implicit atlas with cell-like stratification still exist for RSFT, in particular the space of gluing parameters has a cell-like stratification. However, the combinatorics for RSFT is also ``tree-like" like contact homology, hence the construction is a verbatim account of \cite{pardon2019contact}. More precisely, the only places we need to pay attention to because of the differences of the combinatorics are that the gluing parameter spaces are cell-like stratified and virtual fundamental cycles gives rise to the right coefficient in $BL_\infty$ coefficient. In the following, we give a brief description of the construction.  

\subsubsection{$\cR$-modules}
We first introduce a category $\cR$ which will play the same role of $\cS_I$ in \cite[\S 2.1]{pardon2019contact} to govern the combinatorics of rational holomorphic curves in the symplectization. The objects of $\cR$ are connected non-empty directed graphs without cycles, such that each vertex has at least one incoming edge. Edges with missing source, i.e.\ input edges, and edges with missing sink, i.e.\ output edges, are allowed. Those edges are called external edges and all other edges are called interior edges. The graph $T$ is equipped with decorations as follows.
\begin{enumerate}
	\item For each edge $e\in E(T)$, a Reeb orbit $\gamma_e$.
	\item For each vertex $v\in V(T)$, a relative homology class $\beta_v\in H_2(Y,\{\gamma_{e^+}\}_{e^+\in E^+(v)}\sqcup \{\gamma_{e^-}\}_{e^-\in E^{-}(v)})$, where we denote by $E^{+}(v)$ the set of incoming edges at $v$ and $E^-(v)$ the set of outgoing edges at $v$, which can be empty.
	\item For each external edge $e\in E^{\ext}(T)$, a basepoint $b_e\in \Ima \gamma_e$.
\end{enumerate}
A morphism $\pi:T\to T'$ in $\cR$ consists first of a contraction of the underlying graph of $T$ to $T'$ by collapsing some of the interior edges of $T$. The decorations have the following property.
\begin{enumerate}
	\item For each non-contracted edge $e\in E(T)$, we have $\gamma_{\pi(e)}=\gamma_e$.
	\item For each vertex $v'\in V(T')$, we have $\beta_{v'}=\sum_{\pi(v)=v'}\beta_v$.
\end{enumerate}
Finally, we specify for each external edge $e\in E^{\ext}(T)=E^{\ext}(T')$ a path along $\Ima\gamma_e$ between the basepoints $b_e$ and $b'_e$ modulo the relation that identifies such two paths iff their difference lifts $\gamma_e$. In particular, there are exactly $\kappa_{\gamma_e}$ different equivalence classes of paths. Then the automorphism group of $T$ with a single vertex is a product of cyclic groups and symmetric groups with cardinality $\mu_{\Gamma^+}\mu_{\Gamma^-}\kappa_{\Gamma^+}\kappa_{\Gamma^-}$. For $T\to T'$, we use $\Aut(T/T')$ to denote the subgroup of $\Aut(T)$ compatible with $T\to T'$.

A concatenation in $\cR$ consists of a finite non-empty collection of objects $T_i\in \cR$ along a matching between some pairs of output edges and input edges with matching orbit label, such that the resulting gluing is a directed graph without cycles, along with a choice of paths between the basepoints for each pair of matching edges. Given a concatenation $\{T_i\}_i$ in $\cR$, there is a resulting object $\#_i T_i\in \cR$. A morphism of concatenations $\{T_i\}_i\to \{T'_i\}_i$ means a collection of morphisms $T_i\to T'_i$ covering a bijection of index sets. Then a morphism $\{T_i\}_i\to \{T'_i\}_i$ induces a morphism $\#_i T_i \to \#_i T'_i$.  If $\{T_i\}_i$ is a concatenation and $T_i=\#_j T_{ij}$ for some concatenation $\{T_{ij}\}_j$, then there is a resulting composite concatenation $\{T_{ij}\}_{ij}$ with natural isomorphisms $\#_{ij}T_{ij}=\#_i\#_jT_{ij}=\#_i T_i$. We use $\Aut(\{T_i\}_i/\#_iT_i)$ to represent the group of automorphisms of $\{T_i\}_i$ acting trivially on $\#_iT_i$, i.e.\ the product $\prod_e \Z_{\kappa_{\gamma_e}}$ over junction edges.  

The key concept to organize the moduli spaces, implicit atlases, and virtual fundamental cycles is the following $\cR$-module.
\begin{definition}[{\cite[Definition 4.5]{pardon2019contact}}]
	A $\cR$-module $X$ valued in a symmetric monoidal category $\cC^{\otimes }$ consists of the following data.
	\begin{enumerate}
		\item A functor $X:\cR\to \cC$.
		\item For every concatenation $\{T_i\}_i$ in $\cR$, a morphism
		$$\otimes_i X(T_i)\to X(\#_i T_i),$$
		such that the following diagrams commute:
		$$
		\xymatrix{
	    \otimes_i X(T_i)\ar[r]\ar[d] & X(\#_i T_i)\ar[d]\\
	    \otimes_i X(T'_i)\ar[r] & X(\#_i T'_i)}\qquad 
        \xymatrix{
        & \otimes_i X(T_i) \ar[rd] & \\
        \otimes_{i,j} X(T_{ij}) \ar[rr]\ar[ru] & & X(\#_{ij}T_{ij})}
		$$
		for any any morphism of concatenations and composition of concatenations.
	\end{enumerate}
\end{definition}

\begin{example}
	A holomorphic building of type $T\in \cR$ consists of the following data.
	\begin{enumerate}
		\item For every vertex $v$, a closed, connected nodal Riemann surface of genus zero $C_v$, along with distinct points $\{p_{v,e}\in C_v\}_e$ indexed by the edges incident at $v$ and a $J$ holomorphic map $u_v:C_v\backslash\{p_{v,e}\}_e\to \R \times Y$ up to $\R$-translation.
		\item $u_v$ converges to $\gamma_{e^+}$ near $p_{v,e^+}$ in the sense of \eqref{eqn:asymp1} for $e^+\in E^+(v)$ and converges to $\gamma_{e^-}$ near $p_{v,e^-}$ in the sense of \eqref{eqn:asymp2} for $e^-\in E^-(v)$. We use $(u_v)_{p_{v,e}}:S^1\to Y$ to denote the $Y$-component of the limit map at the puncture $p_{v,e}$, i.e.\ the corresponding orbit.
		\item For every input/output edge $e$, an asymptotic marker $L_e\in S_{p_{v,e}}C_v$ which is mapped to the basepoint $b_e$ by $(u_v)_{p_{v,e}}$.
		\item For every interior edge $v\stackrel{e}{\to} v'$, a matching isomorphism $m_e:S_{p_{v,e}}C_v\to S_{p_{v',e}}C_{v'}$ intertwining $(u_v)_{p_{v,e}}$ and $(u_{v'})_{p_{v',e}}$.
	\end{enumerate}
    An isomorphism between two buildings is a collection of isomorphisms between $C_v$ commuting with all the data. Then we define $\cM(T)$ to be the set of isomorphism classes of holomorphic buildings of type $T$. Note that $\Aut(T)$ acts on $\cM(T)$ by changing markings. Then we define
    $$\overline{\cM}(T):=\bigsqcup_{T'\to T} \cM(T')/\Aut(T'/T).$$
    The union is over the set of isomorphism classes in the over category $\cR_{/T}$. Moreover, $\overline{\cM}(T)$ is endowed with the Gromov topology and is a compact Hausdorff space \cite[\S 2.9, 2.10]{pardon2019contact}. Note that here for each $v\in V(T)$, we view $u_v$ as a curve in its own copy of the symplectization. In particular, we have no level structure and the topology is slightly different from the buildings in \cite{bourgeois2003compactness} by forgetting all trivial cylinders. However this poses no difference for the compactness. In particular, there is a surjective map from the compactification in \cite{bourgeois2003compactness} to $\overline{\cM}(T)$ by collapsing the boundary configurations containing levels with multiple disconnected nontrivial curves into corners. The functor $\overline{\cM}$ is an $\cR$-module in the category of compact Hausdorff spaces with disjoint union as the monoidal structure. The natural map $\overline{\cM}(T)\to \cR_{/T}$ is a stratification in the sense of \cite[Definition 2.15]{pardon2019contact}. We define $\vdim (T)$ as $\sum_{v\in V(T)}(\ind (u_v)-1)$ and $\codim(T'/T)$ is the number of interior edges collapsed in $T'\to T$. Then we have $\codim(T'/T)+\vdim(T')=\vdim(T)$.
\end{example}
\begin{example}\label{ex:orientation}
	For each non-degenerate Reeb orbit $\gamma$ (good or bad) and a basepoint $b\in \Ima\gamma$, \cite[Definition 2.46]{pardon2019contact} constructs a canonical $\Z_2$ graded line $\mathfrak{o}_{\gamma,b}$ with grading $\mu_{CZ}(\gamma)+n-3 \mod 2$. Any path $b\to b'$ gives rise to a functorial isomorphism $\mathfrak{o}_{\gamma,b}\to \mathfrak{o}_{\gamma,b'}$, two paths induces the same isomorphism if the difference is a lift of $\gamma$. As a consequence, $\Z_{\kappa_{\gamma}}$ acts on $\mathfrak{o}_{\gamma,b}$. Then $\gamma$ is good iff the action is trivial. Let $T$ be a tree, then we have the determinant line $\mathfrak{o}^{\circ}_T\vert_v$ of the linearized Cauchy-Riemann operator at the vertex $v$ and a canonical isomorphism from $\mathfrak{o}^{\circ}_T\vert_v$ to $\otimes_{e^+\in E^+(v)} \mathfrak{o}_{\gamma_{e^+},b_{e^+}}\otimes_{e^-\in E^-(v)} \mathfrak{o}^{\vee}_{\gamma_{e^-},b_{e^-}}$. Moreover, $\mathfrak{o}^{\circ}$ is a $\cR$-module \cite[Example 4.7]{pardon2016algebraic}. We define $\mathfrak{o}_T$ by $\mathfrak{o}_T^{\circ}\otimes (\mathfrak{o}_{\R}^{\vee})^{V(T)}$, where the line $\mathfrak{o}_\R$ comes from linearizing the $\R$ translation. $\mathfrak{o}_T$ will be canonically isomorphic to the orientation sheaf of $\overline{\cM}(T)$. Moreover, for $T'\to T$, there is an induced isomorphism $\mathfrak{o}_{T'}\to \mathfrak{o}_T$ by \cite[(2.61)]{pardon2019contact} from an orientation of the gluing parameter space.
\end{example}

\begin{example}
In the construction of the implicit atlas \cite[Definition 3.2]{pardon2019contact} for moduli spaces $\overline{\cM}(T)$, we need to construct a thickened moduli space, which roughly speaking consists of solutions to pseudo-holomorphic curve equation up to a finite dimentional error. The choices involved in defining implicit atlas, called thickening data as defined in \cite[Definition 3.9]{pardon2019contact}, work verbatim for our purpose. Then we have the set of thickening data $A(T)$, and we may define
$$\overline{A}(T):=\bigsqcup_{T'\subseteq T} A(T'),$$
where the disjoint union is over all connected subgraphs that are in $\cR$. Then clearly $\overline{A}$ is an $\cR^{op}$-module to the category of sets.
\end{example}

\begin{proposition}
	$\overline{\cM}(T)$ is equipped with an implicit atlas $\overline{A}(T)$ with oriented cell-like stratification.
\end{proposition}
\begin{proof}
	First of all, we have the space of gluing parameters $G_{T/}$ that associates to each interior edge a number in $(0,\infty]$. Since there are no cycles in $T$, there are no relations among those gluing parameters. In particular $G_{T/}$ has a cell-like stratification \cite[Definition 3.1]{pardon2019contact} over $\cR_{T/}$ like $(G_I)_{T/}$ in \cite[Lemma 3.5]{pardon2019contact}. Then the claim follows from the same proof of \cite[Theorem 3.23]{pardon2019contact}. The analogues of \cite[Theorem 3.31, 3.32]{pardon2019contact} hold for our setup since we only glue one puncture at a time, hence the gluing analysis in \cite[\S 5]{pardon2019contact} applies in a verbatim way.
\end{proof}

With the existence of implicit atlas with cell-like stratification, the machinery of virtual fundamental cycles induces a pushforward map
\begin{equation}\label{eqn:push}
C^{*+\vdim(T)}_{\vir}(\overline{\cM}(T) \rel \partial; \overline{A}(T))\to C_{-*}(E;\overline{A}(T)),
\end{equation}
where $E$ is part of the data in $\overline{A}(T)$ and $C_{-*}(E;\overline{A}(T))$ is quasi-isomorphic to the $\Q$ concentrated in degree $0$ and $H^{*}_{\vir}(\overline{\cM}(T) \rel \partial; \overline{A}(T))$ is isomorphic to $\check{H}^*(\overline{\cM}(T), \mathfrak{o}_{\rel \partial})$ with $\mathfrak{o}_{\rel \partial}=j_{!}\mathfrak{o}$ for the orientation bundle $\mathfrak{o}$ and $j:\overline{\cM}(T)\backslash \partial \overline{\cM}(T)\to \overline{\cM}(T)$\footnote{Here $\partial \overline{\cM}(T)$ is the preimege of $\cR^{<T}$ in the stratification $\overline{\cM}(T)\to \cR_{/T}$.}. Heuristically, the pushforward map can be viewed as integration of compact supported degree $\vdim(T)$ cohomology class on $\overline{\cM}(T)$.

The construction of virtual fundamental cycles for $BL_\infty$ algebra requires combining single pushfordward/virtual fundamental class with respect to the combinatorics of $\cR$. The proof hinges on an induction argument, which in particular requires certain finiteness.  An object $T\in \cR$ is called effective iff $\overline{\cM}(T)\ne \emptyset$. Then for any morphism $T\to T'$, if $T$ is effective, so is $T'$. For any concatenation $\{T_i\}_i$,  every $T_i$ is effective iff $\#_i T_i$ is effective. In the following, $\cR$ will mean the full subcategory spanned by effective objects, which depends on $J$. Then $\cR$ has the following properties, which allows one to apply inductive constructions.
\begin{enumerate}
	\item Every $T$ can be written as a concatenation of maximal elements $\#_iT_i$, where an element $T_i$ is maximal iff every morphism mapping out of $T_i$ is an isomorphism. That is, $T_i$ has only one interior vertex.
	\item Let $T,T'\in \cR$, we say $T'\preccurlyeq T$ iff there is a morphism $\#_i T_i\to T$ with some $T_i$ isomorphic to $T'$. Then there are no infinite strictly decreasing sequences (in the effective version). This is a consequence of compactness or positivity of contact energy \eqref{eqn:positive} in the exact cobordism setting of this paper. 
\end{enumerate}
As a consequence, a lot of the constructions can be built inductively from the minimal elements in $(\cR,\preccurlyeq )$. Note that maximal $T$ is not necessarily maximal in $\preccurlyeq$, i.e.\ commodification of the moduli spaces of curves in $\widehat{Y}$ might involve breaking. But minimal elements of $(\cR,\preccurlyeq )$ are necessarily maximal, i.e.\ compactified moduli spaces without breaking are necessarily without breaking. The induction will typically start at minimal elements of $(\cR,\preccurlyeq )$ and the induction process for any tree $T$ will terminate since there are no infinite strictly decreasing sequence.

Following the same procedure of \cite[Definition 4.19]{pardon2019contact}, there is a canonical construction of $\cR$-module $C^{*+\vdim}_{\vir}(\overline{\cM} \rel \partial)$ by homotopy colimit in the category of cochain complexes such that $C^{*+\vdim(T)}_{vir}(\overline{\cM}(T) \rel \partial)$ is quasi-isomorphic to $C^{*+\vdim(T)}_{\vir}(\overline{\cM}(T) \rel \partial;\overline{A}(T))$. Similarly, by the homotopy colimit as in \cite[Definition 4.20]{pardon2019contact}, there is an $\cR$-module $C_{*}(E)$ and \eqref{eqn:push} leads to a canonical map of $\cR$-modules $C^{*+\vdim}_{\vir}(\overline{\cM} \rel \partial) \to C_{-*}(E)$.
Similar to \cite[Definition 4.14]{pardon2019contact}, there is an $\cR$-module $\Q[\cR]$ governing the boundary information,
$$\Q[\cR](T):=\Q[\cR_{/T}]=\bigoplus_{T'\to T}\mathfrak{o}_{T'}[\vdim(T')]$$
with the differential given by the sum of all codimension one maps $T''\to T'$ in $\cR_{/T}$ of boundary map $\mathfrak{o}_{T'}\to \mathfrak{o}_{T''}$ in Example \ref{ex:orientation}\footnote{It should be viewed as multiplying an additional $\frac{1}{|\Aut(T''/T')|}=1$.}. The $\cR$-module structure on $\Q[\cR]$ is described in \cite[(4.37), (4.38)]{pardon2019contact}. Then a virtual fundamental class recording the combinatorics of $\cR$ is simple a $\cR$-module map $\Q[\cR]\to \Q$. Namely, for every $T$, the image of $\mathfrak{o}_T[\vdim T]\subset \Q[\cR](T)\to \Q$ determines an element of $(\mathfrak{o}_T^{\vee})^{\Aut(T)}$ (the $\Aut(T)$ invariant part\footnote{The $\Aut(T)$-invariance forces the vanishing of the virtual fundamental cycle at $T$ if one of the input/output edges of $T$ is labeled with a bad Reeb orbit.} of $\mathfrak{o}_T^{\vee}$) if $\vdim(T)=0$. On the other hand, it is not hard to believe that the right virtual fundamental cycle should be the pushforward of  $1$. By \cite[Lemma 4.23]{pardon2019contact}, $\Hom_{\cR_{/T}}(\Q[\cR],C^{*+\vdim}_{\vir}(\overline{\cM}\rel \partial))$ is a $\cR^{op}$ module of complexes whose cohomology is isomorphic to the $\cR^{op}$ module $T\to \check{H}^{*}(\overline{\cM}(T))$. In particular, by the same argument of \cite[Lemma 4.31]{pardon2019contact}, there is a $\cR$-module map $\Q[\cR]\to C^{*+\vdim}_{\vir}(\overline{\cM}\rel\partial)$ representing $1\in H^0(\overline{\cM}(T))$ for all $T$.

In order to build an $\cR$ module map $\Q[\cR]\to \Q$ following $\Q[\cR]\to C^{*+\vdim}_{\vir}(\overline{\cM}\rel\partial) \to C^{-*}(E)$, we need to find a quasi-isomorphism of $\cR$-module from $C_*(E)$ to the trivial $\cR$-module $\Q$ to evaluate at the chain level. Instead of building a direct quasi-isomorphism, this can be done with a cofibrant replacement, i.e.\ a diagram of quasi-isomorphisms of $\cR$ modules,
$$C_*(E)\stackrel{\sim}{\leftarrow} C_*^{\cof}(E) \stackrel{\sim}{\to} \Q,$$
where $C_*^{\cof}(E)$ is cofibrant in the sense of \cite[Definition 4.24]{pardon2019contact}. The construction of $C_*^{\cof}(E)$ follows from the same recipe for \cite[Definition 4.28]{pardon2019contact} by induction on $\preccurlyeq$. $\Q[\cR]$ is again cofibrant in the sense of \cite[Definition 4.24]{pardon2019contact} by the same argument of \cite[Lemma 4.26]{pardon2019contact}. Now we can introduce the auxiliary data in the VFC setup of RSFT.

\begin{definition}
	Given $\alpha,J$, an element of $\Theta(\alpha,J)$ consists a commuting diagram of $\cR$-modules
	\begin{equation}\label{diagram:data}
		\xymatrix{ \Q[\cR] \ar[r]^{\tilde{w}_*}\ar[d]^{w_*} & C^{\cof}_{-*}(E)\ar[d]^{\sim} \ar[r]^{\quad p_*} & \Q \\
		C^{*+\vdim}_{\vir}(\overline{\cM} \rel \partial ) \ar[r] & C_{-*}(E) &
	}	
	\end{equation}
	satisfying the following properties.
	\begin{enumerate}
		\item $p_*$ induces the canonical isomorphism $H_*^{\cof}(E)=H_*(E)=\Q$.
		\item $w_*$ satisfies the property that for any $T\in \cR$, $w_*\in \Hom_{\cR_{/T}}(\Q[\cR],C^{*+\vdim}_{\vir}(\overline{\cM} \rel \partial))$ on cohomology level represents the constant function $1\in \check{H}^0(\overline{\cM}(T))$ under the identification in \cite[Lemma 4.23]{pardon2019contact}.
 	\end{enumerate}
\end{definition}

\begin{proof}[Proof of \eqref{BL:1} of Theorem \ref{thm:BL}]
	In the context of VFC, $\Theta(\alpha)=\bigsqcup_{J}\Theta(J,\Theta(\alpha,J))$. Moreover $\Theta(\alpha,J)$ is not empty.  The existence of $p_*$ follows from \cite[Lemma 4.30]{pardon2019contact}, the existence of $w_*$ follows from \cite[Lemma 4.31]{pardon2019contact}, and the existence of lifting $\tilde{w}_*$ follows from the cofibrant property and induction on $\preccurlyeq$ as in \cite[Proposition 4.34]{pardon2019contact}.
	
	Given a diagram \eqref{diagram:data}, we have a $\cR$-module map $p_*\circ \tilde{w}_*:\Q[\cR]\to \Q$, which assigns to each $T$ with $\vdim(T)=0$ an element $\#\overline{\cM}(T)^{\vir}\in (\mathfrak{o}_T^{\vee})^{\Aut(T)}$, which, after fixing a trivialization of $\mathfrak{o}_{\gamma,b}$ for every Reeb orbits, is a rational number. If an exterior edge of $T$ is labeled by a bad orbit, then being $\Aut(T)$ invariant implies that  $\#\overline{\cM}(T)^{\vir}=0$. $\#\overline{\cM}(T)^{\vir}$ is the virtual count of the moduli space of holomorphic curves modeled by the tree $T$ with labels, i.e.\ of $\overline{\cM}(T)$. If $T$ is a maximal tree (without interior edges), then $\#\overline{\cM}(T)^{\vir}$ (after choosing invariant trivializations of $\mathfrak{o}_{\gamma,b}$) is interpreted as $\#\overline{\cM}_{Y,A}(\Gamma_+,\Gamma_-)$ in \eqref{eqn:p}, where the marking of $T$ is determined by $A,\Gamma_+,\Gamma_-$.  Finally, being a $\cR$-module implies that 
	\begin{eqnarray}
	0 & = & \sum_{\codim(T'/T)=1}\frac{1}{|\Aut(T'/T)|}\#\overline{\cM}(T')^{\vir} \nonumber\\
	& = & \sum_{\codim(T'/T)=1}\#\overline{\cM}(T')^{\vir}\text{, since } \Aut(T'/T)=1,  \label{eqn:bounadry}
	\end{eqnarray}
	\begin{eqnarray}
	\#\overline{\cM}(\#_iT_i)^{\vir} & = & \frac{1}{|\Aut(\{T_i\}_i/\#T_i)|} \prod_i \# \overline{\cM}(T_i)^{\vir}. \label{eqn:product} 
	\end{eqnarray}
	Let $T$ be a tree with one interior vertex labeled by homology class $\beta$, $k$ input edges labeled by $\Gamma^+$, and $l$ output edges labeled by $\Gamma^-$. If $\vdim(T)=0$, then we define $q^{\Gamma^-}$ coefficient of $p^{k,l}(q^{\Gamma^+})$, i.e.\  
    $$\langle p^{k,l}(q^{\Gamma^+}), q^{\Gamma^-} \rangle = \sum_T \frac{\mu_{\Gamma^+}\kappa_{\Gamma^+}}{|\Aut(T)|}\#\overline{\cM}(T)^{\vir}=\sum_T \frac{1}{\mu_{\Gamma^-}\kappa_{\Gamma^-}}\#\overline{\cM}(T)^{\vir},$$ where the sum is over isomorphism classes of such maximal trees with $\vdim=0$. 
    
    In view of Proposition \ref{prop:2level}, we need to prove that $\langle p_2^{k,l}(q^{\Gamma^+}), q^{\Gamma^-} \rangle $ is zero for any multisets $\Gamma^+,\Gamma^-$ with $|\Gamma^+|=k,|\Gamma^-|=l$. We claim that 
    \begin{equation}\label{eqn;p2kl}
    \langle p_2^{k,l}(q^{\Gamma^+}), q^{\Gamma^-} \rangle = \frac{1}{\mu_{\Gamma^-}\kappa_{\Gamma^-}}\sum_T \sum_{\codim(T'/T)=1} \#\overline{\cM}(T')^{\vir}=0
    \end{equation}
	where the first sum is over the isomorphism classes of maximal $T$ with $\vdim(T)=1$ with marking given by $\Gamma^+,\Gamma^-$ and all possible homology classes. Recall from \S \ref{ss:tree}, when we apply $p^{k,l}$ as gluing trees, we choose a representative of input by ordering the vertices while the output is understood as an equivalence class or unordered. More precisely, by the definition of $p^{k,l}$ before Proposition \ref{prop:2level}, we have
    \begin{eqnarray*}
        \langle p^{k,l}(q^{\Gamma^+}), q^{\Gamma^-} \rangle & = & \sum  \pm \langle p^{|\Gamma^+_1|,|\Gamma^-_1|+1}(q^{\Gamma_1^+}), q^{\Gamma^-_1}q_{\gamma} \rangle \cdot  \langle p^{|\Gamma^+_2|+1,|\Gamma^-_2|}(q^{\gamma}q^{\Gamma_1^+}), q^{\Gamma^-_2} \rangle
    \end{eqnarray*}
    where the sum is over divisions of  $\Gamma^+$  as an ordered set into $\Gamma_1^+\cup \Gamma_2^+$ with $|\Gamma^+_1|\ge 1$ (hence $2^{|\Gamma^+|}-1$ such divisions), divisions of $[\Gamma^-]$ as a unordered multiset into  $[\Gamma_1^-] \cup [\Gamma_2^-]$ and good Reeb orbits $\gamma$ and a choice of output vertices in $p^{|\Gamma^+_1|,|\Gamma^-_1|+1}$ marked with $\gamma$. Here the sign is determined by the rule in \S \ref{ss:tree} from switch orders. In other words, we have $\langle p^{k,l}(q^{\Gamma^+}), q^{\Gamma^-} \rangle $ is
     \begin{equation}\label{eqn:p2kl'}
        \sum_{\substack{\Gamma^+=\Gamma_1^+\cup \Gamma_2^+\\ [\Gamma^-]=[\Gamma_1]^-\cup [\Gamma_2^-],  \gamma }}   \pm \langle p^{|\Gamma^+_1|,|\Gamma^-_1|+1}(q^{\Gamma_1^+}), q^{\Gamma^-_1}q_{\gamma} \rangle \cdot n(\gamma,\Gamma^-_1) \cdot  \langle p^{|\Gamma^+_2|+1,|\Gamma^-_2|}(q^{\gamma}q^{\Gamma_1^+}), q^{\Gamma^-_2} \rangle
    \end{equation}
    where $n(\gamma,\Gamma^-_1)$ is the number of $\gamma$ in $\Gamma^-_1$ plus $1$, which is the number of choices for $\gamma$-vertices. They have the same sign when $|q_{\gamma}|=0$ and when $|q_{\gamma}|=1$ the claim is tautological  as $q_{\gamma}^2=0$. 
    
    Now Assume $$\Gamma^+:=\{\underbrace{\gamma_1,\ldots,\gamma_1}_{k_1},\ldots, \underbrace{\gamma_m,\ldots,\gamma_m}_{k_m}\},\quad \Gamma^-:=\{\underbrace{\eta_1,\ldots,\eta_1}_{l_1},\ldots, \underbrace{\eta_s,\ldots,\eta_s}_{l_s}\}$$ for $\sum_{i=1}^m k_i=k$, $\sum_{i=1}^s l_i=l$ and $\gamma_i\ne \gamma_j,\eta_i\ne \eta_j$ for $i\ne j$. Note that every $T'$ with $\codim(T'/T)=1$ is determined by ordered divisions $\Gamma^+=\Gamma_1^+\cup \Gamma_2^+$ and $\Gamma^-=\Gamma_1^+\cup \Gamma_2^-$ and one connecting interior edge marked with $\gamma$. Every unordered division $[\Gamma^-]=[\Gamma_1^-] \cup [\Gamma_2^-]$ appears $\prod_{i=1}^s\binom{l_i}{l'_i}$ times in the ordered divisions, where 
    $$[\Gamma^-_2]=[\{\underbrace{\eta_1,\ldots,\eta_1}_{l'_1},\ldots, \underbrace{\eta_s,\ldots,\eta_s}_{l'_s}\}]$$
    for $l'_i\ge 0$. We can cut out the interior edge to obtain $T'_1,T'_2$ with the interior edge turning into an output edge for $T'_1$ marked with $\gamma$, i.e.\ $T'=T'_1\# T'_2$.  Note that by \eqref{eqn:product}, we have
	$$\# \overline{\cM}(T')^{\vir}=\frac{1}{\kappa_{\gamma}} \# \overline{\cM}(T_1')^{\vir}\# \overline{\cM}(T_2')^{\vir}.$$
    
    Therefore we have
    \begin{eqnarray}
        \frac{1}{\mu_{\Gamma^-}\kappa_{\Gamma^-}} \sum_{\codim(T'/T)=1} \#\overline{\cM}(T')^{\vir} & = & \frac{1}{\mu_{\Gamma^-}\kappa_{\Gamma^-}}\sum_{\substack{\Gamma^+=\Gamma_1^+\cup \Gamma_2^+\\ \Gamma^-=\Gamma_1^-\cup \Gamma_2^-,  \gamma }} \#\overline{\cM}(T')^{\vir} \nonumber \\
        & = & \sum_{\substack{\Gamma^+=\Gamma_1^+\cup \Gamma_2^+\\ \Gamma^-=\Gamma_1^-\cup \Gamma_2^-,  \gamma }} \frac{1}{\mu_{\Gamma^-}\kappa_{\Gamma^-}\kappa_{\gamma}}  \# \overline{\cM}(T_1')^{\vir}\# \overline{\cM}(T_2')^{\vir} \nonumber\\
        & = & \sum_{\substack{\Gamma^+=\Gamma_1^+\cup \Gamma_2^+\\ [\Gamma^-]=[\Gamma_1]^-\cup [\Gamma_2^-],  \gamma }} \frac{1}{\mu_{\Gamma^-}\kappa_{\Gamma^-}\kappa_{\gamma}} \prod_{i=1}^s\binom{l_i}{l'_i}  \# \overline{\cM}(T_1')^{\vir}\# \overline{\cM}(T_2')^{\vir} \label{eqn:T1T2}
    \end{eqnarray}    
    Since $\frac{1}{\mu_{\Gamma^-}}\prod_{i=1}^s\binom{l_i}{l'_i} = \frac{n(\gamma,\Gamma^1_-)}{\mu_{\Gamma^-_1\cup\{\gamma\}}\mu_{\Gamma^-_2}}$, \eqref{eqn:p2kl'} and \eqref{eqn:T1T2} imply the claim \eqref{eqn;p2kl} as the extra signs from switching order in \eqref{eqn:p2kl'} are encoded in the operations on $(\mathfrak{o}_T^{\vee})^{\Aut(T)}$ to write down \eqref{eqn:bounadry}, \eqref{eqn:product}. Hence $p^{k,l}$ gives a $BL_\infty$ structure.
\end{proof}

The next proposition follows from \cite[Proposition 4.33]{pardon2019contact}. It is Axiom \ref{axiom} in the context of VFC.
\begin{proposition}\label{prop:transverse}
	If $\overline{\cM}(T)$ is cut out transversely with $\vdim(T)=0$, then $\#\overline{\cM}(T)^{\vir}=\#\overline{\cM}(T)=\#\cM(T)$ for any $\theta\in \Theta(\alpha,J)$.
\end{proposition}	

\subsubsection{$\cR_{II}$,$\cR^{\bullet}$ and $\cR^{\bullet}_{II}$ modules}
In the following, we introduce $\cR_{II},\cR^{\bullet}$ and $\cR^{\bullet}_{II}$ to govern moduli spaces as well virtual fundamental cycles for $BL_\infty$ morphisms, pointed maps and the homotopy in Definition \ref{def:compatible}. 
\begin{enumerate}
	\item The category $\cR_{II}$ is the analogue of $\cS_{II}$ in \cite[\S 2.1]{pardon2019contact}. The objects of $\cR_{II}$ are graphs without cycles as before, but now each edge $e\in E(T)$ is labeled with a symbol $*(e)\in \{0,1\}$ such that all input edges are labeled with $0$ and all output edges are labeled with $1$. For each vertex $v\in V(T)$, we associate it with a pair of symbols $*^{\pm}(v)\in \{0,1\}$ such that $*^+(v)\le *^{-}(v)$ and $*(e^{\pm}(v))=*^{\pm}(v)$. If $*^{+}(v)=*^{-}(v)$, then $v$ is called a symplectization vertex and if $*^{+}(v)<*^{-}(v)$, then $v$ is called a cobordism vertex. Given an exact cobordism $X$ from $Y_-$ to $Y_+$, for every $T\in \cR_{II}$, we can similarly define the moduli space $\cM_{II}(T)$, where the curve attached to a symplectization vertex $v$ with $*^{\pm}(v)=0$ is a holomorphic curve in $\R\times Y_+$ modulo $\R$-translation, the curve attached to a symplectization vertex $v$ with $*^{\pm}(v)=1$ is a holomorphic curve in $\R\times Y_-$ modulo $\R$-translation, the curve attached to a cobordism vertex $v$  is a holomorphic curve in $\widehat{X}$. Then we have the analogous compactification $\overline{\cM}_{II}(T)$ using the over category over $T$, which is a $\cR_{II}$-module. 
	\item The category $\cR^{\bullet}$ is similar to $\cR$ but with exactly one vertex labeled by $\bullet$. The morphisms in $\cR^{\bullet}$ consist again of contractions of graphs such that the $\bullet$ vertex is mapped to the $\bullet$ vertex. For every $T\in \cR^{\bullet}$, we can associate a moduli space $\cM^{\bullet}(T)$, which is defined similarly as $\cM(T)$ but the map associated to $\bullet$ vertex is a holomorphic curve with a marked point mapped to the fixed point $(0,o)\in \R\times Y$. We can similarly define the compactified moduli spaces $\overline{\cM}^{\bullet}(T)$, which is a $\cR^{\bullet}$-module.
	\item The category $\cR^{\bullet}_{II}$ is the combination of $\cR_{II}$ and $\cR^{\bullet}$, i.e.\ the objects are the same as $\cR_{II}$ but where one of the vertices is marked with $\bullet$. In the definition of $\cM_{II}^{\bullet}(T)$, the curve attached to the $\bullet$ vertex is a curve in the symplectization with a point constraint if the vertex is a symplectization vertex, and is a curve in the cobordism with a path constraint if the vertex is a cobordism vertex.
\end{enumerate}

\begin{proof}[Proof of the rest of Theorem \ref{thm:BL}]
	We need to argue that $\overline{\cM}_{II}(T),\overline{\cM}^{\bullet}(T),\overline{\cM}^{\bullet}_{II}(T)$ are equipped with implicit atlases with oriented cell-like stratification. For this, we only need to argue that the gluing parameter spaces are cell-like like, the remaining of the argument is the same as \cite[Theorem 3.23]{pardon2019contact}. The gluing parameter space $(G^{\bullet})_{T/}$ for $\cR^{\bullet}_{T/}$ is same as the $G_{T/}$, i.e.\ $(0,\infty]^{E^{int}(T)}$, since there is no relations among gluing parameters.  The gluing parameter space $(G_{II})_{T/}$ for $(\cR_{II})_{T/}$ is defined as a subset of 
	$$\left\{ (\{g_e\}_e,\{g_v\}_v)\in  (0,\infty]^{E^{int,0}(T)}\times [-\infty,0)^{E^{int,1}(T)}\times (0,\infty]^{V_{00}(T)}\times [-\infty,0)^{V_{11}(T)}\right\},$$
	subject to the constraints
	$$g_v=g_e+g_{v'}, \text{ for } v \stackrel{e}{\to } v' \text{ with }*(e)=0, \quad g_{v'}=g_e+g_v, \text{ for }  v \stackrel{e}{\to } v' \text{ with }*(e)=1.$$
	where $g_{v}$ is interpreted as $0$ if $v\in V_{01}(T)$, $V_{ij}(T)$ is the set of vertices with $*^+(v)=i,*^{-}(v)=j$ and $E^{int,i}(T)$ is the set of interior edges $e$ such that $*(e)=i$. Then $g_v$ can be viewed as the height of the vertex $v$ for $v\in V_{ij}(T)$, where the heights of all cobordism vertices are $0$, as all of them are placed in the same level. Following the argument of \cite[Lemma 3.6]{pardon2019contact}, it is sufficient to prove $(G_{II})_{T/}$ is a topological manifold with boundary. We can perform the the same change of coordinates $h=e^{-g}\in [0,1)$ for $v\in V_{00}(T),e\in E^{int,0}(T)$, and $h=e^{g}\in [0,1)$ for $v\in V_{11}(V), e\in E^{int,1}(T)$.  We allow $h\in [0,\infty)$ for convenience. Then the relation becomes $h_v=h_eh_{v'}$ for $*(e)=0$ and $h_{v'}=h_eh_v$ for $*(e)=1$.  Now the difference with \cite[Lemme 3.6]{pardon2019contact} is that we do not have $v_{\max}$, which in contact homology corresponds to the vertex with the input edge. In our case, the subgraph generated $V_{00}(T)$ is a disjoint union of graphs $\{T^0_i\}_{i\in I^0}$ and the subgraph generated $V_{11}(T)$ is a disjoint union of graphs $\{T^1_i\}_{i\in I^1}$.  We pick a vertex $v^0_i,v^1_i$ in $T^0_i,T^1_i$ respectively. Since $T_i^0$ has no cycles and we can view $v^0_i$ as a root, we can parameterize the gluing parameters associated to $T_i^0$ by $h_{v^0_i}\in [0,\infty),q_e=h_e^2-h_{v'}^2\in \R$ if $e$ is in the same direction with the direction pointed away from the root $v^0_i$, and $h_e\in [0,\infty)$ if $e$ is the opposite direction with the tree direction.  In the same direction case, $h_v\in [0,\infty),q_e=h_e^2-h_{v'}^2\in \R$ determine $h_e,h_{v'}\in [0,\infty)$ as in \cite[Lemma 3.6]{pardon2019contact}. It is clear that such change of coordinate parameterize the gluing parameters by $[0,\infty)\times \R^{|\{e|\text{in same direction}\}|} \times [0,\infty)^{|\{e|\text{in opposite direction}\}|}$. Similarly we parameterize the gluing parameters on $T^1_i$ by $h_{v^1_i}\in [0,\infty),h_e\in [0,\infty)$ if $e$ is in the same direction with the tree direction, and $q_e=h^2_e-h^2_v\in \R$ if $e$ is the opposite direction with the tree direction. As a consequence $(G_{II})_{T/}$ is the product of the such parametrization of $T^0_i,T_i^1$, which is a topological manifold with boundary, and the top stratum corresponds to the interior. The gluing parameter space $(G^{\bullet}_{II})_{T/}$ is same as  $(G_{II})_{T/}$. Therefore $\overline{\cM}_{II}(T),\overline{\cM}^{\bullet}(T),\overline{\cM}^{\bullet}_{II}(T)$ are equipped with implicit atlases with oriented cell-like stratification. 
	
    The virtual fundamental cycles for $BL_\infty$ morphisms, pointed maps and homotopies are module morphisms $\Q[\cR_{II}]\to \Q$, $\Q[\cR^{\bullet}]\to \Q$ and $\Q[\cR^{\bullet}_{II}]\to \Q$ respectively that are derived from diagrams like \eqref{diagram:data}. The non-emptiness of such diagrams and surjectivity of the projections of admissible auxiliary data follows from \cite[Proposition 4.34]{pardon2019contact}. Combined with Propositions \ref{prop:mor2level}, \ref{prop:pointed_2level} and \ref{prop:pqphi_2level}, that module morphisms $\Q[\cR_{II}]\to \Q$, $\Q[\cR^{\bullet}]\to \Q$ and $\Q[\cR^{\bullet}_{II}]\to \Q$ give rise counts to $BL_\infty$ morphisms, pointed maps and homotopies follows from the same proof of \eqref{BL:1} of Theorem \ref{thm:BL}. For \eqref{BL:5} of Theorem \ref{thm:BL}, it is clear the whole construction for $\alpha$ can be identified with the construction for $k\alpha$ as long as we use the same admissible almost complex structure $J$ for $k>0$.
\end{proof}

\subsubsection{Polyfold approach}
The polyfold construction of SFT \cite{SFT}, which is described in \cite{fish2018lectures}, will imply Theorem \ref{thm:BL} as well. However, we can not use the  ``tree-like" compactification as in $\overline{\cM}(T)$ because the gluing parameter space is only topological manifold with boundary. For the analytic requirement in polyfold, it is important to use the building compactification in \cite{bourgeois2003compactness} so that all gluing parameters are independent and form a smooth manifold with boundary and corner. To implement the polyfold construction for our purpose, it is sufficient to build polyfold strong bundles with sc-Fredholm sections for the SFT building compactification, which is sketched in \cite{fish2018lectures}. 
 
Since we will not need to discuss more subtle cases like neck-stretching and homotopies, the abstract theory of polyfold developed in \cite{hofer2017polyfold} suffices to provide transverse perturbations by the similar induction on $(\cR,\preccurlyeq)$ starting from minimal elements in $\preccurlyeq$, which are polyfolds without boundaries.  The non-empty set $\Theta$ in Theorem \ref{thm:BL} now consists pairs $(J,\sigma)$, where $J$ is an admissible almost complex structure and $\sigma$ is a family of compatible $sc^+$-multisections in general position. Then \eqref{eqn:bounadry} and \eqref{eqn:product} follow from the Stokes' theorem in \cite{hofer2017polyfold}, where the coefficients can be explained to be the discrepancies of isotropy among polyfolds with their boundary polyfolds and boundary polyfolds with product polyfolds. 

To verify Axiom \ref{axiom}, we first note that classical transversality implies polyfold transversality by definition. If $\overline{\cM}(T)$ is cut out transversely for $\vdim(T)=0$, we may still need to perturb the sc-Fredholm section on the associated polyfold, because we construct perturbations by induction. Even though we know that the section is transverse on the boundary polyfolds, but the section can be non-transverse on some factor of the boundary, which will be perturbed before we construct perturbations for $\overline{\cM}(T)$. However, we can choose our perturbations small enough to get the local invariance of $\#\overline{\cM}(T)$ when $\vdim(T)=0$. In other words, Axiom \ref{axiom} holds if we choose sufficiently small perturbations. This matches with Proposition \ref{prop:transverse}, as $\Theta(\alpha,J)$ in VFC can be understood as ``infinitesimal" perturbations.

\begin{remark}[Kuranishi approach]
	The Kuranishi approach of SFT \cite{ishikawa2018construction} would also imply Theorem \ref{thm:BL}.  Axiom \ref{axiom} should follow from the same argument above for small enough perturbations in a reasonable measurement.
\end{remark}

\section{Semi-dilations}\label{s4}
In this section, we introduce an inner hierarchy called the order of semi-dilation for the $\Pl=1$ case. 

\begin{remark}
    As explained in \S \ref{ss:FC}, a full implementation of the order of semi-dilation is dependent on the unproven Claim \ref{claim:u} below, as it requires a rigorous implementation of the $U$-map, which we defer to later work.
\end{remark}

Note that if $\Pl(Y)=1$, then $\RSFT(Y)$ admits $BL_\infty$ augmentations and for any $BL_{\infty}$ augmentation $\epsilon$, we have the order is $1$ for any point in $Y$. Note that $(\overline{B}^1 V_{\alpha},\widehat{\ell}_{\epsilon})$ is the chain complex $(V_{\alpha},\ell^1_{\epsilon})$ for the linearized contact homology. Since $\Pl(Y)=1$, for any point in $Y$, we have an class $x\in H_*(V_{\alpha},\ell^1_{\epsilon})$ such that $\ell^1_{\bullet,\epsilon}(x)=1$. 

If the augmentation $\epsilon_W$ is from an exact filling $W$, then by \cite{bourgeois2009exact,bourgeois2017s}, the linearized contact homology $\LCH_*(Y,\epsilon_W)=\LCH_*(W):= H_*(V_{\alpha},\ell^1_{\epsilon_W})$ is isomorphic to the equivariant symplectic (co)homology, which as an $S^1$-equivariant theory carries a $H^*(BS_1)=\Q[U]$-module structure and fits into the following Gysin sequences,
$$
\resizebox{12cm}{!}{
\xymatrix{
	\ldots \ar[r] & SH_+^{2n-3-k}(W) \ar[r]\ar[d] & \LCH_k(W) \ar[r]^{U}\ar[d]^{\simeq} & \LCH_{k-2}(W) \ar[r]\ar[d]^{\simeq} & SH_+^{2n-2-k}(W)\ar[d] \ar[r] & \ldots \\ 
	\ldots \ar[r] & SH_+^{2n-3-k}(W) \ar[r] & SH_{+,S^1}^{2n-3-k}(W) \ar[r]^{U} & SH_{+,S^1}^{2n-1-k}(W) \ar[r] & SH_+^{2n-2-k}(W) \ar[r] & \ldots }
}$$
As we use homological convention in this paper, $U$ has degree $-2$ in the case with a $\Z$ grading for $\LCH$. And $U$ has degree $2$ for the $S^1$-equivariant symplectic cohomology, which is graded by $n-\mu_{CZ}$. 

Strictly speaking, \cite{bourgeois2009exact,bourgeois2017s} make several transversality assumptions which limit the collection of exact fillings where the isomorphism is rigorously established. However, the geometric ideas behind  \cite{bourgeois2009exact,bourgeois2017s} should work when implemented using a suitable virtual machinery to prove isomorphisms for general cases. We first carry out the discussion neglecting the foundational issues for now. The $U$ map is defined on the linearized contact homology $H_*(V_\alpha,\ell^1_{\epsilon})$ for any augmentation $\epsilon$. And for any element $x\in H_*(V_\alpha,\ell^1_{\epsilon})$ there exists $k\in \N_+$ such that $U^k(x)=0$. In the following, we first recall the definition of $U$-map for linearized contact homology.

\subsection{$H_*(V_\alpha,\ell^1_{\epsilon})$ as a $\Q[U]$ module}\label{sec:U-map}
To explain the $U$-map, we recall the following two moduli spaces from \cite[\S 7.2]{bourgeois2009exact}. In some sense, the following moduli spaces should be viewed as a version of cascades moduli spaces.
\subsubsection{$\bm{\cM^1_{Y,A}(\gamma^+,\gamma^-,\Gamma^-)}$.} Let $\gamma^+,\gamma^-$ be two good Reeb orbits and $\Gamma^-$ be an ordered multiset of good Reeb orbits of cardinality $k\ge 0$. Then an element in $\cM^1_{Y,A}(\gamma^+,\gamma^-,\Gamma^-)$ consists of the following data.
\begin{enumerate}
	\item A sphere $(\Sigma,j)$, with one positive puncture $z^+$ and $1+k$ negative punctures $z^{-},z^{-}_1,\ldots,z^{-}_k$. We pick an asymptotic marker on $z^{+}$, then by choosing a global polar coordinate on $\Sigma \backslash\{z^+,z^-\}$, there is a canonically induced asymptotic marker on $z^{-}$ by requiring it having the same angle as the asymptotic marker at $z^+$ in the polar coordinate. We also pick free asymptotic markers on $z^-_i$ for $1\le i \le k$.
	\item A map $u:\dot{\Sigma}\to \R\times Y$ such that $\rd u \circ j =J\circ \rd u$ and $[u]=A$ modulo automorphism and the $\R$-translation, where $\dot{\Sigma}$ is the $2+k$ punctured sphere.
	\item $u$ is asymptotic to $\gamma^+,\gamma^-,\Gamma^-$ near $z^+,z^-$ and $\{z_i^-\}_i$ w.r.t. to the asymptotic markers and the chosen marked points on the image of Reeb orbits as before.
\end{enumerate}
We use $\overline{\cM}^1_{Y,A}(\gamma^+,\gamma^-,\Gamma^-)$ to denote the compactification of $\cM^1_{Y,A}(\gamma^+,\gamma^-,\Gamma^-)$. Then for an exact cobordism $X$, we can similarly define $\overline{\cM}^1_{X,A}(\gamma^+,\gamma^-,\Gamma^-)$, where we do not modulo the $\R$ translation.
\subsubsection{$\bm{\cM^2_{Y,A}(\gamma^+,\gamma,\gamma^-,\Gamma^-_1,\Gamma^-_2)}$}
Let $\gamma^+, \gamma^-$ be two good Reeb orbits, $\gamma$ a Reeb orbit that could be bad\footnote{The necessity of $\gamma$ being potentially bad is explained \cite[\S 2]{MR3348144}.} and $\Gamma^-_1,\Gamma^-_2$  two ordered multisets of good Reeb orbits of cardinality $k_1,k_2\ge 0$. Then an element in $\cM^2_{Y,A}(\gamma^+,\gamma,\gamma^-,\Gamma^-_1,\Gamma^-_2)$ consists of the following data.
\begin{enumerate}
	\item Two spheres $(\Sigma,j),(\tilde{\Sigma},j)$, each with one positive puncture $z^+,\tilde{z}^+$, and $1+k_1$ negative punctures $z^{-},z^{-}_1,\ldots,z^{-}_{k_1}$, $1+k_2$ negative punctures,  $\tilde{z}^{-},\tilde{z}^{-}_1,\ldots,\tilde{z}^{-}_k$ respectively. Each puncture is equipped with an asymptotic marker.
	\item Two holomorphic curves $u,\tilde{u}$ from $\dot{\Sigma},\dot{\tilde{\Sigma}}$ to $\R\times Y$ modulo automorphism and $\R$-translations, such that $[u]\#_{\gamma}[\tilde{u}]=A$.
	\item $u$ is asymptotic to $\gamma^+,\gamma,\Gamma^-_1$ and $\tilde{u}$ is asymptotic to $\gamma,\gamma^-,\Gamma^-_2$.
	\item Let $L_-$ and $L_+$ be two asymptotic markers on $z^-$ and $\tilde{z}^+$ that are induced from the chosen asymptotic markers on $z^+$ and $\tilde{z}^-$ by global polar coordinates\footnote{In particular, they may be different from the chosen  asymptotic markers on $z^-$ and $\tilde{z}^+$}. We can define $ev_{L-}(u),ev_{L_+}(\tilde{u})$ to be the limit point in the $Y$ component evaluated along the asymptotic markers $L_-,L_+$. Then we require $(b_{\gamma},ev_{L_-}(u),ev_{L_+}(\tilde{u}))$ is the natural order on $\Ima \gamma$, where $b_{\gamma}$ is the chosen marked point on $\Ima\gamma$.	
\end{enumerate}
We use $\overline{\cM}^2_{Y,A}(\gamma^+,\gamma,\gamma^-,\Gamma^-_1,\Gamma^-_2)$ to denote the compactification. Note that we need to add in the stratum corresponding to the collision of $(b_{\gamma},ev_{L-}(u),ev_{L_+}(\tilde{u}))$ in addition to usual building structures. We can similarly define $\overline{\cM}^{2,\uparrow}_{X,A}(\gamma^+,\gamma,\gamma^-,\Gamma^-_1,\Gamma^-_2)$ and $\overline{\cM}^{2,\downarrow}_{X,A}(\gamma^+,\gamma,\gamma^-,\Gamma^-_1,\Gamma^-_2)$ for an exact cobordism $X$. The difference is that the former one has $u$ in $\widehat{X}$ and the latter one has $\tilde{u}$ in $\widehat{X}$. 

Given a dga augmentation $\epsilon^1$ to $\CHA(Y)$, i.e.\ a map $\epsilon^1:V_\alpha \to \Q$, which extends to an algebra map $\widehat{\epsilon}^1:\CHA(Y)\to \Q$ such that $\widehat{\epsilon}^1\circ \widehat{p}^1=0$. Then $U:V_\alpha \to V_\alpha$ is defined by
\begin{eqnarray}
& & U(q_{\gamma^+}):= \nonumber \\
& & \sum_{\gamma^-,[\Gamma^-]}\frac{1}{\kappa_{\gamma^-}\mu_{\Gamma^-}\kappa_{\Gamma^-}} \#\overline{\cM}^1_{Y,A}(\gamma^+,\gamma^-,\Gamma^-)\prod_{\gamma'\in \Gamma^-}\epsilon(\gamma')q_{\gamma^-} + \nonumber\\
 & & \sum_{\substack{\gamma^-,\gamma,\\ [\Gamma^-_1],[\Gamma^-_2]}}\frac{1}{N}\#\overline{\cM}^2_{Y,A}(\gamma^+,\gamma,\gamma^-,\Gamma^-_1,\Gamma^-_2)\prod_{\gamma'\in \Gamma^-_1\cup \Gamma^-_2}\epsilon(\gamma')q_{\gamma^-}, \label{eqn:u}
\end{eqnarray}
 where $N=\kappa_{\gamma^-}\kappa_{\gamma}\mu_{\Gamma^-_1}\mu_{\Gamma^-_2}\kappa_{\Gamma^-_1}\kappa_{\Gamma^-_2}.$
\begin{remark}
	The $\cM^2_{Y,A}$ in \cite[\S 7.2]{bourgeois2009exact} requires modulo an equivalence $(L_{-},L_{+})\simeq (L_{-}+\frac{2\pi}{\kappa_{\gamma}},L_{+}+\frac{2\pi}{\kappa_{\gamma}})$, i.e.\ the moduli space of ``glued" two-level buildings. Here we do not introduce the equivalence, the discrepancy is just the extra $\frac{1}{\kappa_{\gamma}}$ in \eqref{eqn:u} compared to \cite[(85)]{bourgeois2009exact}\footnote{The extra coefficient $\mu_{\Gamma}$ comes from that we consider $\Gamma$ as an ordered set, and $\mu_{\Gamma}$ is the size of the isotropy coming from permutation.}.
\end{remark}
The reason that $U$ is a chain map from $(V_\alpha,\ell^1_{\epsilon})$ to itself follows from the boundary of $1$-dimensional $\overline{\cM}^1_{Y,A}$ and $\overline{\cM}^2_{Y,A}$. More precisely, the codimension $1$ boundary of $\overline{\cM}^1_{Y,A}$ consists of (1) a level breaking where the lower level does not contain $z_-$ and (2) a level breaking where the lower level contains $z_-$. For case one, such contribution is zero when capping $\Gamma^-$ off with $\epsilon^1$ by the relation $\widehat{\epsilon}^1\circ \widehat{p}^1=0$. For case two, the contribution will cancel with the codimension $1$ boundary part of $\overline{\cM}^2_{Y,A}$ corresponding to the collision of $ev_{L_-}(u),ev_{L_+}(\tilde{u})$. The other parts of codimension $1$ boundary of $\overline{\cM}^2_{Y,A}$ consists of (1) a level breaking of $u$ where the lower level does not contain $z_-$, this is again killed by the capping off with $\epsilon^1$; (2) A level breaking of $u$ where the lower level contains $z_-$, the corresponds to a component of $\ell^1_{\epsilon}\circ U$, where the $U$ part is contributed by a $\overline{\cM}^2_{Y,A}$; (3) Similar level breakings for $\tilde{u}$; (4) The collision of $b_{\gamma}$ and $ev_{L_-}(u)$, this corresponds a component of $ \ell^1_{\epsilon}\circ U$, where the $U$ part is contributed by a $\overline{\cM}^1_{Y,A}$, similarly,  the collision of $b_{\gamma}$ and $ev_{L_+}(\tilde{u})$ is the remaining part of $U\circ\ell^1_{\epsilon}$.

Similarly, given an exact cobordism, we can show that chain morphism $\phi^{1,1}_{\epsilon}:V_\alpha\to V_{\alpha'}$ is commutative with $U$ up to homotopy, where the homotopy is defined by $\overline{\cM}^1_{X,A}$,  $\overline{\cM}^{2,\uparrow}_{X,A}$, and  $\overline{\cM}^{2,\downarrow}_{X,A}$ by a similar formula to \eqref{eqn:u} with a similar argument. 

In order to define the order of semi-dilation, we need to define the $U$-map to the following extent.
\begin{claim}\label{claim:u}
	Let $(Y,\alpha)$ be a non-degenerate contact manifold and $\theta$ be an auxiliary data which is used in defining a $BL_\infty$ structure $p_{\theta}$. 
	\begin{enumerate}
		\item There is an auxiliary data $\theta_U$ for the definition of $U$, such that for any $BL_\infty$ augmentation $\epsilon$ of $(V_\alpha,p_{\theta})$, we have a map $U_{\theta_U}:H_*(V_\alpha,\ell^1_{\epsilon})\to H_{*-2}(V_\alpha,\ell^1_{\epsilon})$ and for any $x\in H_*(V_\alpha,\ell^1_{\epsilon})$ there exists $k$ such that $U^k_{\theta_U}(x)=0$ (This nilpotent property follows from that $U$ decreases the contact action for non-degenerate $\alpha$). 
		\item When there is a strict exact cobordism $X$ from $(Y',\alpha')$ to $(Y,\alpha)$ with admissible auxiliary data $\theta,\theta', \theta_U,\theta'_U$ for $\alpha,\alpha'$ and their $U$-maps respectively, then there exists auxiliary data $\xi$, such that the $\phi^{1,1}_{\xi,\epsilon}:H_*(V_\alpha,\ell^1_{\epsilon\circ \phi_{\xi}})\to H_*(V_{\alpha'},\ell^1_{\epsilon})$ commutes with the $U$-maps for any $BL_\infty$ augmentation $\epsilon$ for $(V_{\alpha'},p_{\theta'})$.
		\item For any $k\in \R_+$, there exists $k\theta_U$, such that $U_{k\theta_U}$ is canonically identified with $U_{\theta_U}$
        \item\label{iso} When the augmentation is from an exact filling $W$, then there is an isomorphism $H_*(V_\alpha,\ell^1_{\epsilon})\to SH^{2n-3-*}_{+,S^1}(W)$ preserving the $U$-map.
	\end{enumerate} 
\end{claim}
\begin{remark}
    Note that, in Claim \ref{claim:u}, we are not claiming that linearized contact homology for all possible augmentations forms a contact invariant, which requires establishing the homotopy property of contact homology.
\end{remark}

Assuming Claim \ref{claim:u}, let $Y$ be a contact manifold with $P(Y)=1$, we can define the order of semi-dilation $\SD(Y)$ by
\begin{equation}\label{eqn:SD}
    \SD(Y):=\max\left\{\min\left\{k\left| U^{k+1}(x)=0, x\in H_*(V_\alpha,\ell_\epsilon^1), \ell^1_{\bullet,\epsilon}(x)=1\right.    \right\}\left| o\in Y, \epsilon\in \Aug_{\Q}(V_\alpha)\right. \right\}.
\end{equation}
\begin{proposition}\label{prop:SD}
	For those contact manifolds $Y$ with $\Pl(Y)=1$, the assignment $\SD(Y)$ is well-defined and is a monoidal functor from the full subcategory $\Pl^{-1}(1)$ of $\cont$ to $\N\cup \{\infty \}$, where the monoidal structure on $\N\cup \{\infty\}$ is defined by $a\otimes b=\max\{a,b\}$,
\end{proposition}
\begin{proof}
That $\SD(Y)$ is independent of all choices follows from the same argument of Proposition \ref{prop:PT}. The monoidal structure follows from $H_*(V\oplus V',\ell^1_{\epsilon}\oplus \ell^1_{\epsilon'}) = H_*(V,\ell^1_{\epsilon})\oplus H_*(V',\ell^1_{\epsilon'})$ as $\Q[U]$-modules.
\end{proof}
\begin{proof}[Proof of Theorem \ref{thm:main}]
	It follows from  Proposition \ref{prop:PT}, Proposition \ref{prop:P}, and Proposition \ref{prop:SD}.
\end{proof}

Although Claim \ref{claim:u} is expected to hold in any virtual implementation of SFT, it does not follow from a direct generalization of Pardon's construction of (linearized) contact homology \cite{pardon2019contact}. One of the main issues is the different perspectives of the contact homology, namely as a quotient theory in \cite{pardon2019contact} v.s.\ as an equivariant theory in \cite{bourgeois2009exact}. More precisely, in the definition of contact homology or RSFT, the count $\#\overline{\cM}_{Y,A}(\Gamma^+,\Gamma^-)$ does not depend on the specific choice of base point $b_{\gamma}$ for $\gamma\in \Gamma^+\cup \Gamma^-$. This is clear if $\overline{\cM}_{Y,A}(\Gamma^+,\Gamma^-)$ is geometrically cut out transversely and it is also true for the VFC construction by the definition of $\cR$-module (different choices of $b_{\gamma}$ gives to isomorphic objects in $\cR$). On the other hand, the definition of $\overline{\cM}^1_{Y,A},\overline{\cM}^2_{Y,A}$ is actually sensitive to the choice of $b_{\gamma}$ both geometrically and in virtual constructions. For example, as we explained that the collision of $b_{\gamma}$ with $ev_{L-}(u)$ should correspond to a component of $\ell^1_{\epsilon}\circ U$, where the $U$ part is contributed by a $\overline{\cM}^1_{Y,A}$, with one caveat that we must have $ev_{L_+}(\tilde{u})\ne b_{\gamma}$, for otherwise such degeneration is from a corner component instead of a boundary component of the compactification of $\cM^2_{Y,A}$. If $\tilde{u}$ is cut out transversely, we can certainly arranges this by choosing a slightly different $b_{\gamma^-}$. This shows the crucial dependence on base points and we can not count $\tilde{u}$ from $\#\overline{\cM}_{Y,A}(\gamma,\{\gamma^-\}\cup \Gamma^-_2)^{\text{vir}}$ defined in \cite{pardon2019contact}, which is independent of the choice of base points.  In particular, the virtual implementation of $U$-map can not be built directly on the contact homology algebra in \cite{pardon2019contact}.

Instead of making sense of $\#\overline{\cM}^1_{Y,A},\#\overline{\cM}^2_{Y,A}$ directly and facing the difficulties mentioned above, we can follow an alternative way using the methods in \cite{bourgeois2009exact}. More precisely, we can first define a positive $S^1$-equivariant symplectic cohomology for algebraic augmentations using VFC, where the $U$-map is more natural. Then there is an isomorphism from the positive $S^1$-equivariant symplectic cohomology to the linearized contact homology following the idea in \cite{bourgeois2009exact} and we can transfer the $U$-map on the positive $S^1$-equivariant symplectic cohomology to the linearized contact homology. Indeed, this is how the $U$-map on linearized contact homology in \cite[\S 7.2]{bourgeois2009exact} is motivated. By taking this detour, we do not need to modify the construction of contact homology, while the construction of the positive $S^1$-equivariant symplectic cohomology and the isomorphism do not exceed the techniques provided in \cite{pardon2016algebraic,pardon2019contact}. For the sake of simplicity, we will defer this alternative approach to Claim \ref{claim:u} for later work. 

\subsection{Order of semi-dilation for fillings} $k$-dilation and $k$-semi-dilation were introduced in \cite{zhou2019symplectic} as structures on $S^1$-equivariant symplectic cohomology, which are generalizations of symplectic dilation of Seidel-Solomon \cite{seidel2012symplectic}. More precisely, an exact domain $W$ carries a $k$-dilation, iff there is a class $x\in SH^*_{+,S^1}(W)$ such that $x$ is sent to $1$ by  $SH^*_{+,S^1}(W) \to H^{*+1}_{S^1}(W)$ and $U^{k+1}(x)=0$,  where $H^{*}_{S^1}(W):=H^{*}(W)\otimes(\Q[U,U^{-1}]/[U])$. $W$ carries a $k$-semi-dilation iff $x$ is sent to $1$ in $SH^*_{+,S^1}(W) \to H^{*+1}_{S^1}(W)\to  H^{0}(W)$ and $U^{k+1}(x)=1$. Under the isomorphism $SH^*_{+,S^1}(W)=\LCH_{2n-3-*}(W)$ \cite{bourgeois2009exact,bourgeois2017s} in Theorem \ref{thm:BO-iso} below, the element we are looking for in \eqref{eqn:SD} is exactly the $k$-semi-dilation in \cite{zhou2019symplectic}. It is natural to expect that examples with nontrivial $\SD$ come from examples with nontrivial $k$-(semi)-dilation found in \cite{zhou2019symplectic}. Indeed, it is the case and we will show in \S \ref{s7} that $\SD$ is surjective. The only extra thing we need to argue compared to \cite[Definition 3.4]{zhou2019symplectic} is that the computation is independent of the augmentation. To make the connection, we restate the main theorem of \cite{bourgeois2009exact} as follows. We use $\LCH^{<A}_*(Y,\epsilon)$ to denote the truncated linearized contact homology generated by Reeb orbits of period smaller than $A$.

\begin{theorem}[\cite{bourgeois2009exact}]\label{thm:BO-iso}
Let $W$ be a strict exact filling of $(Y,\alpha)$. Assume $\LCH^{<A}_*(Y,\epsilon_W)$ is defined using a generic $J$, where $\epsilon_W$ is the augmentation from $W$ (e.g.\ assume either conditions in \cite{MR3348144} or Reeb orbits with period smaller than $A$ are simple).  Then $\LCH^{<A}_*(Y,\epsilon_W)$ is isomorphic to the positive $S^1$-equivariant symplectic cohomology $SH^{*,<A}_{+,S^1}(W)$ of $W$ using a Hamiltonian of slope $A$ as $\Q[U]$ modules. Assume $Y$ is connected, under this isomorphism, that $\ell^1_{\bullet,\epsilon_W}:\LCH^{<A}_*(Y,\epsilon_W)\to \Q$ is isomorphic to $SH^{*,<A}_{+,S^1}(W)\to H^{*+1}_{S^1}(W)=H^{*+1}(W)\otimes (\Q[U,U^{-1}]/U)\to H^0(W)=\Q$.
\end{theorem}
Then \eqref{iso} of Claim \ref{claim:u} implies that we can use the $U$-map on positive $S^1$-equivariant symplectic cohomology to estimate $\SD$ in \eqref{eqn:SD}.

%\begin{remark}
%	In fact, $\infty^{\PT},\infty^{\SD}$ are the only two elements in $\cH$ that we do not know if it is in the image of $\Hcx$. $\Hcx(Y)=\infty^{\PT}$ corresponds to that $\RSFT(Y)$ has no $BL_\infty$ augmentation while $\RSFT(Y)$ has infinite torsion. Note that a $BL_\infty$ augmentation is essentially a solution to a family of algebraic equations (in a infinite dimensional space). Using that $\Q$ is not algebraically closed, it is easy to construct a (finite dimensional) $BL_\infty$ algebra over $\Q$ with no augmentation and infinite algebraic planar torsion.  However, it is unclear how to construct a geometric example. It is also an interesting question on obstructions to $BL_\infty$ augmentations beyond torsion besides using that $\Q$ is not algebraically closed.
%\end{remark}

\begin{remark}\label{rmk:func}
    From the Viterbo transfer map, a functor $\SD$ from $\cont_*\to \N \cup \{\infty\}$ was defined in \cite[Corollary D]{zhou2019symplectic} using $S^1$-equivariant symplectic cohomology. In the context of this paper, the order of semi-dilation in \cite{zhou2019symplectic} is \eqref{eqn:SD} using an augmentation from the exact filling. One can similarly define planarity of an exact filling. However, to establish the well-definedness and functoriality, we need to introduce the notation of homotopy between $BL_\infty$ augmentations for linearized theories, which is beyond the scope of this paper. The trick in Proposition \ref{prop:PT}, \ref{prop:P}, \ref{prop:SD} can not help dropping the dependence on contact forms or auxiliary data,  since it requires comparing the composition of the morphism from an exact cobordism $X$ and the augmentation from an exact filling $W$ to an augmentation from $X\circ W$. However this is essentially a $BL_\infty$ homotopy from neck-stretching.
\end{remark}

\section{Lower bounds for planarity}\label{s5}
As explained in \S \ref{s3}, the curve responsible for finiteness of planarity is a curve with multiple positive punctures and a point constraint. Since planarity does not depend on the choice of the point, one should expect that finiteness of planarity implies uniruledness. In this section, we will prove such implication and a lower bound for planarity. We first recall the notion of uniruledness from \cite{mclean2014symplectic}. 
\subsection{Order of uniruledness}
\begin{definition}[{\cite[\S 2]{mclean2014symplectic}}]\label{def:J}
	Let $(W,\lambda)$ be an exact domain. A $\rd \lambda$-compatible almost complex structure $J$ on $W$ is convex iff there is a function $\phi$ such that
	\begin{enumerate}
		\item $\phi$ attains its maximum on $\partial W$ and $\partial W$ is a regular level set,
		\item $\lambda\circ J=\rd \phi$ near $\partial W$.
	\end{enumerate}
\end{definition}

\begin{definition}[{\cite[Definition 2.2]{mclean2014symplectic}}]
	Let $k>0$ be an integer and $\Lambda>0$ a real number. We say that an exact domain $(W,\lambda)$ is $(k,\Lambda)$ uniruled if, for every convex almost complex structure $J$ on $W$ and every $p\in W^\circ$ (the interior of $W$) where $J$ is integrable near $p$, there is a proper $J$-holomorphic map $u:S \to W^{\circ}$ passing through $p$ and the following holds:
	\begin{enumerate}
		\item $S$ is a genus $0$ Riemann surface and $\rank(H_1(S;\Q))\le k-1$, 
		\item $\int_S u^*\rd \lambda \le \Lambda$.
	\end{enumerate} 
    We say  $W$ is $k$-uniruled if $W$ is $(k,\Lambda)$ uniruled for some $\Lambda>0$. 
\end{definition}
The number $\Lambda$ depends on the Liouville form $\lambda$; this is not relevant for our purposes. However, the number $k$ only depends on the Liouville structure up to homotopy.
\begin{definition}
	Let $W$ be an exact domain. We define the order of uniruledness by $$\U(W):=\min\{k|W \text{ is } k \text{ uniruled}\}.$$
\end{definition}

The following was (inexplicitly) proven by McLean \cite{mclean2014symplectic}.
\begin{proposition}
	$\U$ is a functor from $\cont_*$ to $\N_+\cup \{\infty\}$.
\end{proposition}
\begin{proof}
	Let $V\subset W$ be an exact subdomain, then $\U(V)\le \U(W)$ by \cite[Proposition 3.1]{mclean2014symplectic}. It is clear from definition that $\U(V,\lambda)=\U(V,t\lambda)$ for $t>0$. Since for any Liouville structure $\theta$ on $V$ that is homotopic to $\lambda$, we have exact embeddings $(V,t^{-1}\lambda)\subset (V,\theta)\subset (V,t\lambda)$ for $t\gg 0$, it follows that $\U$ is a well-defined functor on $\cont_*$.
\end{proof}

\begin{remark}
	A point worth noting is that the definition and functorial property of $\U$ do not depend on any Floer theory. However $\U$ gives a measurement of ``complexity" of exact domains. By \cite[Theorem 3.27]{zhou2019symplectic}, the existence of a $k$-(semi)-dilation implies that the order of uniruledness is $1$. Hence the order of (semi)-dilation in \cite[Corollary D]{zhou2019symplectic} is a refined hierarchy in $\U=1$.
\end{remark}

For an affine variety $V$, we define the order of algebraically uniruledness $\AU(V)$ to be the minimal number $k$ such that $V$ is algebraically $k$ uniruled, i.e. through every generic $p\in V$ there is a polynomial map $S\to V$ passing through $p$, where $S$ is a punctured $\CP^1$ with at most $k$ punctures.
\begin{proposition}[{\cite[Theorem 2.5]{mclean2014symplectic}}]\label{prop:Ulower}
Let $V$ be an affine variety. Then $\U(V)\ge \AU(V)$.
\end{proposition}

\begin{example}
Let $S_k$ be the sphere with $k$ disjoint disks removed. Then $\U(S_k)=k$. Let $\Sigma_{g,k}$ be the genus $g\ge 1$ surface with $k$ disjoint disks removed, then $\U(\Sigma_{g,k})=\infty$. It is clear that $S_k$ embeds exactly into $S_{k+1}$. However $S_{k+1}$ can only be embedded in $S_k$ symplectically but not exactly.
\end{example}

In general we have the following.
\begin{theorem}
	We have $\U((S_k)^n)=k$ and $\U((\Sigma_{g,k})^n)=\infty$ for $g\ge 1$. In particular, $\U$ is a surjective functor in any dimension $\ge 2$.
\end{theorem}
\begin{proof}
	Note that $S_k^n$ has a projective compactification $(\CP^1)^n$, we may assume the symplectic form is a product of the same symplectic form on $\CP^1$. Then, for any compatible almost complex structure $J$ on $(\CP^1)^n$, there is a holomorphic curve passing through any fixed point, in the class $[\CP^1\times \{pt\}\times \ldots \times \{pt\}]$ and intersecting each divisor $\{p_i\}\times (\CP^1)^{n-1}$ exactly once for $1\le i \le k$, where $p_i$ is the $i$th puncture. We may assume $J$ is an extension of a convex almost complex structure on $S_k^n$ (which is not necessarily split). Therefore $\U((S_k)^n)\le k$, by neck-stretching. 
	
	On the other hand, the affine varieties corresponding to $(S_k)^n$ and $(\Sigma_{g,k})^n$ are $(\CP^1\backslash\{p_1,\ldots,p_k\})^n$ and  $(\Sigma_g\backslash\{p_1,\ldots,p_k\})^n$. We know that a rational algebraic curve in $(\CP^1\backslash\{p_1,\ldots,p_k\})^n$ and  $(\Sigma_g\backslash\{p_1,\ldots,p_k\})^n$ projects to each factor as a rational algebraic curve. Then every rational curve must have at least $k$ punctures. Therefore $\AU((S_k)^n)=k$ and $\AU((\Sigma_{g,k})^n)=\infty$, and the claim follows from Proposition \ref{prop:Ulower}.
\end{proof}

As a consequence, we find in each dimension a nested sequence of exact domains $V_1\subset V_2 \ldots$, such that $V_i$ can not be embedded into $V_j$ exactly if $i>j$. Sequences with such property in $\dim \ge 10$ were also obtained in \cite[Corollary 1.5]{lazarev2020prime}.  

\begin{remark}
	In \S \ref{s6}, we will show that $\Pl(\partial (S_k)^n)=k$ if $n\ge 2$. Therefore not only there is no exact embedding from $(S_{k+1})^n$ to $(S_{k})^n$, but also there is no exact cobordism from $\partial (S_{k+1})^n$ to $\partial (S_{k})^n$.
\end{remark}

\begin{remark}
	From $\U$ on $\cont_*$, we can build a functor $\U_{\partial}$ on $\cont$ as follows
	$$\U_{\partial}(Y):=\max\{\U(W)|W \text{ is an exact filling of } Y \},$$ where the maximum of the empty set is defined as zero. Then Corollary \ref{cor:lower} below implies that $\U_{\partial}\le \Pl$. The equality does not always hold. For example, $\U_{\partial}(\mathbb{RP}^{2n-1},\xi_{std})=0$ for $n\ne 2^k$ by \cite{zhou2020mathbb}, but $\Pl(\mathbb{RP}^{2n-1},\xi_{std})=1$ when $n\ge 3$ by Theorem \ref{thm:quotient}. Those discrepancies come from the difference between fillings and augmentations. It is possible to generalize the notion of order of uniruledness $\U,\U_{\partial}$ to strong fillings or even weak fillings, but we will not pursue this in this paper. 
\end{remark}

In the following, we introduce an alternative definition of $k$-uniruledness but on the completion $\widehat{W}$, which is suitable to be related to SFT.

\begin{definition}
	Let $(W,\lambda)$ be an exact filling with a non-degenerate contact boundary. We say that the completion $\widehat{W}$ is $k$-uniruled if there exists $\Lambda>0$, such that for every $p\in W^{\circ}$ and every admissible almost complex structure $J$ that is integrable near $p$, there is a rational holomorphic curve passing through $p$ with at most $k$ positive punctures and contact energy of the curve is at most $\Lambda$. 
\end{definition}

\begin{proposition}\label{prop:unirule_equiv}
	An exact filling $(W,\lambda)$ is $k$-uniruled iff $\widehat{W_{\epsilon}}$ is $k$-uniruled, where $W_{\epsilon}$ is Liouville homotopic to $W$ with a non-degenerate contact boundary.  
\end{proposition}
\begin{proof}
	We first show that  $(W,\lambda)$ is $k$-uniruled implies $\widehat{W_{\epsilon}}$ is $k$-uniruled. WLOG, we can take $W_{\epsilon}\subset W$, since we can rescale $W_{\epsilon}$. By assumption there is a $\Lambda>0$ such that for any $p\in W^{\circ}_{\epsilon}$ and any $J$ integrable near $p$ and convex near $\partial W$, there is a $J$-rational curve $u:S \to W$ with $\int_S u^*\rd \lambda < \Lambda$ and $H_1(S;\Q)\le k-1$. In particular, we can choose $J$ to be cylindrical convex near $\partial W_{\epsilon}$. Then by applying neck-stretching along $\partial W_{\epsilon}$, we must have a rational holomorphic curve $u:S_{\epsilon}\to \widehat{W_{\epsilon}}$ passing through $p$ with contact energy smaller than $\Lambda$. We know that $S_{\epsilon}$ is a punctured sphere, as $\partial W_{\epsilon}$ is non-degenerate. It is sufficient to prove $\rank H_1(S_{\epsilon};\Z)=\rank H_1(S_{\epsilon};\Q)\le k-1$. Assume otherwise, then we know that $H_1(S_{\epsilon};\Z)\to H_1(S;\Z)$ is not injective, for if not, we have $\rank H_1(S;\Q)\ge \rank H_1(S_{\epsilon};\Q)\ge k$. Therefore we find a class $[\gamma]\in H_1(S_{\epsilon};\Z)$, such that $[\gamma]$ is represented by a disjoint union $\gamma$ of possibly multiply covered loops around punctures of $S_{\epsilon}$, and there is an immersed surface $A$ in $S\backslash S_{\epsilon}$ whose boundary is $\gamma$. Then in the fully stretched case, $u|_A$ corresponds to a holomorphic building with only negative punctures, which is impossible for energy reasons.
	
	Now we assume $\widehat{W_{\epsilon}}$ is $k$-uniruled. WLOG, we can assume $W\subset W_{\epsilon}$. By \cite[Proposition 5.3]{zhou2019symplecticI}, any convex almost complex structure on $W$ can be extended to an admissible almost complex structure on $\widehat{W_{\epsilon}}$. By assumption, there is a rational curve $u:S\to \widehat{W_{\epsilon}}$ passing through the chosen point $p\in W$ with $S$ an at most $k$ punctured sphere and the contact energy of $u$ is at most $\Lambda$. Let $S'$ be the connected component of $u^{-1}(W^{\circ})$ containing the point mapped to $p$. It clear that the area of $u|_{S'}$ is bounded by $\Lambda$. We claim that $H_1(S';\Z)\to H_1(S;\Z)$ is injective. For otherwise, there is a class $A\in H_2(S,S';\Z)$ mapped to a nontrivial element by $H_2(S,S';\Z)\to H_1(S';\Z)$. Then we can find a $S''\subset S'$ such that $\lambda \circ J =\rd \phi$ on $u|_{S'\backslash S''}$, where $\phi$ is the function in Definition \ref{def:J}. Then by excision, we have $A$ represented by an immersed surface in $S\backslash S''$ not contained completely in $S'\backslash S''$ with boundary in $S'\backslash S''$. Let $\widehat{\phi}$ be the extension of $\phi$ on $\widehat{W_{\epsilon}}$ by \cite[Proposition 5.3]{zhou2019symplecticI}. In particular, the maximum principle holds for $\widehat{\phi}$. Then we reach at a contradiction, since $\widehat{\phi}(u)|_{\partial A}<\max \widehat{\phi}(u)|_A$. Since $\Q$ is flat, we know that $H_1(S';\Q)\to H_1(S;\Q)$ is also injective, hence $\rank H_1(S';\Q)\le \rank H_1(S;\Q)\le k-1$. 
	
\end{proof}

\subsection{Uniruledness and planarity}
 
The main theorem of this section is the following.
\begin{theorem}\label{thm:unirule}
	If $\Pl(Y)=k$, then any exact filling of $Y$ is $k$-uniruled.
\end{theorem}
An exact filling $W$ gives rise to a $BL_\infty$ augmentation $\epsilon_W$ over $\Q$. As a consequence we have a chain morphism $\widehat{\ell}_{\bullet,\epsilon_W}:\overline{B}V\to \Q$ after fixing a point $o$ in $Y$ and an auxiliary data. We can define a different map $\eta_W: \overline{B}V\to \Q$ by 
\begin{equation}\label{eqn:planarity}
\eta_W(q^{\Gamma^+}) = \sum_{A}  \# \overline{\cM}_{W,A,p}(\Gamma^+,\emptyset)
\end{equation}
for a fixed point $p\in W$ such that $p,o$ are in the same connected component of $W$, and $|\Gamma^+|=k$.
\begin{proposition}\label{prop:equi}
	$\eta_W$ is a chain morphism and is homotopic to $\widehat{\ell}_{\bullet,\epsilon_W}$ with appropriate choices of auxiliary data, where $\epsilon_W$ is the augmentation from $W$. Moreover $\eta_W$ is compatible with the word length filtration, i.e.\ for $k\ge 1$, the following diagram that is commutative up to homotopy $H$. 
    $$
    \xymatrix{
     (\overline{B}^kV, \widehat{\ell}_{\epsilon_W}) \ar[r]^{\quad\widehat{\ell}_{\bullet,\epsilon_W}}\ar[d]^{\Id}\ar@{=>}[rd]^H & (\Q,0) \ar[d]^{\Id} \\
     (\overline{B}^kV, \widehat{\ell}_{\epsilon_W}) \ar[r]^{\quad \eta_W} & (\Q,0) 
    }
    $$

\end{proposition}
\begin{proof}
    That $\eta_W$ defines a chain morphism, i.e.\ $\eta_W\circ \widehat{\ell}_{\epsilon_W}=0$, follows from the boundary of  $1$-dimensional $\overline{\cM}_{W,A,p}(\Gamma^+,\emptyset)$. Let $\gamma$ be a path in $W$ connecting $p$ to $o$ and we use $\widehat{\gamma}$ to denote the completion ray of $\gamma$ in $\widehat{W}$.  Then the homotopy $H:\overline{B}V \to \Q$ is defined by 
    $$H(q^{\Gamma^+}) = \sum_{A} \# \overline{\cM}_{W,A,\gamma}(\Gamma^+,\emptyset),$$
    where $\overline{\cM}_{W,A,\gamma}(\Gamma^+,\emptyset)$ is defined similarly to \eqref{path} before \eqref{eqn:positive}.
    The realization of those operators using virtual techniques is similar to $\phi_\bullet$ in \eqref{BL:4} of Theorem \ref{thm:BL} in \S \ref{s:virtual}. The homotopy relation, $\widehat{\ell}_{\bullet,\epsilon_W}-\eta_W = H\circ \widehat{\ell}_{\epsilon_W}$, comes from the boundary of $1$-dimensional $\overline{\cM}_{W,A,\gamma}(\Gamma^+,\emptyset)$. It is clear that both $\eta_W$ and $H$ are compatible with the word length filtration.
\end{proof}

\begin{proof}[Proof of Theorem \ref{thm:unirule}]
	The theory of $BL_\infty$ algebras considered for contact manifolds is equipped with a filtration by the contact action, where the action $\cA(q_{\gamma})$ of a generator is $\int \gamma^*\alpha$. Then the action can be extended to $EV$ and $SV$, by declaring the action of a monomial in SV, or a monomial of monomials in $EV$ using the $\{q_{\gamma}\}$ as basis is the sum of actions, the action of an element using those monimals as basis is the maximum of the actions of the monomials with nonzero coefficient. Then all of the operators for contact manifolds and exact cobordisms will decrease the action. It may not be true that the spectral invariant for $\Pl(Y)=k$ is bounded for all $BL_\infty$ augmentations. But for an augmentation $\epsilon_W$ coming from an exact filling $W$, we have the spectral invariant is bounded, i.e.\ there is a $\Lambda>0$ and an $x\in \overline{B}^kV$ with $\cA(x)\le \Lambda$, $\widehat{\ell}_{\epsilon_W}(x)=0$, and $\widehat{\ell}_{\bullet,\epsilon_W}(x)=1$. Then by Proposition \ref{prop:equi}, we have $\eta_W(x)=1$.  We must have the unperturbed/geometric $\overline{\cM}_{W,A,p}(\Gamma_+,\emptyset)$ is not empty for  some $\Gamma_+$ with $|\Gamma^+|\le k$ and $\sum_{\gamma \in \Gamma^+}\int \gamma^*\alpha \le \Lambda$, for otherwise we will have $\eta_W(x)=0$ by Axiom \ref{axiom}. It is clear from the proof of Proposition \ref{prop:equi} that $p$ can be any point in $W$. This shows that $\widehat{W}$ is $(k,\Lambda)$ uniruled; by Proposition \ref{prop:unirule_equiv}, $W$ is $k$-uniruled.
\end{proof}

Theorem \ref{thm:unirule} provides a lower bound for $\Pl$.
\begin{corollary}\label{cor:lower}
	Let $W$ be an exact filling of $Y$, then $\Pl(Y)\ge \U(W)$. If $W$ is an affine variety, then $\Pl(Y)\ge \AU(W)$.
\end{corollary}

\section{Upper bounds for planarity}\label{s6}
In this section, we will obtain upper bounds of $\Pl$ for the following two cases.
\subsection{Iterated planar open books}

\begin{definition}[{\cite[Definition 2.2]{acu2017weinstein}}] \label{IPLF}
An \emph{iterated planar Lefschetz fibration} $f: (W^{2n}, \omega) \to \mathbb{D}$ on a $2n$-dimensional Weinstein domain $(W^{2n}, \omega)$ is an exact symplectic Lefschetz fibration satisfying the following properties: 

\begin{enumerate}

\item There exists a sequence of exact symplectic Lefschetz fibrations $f_i:(W^{2i}, \omega_i) \to \mathbb{D}$ for $i=2, \dots, n$ with $f=f_n$.

\item The total space $(W^{2i}, \omega_{i})$ of $f_{i}$ is a regular fiber of $f_{i+1}$, for $i=2, \dots, n-1$.

\item $f_2: (W^4, \omega_2) \to \mathbb{D}$ is a planar Lefschetz fibration, i.e. the regular fiber of $f_2$ is a genus zero surface with nonempty boundary, which we denote by $W^2$.
\end{enumerate}
\end{definition}

\begin{definition}[{\cite[Definition 2.3, 2.4]{acu2017weinstein}}] \label{IPOB}
An \textit{iterated planar open book decomposition} of a contact manifold $(Y^{2n+1}, \xi)$ is an open book decomposition for $Y$ whose page $W$ admits an iterated planar Lefschetz fibration, which supports the contact structure $\xi$ in the sense of Giroux. We say that $(Y,\xi)$ is iterated planar (IP).
\end{definition}

If the number of boundary components of $W^2$ in the above definition is $k$, we say that $(Y,\xi)$ is \emph{$k$-iterated planar} or $k$-IP. We remark that the collection of IP contact manifolds is already a large class of examples, as e.g.\ the fundamental group is not an obstruction in any fixed dimension at least $5$ \cite[Theorem 1.4]{acu2020generalizations}.

\begin{theorem}\label{thm:upper}
	Let $(Y,\xi)$ be a $k$-IP contact manifold. Then $\Pl(Y)\le k$.
\end{theorem}

The strategy for obtaining an upper bound $\Pl(Y)\leq k$ on the planarity of a contact manifold $(Y,\xi)$ is via the following algebraic-geometric condition.

\begin{lemma}\label{lemma:upperbound} Let $(Y,\xi)$ be a contact manifold. Assume the following holds.
\begin{itemize}
    \item[$(\star)_k$] There exists a point $o\in\mathbb{R}\times Y$, a contact form $\alpha$ for $\xi$, a choice of  $\alpha$-compatible cylindrical almost complex structure $J$ on $\mathbb{R}\times Y$, and some collection $\Gamma=(\gamma_1,\dots,\gamma_k)$ of precisely $k$ distinct, non-degenerate and simply-covered $\alpha$-Reeb orbits, for which the following holds.
    \begin{itemize}
        \item[$(1)_k$]\label{1k} If $\Gamma^+\subseteq \Gamma$, and $\Gamma^-\neq \emptyset$, then $\mathcal{M}_{Y,A}(\Gamma^+,\Gamma^-)=\emptyset$ for every homology class $A$.
        \item[$(2)_k$] The moduli space $\mathcal{M}_{Y,A,o}(\Gamma,\emptyset)$ is transversely cut out for every $A$ with expected dimension $0$.
        \item[$(3)_k$] For some choice of coherent orientations, the algebraic count of the $k$-punctured spheres in $\bigcup_{A,\vdim=0} \mathcal{M}_{Y,A,o}(\Gamma,\emptyset)$ is nonzero.
    \end{itemize}
\end{itemize}  

Then $\Pl(Y)\leq k$. 
\end{lemma}
\begin{proof}
    By the first property, we have $\overline{\mathcal{M}}_{Y,A,o}(\Gamma^+,\Gamma^-)=\emptyset$ for any $\Gamma^+\subseteq \Gamma$, and $\Gamma^-\neq \emptyset$. Therefore by Axiom \ref{axiom}, the second and third condition implies that $\widehat{\ell}_{\bullet,\epsilon}(q^{\Gamma^+})\ne 0$ for any augmentation $\epsilon$ (if there is no augmentation, then $\Pl(Y)=0$ by definition). Moreover by the first property, we have $q^{\Gamma^+}$ is closed in $(\overline{S}V,\widehat{\ell}_{\epsilon})$ for any augmentation $\epsilon$. Then the claim follows.
\end{proof}

\begin{proof}[Proof of Theorem \ref{thm:upper}]
	We proceed by induction on dimension.
	
	If $\dim Y=3$, then an IP contact $3$-manifold is simply a planar contact $3$-manifold. Fix a choice of planar open book supporting the contact structure, with page a sphere with $k$-disks removed, and so with binding consisting of $k$ circles. One then constructs an adapted Giroux form, so that each component of the binding is a non-degenerate and simply-covered orbit, and a \emph{holomorphic} open book as e.g.\ in \cite{Wendlopenbook}. This provides a Fredholm-regular foliation of $\mathbb{R}\times Y$ whose leaves are either trivial cylinders over the binding, or holomorphic Fredholm-regular $k$-punctured spheres projecting to pages and asymptotic to the binding. One can prove via standard $4$-dimensional arguments coming from Siefring intersection theory (the same as in higher-dimensions, as used below), that any curve whose positive asymptotics are a subset of the binding, is a leaf of this foliation. The argument is as follows: given such a curve $v$, let $u$ be a leaf in the foliation which is not a trivial cylinder (and therefore has no negative ends). Denote the set of positive asymptotics of $v$ by $\Gamma_v^+$. By \cite[Lemma 4.9]{LW}, the Siefring intersection between $u$ and $v$ is given by
	\begin{equation}\label{Siefringint}
	    u*v=\sum_{\gamma_\zeta \in \Gamma_v^+}u*(\mathbb{R}\times \gamma_\zeta).
	\end{equation}
	In a suitable trivialization, we have that the Conley-Zehnder index of every binding component is $\mu_{CZ}(\gamma)=1$. Moreover, we have the relation $\mu_{CZ}(\gamma)=2\alpha_-(\gamma)+p(\gamma)$, where $\alpha_-(\gamma)$ is the largest winding number of an eigenfunction of negative eigenvalue for the asymptotic operator of $\gamma$, and $p(\gamma)$ is the parity of the Conley-Zehnder index. We conclude that $\alpha_-(\gamma)=0$. This coincides with the asymptotic winding number of $u$ about each $\gamma$, which is therefore an extremal winding number. This means that there are no intersections between $u$ and $\mathbb{R}\times \gamma$ coming from infinity, i.e.\ $u*(\mathbb{R}\times \gamma_\zeta)=0$ for each binding component $\gamma$. Combining with \eqref{Siefringint}, we conclude that $u*v=0$. On the other hand, if $v$ were not a leaf in the foliation, by positivity of intersections, we would have $u*v>0$ for some leaf $u$; a contradiction.

	While this a priori holds for an almost complex structure which is compatible with a SHS deforming the contact form, one may perturb this SHS to nearby contact data without changing the isotopy class of the contact form, and for which the binding still consists of closed Reeb orbits. After perturbing the original almost complex structure $J_0$ to a $J$ which is compatible with this nearby contact data and generic, the curves in the foliation survive by Fredholm regularity (an open condition). Moreover, the uniqueness statement still holds if the perturbation is small enough, i.e.\ every holomorphic curve whose positive asymptotics are a subset of the binding is a perturbation of a leaf in the original foliation. Indeed, if $v_k$ is a sequence of such $J_k$-holomorphic curves with $J_k\rightarrow J_0$, by SFT compactness we obtain a building $v_\infty$ as a limit configuration. Applying uniqueness for $J_0$ to the topmost floor, we conclude that $v_\infty$ is a non-cylindrical leaf in the foliation. By Fredholm regularity, it follows that $v_k$ is a perturbed leaf of the foliation for $J_0$, if $k$ is sufficiently large.

    In particular, $(1)_k$ and $(2)_k$ in Lemma \ref{lemma:upperbound} hold, for the perturbed $J$. In this case the geometric (and hence the algebraic) count of these curves with a point constrain is $1$ for any generic point $o$, and so $(3)_k$ is also satisfied. We fix such a $o$ for which we have this uniqueness property.
	
	If $\dim Y \geq 5$, we fix an IP open book $\pi: Y\backslash B \rightarrow S^1$ supporting $\xi$, with binding $B\subset Y$, a codimension-$2$ contact submanifold. Since $B$ is also $k$-IP if $Y$ is, we may assume by induction that $(\star)_k$ holds for $B$. We may then extend the Giroux contact form $\alpha_B$ on $B$ for which $(\star)_k$ holds to a Giroux contact form $\alpha$ on $Y$, in such a way that all $k$ Reeb orbits $\Gamma=(\gamma_1,\dots,\gamma_k)$ from the induction step are still non-degenerate orbits in $Y$. Indeed, one can do this in a local model near the binding: take a collar neighbourhood $B\times \mathbb{D}$ of $B$, and we let $\alpha=(1-r^2)(\alpha_B+r^2 d\theta)$ where $(r,\theta)$ are polar coordinates on the $\mathbb{D}^2$-factor; then $\alpha$ is extended to the mapping torus piece as in Giroux's construction. The Reeb vector field of $\alpha$ coincides with that of $\alpha_B$ along $B$, and so each $\gamma_i$ is still an orbit for $\alpha$. Moreover, the linearized Reeb flow for each $\gamma_i$ splits into components tangent and normal to $B$. The first component is non-degenerate by assumption; the second one is also since the Hessian of the fuction $1-r^2$ is non-degenerate at the critical value $r=0$.
	
	On the other hand, the holomorphic open book construction can also be done in arbitrary dimensions (again, after deforming the Giroux form away from $B$ to a stable Hamiltonian structure, cf.\ \cite[Appendix A]{BGM}, \cite{MorPhD, Mor3bp}). The choice of almost complex structure can be taken to agree with the one from the inductive step along $H_B:=\mathbb{R}\times B$, which is then a holomorphic submanifold. The leaves of the resulting codimension-2 foliation are now either $H_B$, or a codimension-$2$ holomorphic submanifold which is a copy of the Liouville completion of the page, and which is asymptotic to $H_B$ at infinity in the sense of \cite{MS}. We let $\mathcal{F}$ denote this codimension-$2$ foliation on $\mathbb{R}\times Y$. The moduli space $\mathcal{M}_B$ of $k$-punctured spheres defined on $\mathbb{R}\times B$ in the inductive step extends to a moduli space $\mathcal{M}_Y$ on $\mathbb{R}\times Y$, consisting of curves having the same positive asymptotics as curves in $\mathcal{M}_B$. An application of Siefring intersection theory as in \cite{Mor3bp,MS} (this is \emph{exactly} the same argument as in the $3$-dimensional case done above) shows that \emph{any} holomorphic curve $u$ whose positive asymptotics are a subset of $\Gamma$ either completely lies in $H_B$, or its image lies completely in a non-cylindrical leaf $H$ of $\mathcal{F}$. In the first case, $u$ cannot have negative ends by $(1)_k$ applied to $B$. In the second case, since $H$ has no negative ends, $u$ also, and so $(1)_k$ holds on $\mathbb{R}\times Y$. This also shows that $\mathcal{M}_{Y,A,o}(\Gamma,\emptyset)=\mathcal{M}_{B,A,o}(\Gamma,\emptyset)$ for every $o\in H_B$. If moreover we take $o\in H_B$ to be the point given by the inductive step, for which (by the base case) we may assume that we have the uniqueness property that curves in $\mathcal{M}_{B,A,o}(\Gamma;\emptyset)$ are necessarily elements in $\mathcal{M}_B$ (for every $A$), and in particular transversely cut out inside $H_B$. The same analysis as carried out in \cite[Lemma 4.13]{MorPhD} (by splitting the normal linearized CR-operator into tangent and normal components with respect to $H_B$, and using automatic transversality on the normal summand) shows that curves in $\mathcal{M}_B$ are transversely cut out in $\mathbb{R}\times Y$. Then $(2)_k$ holds, and $(3)_k$ also (and in fact the geometric count is $1$ for our particular choice of $o$). Note that all of these conditions still hold after perturbing to nearby contact data, via the same argument as above, i.e.\ using SFT compactness and using uniqueness on the topmost level together with Fredholm regularity. An appeal to Lemma \ref{lemma:upperbound} finishes the proof.    
\end{proof}
It is clear from the definition that the Weinstein conjecture holds for contact manifolds with finite planarity. In the case of iterated planar open books, the Weinstein conjectures was proven for dimension $3$ in \cite{abbas2005weinstein} and higher dimensions in \cite{acu2017weinstein,acu2018planarity}.  In view of the proof of Weinstein conjecture, Theorem \ref{thm:upper} is of the same spirit as the proofs in \cite{abbas2005weinstein,acu2017weinstein,acu2018planarity}. However, more importantly, Theorem \ref{thm:upper} endows the holomorphic curve with SFT meaning. We further remark that the proof of the above theorem in the $5$-dimensional case actually provides a foliation, as opposed to a homological one, as shown in \cite{Mor3bp}.

Theorem \ref{thm:upper} can be viewed as a special case of the following conjecture.
\begin{conjecture}\label{conj}
	Let $Y$ be an open book whose page is $W$, then $\Pl(Y)\le \Pl(W)$ and $\SD(Y)\le \SD(W)$.
\end{conjecture}
In the context of semi-dilations in symplectic cohomology, such a claim was proven in \cite[Proposition 3.31]{zhou2019symplectic} for Lefschetz fibrations. The geometric intuition behind the conjecture is clear and was used in Theorem \ref{thm:upper}, the difficulty lies in making the virtual machinery compatible with the geometry for general $Y$ and $W$.

\subsection{Trivial planar SOBDs} We now consider a related example as to the ones considered above. Fix $(S_k,d\lambda)$ a sphere with $k$-disks removed together with a Liouville form $\lambda$, and let $(M,d\alpha)$ be any $2n$-dimensional Liouville domain. Define $(V:=S_k \times M,\omega=\rd(\lambda+\alpha))$, endowed with the product Liouville domain structure. Let $(Y=\partial V,\xi=\ker (\alpha+\lambda))$, the contact manifold filled by $V$.

\begin{theorem}\label{thm:SOBD}
 If $c_1(M)=0$ or $M$ supports a perfect exhausting Morse function, then we have $\Pl(Y)\le k$.
\end{theorem}

We first note that $Y$ admits a supporting (trivial) SOBD, as considered in \cite{LVHMW,MorPhD}. In other words, we have a decomposition
    $$Y=Y_S\cup Y_P,$$
where $Y_S=\partial S_k \times M$ is the \emph{spine} and $Y_P=S_k \times \partial M$ is the \emph{paper}, and we have trivial fibrations 
    $$
    \pi_S: Y_S\rightarrow M, \quad
    \pi_P: Y_P\rightarrow \partial M.
    $$
We view the first one as a contact fibration over a Liouville domain, and the second, as a Liouville fibration over a contact manifold (its fibers are called the \emph{pages}). By choosing a small Morse function $H$ on the \emph{vertebrae} $M$ (the base of the spine) which vanishes near $\partial M$ and has a unique maximum and no minimum (as critical points), we may perturb the contact form along $Y_S$ to $e^{H}(\alpha + \lambda)$.
    
As explained in \cite[\S 3.1]{MorPhD}, each critical point $p\in M$ of $H$ gives rise to $k$ non-degenerate Reeb orbits of the form $\gamma_{p,i}=S^1\times \{p\}\subset Y_S$ where $S^1$ is the boundary of the $i$th puncture of $S_k$. One then deforms the contact form to a stable Hamiltonian structure which coincides with $\mathcal{H}=(\alpha, \rd(\alpha+\lambda))$ on $Y_P$, and so its kernel there is $TS_k \oplus \ker \alpha$, tangent to $S_k$. After this, one can construct a compatible cylindrical almost complex structure $J_0$, which preserves the splitting $T(\mathbb{R}\times Y_P)=(T\mathbb R \oplus \langle R_\alpha\rangle)\oplus TS_k\oplus \ker \alpha$. We first prove the following proposition which restrains the possibility of holomorphic curves:
\begin{proposition}\label{prop:SOBD}
Let $\Gamma = \{ \gamma_{p_1,j_1},\ldots,\gamma_{p_{s},j_s}\}$ be a set of simple Reeb orbits, such that $p_i$ is a critical point of $H$ and $1\le j_1<j_2\ldots<j_s\le k$. Then there exists an almost complex structure $J$, such that if $H$ is chosen sufficiently $C^2$-small, then any rational $J$-holomorphic curve $u$ in $\mathbb R\times Y$ whose positive asymptotics form a subset of $\Gamma$ is one of the following.
\begin{enumerate}
    \item A trivial cylinder over $\gamma_{p,i}$.
    \item $u$ is a cylinder from $\gamma_{p,j}$ to $\gamma_{q,j}$ in $\R\times S^1\times M \subset \R \times Y_S$, where $S^1$ is the boundary of the $j$th puncture of $S_k$. Moreover, there is a one-to-one correspondence from such cylinders to negative gradient flow lines of $H$ from $p$ to $q$.
    \item $\Gamma=\{ \gamma_{p_1, 1},\ldots,\gamma_{p_k,k}\}$ and $u$ has no negative puncture. 
\end{enumerate}
\end{proposition}
\begin{proof}
We first prove the claim for the stable Hamiltonian structure $\cH$. First of all, one can achieve that all orbits below a fixed action threshold correspond to critical points of the Morse function. By energy reasons, and using that the number of positive ends of $u$ is a priori bounded by $k$, one can choose the action threshold large enough (depending on $k$) so that the negative ends of $u$ also necessarily correspond to critical points, and the number of negative ends (counted with covering multiplicity) is bounded above by the number of positive ends; cf.\ \cite[Lemma 3.8]{MorPhD}. We then separate in two cases: either the image of $u$ is fully contained in $\mathbb{R}\times Y_S$ (case A), or it is not (case B).

\medskip

\textbf{Case A}. Since $Y_S$ has $k$ connected components, by the assumptions on $\Gamma$, $u$ has precisely one positive end, and by the above discussion, we have that $u$ has at most one negative end. Since orbits $\gamma_{p,i}$ are not contractible in $Y_S$, we then see that $u$ has precisely one negative end, which is simply covered. 

We now argue via holomorphic cascades. In the degenerate case $H\equiv 0$, the trivial projection $\pi_S:\mathbb{R}\times Y_S\rightarrow M$ is actually $J_0$-holomorphic. Therefore $v=\pi_S\circ u$ is a holomorphic curve in $M$. Since the asymptotics of $u$ project to points in $M$, $v$ extends to a holomorphic map on a closed surface. But $M$ is exact, and so $v$ is constant, and we deduce that $u$ is a trivial cylinder.
 
If we take $H_t=tH$, we denote by $J_t$ the corresponding almost complex structure (which only differs in that it intertwines the varying $R_t$ with the $\mathbb R$-direction). If we assume we have a sequence $\{u_n\}$ of $J_{t_n}$-holomorphic maps with $t_n\rightarrow 0$, with one positive and negative simply covered orbits corresponding to critical points $p_\pm$, then we obtain a stable holomorphic cascade \cite{BourgeoisPhD} $\mathbf{u}^H_\infty$ as a limiting object. 
Since the positive end of $u_n$ is simply covered, by energy reasons as explained above, \emph{every} Reeb orbit appearing in $\mathbf{u}^H_\infty$ is simply covered, and therefore every of its holomorphic map components cannot be multiply covered. Stability of the cascade means that it does not have trivial cylinder components. We conclude that the space of holomorphic cascades which glue to curves as in our hypothesis consists solely of flow-lines, which are regular by the Morse--Smale condition, and come in a $(ind_{p_+}(H)-ind_{p_-}(H)-1)$-dimensional family. Note that such flow-lines cylinders can \emph{always} be glued to another flow-line cylinder. The implicit function theorem then implies that $u$ is a flow-line cylinder if $H$ is sufficiently small. Indeed, we can argue as follows. Consider the parametric moduli space
$$
\overline{\mathcal{M}}:=\overline{\mathcal{M}}(\mathbb{R}\times M_Y,\{J_t\}_{t\in [0,1]};\gamma_{p_+},\gamma_{p_-})
$$
$$
=\{(t,u): t \in [0,1], u \in \overline{\mathcal{M}}(\mathbb{R}\times M_Y,J_t;\gamma_{p_+},\gamma_{p_-})\} 
$$ 
where $\overline{\mathcal{M}}(\mathbb{R}\times M_Y,J_t;\gamma_{p_+},\gamma_{p_-})$ denotes the Gromov-compactified moduli space of $J_t$-holomorphic curves in $\mathbb{R}\times M_Y$ (of any genus) which have simply covered positive asymptotic $\gamma_{p_+}$, and negative asymptotic $\gamma_{p_-}$, and where the compactification at $t=0$ corresponds to stable holomorphic cascades. Since the Reeb orbits are simply covered, curves in the parametric moduli are somewhere injective. Then, by the Morse--Smale condition and the implicit function theorem, the parametric moduli space is, for sufficiently small $H$, a $(ind_{p_+}(H)-ind_{p_-}(H))$-dimensional compact manifold whose boundary contains $\overline{\mathcal{M}}(\mathbb{R}\times M_Y,J^0;\gamma_{p_+},\gamma_{p_-})$, only consisting of flow line cylinders. The flow-line parametric moduli space 
$$
\overline{\mathcal{M}}_{\mbox{flow-line}}:=\overline{\mathcal{M}}_{\mbox{flow-line}}(\mathbb{R}\times M_Y,\{J_t\}_{t\in [0,1]};\gamma_{p_+},\gamma_{p_-})
$$
$$
:=\{(t,u): t \in [0,1], u \mbox{ corresponds to a flow-line in } \overline{\mathcal{M}}(\mathbb{R}\times M_Y,J_t;\gamma_{p_+},\gamma_{p_-})\}
$$
is, a priori, a submanifold of $\overline{\mathcal{M}}(\mathbb{R}\times M_Y,\{J_t\}_{t\in [0,1]};\gamma_{p_+},\gamma_{p_-})$, which shares the boundary component $\overline{\mathcal{M}}(\mathbb{R}\times M_Y,J^0;\gamma_{p_+},\gamma_{p_-})$, and has its same dimension. By thinking of $\overline{\mathcal{M}}$ as a collar neighbourhood of this boundary component, we obtain 
$$\overline{\mathcal{M}}=\overline{\mathcal{M}}_{\mbox{flow-line}},$$ from which our uniqueness follows. Observe that since there are finitely many critical points, we can take $H$ (or $\epsilon>0$) uniformly small.

In what follows, we start with a $J=J_t$ corresponding to $t$ sufficiently small.

\medskip

\textbf{Case B.} We argue via energy considerations. Denote by $Z^\pm \subset \Sigma$ the positive/negative punctures of $u$, where $\Sigma$ is a closed Riemann surface, and $\dot \Sigma=\Sigma\backslash(Z^+\cup Z^-)$ is the domain of $u$. Since orbits associated to $Z^+$ correspond to critical points of $H$, so does every orbit associated to $Z^-$ (which may a priori be multiply covered). For $z\in Z^\pm$, denote by $\gamma_{p_z}^{\kappa_z}$ the corresponding orbit, where $p_z \in \mbox{crit}(H)$, and $\kappa_z\geq 1$ is the covering multiplicity ($\kappa_z=1$ for all $z\in Z^+$). By assumption, $\# Z^+\leq k=\# \pi_0(Y_S)$. We now estimate the $\Omega$-energy of $u$, where $\mathcal{H}=(\beta,\Omega)$ is the SHS, and $\Omega$ is exact, the primitive being $e^H(\alpha+d\theta)$ along $Y_S$. Using Stokes' theorem, we obtain
\begin{eqnarray*}
\mathbf{E}(u)&:= &\int_{\dot \Sigma} u^*\Omega=2\pi\left(\sum_{z\in Z^+} e^{H(p_z)}-\sum_{z\in Z^-}\kappa_ze^{H(p_z)}\right) \\
&\leq & 2\pi\Vert e^H\Vert_{C^0}\#Z^+\leq 2\pi k\Vert e^H\Vert_{C^0}. 
\end{eqnarray*}
We now bound $\mathbf{E}(u)$ from below. The projection $p: \mathbb{R}\times Y_P\rightarrow S_k$ is also holomorphic. Moreover, the same is true in a small closed $\epsilon$-neighbourhood of $\mathbb{R}\times Y_P$ inside $\mathbb{R}\times Y$, which intersects $\mathbb{R}\times Y_S$ along a collar of the form $\mathbb{R}\times \partial S_k \times [0,\epsilon] \times \partial M$, along which $\beta$ picks up a $\rd\theta$ component. Let $S:=S_k\cup_\partial \partial S_k \times [0,\epsilon]$ this small extension of $S_k$. For each $s\in S$, $E_s:=\mathbb{R}\times \{s\} \times \partial M$ is a holomorphic hypersurface, the fiber of $p$ over $s$. By assumption, $u$ has non-trivial intersection with $\mathbb{R}\times Y_P$, which is foliated by these holomorphic hypersurfaces $E_s$.  By Sard's theorem, we may shrink $\epsilon$ if necessary so that $u$ is transverse to $\mathbb{R}\times \partial S \times M$. Therefore $\Sigma':=u^{-1}(\mathbb{R}\times S\times M)$ is a smooth (possibly disconnected) surface with boundary. Note that since the asymptotics of $u$ are away from $S$, the intersection of $u$ with any of the $E_s$ consists of a finite collection of points in the domain of $u$, which are away from the punctures. We then see that the algebraic intersection $\deg(u):=u\cdot E_s$ is \emph{positive}, independent of $s$, and in fact is the degree of $F:=p \circ u: \Sigma' \rightarrow S$, which is a holomorphic branched cover.

Write $\partial \Sigma'=\bigcup_{i=1}^lC_i$, oriented with the boundary orientation, where $C_i$ is a simple closed curve whose image under $F$ wraps around one of the boundary components of $S$ with winding number $n_i\neq 0$ (measured with respect to $d\theta$). By holomorphicity of $F$, one easily sees that $n_i>0$. Since $u$ intersects $E_s$ for every $s\in \partial S$, we have $l\geq k$. By counting preimages of a point $s\in \partial S$, we obtain
$$
\int_{\partial \Sigma'} u^*\rd\theta=2\pi \sum_{i=1}^ln_i=2\pi k \deg(u).
$$
Using that $\Omega=\rd\alpha+\rd\lambda$ along $S$, $\lambda=e^t\rd\theta$ near $\partial S$ where $t\in [0,\epsilon]$, $u^*\rd\alpha\geq 0$ by choice of $J_0$, and Stokes' theorem, we see that
$$
\mathbf{E}(u)\geq \int_{\Sigma'} u^*\Omega=\int_{\Sigma'} u^*(\rd\alpha + \rd\lambda)\geq \int_{\Sigma'} u^*\rd\lambda=e^{\epsilon}\int_{\partial\Sigma'} u^*\rd\theta=2\pi ke^{\epsilon} \deg(u).
$$
Combining with the previous upper bound for $\mathbf{E}(u)$, we obtain
\begin{eqnarray*}
2\pi k & \leq & 2\pi ke^{\epsilon} \deg(u)\leq 2\pi\left(\sum_{z\in Z^+} e^{H(p_z)}-\sum_{z\in Z^-}\kappa_ze^{H(p_z)}\right)\\ & \leq & 2\pi\Vert e^H\Vert_{C^0}\#Z^+\leq 2\pi k\Vert e^H\Vert_{C^0}.
\end{eqnarray*}
Since $H$ can be taken arbitrarily close to $0$, we see that $\#Z^+=k$, $\#Z^-=0,$ $\deg(u)=1$. This proves the proposition for the SHS $\cH$ and $J_0$. It is clear the same holds for any nearby contact structure and almost complex structure.
\end{proof}

\begin{proof}[Proof of Theorem \ref{thm:SOBD}]
If $\Gamma^+$ is set of Reeb orbits in form of $\gamma_{p,i}$ of cardinality at most $k$, then Proposition \ref{prop:SOBD} implies that $\cM_{Y}(\Gamma^+,\Gamma^-)$ is empty if $\Gamma^-\ne \emptyset$ unless $\Gamma^+=\{\gamma_{p,i}\},\Gamma^-=\{\gamma_{q,i}\}$ with $\ind p>\ind q$. In particular, this implies that $\widehat{\ell}_{\epsilon}(q^{\Gamma^+})$ is independent of $\epsilon$. Moreover, we have $\vdim \cM_{Y,o}(\{\gamma_{p,i}\},\{\gamma_{q,i}\})=\ind(p)-\ind(q)-2n<0$, hence $\widehat{\ell}_{\bullet,\epsilon}(q^{\Gamma^+})$ is also independent of $\epsilon$. As a result, it suffices to find one $\epsilon$ and one such $\Gamma^+$, such that $\widehat{\ell}_{\epsilon}(q^{\Gamma^+})=0$ and $\widehat{\ell}_{\bullet,\epsilon}(q^{\Gamma^+})\ne 0$. In view of Proposition \ref{prop:equi}, the latter condition is satisfied if $\eta_{M\times S_k} (q^{\Gamma^+})\ne 0$. In particular, we can work with the nice exact filling $M\times S_k$.

Assume we start with a Liouville form on $S_k$ such that the simple Reeb orbits on each boundary component have period $1$ and hence the area of $S_k$ is $k$. Following \cite[\S 2.1]{zhou2020filings}, by choosing the Liouville from on $M$ sufficiently large, i.e.\ the Reeb orbits on $\partial M$ have sufficiently large period, we may assume all Reeb orbits below a period threshold are in the form of $\gamma^m_{p,i}$, i.e.\ the $m$-th multiple cover of the Reeb orbit over the critical point $p$ of $H$ and the $i$-th boundary of $S_k$, such that the period of $\gamma^m_{p,i}$ is approximately $m$. On homology level (both in $H_1(\partial(M\times S_k))$ and $H_1(M\times S_k)$), the only relations are $[\gamma^m_{p,i}]=m[\gamma_{p,i}], [\gamma_{p,i}]=[\gamma_{q,i}]$ and $\sum_{i=1}^k [\gamma_{p,i}]=0$. 

Using the existence of symplectic caps (see \cite[Theorem 1.6]{MR4171373}, \cite[Corollary 1.14]{lazarev2020maximal}, \cite[Theorem 3.2]{MR1465333}), we know that there exists an exact embedding of $M$ into $W$, where the contact boundary $\partial W$ is the contact boundary of some affine variety. Then we consider a projective compactification of the affine variety by adding simple normal crossing divisors $D$, but instead we add the divisors $D$ to $W$ to obtain $\overline{W}$. Then $\overline{W}$ is closed symplectic manifold, such that the symplectic form is Poincar\'e dual to $D$. Now we consider holomorphic curves in $\overline{W}\times \CP^1$, which contains $M\times S_k$ as a domain with contact boundary, here the symplectic area of $\CP^1$ is slightly bigger than $k$ to contain $S_k$. We use $\overline{\cM}_{J,o}$ to denote the compactified moduli spaces of holomorphic spheres passing through $o\in M\times S_k$ with homology class $A+[*\times \CP^1]$ for $A\in H_2(M)$ with $c_1(A)=0$. We can form such compactification, since $M$ is exact, in particular, those curve have uniformly bounded symplectic area. If we pick $J$ to be the splitting almost complex structure $\overline{W}\times \CP^1$, then $\overline{\cM}_{J,o}$ is cut out transversely and consists of exactly one point. Then we deform the almost complex structure such that it coincides with the almost complex structure in Proposition \ref{prop:SOBD} near $\partial (M\times S_k)$ and also apply neck-stretching along $\partial (M\times S_k)$. We claim transversality holds for this process for generic choices. First we will prove that the moduli space of curves in class  $A+[*\times \CP^1]$ for $A\in H_2(M)$ (hence $A$ has zero symplectic area) with $c_1(A)=0$ has no bubble degeneration in the compactification. For this, we pick an almost complex structure that makes $D\times \CP^1$ holomorphic and is in a split form near $D\times \CP^1$. Suppose there is a bubble degeneration with a component in class $B+k[*\times \CP^1]$ for $B\in H_2(\overline{W})$. If $B$ has negative symplectic area, i.e.\ the intersection with $D\times \CP^1$ is negative, by the positivity of intersection, we have the component is contained in $D\times \CP^1$. Since the almost complex structure is splitting, the projection to $D$ has negative symplectic area, contradiction. Therefore all components in a hypothetical bubble degeneration are in the form of $B+k[*\times \CP^1]$ such that the symplectic area of $B$ is zero. Now since $[\CP^1]$ in $H_2(\CP^1)$ can not be decomposed into multiple components with positive symplectic area, we see that there is no bubble degeneration. Note that the region where we make assumptions on the almost complex structure is outside the contact boundary that we apply neck-stretching to. In the fully stretched case, due to action and first homology class reasons, we must have the bottom curve, the one with point constraint, has asymptotics in the form of $\Gamma^+=\{\gamma_{p_1,1},\ldots,\gamma_{p_k,k}\}$. Moreover, since all relevant Reeb orbits are simple, we may assume the full-stretched moduli space is also cut out transversely.

For $\Gamma^+=\{\gamma_{p_1,1},\ldots,\gamma_{p_k,k}\}$, we first note that $\vdim \cM_{M\times S_k,A,o}(\Gamma^+,\emptyset)$ is 
\begin{equation}\label{eqn:vdim}
 2\langle c_1(M), A \rangle +\sum_{i=1}^k\ind(p_i)-2nk
\end{equation}
where we view the relative homology class $A$ as in $H_2(M)$. To see this, we first consider $S_k$. Let $\gamma_i$ be the Reeb orbit on the $i$-th boundary, using a trivialization of the symplectic vector bundle over $S_k$, the Conley-Zehnder indices satisfies
$$\sum_{i=1}^k\mu_{CZ}(\gamma_i)=2(2-k).$$
Following \cite[The proof of Theorem 6.3]{zhou2019symplecticI}, in the product $M\times S_k$, the normal direction picks up an extra Conley-Zehnder index of $\ind(p_i)-n$, i.e.\ using the trivialization over the bounding surface $S_k$, we have 
$$\sum_{i=1}^k\mu_{CZ}(\gamma_{p_i,i})=\sum_{i=1}^k(\ind(p_i)-n)+2(2-k).$$
Hence 
$$\vdim \cM_{M\times S_k,A,o}= 2\langle c_1(M), A \rangle + (n-2)(2-k)+\sum_{i=1}^k\mu_{CZ}(\gamma_{p_i,i})-2n=\eqref{eqn:vdim}.$$

Next we separate the proof into two cases. (1) If $c_1(M)=0$, then $\vdim \cM_{M\times S_k,A,o}(\Gamma^+,\emptyset)$ is independent of the relative homology class $A$. Moreover, the expected dimension in \eqref{eqn:vdim} is non-positive and is zero if and only if  $p_1=\ldots=p_k$ is the unique maximum $p_{\max}$ of $H$.  Then the above neck-stretching argument implies that 
$$\bigcup_A \cM_{M\times S_k,A,o}(\Gamma^+,\emptyset)=\bigcup_A \overline{\cM}_{M\times S_k,A,o}(\Gamma^+,\emptyset)$$
and has nontrivial algebraic count for $\Gamma^+=\{\gamma_{p_{\max},1},\ldots,\gamma_{p_{\max},k}\}$. In particular, we have $\eta_{M\times S_k}(q^{\Gamma^+})\ne 0$. By Proposition \ref{prop:SOBD}, we have 
$$\widehat{\ell}_{\epsilon}(q^{\Gamma^+})=\sum_{i=1}^k \sum_{\ind q=2n-1} \#\overline{\cM}_{Y}(\{\gamma_{p_{\max},i}\},\{\gamma_{q,i}\})q_{\gamma_{p_{\max},1}}\ldots q_{\gamma_{q,i}}\ldots q_{\gamma_{p_{\max},k}}=0,$$
since $p_{\max}$ is the unique maximum. Hence we have $\Pl(Y)\le k$. (2) If $M$ supports a perfect exhausting Morse function, then we can assume the perturbation $H$ is perfect. Since $c_1(M)$ is not necessarily $0$, the neck-stretching argument above only gives us 
\begin{align*}
\bigcup_{2c_1(A)+\sum_{i=1}^k \ind (p_i)-2nk=0} & \cM_{M\times S_k,A,o}(\Gamma^+,\emptyset) = \\
& \bigcup_{2c_1(A)+\sum_{i=1}^k \ind (p_i)-2nk=0} \overline{\cM}_{M\times S_k,A,o}(\Gamma^+,\emptyset)    
\end{align*}
with nontrivial algebraic count for some $\Gamma^+=\{\gamma_{p_1,1},\ldots,\gamma_{p_k,k}\}$.  Since $H$ is perfect, i.e.\ the count of rigid gradient flow lines between every pair of critical points is zero,  Proposition \ref{prop:SOBD} implies that $\widehat{\ell}_{\epsilon}(q^{\Gamma^+})=0$. Therefore we have $\Pl(Y)\le k$.
\end{proof}

%\begin{remark}
%Similar proofs work for the SOBDs considered in \cite{MorPhD}. The difference for those examples is that the paper has two connected components $Y_P^\pm$, $Y_P^-$ having genus zero pages, $Y_P^+$ having positive genus ones (i.e.\ the SOBD is not symmetric). Although there is a rational curve passing through a generic point $o\in Y_P^-$, this is not true for $o \in Y_P^+$. This is explained as follows: these examples have finite algebraic planar torsion by the proof of \cite[Theorem 1.4]{MorPhD} , in particular, there is no $BL_\infty$ augmentation, hence the planarity is $0$. Note that the proof of Lemma \ref{lemma:upperbound} and Theorem \ref{thm:SOBD} shows that the ruling curve has homological meaning only if there are augmentations.
%\end{remark}

\begin{corollary}\label{cor:product}
	Let $V$ be an affine variety with $c_1(V)=0$, such that $\AU(V)\ge k$. Then $\Pl(\partial(S_k\times V))=k$.
\end{corollary}
\begin{proof}
	By Theorem \ref{thm:SOBD}, we have $\Pl(\partial(S_k\times V))\le k$. One the other hand it is easy to see that $\AU(S_k\times V)=k$, as every algebraic curve in $V\times S_k$ projects to an algebraic curve in both $V$ and $S_k$. Then the claim follows from Corollary \ref{cor:lower}.
\end{proof}

\section{Examples and applications}\label{s7}
In this section, we will discuss two more classes of examples, where we can compute the hierarchy functors. The first case is smooth affine varieties with a $\CP^n$ compactification, or more generally a Fano hypersurface compactification. The second case is links of singularities, including links of Brieskorn singularities and quotient singularities by the diagonal action of cyclic groups. In particular, we will finish the proof of Theorem \ref{thm:RSFT}.

\subsection{Affine varieties}\label{ss:affine}
Let $V$ be a smooth affine variety, then $V$ is naturally a Weinstein manifold by viewing $V\subset \C^N$, and the function $|x-x_0|^2$ on $\C^N$ restricted to $V$ is a Morse function with finitely many critical points for a generic $x_0\in \C^N$ \cite[\S 6]{milnor2016morse}. In particular, we obtain a contact manifold by taking the intersection of $V$ with the boundary of a large enough ball. We will use $\partial V$ to denote the contact boundary of the intersection, the interior of intersection is called the associated Liouville domain. Both notions up to homotopy are independent of the size of large ball.

An alternative way of associating a Weinstein structure to $V$ is by using a smooth projective compactification $\overline{V}$, with an ample line bundle $\cL$ with a holomorphic section $s$ such that $s^{-1}(0)$ is a normal crossing divisor such that $V=\overline{V}\backslash{s^{-1}(0)}$. We choose a metric on $\cL$, such that the curvature is a K\"ahler form $\omega$ on $\overline{V}$. Then by \cite[Lemma 4.3]{seidel2006biased}, $h=-\log |s|$ and $-\rd^{\C}h$ defines a Weinstein structure (possibly after a compactly supported perturbation) on $V$. The equivalence of these two definitions can be found in \cite{mclean2014symplectic}. 

We first give a description on the embedding relations of affine varieties with the same projective compactification. 

\begin{lemma}\label{lemma:perturb}
	Let $X$ be a smooth projective variety with a very ample line bundle $\cL$ endowed with a Hermitian metric. For $s\in H^0(\cL)$, we use $V_s$ to denote the Liouville domain associated to the affine variety $X\backslash s^{-1}(0)$. Then for $s\ne 0\in H^0(\cL)$, there exists $\epsilon>0$, such that for all $t\in H^0(\cL)$ with $|s-t|<\epsilon$, we have $V_s$ embeds exactly into $V_t$.
\end{lemma}
\begin{proof}
	With the very ample line bundle $\cL$, $X$ can be embedded in $\mathbb{P}H^0(\cL)$ such that every $s\in H^0(\cL)$ corresponds to a hyperplane $H_s \subset \mathbb{P}H^0(\cL) $ and $s^{-1}(0)=X\cap H_s$. We can view the Liouville domain $V_s$ as the intersection of $X$ with a large ball in the identification of  $\C^{N}$ with $\mathbb{P}H^0(\cL) \backslash H_s$. Then for $t$ sufficiently close to $s$, i.e.\ $H_t$ sufficiently closed to $H_s$, the Liouville form of $V_t$ restricted to $V_s\cap S_{R}$ is a contact form, where $S_R$ the radius $R\gg 0$ sphere in $\C^N$. The Gray stability theorem implies that all of them induce the same contact structure on $ V_s\cap S_R$ for $t$ sufficiently close to $s$, and the Liouville form of $V_t$ restricted to $V_s\cap B_R$ is homotopic to the Liouville structure on $V_s$; hence $V_s$ embeds exactly into $V_t$.
\end{proof}

\begin{remark}
	In principle, the exact embedding from $V_s$ to $V_t$ should be built from a Weinstein cobordism. Hence one expects a more precise description of the Weinstein cobordism, which depends on the deformation from $s$ to $t$. Some results in this direction can be found in \cite{acu2020introduction,nguyen2015complement}.
\end{remark}

Roughly speaking, we should have a stratification on $\mathbb{P} H^0(\cL)$ indexed by the singularity type of $s^{-1}(0)$. The index set forms a category by declaring a morphism from stratum $A$ to stratum $B$ if the closure of $B$ contains $A$. Then Lemma \ref{lemma:perturb} implies that we have a functor from the index set (which should be a poset) to $\cont_*$. Making such description precise is not easy, as we do not have a classification of singularities of $s^{-1}(0)$ in general. However, we can describe some subcategory of the index set. The following lemma is also very useful in understanding the embedding relations of affine varieties which arise from different line bundles.
\begin{lemma}[{\cite[Lemma 4.4]{seidel2006biased}}]\label{lemma:multi}
Assume that the smooth affine variety $V$ has a smooth projective compactification $\overline{V}$. Assume there are two ample line bundles $\cL_i$ with sections $s_i$, such that $s_1^{-1}(0)=s_2^{-1}(0)=\overline{V}\backslash V$ is normal crossings, but possibly with different multiplicities. Then the Liouville structures on $V$ defined by the $s_i$ are homotopic.
\end{lemma}

\begin{example}
	$\CP^{n}$ minus $k$ generic hyperplanes can be viewed as the complement of $(s_1\otimes \ldots \otimes s_k)^{-1}(0)$ for generic sections $s_i$ of $\cO(1)$. On the other hand, $\CP^{n}$ minus $k-1$ generic hyperplanes can be viewed as the complement of $(s_1\otimes s_1 \otimes s_3 \otimes  \ldots \otimes s_k)^{-1}(0)$ by Lemma \ref{lemma:multi}. As a consequence of Lemma \ref{lemma:perturb}, we have an exact embedding of $\CP^{n}$ minus $k-1$ generic hyperplanes to $\CP^{n}$ minus $k$ generic hyperplanes. As a simple example, $\CP^2$ minus a line is $\C^2$, $\CP^2$ minus two generic lines is $\C \times T^*S^1$ and $\CP^2$ minus three generic lines is $T^*T^2$. It is clear that we have the embedding relations. Moreover some of the relations can not be reversed, e.g.\ $T^*T^2$ can not be embedded exactly into $\C^2$ or $\C\times T^*S^1$. But $\C\times T^*S^1$ can be embedded back into $\C^2$ by adding a $2$ handle corresponding to the positive Dehn twist in the trivial open book for $\partial (\C\times T^*S^1 )$. More generally, $\CP^n$ minus $k$ generic hyperplanes is $\C^{n+1-k}\times T^*T^{k-1}$ for $k\le n$, and they can be embedded into each other exactly. 
\end{example}	

\begin{example}
	$\CP^2$ minus $3$ hyperplanes passing through the same point is $\C \times S_3$, where $S_3$ is the thrice punctured sphere. Since $\CP^2$ minus $2$ hyperplanes can still be viewed as a further degeneration, we have $\C\times T^*S^1$ embeds to $\C \times S_3$, which is obviously true. On the other hand, $\CP^2$ minus $3$ generic hyperplanes, i.e.\ $T^*T^2$, contains $\C \times S_3$ as an exact subdomain. Moreover, $\CP^2$ minus a smooth degree $2$ curve is $T^*\mathbb{RP}^2$, which is obtained from attaching a $2$-handle to $\C\times T^*S^1$, i.e.\ $\CP^2$ minus $2$ generic lines. $\CP^2$ minus a smooth degree $3$ curve can be described as attaching three $2$-handles to $T^*T^2$, see \cite{acu2020introduction} for details. It is not obvious if the complement of a smooth degree $2$ curve embeds exactly into the complement of a smooth degree $3$ curve. However, the former embeds exactly into the complement of a smooth degree $4$ curve by Lemma \ref{lemma:perturb} and Lemma \ref{lemma:multi}.
\end{example}

Let $D$ be a divisor, we use $D^{\com}$ to denote the complement affine variety. Our main theorem in this section is the following.
\begin{theorem}\label{thm:CP}
	Let $D$ be $k$ generic hyperplanes in $\CP^n$ for $n\ge 2$, then we have the following.
     \begin{enumerate}
     	\item $\Pl(\partial D^{\com})\ge k+1-n$ for $k>n+1$.
     	\item $\Pl(\partial D^{\com})= k+1-n$ for  $n+1<k<\frac{3n-1}{2}$ and $n$ odd.
     	\item $\Pl(\partial D^{\com})=2$ for $k=n+1$. 
     	\item $\Hcx(\partial D^{\com})=0^{\SD}$ for $k\le n$.
     \end{enumerate}
\end{theorem}
The strategy to obtain Theorem \ref{thm:CP} is to first prove $\Pl(\partial D^{\com})\ge \max \{1, k+1-n\}$ by index computations. Then we obtain that the planarity of the affine variety $D^{\com}$ is at most $\max \{1, k+1-n\}$ by looking at the affine variety $D^{\com}_s$, where $D_s$ is the smoothing of $D$, i.e.\ a smooth degree $k$ hypersurface. Finally, we use index computations to show that the relevant portion of the computation of planarity is independent of the $BL_\infty$ augmentation for $\RSFT(\partial D^{\com})$ when  $n+1<k<\frac{3n-1}{2}$. The benefit of considering the smoothing is twofold. First, the relevant holomorphic curve theory for the pair $(\CP^n,D_s)$ is easier both in terms of analysis and computation. In particular, we can use the computation of relative Gromov-Witten invariants in \cite{gathmann2002absolute} to supply the holomorphic curve we need. Second, $\partial D^{\com}$ carries more complicated Reeb dynamics than $\partial D_s^{\com}$, which makes it harder to prove the independence of augmentations by looking at $\partial D^{\com}$ alone. More precisely, for any augmentation of $\RSFT(\partial D^{\com}_s)$ given by the composition of the $BL_{\infty}$ morphism from $\RSFT(\partial D^{\com}_s)$ to $\RSFT(\partial D^{\com})$ with an augmentation of $\RSFT(\partial D^{\com})$, we will show that the planarity for $\RSFT(\partial D^{\com}_s)$ of that augmentation is $\max \{1, k+1-n\}$. Then, planarity of $\RSFT(\partial D^{\com})$ is  $\max \{1, k+1-n\}$ for any augmentation by functoriality. The condition of $n$ being odd is to obtain automatic closedness of a suitable chain in $\overline{S}V_{\partial D_s^{\com}}$ for any augmentation, and is expected to be irrelevant. However, to drop this constraint, we need to use stronger transversality properties supplied by \cite{zhou2020quotient}, see Remark \ref{rmk:n=3} for more discussion. The $n+1<k<\frac{3n-1}{2}$ condition is likely not optimal but possibly necessary. It is a difficult task to compute planarity for all augmentations. There are many affine varieties with a $\CP^n$ compactification such that the contact boundary has infinite planarity, while the planarity of the affine variety, i.e.\ using the augmentation from the affine variety, is finite, see Theorem \ref{thm:infinity}. In particular, different augmentations do make a difference in general. Therefore it is a subtle question to determine which affine variety has finite planarity. In general, we need to develop a computation method of RSFT from log/relative Gromov-Witten invariants like the symplectic (co)homology computation in \cite{diogo2019symplectic}.

\subsubsection{Reeb dynamics on the divisor complement}
In this part, we describe the Reeb dynamics on the boundary of a tubular neighborhood of a simple normal crossing divisor. The general description was obtained in \cite[\S 5]{mclean2012growth}, see also \cite[\S 2.1]{ganatra2020symplectic}, \cite{affine}, \cite{MR3866910} or \cite[\S 2.4]{LC}. In the following, we state the special cases for Theorem \ref{thm:CP}.

\medskip

\textbf{Case 1: a smooth degree $k$ hypersurfaces in $\CP^n$ for $n\ge 2$.} Let $D$ denote a smooth degree $k$ hypersurface in $\CP^n$, then the contact boundary is the concave boundary of the $\cO(k)$ line bundle over $D$, which carries a natural Morse-Bott contact form whose Reeb flow is the $S^1$ action on the circle bundle. Consider the hypersurface $r=f$ in $\cO(k)$, where $r$ is the radial coordinate. Given an action threshold $C$, one can choose a $C^2$-small and positive Morse function $f$ such that all Reeb orbits in this hypersurface, of action less than $C$, have the following properties. 
\begin{enumerate}
	\item There is a simple Reeb orbit $\gamma_p$ over every critical point $p$ of $f$ and these are all of the simple Reeb orbits. We use $\gamma^m_p$ to denote the $m$-th cover of $\gamma_p$. All of the Reeb orbits are good and non-degenerate. The period of $\gamma_p^m$ is greater than than the period of $\gamma^l_q$ iff $m=l$ and $f(p)<f(q)$\footnote{Note that $f$ is the $r$-coordinate of the perturbed hypersurface in the symplectic cap $\cO(k)$, hence larger value of $f$ means smaller period of the $S^1$-fiber.} or $m>l$. As the contact structure is the Boothby-Wang contact structure, computation can be found in e.g.\ \href{https://www.math.snu.ac.kr/~okoert/tools/CZ_index_BW_bundle.pdf}{this note} by van Koert. It is also a special case of the more general normal crossing case in \cite[Theorem 2.7]{LC}.
	\item  Using the obvious disk cap bounded by $\gamma^m_q$ in the symplectic cap $\cO(k)|_D$\footnote{i.e.\ map from $\CP^1\backslash \D$ and view the boundary as a negative end. In other words, it is the unit disk $\D$ mapped to $\cO(k)|_D$, such that the induced boundary map with the usual boundary orientation is $-\gamma^m_q$.} that intersects $D$ once with order $m$, which induces a trivialization of $\det_{\C}\xi$, we have that the Conley-Zehnder index satisfies
	\begin{equation}\label{eqn:index}
	    n-3-\mu_{CZ}(\gamma^m_p)=2m-2+\ind(p),
	\end{equation}
	where $\ind(p)$ is the Morse index of $p$. This is a special case of \cite[Proposition 2.10]{LC}.
 	\item $[\gamma_p] \in H_1(D^{\com})$ is a generator of order $k$, as the intersection of a generic line $A\simeq \CP^1\subset \CP^n$ with $D^{\com}$ shows that $k[\gamma_p]$ is null-homologous. On the other hand, if $l[\gamma_p]$ is null-homologous in $D^{\com}$ for $l<k$, by gluing with the obvious caps, we get a second homology class in $\CP^n$ whose intersection number with $D$ is $l$, contradiction.
\end{enumerate}
One way to understand \eqref{eqn:index} is following: $n-3-\mu_{CZ}(\gamma^m_p)$ is the virtual dimension of the moduli space of holomorphic disks in $\cO(k)$ with the same relative homology class as the bounding disk. Then $n-3-\mu_{CZ}(\gamma^m_p)+2-2m$ is the virtual dimension of the moduli spaces of holomorphic disks in the bounding disk homology class, with one marked point intersecting $D$ with order $m$. In the Morse-Bott case, i.e.\ the contact form is Boothby–Wang and $f$ is used in a cascades model, then the moduli space mentioned above is cut out transversely and is identified with the stable manifold of $p$, whose dimension is $\ind(p)$. When $p_{\min}$ is the unique minimum of $f$, by the argument in \cite{bourgeois2009symplectic} and transversality for the cascade model, we know that the moduli space of holomorphic disks in the cap $\cO(k)$ asymptotic to $\gamma^m_{p_{\min}}$ with one marked point intersecting $D$ of order $m$ and homology class the bounding disk is cut out transversely, compact, with algebraic count $\pm 1$.

\medskip

\textbf{Case 2: $k$ generic hyperplanes in $\CP^n$ for $k\ge n+1$.}
Let $D_1,\ldots,D_k$ denote the $k$ hyperplanes. Let $I\subset \{1,\ldots,k\}$ be a set of cardinality at most $n$. We define $D_I$ to be the intersection $\cap_{i\in I} D_i$, which is a copy of $\CP^{n-|I|}$. We define $\check{D}_I$ by $D_I\backslash (D_I\cap \cup_{i\notin I} D_i)$. Let $N_i$ be the normal disk bundle over $D_i$. Then $\bigoplus_{i\in I} \partial N_i|_{D_I}$ is a $T^{|I|}$ bundle over $D_I$. Then the contact boundary is topologically decomposed as $\bigcup_{I\subset \{1,\ldots,k\}} \bigoplus_{i\in I} \partial N_i|_{\check{D}_I}$. We pick an exhausting Morse function $f_I$ on each $\check{D}_I$, i.e.\ the gradient of $f_I$ is pointing out along $\partial \check{D}_I$. The Reeb dynamics has the following properties.
\begin{enumerate}
	\item For each critical point $p$ of $f_I$ and any function $t:\{1,\ldots,k\} \to \N$ with $\supp t=I$,  we have a $T^{|I|-1}$ Morse-Bott family of Reeb orbits $\gamma^{t}_{p}$. The homology class of $\gamma^t_p$ is given by the image of $t\in H_1(T^{|I|})=\Z^{|I|}$ in $H_1(D^{\com})$ by inclusion, where $T^{|I|}$ is the $T^{|I|}$ fiber of  $\bigoplus_{i\in I} \partial N_i|_{\check{D}_I}$  over $p$. This follows from \cite[\S 5]{mclean2012growth} or \cite[Theorem 2.17]{MR3866910}, which is summarized in \cite[Theorem 2.7]{LC}.
	\item $H_1(D^{\com})$ is generated by the simple circles $[\beta_i]$ wrapping around $D_i$ once (i.e.\ the oriented boundary of small disks that intersect with $D_i$ negatively once) subject to the relation $\sum_{i=1}^k [\beta_i]=0$, see \cite[Proposition 2.3]{zbMATH05183749}. The homology class $[\gamma^{t}_p]$ over $\check{D}_I$ is $\sum_{i\in I} t(i) [\beta_i]$ \cite[Theorem 2.7 (5)]{LC}. Moreover, there is a natural disk cap with boundary $\gamma^t_p$ intersecting $D_i$ of order $t(i)$ (by cap instead of disk/filling, we emphasize the boundary map with induced orientation is $-\gamma^t_p$.).
	\item The generalized Conley-Zehnder index, using the obvious disk cap whose intersection number with $D_i$ is $t(i)$ for $i\in I$, is given by 
    $$n-3-\mu_{CZ}(\gamma^{t}_p)=2\sum_{i\in I} t(i)-2+\ind(p)+\frac{|I|-1}{2}.\footnote{The extra $\frac{|I|-1}{2}$ is from the $T^{|I|-1}$ Morse-Bott family, which after perturbation spans the region of Morse indices of the a Morse function $T^{|I|-1}$ as in \eqref{eqn:region}. It particular, it is analogous to case 1 for smooth divisors. An analogous situation can be found in \cite[Theorem 5,16]{MR3489704}. The two situations are different in the sense that \cite{MR3489704} considered the symplectic filling by normal crossing divisors instead of a symplectic cap and used a preferred global trivialization of a power of the contact distribution.}$$
    In the following, we use $\sum t$ as a shorthand for $\sum_{i\in I} t(i)$. After a perturbation, the $T^{|I|-1}$ family of Reeb orbits degenerate to $2^{|I|-1}$ many non-degenerate orbits corresponding to generators of $H_*(T^{|I|-1})$, and the Conley-Zehnder indices span the following region: 
   \begin{equation}\label{eqn:region}
       n-3-\mu_{CZ} \in \left[2\sum t-2+\ind(p), \quad 2\sum t-2+\ind(p)+|I|-1\right].
   \end{equation}
    This follows from \cite[Proposition 2.10]{LC}. 
    We use $\check{\gamma}^t_p$ to denote the orbit with $n-3-\mu_{CZ}(\check{\gamma}^t_p)=2\sum t-2+\ind(p)$ and $\hat{\gamma}^t_p$ to denote the orbit with $n-3-\mu_{CZ}(\hat{\gamma}^t_p)= 2\sum t-2+\ind(p)+|I|-1$.
\end{enumerate}
\begin{remark}\label{rmk:action}
    Following \cite[Lemma 5.17,5.18]{mclean2012growth}, the period of $\gamma^t_p$ is close to (smaller than) $\sum_{i\in I} t(i)$ with a small discrepancy depending on the symplectic size of the neighborhood of $D$ that is removed and a smaller discrepancy from the perturbation from $f$, with the property that the period of $\gamma^t_p$ is smaller than that of $\gamma^t_q$ iff $f(p)>f(q)$. In particular, $\sum t$ can be thought as the period of the Morse-Bott family with the ``ideal" case of only removing the divisor $D$. The period will be further perturbed after we perfect the $T^{|I|-1}$ family into non-degenerate orbits. Those newly created orbits have periods arbitrarily close to the period of $\gamma_p^t$ and the period of $\check{\gamma}_p^t$ is larger than the period of $\hat{\gamma}_p^t$. In the general case, for $s\in H^0(\cL)$ with $s^{-1}(0)$ normal crossings, and $s^{-1}(0)=\sum_{i=1}^k a_i D_i$ as a divisor, when we use $-\rd^{\C}\log |s|$ as the Liouville structure on $X\backslash s^{-1}(0)$, after the deformation as in \cite[Lemma 5.17]{mclean2012growth} to organize the Liouville form nicely near the boundary, we have a similar description of Reeb orbits as the above and the period of the Reeb orbits with intersection order $t:\{1,\ldots,k\}\to \N$ with $\supp t =I$ over $\check{D}_I$ is given by $\sum_{i=1}^k a_i t(i)$ minus a (arbitrarily) small discrepancy. 
\end{remark}

\subsubsection{Lower bound of $\Pl(\partial D^{\com})$}
\begin{proposition}\label{prop:unirule}
	Let $D=D_1\cup \ldots \cup D_k$ denote the $k> n+1$ generic hyperplanes in $\CP^n$ for $n\ge 2$.
	\begin{enumerate}
		\item For any Reeb orbits set $\Gamma:=\{\gamma_1,\ldots,\gamma_r\}$ for $r<k+1-n$ with $\sum [\gamma_i]=0\in H_1(D^{\com})$, the virtual dimension of the moduli space $\overline{\cM}_{D^{\com}, A, o}(\Gamma,\emptyset)$ is less than $0$ for any $A$.
		\item For any Reeb orbits set  $\Gamma:=\{\gamma_1,\ldots,\gamma_{k+1-n}\}$ with  $\sum [\gamma_i]=0\in H_1(D^{\com})$ and such that the virtual dimension of the moduli space $\overline{\cM}_{D^{\com}, A, o}(\Gamma,\emptyset)$ is non-negative, there is a partition of $\{1,\ldots,k\}$ into subsets $I_1,\ldots,I_{k+1-n}$, such that $\Gamma=\{\check{\gamma}^{\sigma_{I_i}}_{p_{i,\min}}\}_{i}$, where $p_{i,\min}$ is the minimum on $\check{D}_{I_i}$ and $\sigma_{I_i}$ is the indication function supported on $I_i$, i.e.\ $\sigma_{I_i}(j)=1$ if $j\in I_i$ and is otherwise zero.
	\end{enumerate}
\end{proposition}
\begin{proof}
	Note that $c_1(D^{\com})=0$, hence the virtual dimension does not depend on $A$, and we will suppress it from the notation in the following discussion (the same applies everywhere in this subsection). Given a curve $u$ in the same homotopy class of a curve in $\cM_{D^{\com}, o}(\Gamma,\emptyset)$, we use $u_i$ to denote the natural disk cap of $\gamma_i$. Then we have
	$$\ind(u)+\sum_{i=1}^{r}(n-3-\mu_{CZ}(\gamma_i))=2c_1(u\#_{i=1}^r u_i)+2(n-3)-2n+2=2c_1(u\#_{i=1}^r u_i)-4.$$ Here, note that the $-2n+2$ comes from the point constraint. We assume $\gamma_i$ is of the form $\gamma^{t_i}_{p_i}$ after perturbations. Since $\sum_{i=1}^r [\gamma_i]=0$ in homology, if we view each $t_i$ as a $k$-dimensional integral vector, we have  $\sum_{i=1}^r t_i =(N,\ldots,N)$ for some $N\in \N_+$. Since $t_i$ keeps track of the intersection of the natural disk $u_i$ with $D$, we know that $[u\#_{i=1}^r u_i]\cap D$ is $Nk$ points. In particular, we have $[u\#_{i=1}^r u_i]$ is $N$ times the generator in $H_2(\CP^n)$ and $c_1(u\#_{i=1}^r u_i)=N(n+1)$. Then we have
	\begin{eqnarray}
		\ind(u) & \le & 2N(n+1)-4-\sum_{i=1}^r(2\sum t_i-2+\ind(p_i) ) \label{eqn:ineq1}\\
		& \le & 2N(n+1)-4-2\sum_{i=1}^r \sum t_i+2r \label{eqn:ineq2}\\
		& = & 2N(n+1-k)-4+2r\nonumber\\
		& = & 2(N-1)(n+1-k)+2(r+n-k-1) < 0\nonumber
 	\end{eqnarray}
    when $r<k+1-n.$ If $r=k+1-n$, to have $\ind(u)\ge 0$, we must have $N=1$. In this case, both inequalities \eqref{eqn:ineq1} and \eqref{eqn:ineq2} must be equality. In particular, $\ind(p_i)=0$ and $\gamma_i$ must be a check orbit $\check\gamma_{p_i}$, i.e.\ the claim holds.
\end{proof}

Then by Proposition \ref{prop:equi}, we have the following.
\begin{corollary}[lower bound on $P(\partial D^c)$]\label{cor:lower_bound} 
	If $D$ denotes $k$ generic hyperplanes for $k> n+1$ and $n\ge 2$, then $\Pl(\partial D^c)\ge k+1-n$. The same holds for $\partial V_s$, if $s$ is a perturbation of such $D$.
\end{corollary}
\begin{proof}
	Let $D$ be $k$ generic hyperplanes as in the statement, then $\eta_{D^{\com}}$ on $\overline{B}^r V_{\partial D^{\com}}$ in Proposition  \ref{prop:equi} is zero for $r<k+1-n$ by Proposition \ref{prop:unirule}. Therefore $\Pl(\partial D^{\com})\ge k+1-n$. The remaining of the claim follows from Lemma \ref{lemma:perturb} and functorial property of $\Pl$.
\end{proof}
In the following, we will separate the proof of upper bound of $\Pl(\partial D^{\com})$ into two steps, namely, we will first show the existence of a holomorphic curve that is responsible for the finite planarity, and then we will argue that the phenomena is independent of augmentations for certain $k$.

\subsubsection{Step one for the upper bound of $\Pl(\partial D^{\com})$ -- source of holomorphic curves}
\begin{proposition}\label{prop:curve}
	Let $D_s$ be a smooth degree $k$ hypersurfaces in $\CP^n$, then the following holds.
	\begin{enumerate}
		\item\label{curve1} If $k\le n$, for a point $o\in D_s^{\com}$, there is a Reeb orbit $\gamma^k_p$ with $\ind(p) =2(n-k)$ and an admissible complex structure, such that $\overline{\cM}_{D_s^{\com}, o}(\{\gamma^k_p\},\emptyset)$ is cut out transversely and $\#\overline{\cM}_{D_s^{\com}, o}(\{\gamma^k_p\},\emptyset)\ne 0$.
		\item\label{curve2} If $k \ge n+1$, for a point $o\in D_s^{\com}$, there are two Reeb orbits $\gamma^n_{p_{\min}},\gamma_{p_{\min}}$ with $p_{\min}$ the unique minimum on $D_s$ and an admissible almost complex structure, such that  $$\overline{\cM}_{D_s^{\com}, o}(\{\gamma^n_{p_{\min}},\underbrace{\gamma_{p_{\min}},\ldots,\gamma_{p_{\min}}}_{k-n}\},\emptyset)$$ is cut out transversely with nontrivial algebraic count.
	\end{enumerate}
\end{proposition}

\begin{proof}
	This follows from applying neck-stretching to $\CP^n$ along $\partial D^{\com}_s$. We denote the relative Gromov-Witten invariant that counts genus $0$ holomorphic curves in class $A$ with $k$ marked points going through $C_1,\ldots,C_k\in H_*(\CP^n)$ and $l$ marked point going through $E_1,\ldots,E_l\in H_*(D_s)$ and intersect $D_s$ with multiplicity at least $s_1,\ldots,s_l$ respectively by $\GW^{\CP^n,D_s}_{0,k,(s_1,\ldots,s_l),A}(C_1,\ldots,C_k,E_1,\ldots,E_l)$ \cite{ionel2003relative}.  The source of holomorphic curves is from the non-vanishing relative Gromov-Witten invariants $\GW^{\CP^n,D_s}_{0,1,(k),A}([pt],[D_s]\cap^{n-k}[H])$ for $k\le n$ and $\GW^{\CP^n,D_s}_{0,1,(n,1,\ldots,1),A}([pt],[D_s],\ldots, [D_s])$ for $k>n$ respectively from \cite{gathmann2002absolute}, where $H\in H_{2n-2}(\CP^n)$ is the hyperplane class and $A$ is the generator of $H_2(\CP^n)$. More precisely, when $k>n$, by the divisor axiom, we have, 
    $$\GW^{\CP^n,D_s}_{0,1,(n,1,\ldots,1),A}([pt],[D_s],\ldots, [D_s])=(k)^{k-n}\GW^{\CP^n,D_s}_{0,1,(n),A}([pt],[D_s]).$$
    By applying \cite[Theorem 2.6]{gathmann2002absolute}\footnote{$D_{\alpha,k}(X,\beta)$ in \cite[Theorem 2.6]{gathmann2002absolute}, i.e.\ those bubble trees with a subtree contained in the divisor, is empty, as we are considering degree $1$ curves.}  $n-1$ times, we have 
    \begin{eqnarray*}
        \GW^{\CP^n,D_s}_{0,1,(n),A}([pt],[D_s]) & = & \int_{\overline{\mathcal{M}}_{0,2}(\CP^n,A)} ev_1^*\mathrm{PD}([pt])\wedge ev_2^*\mathrm{PD}([D_s]) \wedge \prod_{i=1}^{n-1}\left(\ev_2^*\mathrm{PD}([D_s])+i\psi\right)\\
        & = & \int_{\overline{\mathcal{M}}_{0,2}(\CP^n,A,pt)} ev_2^*\mathrm{PD}([D_s]) \wedge  \prod_{i=1}^{n-1}\left(\ev_2^*\mathrm{PD}([D_s])+i\psi\right)
    \end{eqnarray*}
    where $\overline{\mathcal{M}}_{0,2}(\CP^n,A)$ is the compactified moduli space of holomorphic spheres in class $A\in H_2(\CP^n)$ with two marked points ($ev_1,ev_2$ are two evaluation maps), and $\overline{\mathcal{M}}_{0,2}(\CP^n,A,pt)$ is the one with the first marked point is subject to a point constraint. $\psi$ is the psi class (\cite[Page 183]{gathmann2002absolute}) in  descendant Gromov-Witten invariants \cite[\S 4.5.5]{MR2262630}. Note that  $\overline{\mathcal{M}}_{0,2}(\CP^n,A,pt)\simeq \CP^n$ and $\psi$ is the first Chern class of the tautological line bundle, as a consequence, we have 
    $$\GW^{\CP^n,D_s}_{0,1,(n),A}([pt],[D_s]) = \prod_{i=1}^n(k-i)>0.$$
    Similarly, when $k<n$, we have 
    $$\GW^{\CP^n,D_s}_{0,1,(k),A}([pt],[D_s]\cap^{n-k}[H]) = \int_{\overline{\mathcal{M}}_{0,2}(\CP^n,A,pt)} ev_2^*\mathrm{PD}([D_s]\cap^{n-k}[H]) \wedge  \prod_{i=1}^{k-1}\left(\ev_2^*\mathrm{PD}([D_s])+i\psi\right)=k!$$
    
    Since a curve in class $A$ is necessarily somewhere injective and not contained in $D_s$ because we can choose the $[pt]$ class in $D_s^{\com}$, one can assume transversality in the process of neck-stretching. In the fully stretched picture, each connected component of the bottom curve has at most $\max\{1,k+1-n\}$ positive punctures for otherwise genus has to be created. If the component of the bottom curve with the point constraint has $0<r<\max\{1,k+1-n\}$ positive punctures, in particular, $k\ge n+1$, we assume the positive asymptotics are $\Gamma^+=\{\gamma^{d_i}_{p_i}\}_{1\le i \le r}$. Since $\sum [\gamma_{p_i}^{d_i}]=0$ in homology and $[\gamma_{p_i}]$ is the generator of $H_1(D^{\com}_s)$ of order $k$, we have $\sum d_i=km$ for some $m$. Then the expected dimension of $\overline{\cM}_{D^{\com}_s,o}(\Gamma^+,\emptyset)$ is given by
	\begin{eqnarray*}
	\ind(u) & = & 2m(n+1)-4-\sum_{i=1}^r(n-3-\mu_{CZ}(\gamma^{d_i}_{p_i}))\\
	& = & 2m(n+1)-4 -\sum_{i=1}^r(2d_i-2+\mbox{ind}(p_i))\\
            & \le & 2m(n+1)-4-2\sum_{i=1}^r d_i+2r \\
            & = & 2m(n+1-k)+2r-4 = 2(m-1)(n+1-k)+2(r+n-k-1)<0
	\end{eqnarray*}
    Therefore the component of the bottom curve with the point constraint must have $\max\{1,k+1-n\}$ positive punctures. Moreover, from the above computation, we separate the proof into three cases.
    
    \medskip
    
    \textbf{(1) Case $k>n+1$.} To have $\ind(u)\ge 0$, we must have that $p_i$ is the minimum $p_{\min}$ and $m=1$. That is, the positive asymptotics of the bottom curve are $\{\gamma^{d_i}_{p_{\min}}\}_{1\le i \le k+1-n}$ with $\sum d_i=k$. Next, we consider the curves in the symplectic cap. Because a curve in the symplectic cap must intersect $D_s$, there are at most $k+1-n$ connected components of the top curve. If there are less than $k+1-n$ connected components of the top curve, genus must be created because the bottom curve component with a point constraint has $\max\{1,k+1-n\}$ positive punctures. As a consequence, there is one component intersects $D_s$ with order $n$ and $k-n$ components intersect $D_s$ with order $1$.  Note that component $v$ in the symplectic cap that intersects $D_s$ with order $n$ must be asymptotic to $\{\gamma^{n_1}_{q_1},\ldots \gamma^{n_l}_{q_l}\}$ for critical points $q_*$ with $\sum_{i=1}^l n_i=n$. Assume otherwise that the sum of multiplicity is $n+km$ for $m\ge 1$, then the relative homology class of $v$ is the same as the sum of the natural disk of $\gamma^{n+km}_q$ and $-mA$ for the positive generator $A\in H_2(D_s)$ that is mapped to the generator of $H_2(\CP^n)$. Then the symplectic area of $v$ is the sum of the area of the natural disk and $-m\omega_{\CP^n}(A)$. Since the symplectic area of the natural disk can be arbitrarily small if we only remove a sufficiently small neighborhood of the divisor, the symplectic area of $v$ is negative, contradiction. Similarly, the component intersecting $D_s$ with order $1$ must be a sphere with one negative puncture asymptotic to $\gamma_{q}$ for a critical point $q$. The component intersecting $D_s$ with order $n$ must be a sphere with one negative puncture asymptotic to $\gamma^n_{q}$ for a (potentially different) critical point $q$. Assume otherwise, to glue to a sphere in $\CP^n$, we must have at least another component in the bottom level. Since the total symplectic area of the curves outside the symplectic cap is approximately $k$ times the period of $\gamma_q$, which is approximately the symplectic area of the bottom curve with a point constraint, there is no action room for another bottom-level curve. Therefore the curves in the middle symplectization level must be cylinders, we must have $(d_1,\ldots d_{k+1-n})=(n,1,\ldots,1)$. Moreover, since $\gamma_{p_{\min}}^n,\gamma_{p_{\min}}$ has the maximum period in the respective homology class, there is no room for nontrivial curves in the symplectization level. The bottom curve moduli space $$\cM_{D_s^{\com}, o}(\{\gamma^n_{p_{\min}},\underbrace{\gamma_{p_{\min}},\ldots,\gamma_{p_{\min}}}_{k-n}\},\emptyset)$$ consists of somewhere injective curves, for otherwise, we assume that $u\in \cM_{D_s^{\com}, o}(\{\gamma^n_{p_{\min}},\gamma_{p_{\min}},\ldots,\gamma_{p_{\min}}\},\emptyset)$ is a branched cover over $u'$, then we can cap off $u'$ with natural disks to obtain a homology class $A$ in $H_2(\CP^2)$ with $A\cap D_s<k$, which is a contradiction. It is direct to check that the holomorphic disks in the symplectic cap (i.e.\ disk fibers) are cut out transversely (see the discussion after \eqref{eqn:index}), hence transversality holds for the fully stretched situation. Therefore we have $\#\overline{\cM}_{D_s^{\com}, o}(\{\gamma^n_{p_{\min}},\gamma_{p_{\min}},\ldots,\gamma_{p_{\min}}\},\emptyset)\ne 0$.
    
    \medskip
    
    \textbf{(2) Case $k=n+1$.} Then to have $\ind(u)\ge 0$, we must have that $p_i$ is the minimum of $p_{\min}$ but $m\ge 1$. By the same area argument for the cap, the total contact action of negative asymptotics of curves in the symplectic cap is close to $k$. On the other hand, the total contact action of $\{\gamma_{p_i}^{d_i}\}_{1\le i \le 2}$ is close to $mk$. Hence we must have $m=1$ and curves in the symplectic cap must be once punctured spheres that are asymptotic to  $\gamma^n_q$ and $\gamma_{q'}$. Then the remaining of the argument is the same as before.
    
    \medskip
    
    \textbf{(3) Case $k\le n$}. Since the bottom level with the point constraint has one positive puncture, that is asymptotic to $\gamma^{km}_q$. By the same area and action argument, we have $m=1$ and the both top and bottom level has one component with one puncture. Assume the negative asymptotics of the disk cap is $\gamma^k_p$, then we must have $\ind(p) \ge 2(n-k)$ to have non-negative expected dimension for the disk. On the other hand, for the bottom curve, we must have $\ind(q)\le 2(n-k)$ to have non-negative expected dimension.  Therefore if $p\ne q$,  the expected dimension of the cylinders in the symplectization is negative. Hence we have $p=q$ with $\ind(p)=2(n-k)$.  Then we know that there is at least one critical point $p$ with $\ind(p)=2(n-k)$ such that $\#\overline{\cM}_{D_s^{\com},o}(\{\gamma^k_p\},\emptyset)\ne 0$ and the unstable manifold of $p$ represents multiples of $[D_s]\cap^{n-k}[H]$. 
\end{proof}

\begin{remark}[Algebraic Gromov-Witten invariants v.s.\ symplectic Gromov-Witten invariants]
    The proof of Proposition \ref{prop:curve} makes an inexplicit assumption that the algebraically defined Gromov-Witten invariants used in \cite{gathmann2002absolute} are the same as the symplectic version, where we only use a compatible almost complex structure. Without such assumption, the above argument only shows that moduli spaces in Proposition \ref{prop:curve} are not empty by compactness but the algebraic count might be zero. Such equivalence is expected, but not established. However, for the very special case of degree $1$ curves in Proposition \ref{prop:curve}, such equivalence is easier to establish. Here we only mention two strategies to establish this special equivalence to make  Proposition \ref{prop:curve} completely self-contained: (1) Translate the proofs in \cite{gathmann2002absolute} into the symplectic version; (2) Interpret the algebraic Gromov-Witten invariants in Proposition \ref{prop:curve} as Euler classes of an obstruction bundle over $\overline{\mathcal{M}}_{0,2}(\CP^n,A,pt)\simeq \CP^n$, then establish the equivalence with the symplectic version. We will not pursue the details of those arguments in this paper.
\end{remark}

\begin{corollary}\label{cor:planarity}
	Let $D_s$ be the smooth degree $k$ hypersurface in $\CP^n$ for $k>n+1$ and $n\ge 2$. Then $\eta_{D_s^{\com}}(q_{\gamma^n_{p_{\min}}}q^{k-n}_{\gamma_{p_{\min}}})\ne 0$ and $q_{\gamma^n_{p_{\min}}}q^{k-n}_{\gamma_{p_{\min}}}$ is closed in $(\overline{B}^{k+1-n}V_{\partial D^{\com}},\widehat{\ell}_{\epsilon_{D_s^{\com}}})$.
\end{corollary}
\begin{proof}
	We may assume the Morse function on $D_s$ is perfect; this follows from direct check for $n=2$, from \cite{harer1978handlebody} for $n=3$, and the Lefschetz hyperplane theorem and the $h$-cobordism theorem for $n\ge 4$. Then Proposition \ref{prop:curve} implies that $\eta_{D_s^{\com}}(q_{\gamma^n_{p_{\min}}}q^{k-n}_{\gamma_{p_{\min}}})\ne 0$. It suffices to prove that $q_{\gamma^n_{p_{\min}}}q^{k-n}_{\gamma_{p_{\min}}}$ is closed. Since we have the parity of the SFT grading is the same as the Morse index, and $q_{\gamma^n_{p_{\min}}}q^{k-n}_{\gamma_{p_{\min}}}$ has even grading, we only need to consider if $ \widehat{\ell}_{\epsilon_{D_s^{\com}}}(q_{\gamma^n_{p_{\min}}}q^{k-n}_{\gamma_{p_{\min}}})$ contains any $q_{\gamma^m_p}$ with $\ind(p) = n-1$ for $n$ even. As a consequence, we need consider $\overline{\cM}_{\partial D_s^{\com}}(\Gamma^+,\Gamma^-)$ for $\Gamma^+\subset \{\gamma^n_{p_{\min}},\gamma_{p_{\min}},\ldots,\gamma_{p_{\min}} \}$ and $\Gamma^-=\{\gamma^m_p,\gamma^{d_1}_{p_1},\ldots \gamma^{d_s}_{p_s}\}$, then we close off $\{\gamma^{d_i}_{p_i}\}_{1\le i \le s}$ and a subset of the complement of $\Gamma^+$ by the augmentation from $D_s^{\com}$. On the other hand, by homology reason, we know the sum of multiplicities of $\Gamma^-$ equals to the sum of multiplicities of $\Gamma^+$. As a consequence, there is no subset of $\{\gamma^{d_i}_{p_i}\}_{1\le i \le s}$ whose sum with a subset of complement of $\Gamma^+$ represents a null-homologous class in $D_s^{\com}$. In particular, there is no room for augmentation from $D_s^{\com}$ to apply and we only need to consider $\overline{\cM}_{\partial D_s^{\com}}(\Gamma^+,\{\gamma^m_p\})$ where $m$ is the sum of the multiplicities of $\Gamma^+$. It is direct to check  the expected dimension of this moduli space is $(n-2)+2|\Gamma^+|-2$, which is strictly positive whenever $n\ge 3$. When $n=2$, it is direct check that only case with expected dimension $0$ is 
	$\overline{\cM}_{\partial D_s^{\com}}(\{\gamma^2_{p_{\min}}\},\{\gamma^2_p\})$ and $\overline{\cM}_{\partial D_s^{\com}}(\{\gamma_{p_{\min}}\},\{\gamma_p\})$, each of them corresponds to moduli space of gradient trajectories from minimum $p_{\min}$ to the index $1$ critical point $p$, whose algebraic count is zero, as our Morse function is perfect. Therefore $q_{\gamma^n_{p_{\min}}}q^{k-n}_{\gamma_{p_{\min}}}$ is closed in $(\overline{B}^{k+1-n}V_{\partial D_s^{\com}},\widehat{\ell}_{\epsilon_{D_s^{\com}}})$.
\end{proof}

\begin{remark}
    Corollary \ref{cor:planarity} and Proposition \ref{prop:unirule} essentially imply that the planarity of $D_s^{\com}$ is $k+1-n$ for $k>n+1$. In the proof of Corollary \ref{cor:planarity}, we use the topology of the filling $D^{\com}_s$ to get some restrictions on the augmentation, in particular the augmentation respects the homology classes of orbits. However, we can not run such an argument for general augmentations to obtain Theorem \ref{thm:proj}.
\end{remark}

\subsubsection{Step two for the upper bound of $\Pl(\partial D^{\com})$ -- independence of augmentations} So far, we proved that $\Pl(D^{\com}_s)\le k+1-n$ (the planarity using the augmentation from the exact filling $D^{\com}_s$) for a smooth degree $k>n+1$ divisor $D_s$. Even if we assume the functoriality of $\Pl$ for exact domains in Remark \ref{rmk:func} was proven, and then we have that $\Pl(D^{\com})\le k+1-n$ for $D$ the $k$ generic hyperplanes, we still need to argue that the computation is independent of augmentations. In the following, we first show that the independence of augmentation is not tautological.

\begin{theorem}\label{thm:infinity}
	Let $D_s$ be a smooth degree $k>2n-3$ hypersurface in $\CP^{n}$, then $\Pl(\partial D^{\com}_s)=\infty$.
\end{theorem}
\begin{proof}
We claim $\# \overline{\cM}_{\partial D_s^{\com}}(\Gamma^+,\emptyset)=0$ and $\# \overline{\cM}_{\partial D_s^{\com},o}(\Gamma^+,\emptyset)=0$. For this we use a cascades model (but only the compactness), i.e. we consider the Boothby-Wang contact form on $\partial D^{\com}_s$.	Following the compactness argument in \cite{bourgeois2009symplectic}, if we degenerate the contact form on $D_s$ (as perturbed by the Morse function) to the Boothby-Wang contact form, the curves in $\overline{\cM}_{\partial D_s^{\com}}(\Gamma^+,\emptyset),\overline{\cM}_{\partial D_s^{\com}, A,o}(\Gamma^+,\emptyset)$ degenerate to cascades. But since $|\Gamma^+|\ne \emptyset$, there is one level containing nontrivial holomorphic curves in the symplectization of the Boothby-Wang contact form, which projects to a holomorphic sphere in $D_s$. However since $k>2n-3$, there is no holomorphic sphere in $D_s$. Hence the claim follows. $\# \overline{\cM}_{\partial D_s^{\com}}(\Gamma^+,\emptyset)=0$ implies that  $\epsilon^k=0$ for all $k\ge 1$ form a $BL_{\infty}$ augmentation. Then  $\# \overline{\cM}_{\partial D_s^{\com},o}(\Gamma^+,\emptyset)=0$ implies that the planarity is $\infty$ using such augmentation.
\end{proof}
If one applies neck-stretching to the curve found in Proposition \ref{prop:curve}, we will get a SFT building, which might contain curves with negative punctures subject to a point constraint and augmentation curves in the filling. Theorem \ref{thm:infinity} is a situation where the augmentation from the natural filling and the trivial algebraic augmentation yield different computations. The following proposition singles out the module spaces for $\partial D^{\com}_s$ that might influence the computation of planarity for different augmentations. 

\begin{proposition}\label{prop:max}
	Let $D_s$ be a smooth degree $k\ge n+1$ hypersurface in $\CP^{n}$, assume $\Gamma^+$ is a subset of $\{\gamma^n_{p_{\min}},\underbrace{\gamma_{p_{\min}},\ldots,\gamma_{p_{\min}}}_{k-n} \}$. Then for $\Gamma^-\ne \emptyset$, $\# \overline{\cM}_{\partial D_s^{\com},o}(\Gamma^+,\Gamma^-)=0$ unless $\Gamma^+=\{\gamma^n_{p_{\min}}\}$, $\Gamma^-=\{\gamma^n_{p_{\max}}\}$ or $\Gamma^+=\{\gamma_{p_{\min}}\}$, $\Gamma^-=\{\gamma_{p_{\max}}\}$.
\end{proposition}
\begin{proof}
We can assume $\Gamma^-=\{\gamma^{d_1}_{p_1},\ldots,\gamma^{d_r}_{p_r}\}$ with $\sum d_i$ is the total multiplicity of $\Gamma^+$ (which is at most $k$) by homology and action reasons, as the total multiplicity of $\Gamma^-$ is at most the total multiplicity of $\Gamma^+$ with the difference a multiple of $k$.
Then we can run the Morse-Bott compactness argument as in Theorem \ref{thm:infinity} for $\overline{\cM}_{\partial D_s^{\com},o}(\Gamma^+,\Gamma^-)$, in the limit cascades moduli space, we necessarily have the holomorphic curve part have zero energy and hence a constant. Therefore due to the generic point constraint $o$, we must have $p_1=\ldots = p_r=p_{\max}$.  Then the expected dimension of such moduli space is computed by
	$$2k-2|\Gamma^+|+\vdim \overline{\cM}_{\partial D_s^{\com},A,o}(\Gamma^+,\Gamma^-) - 2|\Gamma^-|-2k = -4. $$
	Hence $\vdim \overline{\cM}_{\partial D_s^{\com},A,o}(\Gamma^+,\Gamma^-)=2|\Gamma^+|+2|\Gamma^-|-4$, which is zero iff $|\Gamma^+|=|\Gamma^-|=1$. Hence the claim follows.
\end{proof}

\begin{remark}\label{rmk:max}
	In the case considered in Proposition \ref{prop:max}, the only non-empty moduli spaces contributing to the pointed map are $\overline{\cM}_{\partial D_s^{\com},o}({\gamma^n_{p_{\min}}},{\gamma^n_{p_{\max}}})$ and  $\overline{\cM}_{\partial D_s^{\com},o}({\gamma_{p_{\min}}},{\gamma_{p_{\max}}})$. Moreover, the algebraic count is not zero as the gradient trajectories from $p_{\min}$ to $p_{\max}$ traverse the whole manifold. This follows from a cascades construction with gluing as in \cite{bourgeois2009symplectic}.
\end{remark}

Theorem \ref{thm:infinity} along with Corollary \ref{cor:planarity} shows that computation of planarity can \emph{depend} on the augmentations. In the special case of the contact boundary of smooth divisor complements, Proposition \ref{prop:max} isolates how this dependence works, i.e.\ $BL_\infty$ augmentations to $q_{\gamma^n_{p_{\max}}} q^{k-n}_{\gamma_{p_{\min}}}$ and  $q_{\gamma^n_{p_{\min}}} q_{\gamma_{p_{\max}}} q^{k-n-1}_{\gamma_{p_{\min}}}$ determined whether $\{\gamma^n_{p_{\min}},\underbrace{\gamma_{p_{\min}},\ldots,\gamma_{p_{\min}}}_{k-n} \}$ contributes to finite planarity. Indeed if we apply neck-stretching to the moduli space for  $\eta_{D_s^{\com}}(q_{\gamma^n_{p_{\min}}}q^{k-n}_{\gamma_{p_{\min}}})\ne 0$ in Corollary \ref{cor:planarity} with the point constraint picked from the contact boundary along a copy of the contact boundary that is pushed in a bit, Proposition \ref{prop:max} and the proof of Theorem \ref{thm:infinity} imply that we must have non-trivial augmentations\footnote{Indeed, this is case. Roughly speaking, the augmentation to  $q_{\gamma^n_{p_{\max}}} q^{k-n}_{\gamma_{p_{\min}}}$. } to  $q_{\gamma^n_{p_{\max}}} q^{k-n}_{\gamma_{p_{\min}}}$ or $q_{\gamma^n_{p_{\min}}} q_{\gamma_{p_{\max}}} q^{k-n-1}_{\gamma_{p_{\min}}}$ using the standard filling $D_s^{\com}$\footnote{Indeed, the augmentation to  $q_{\gamma^n_{p_{\max}}} q^{k-n}_{\gamma_{p_{\min}}}$ (up to a multiple related to the multiplicity of Reeb orbits) counts degree $1$ curves in $\CP^n$ passing through a fixed point of the divisor $D_s$ with multiplicity $n$ and $k-n$ marked point passing through $D_s$, using the divisor axiom and \cite[Proposition 3.4]{cieliebak2018punctured}, such counting should be $(n-1)!k^{k-n}\ne 0$.}. However in Theorem \ref{thm:infinity}, we choose the trivial augmentation which kills the planarity.

Let $D$ be $k$ generic hyperplanes. To prove the upper bounds for Theorem \ref{thm:CP}, we need to
\begin{enumerate}
    \item find a collection of Reeb orbits on $\partial D^{\com}$ that bounds a rational curve with a point constraint;
    \item show that the collection of Reeb orbits represents a closed class in $(\overline{B}^*V, \widehat{\ell}_{\epsilon})$ for any augmentation;
    \item show that the non-trivial planarity also does not depend on augmentation. 
\end{enumerate}
Although there is an obvious candidate for step 1, to establish results in the same spirit of Corollary \ref{cor:planarity}, we need to understand curves in the symplectic cap/neighborhood of the normal crossing divisors. In principle, one should be able to set up a relation between RSFT curves in an affine variety (complement of simple normal crossing divisors) with the log Gromov-Witten invariants. However, this is technically much harder than relating the relative Gromov-Witten invariants with RSFT curves of the complement of a smooth divisor in Corollary \ref{cor:planarity}. For step 2 and 3, as the contact boundary $\partial D^{\com}$ has a much more complicated Reeb dynamics compared to the smooth case, it is highly non-trivial to establish them. Therefore the strategy to obtain the upper bounds in Theorem \ref{thm:CP} is to take advantage of simple Reeb dynamics on $\partial D_s^{\com}$ and functoriality, where $D_s$ is a degree $k$ smooth hypersurface. More precisely, we use $X$ to denote the exact cobordism from $\partial D^{\com}$ to $\partial D^{\com}_s$ from Lemma \ref{lemma:perturb}. From the discussion above the planarity of $\partial D^{\com}_s$ depends on augmentations, it is still possible for the planarity for  $\partial D^{\com}_s$ to be finite if we only use augmentation in the form of $\epsilon\circ \phi$, where $\epsilon$ is the augmentation of $\RSFT(\partial D^{\com})$ and $\phi$ is the $BL_\infty$ morphism from $X$, this the content of Proposition \ref{prop:ind'}. Then we use the functoriality in Proposition \ref{prop:order} and argue that the computation we did with the filling $D_s^{\com}$ in Corollary \ref{cor:planarity} is in the form  $\epsilon_{D^{\com}}\circ \phi$, where $\epsilon_{D^{\com}}$ is the augmentation of $\RSFT(\partial D^{\com})$ from $D^{\com}$. In principle, this involves a homotopy argument by neck-stretching. To avoid the overhead of introducing homotopies of $BL_\infty$ morphisms, we show that the formula can be identified on the nose, due to the fact that when transversality in neck-streaking holds, we can identify a fully-stretched moduli space with a sufficiently stretched moduli space by classical gluing. This is the content of Proposition \ref{prop:neck}. In the following, we first prove a property explaining the role of $k<\frac{3n-1}{2}$.

\begin{proposition}\label{prop:ind}
	Let $X$ be the cobordism from $\partial D^{\com}$ to $\partial D^{\com}_s$ as above. Assume $\Gamma^+=\{\gamma^{k_1}_{p_{\max}},\gamma^{k_2}_{p_{\min}},\ldots,\gamma^{k_s}_{p_{\min}}\}$ for $\sum_{i=1}^s k_i\le k$, we have $\vdim \cM_{X}(\Gamma^+,\Gamma^-)<0$ or $\overline{ \cM}_{X}(\Gamma^+,\Gamma^-)=\emptyset$ if $\Gamma^-\ne \emptyset$ and $s<\frac{n+1}{2}$.  
\end{proposition}
\begin{proof}
	 Assume $\Gamma^-=\{\gamma^-_r \}_{r\in R}$, which are perturbations from $\{\gamma^{t_r}_{p_r}\}_{r\in R}$, such that $\overline{ \cM_{X}}(\Gamma^+,\Gamma^-)\ne \emptyset$. Then we have $\sum_{r\in R} \sum t_r = \sum_{i=1}^s k_i \mod k$ by homology reasons. On the other hand, as explained in Remark \ref{rmk:action}, the total contact action of $\Gamma^-$ is close to $\sum_{r\in R} \sum t_r$. If we choose the smoothing $D_s$ contained in the neighborhood of $D$ that is removed to get the nice contact boundary to construct the exact cobordism $X$, we have the total contact energy of $\Gamma^+$ is approximately $\sum_{i=1}^s k_i$ and hence $\sum_{r\in R} \sum t_r \le \sum_{i=1}^s k_i$. Therefore we have  $\sum_{r\in R} \sum t_r = \sum_{i=1}^s k_i$. Then the expected dimension of $\cM_{X,A}(\Gamma^+,\Gamma^-)$, i.e.\ $\ind(u)$ for $u\in\cM_{X,A}(\Gamma^+,\Gamma^-)$, satisfies
	 $$2n-2+\sum_{i=1}^s 2(k_i-1)+\ind(u)+\sum_{r=1}^R(\mu_{CZ}(\gamma^{-}_r)+n-3)=2n-6.$$
	 Since $(\mu_{CZ}(\gamma^{-}_r)+n-3)\ge 2n-3-2\sum t_r-\ind(p_r)-|\supp t_r|$ and  $\check{D}_{\supp t_r}$ is Weinstein by $k\ge n+1$,  we have $\ind(p_r)\le n-|\supp t_r|$ and  $(\mu_{CZ}(\gamma^{-}_r)+n-3)\ge (n-3)-2\sum t_r$. As a consequence, we have 
	 $$\ind(u)\le -4-|R|(n-3)+2s.$$
	 In particular, $\ind(u)<0$ if $|R|\ne 0$ and $s<\frac{n+1}{2}$.
\end{proof}

\begin{proposition}\label{prop:neck}
	Let $\phi$ denote the $BL_\infty$ morphism from the cobordisms $X$ and $\epsilon_{D^{\com}_s},\epsilon_{D^{\com}}$ augmentations from $D^{\com}_s,D^{\com}$ respectively, then we have
	$$\widehat{\ell}_{\bullet,\epsilon_{D^{\com}_s}}(q_{\gamma^n_{p_{\min}}}q^{k-n}_{\gamma_{p_{\min}}})=\widehat{\ell}_{\bullet,\epsilon_{D^{\com}}\circ \phi }(q_{\gamma^n_{p_{\min}}}q^{k-n}_{\gamma_{p_{\min}}})\stackrel{\text{Prop }\ref{prop:order}}{=} \widehat{\ell}_{\bullet,\epsilon_{D^{\com}}}\circ \widehat{\phi}^1_{\epsilon_{D^{\com}}}(q_{\gamma^n_{p_{\min}}}q^{k-n}_{\gamma_{p_{\min}}}) \ne 0,$$
	where  $\widehat{\phi}^1_{\epsilon_{D^{\com}}}$ is defined in Proposition \ref{prop:order}, i.e.\ the map on the bar complex for the linearized $L_\infty$ morphism from $(V_{\partial D_s^{\com}}, \{\ell^k_{\epsilon_{D^{\com}}\circ \phi}\}_{k\ge 1})$ to $(V_{\partial D^{\com}},\{\ell^k_{\epsilon_{D^{\com}}}\}_{k\ge 1})$ induced by $\phi$. 
\end{proposition}
\begin{proof}
	 We will apply a neck-stretching for $\overline{\cM}_{D^{\com}_s, o}(\{\gamma^n_{p_{\min}},\gamma_{p_{\min}},\ldots,\gamma_{p_{\min}}\},\emptyset)$ in \eqref{curve2} of Proposition \ref{prop:curve} along $\partial D^{\com}$ for $o\in D^{\com}$. We first claim that every curve in $\overline{\cM}_{D^{\com}_s, o}(\{\gamma^n_{p_{\min}},\gamma_{p_{\min}},\ldots,\gamma_{p_{\min}}\},\emptyset)$ is somewhere injective. For otherwise, assume $u\in \overline{\cM}_{D^{\com}_s, o}(\{\gamma^n_{p_{\min}},\gamma_{p_{\min}},\ldots,\gamma_{p_{\min}}\},\emptyset)$ is a branched cover over $u'$, then we can cap off $u'$ with natural disks to obtain a homology class $A$ in $H_2(\CP^2)$ with $A\cap D_s<k$, which is a contradiction. Therefore it is safe to assume $\overline{\cM}_{D^{\com}_s, o}(\{\gamma^n_{p_{\min}},\gamma_{p_{\min}},\ldots,\gamma_{p_{\min}}\},\emptyset)$ is cut out transversely for the stretching $J_t$. In the fully stretched picture, the bottom level containing the marked point $o$ must have $k+1-n$ positive punctures. This is because we must have the number of positive punctures no larger than $k+1-n$ for otherwise genus has to be created. If there are fewer punctures, then by Proposition \ref{prop:unirule}, the curve can not exist by dimension reasons. By the dimension computation in Proposition \ref{prop:unirule}, the only possible bottom level is described in Proposition \ref{prop:unirule}. Then by the same capping argument, we know that the bottom curve is necessarily somewhere injective. As a consequence, all the levels above the bottom level must be unions of cylinders because of the number of positive punctures. Then by considering homology of the cobordism $X$, we must have the positive asymptotics of the bottom level is the form of ${\check{\gamma}^{\sigma_I}_{p_{I,\min}}}\cup\{\gamma_{p_{i,\min}}\}_{i\in I^{\com}}$, where $I\subset \{1,\ldots,k\}$ is a subset of size $n$, $p_{I,\min},p_{i,\min}$ are minimums. Next we still have multiple sympletization levels of cylinders for $\partial D^{\com}$ and one level of cylinders in $X$. Since ${\check{\gamma}^{\sigma_I}_{p_{I,\min}}}\cup\{\gamma_{p_{i,\min}}\}_{i\in I^{\com}}$ have the maximal period in their respective homology classes, there is no action room for the sympletization levels. The top level cylinders in $X$, i.e.\ $\overline{\cM}_{X}(\{\gamma^n_{p_{\min}}\},\{\check{\gamma}^{\sigma_I}_{p_{I,\min}}\})$ and  $\overline{\cM}_X(\{\gamma_{p_{\min}}\},\{\gamma_{p_{i,\min}}\})$ for $i\in I^{\com}$, are rigid and cut out transversely, as the negative asymptotic orbits are simple. Therefore, the fully stretched moduli space is cut out transversely. The transversality of neck-stretching implies that this $2$-level breaking can be identified with $\overline{\cM}_{D^{\com}_s, o}(\{\gamma^n_{p_{\min}},\gamma_{p_{\min}},\ldots,\gamma_{p_{\min}}\},\emptyset)$ for sufficiently stretched $J_t$.  By Axiom \ref{axiom}, we can count them to obtain that
	$$\eta_{D_s^{\com}}(q_{\gamma^n_{p_{\min}}}q^{k-n}_{\gamma_{p_{\min}}})=\eta_{D^{\com}}\circ\widehat{\phi}^1_{\epsilon_{D^{\com}}}(q_{\gamma^n_{p_{\min}}}q^{k-n}_{\gamma_{p_{\min}}}).$$   
	Then we can use Proposition \ref{prop:equi} to relate $\eta$ back to $\widehat{\ell}_{\bullet,\epsilon}$, since $q_{\gamma^n_{p_{\min}}}q^{k-n}_{\gamma_{p_{\min}}}$ is closed in  $(\overline{B}^{k+1-n}V_{\partial D_s^{\com}},\widehat{\ell}_{\epsilon_{D_s^{\com}}})$ by Corollary \ref{cor:planarity}. The non-vanishing follows from Corollary \ref{cor:planarity}.
	
\end{proof}

\begin{proposition}\label{prop:ind'}
	If $n+1\le k<\frac{3n-1}{2}$, then $\widehat{\ell}_{\bullet,\epsilon\circ \phi}(q_{\gamma^n_{p_{\min}}}q^{k-n}_{\gamma_{p_{\min}}})\ne 0$ is independent of the augmentation $\epsilon$ of $\RSFT(\partial D^{\com})$.
\end{proposition}
\begin{proof}
	When $k<\frac{3n-1}{2}$, we have $1+k-n<\frac{n+1}{2}$. By Proposition \ref{prop:max}, a component to $\widehat{\ell}_{\bullet,\epsilon\circ \phi}(q_{\gamma^n_{p_{\min}}}q^{k-n}_{\gamma_{p_{\min}}})$ with influence from $\epsilon$ is described in the graph below (Figure \ref{fig:ind_aug}), which does not exist by dimension reasons by  Proposition \ref{prop:ind}.   Therefore $\widehat{\ell}_{\bullet,\epsilon\circ \phi}(q_{\gamma^n_{p_{\min}}}q^{k-n}_{\gamma_{p_{\min}}})$ is independent of $\epsilon$. The non-vanishing then follows from Proposition \ref{prop:neck} by taking $\epsilon=\epsilon_{D^{\com}}$.
	\begin{figure}[H]
		\begin{center}
			\begin{tikzpicture}
	         \node at (0,0) [circle,fill,inner sep=1.5pt] {};
	         \node at (0,0.3) {$\gamma^n_{p_{\min}}$};
 	         \node at (2,0) [circle,fill,inner sep=1.5pt] {};
 	         \node at (2,0.3) {$\gamma_{p_{\min}}$};
	         \node at (4,0) [circle,fill,inner sep=1.5pt] {};
	         \node at (4,0.3) {$\gamma_{p_{\min}}$};
	         \node at (6,0) [circle,fill,inner sep=1.5pt] {};
	         \node at (6,0.3) {$\gamma_{p_{\min}}$};
	         
	         \draw (0,0) to (0,-2);
	         \draw[dashed] (2,0) to (2,-2);
	         \draw[dashed] (4,0) to (4,-2);
	         \draw[dashed] (6,0) to (6,-2);
	         \node at (0,-1) [circle, fill=white, draw, outer sep=0pt, inner sep=3 pt] {};
             \node at (0,-1) [circle, fill=black, draw, outer sep=0pt, inner sep=1.5 pt] {};
             \node at (0.5,-1) {$p_{\bullet}^{1,1}$};
	          
	         \node at (0,-2) [circle,fill=red,inner sep=1.5pt] {};
	         \node at (0.5,-2) {$\gamma^n_{p_{\max}}$};
	         \node at (2,-2) [circle,fill=red,inner sep=1.5pt] {};
	         \node at (4,-2) [circle,fill,inner sep=1.5pt] {};
	         \node at (6,-2) [circle,fill,inner sep=1.5pt] {};
	         
	         \draw[color=red] (0,-2) to (1,-3) to (0,-4);
	         \draw[color=red] (2,-2) to (1,-3) to (2,-4);
	         \draw[color=red] (1,-3) to (1,-4);
	         \node at (1,-3) [circle, fill=red, draw, inner sep=3 pt] {}; 
	         \node at (1.5,-3) {$\phi^{2,3}$};
	         
	         \draw (4,-2) to (5,-3) to (5,-4);
	         \draw (6,-2) to (5,-3);
	         \node at (5,-3) [circle, fill, draw, inner sep=3 pt] {}; 
	         \node at (5.5,-3) {$\phi^{2,1}$};
	         \node at (5,-4) [circle,fill,inner sep=1.5pt] {};
	         
	         \draw (0,-4) to (0,-5);
	         \draw (1,-4) to (1,-5);
	         \draw (2,-4) to (3.5,-5) to (5,-4);
	         \node at (0,-5) [circle, fill, draw, inner sep=3 pt] {}; 
	         \node at (0.5,-5) {$\epsilon^{1}$};
	         \node at (1,-5) [circle, fill, draw, inner sep=3 pt] {}; 
	         \node at (1.5,-5) {$\epsilon^{1}$};
	         \node at (3.5,-5) [circle, fill, draw, inner sep=3 pt] {}; 
	         \node at (4,-5) {$\epsilon^{2}$};
	         
	         \node at (0,-4) [circle,fill=red,inner sep=1.5pt] {};
	         \node at (1,-4) [circle,fill=red,inner sep=1.5pt] {};
	         \node at (2,-4) [circle,fill=red,inner sep=1.5pt] {};
	
			\end{tikzpicture}
        \end{center}
        \caption{One generic example of many possible configurations contributing to $\widehat{\ell}_{\bullet,\epsilon\circ \phi}(q_{\gamma^n_{p_{\min}}}q^{k-n}_{\gamma_{p_{\min}}})$ with influence from $\epsilon$.  The red part has negative dimension}\label{fig:ind_aug}
    \end{figure} 
\end{proof}

\begin{proof}[Proof of Theorem \ref{thm:CP}].
	If $k\le n$, then $D^{\com}=T^*T^{k-1}\times \C^{n-k+1}$, then $\Hcx(\partial D^{\com})=0^{\SD}$ by Theorem \ref{thm:product}. If $k=n+1$, then $\Pl(\partial D^{\com})=2$ by Corollary \ref{cor:product}. For $k>n+1$, the lower bound follows from Corollary \ref{cor:lower_bound}. When $k<\frac{3n-1}{2}$ and $n$ odd, we have for any augmentation $\epsilon$ of $\RSFT(\partial D^{\com})$, $q_{\gamma^n_{p_{\min}}}q^{k-n}_{\gamma_{p_{\min}}}$ represents a closed class in $(\overline{B}^{k+1-n}V_{\partial D_s^{\com}},\widehat{\ell}_{\epsilon\circ \phi})$, as the SFT grading of $\RSFT(\partial D^{\com}_s)$ is even for all generators. In particular, $\widehat{\phi}^1_{\epsilon}(q_{\gamma^n_{p_{\min}}}q^{k-n}_{\gamma_{p_{\min}}})$ is closed in $(\overline{B}^{k+1-n}V_{\partial D^{\com}},\widehat{\ell}_{\epsilon})$ for any $\epsilon$.	Then by Proposition \ref{prop:ind'}, $\widehat{\ell}_{\bullet,\epsilon\circ \phi}(q_{\gamma^n_{p_{\min}}}q^{k-n}_{\gamma_{p_{\min}}})\ne 0$ for any $\epsilon$, and we conclude that $\Pl(\partial D^c)=k+1-n$ if $n+1<k<\frac{3n-1}{2}$ by Proposition \ref{prop:order}.
\end{proof}

\begin{remark}\label{rmk:general} 
	Our computation method above can be summarized as finding a curve contributing to the planarity by relative Gromov-Witten invariants and then arguing independence of augmentation by index computations. The trick we use is arguing closedness in the smooth divisor, where generators are simpler, proving the upper bounds using the functoriality, and arguing everything interesting about the functoriality happens purely in $X$ (i.e.\ not dependent on augmentation for $\RSFT(\partial D^{\com})$). A more systematic way of computing planarity is deriving a formula for the $BL_\infty$ algebra as well as the augmentation from the affine variety using log/relative Gromov-Witten invariants. In the context of symplectic (co)homology, such formula was obtained in \cite{diogo2019symplectic}. 
\end{remark}

Even though Theorem \ref{thm:CP} depends on the parity of $n$ and the size of $k$, things get easier if we only consider the augmentation from the affine variety (Corollary \ref{cor:planarity}). The following can be viewed as the geometric reinterpretation of Corollary \ref{cor:planarity}, which does not depend on $n,k$ and is sufficient for obstructing exact embeddings.
\begin{theorem}\label{thm:embedding}
	Let $D$ be $k$ generic hyperplanes in $\CP^n$ for $n\ge 1$, then $\U(D^{\com})=\max\{1,k+1-n\}$.
\end{theorem}
\begin{proof}
	The $n=1$ case is obvious. For $n\ge 2$, we can use Proposition \ref{prop:unirule} to claim that $\overline{\cM}_{D^{\com},o}(\Gamma^+,\emptyset)=\emptyset$ for generic $J$ as along as $|\Gamma^+|<\max \{1,k+1-n\}$. This is because we can obtain the classical transversality of  $\overline{\cM}_{D^{\com},o}(\Gamma^+,\emptyset)=\emptyset$, as every curve is a branched cover of a somewhere injective curve with negative expected dimension. Therefore, we have $\U(D^{\com})\ge \max\{1,k+1-n\}$ by Proposition \ref{prop:unirule_equiv}. On the other hand, the nontrivial relative Gromov-Witten invariant used in Proposition \ref{prop:curve} implies that $\U(D^{\com})\le \max\{1,k+1-n\}$ by neck-stretching.
\end{proof}

\subsubsection{Examples with nontrivial $\SD$ when $k$ is small} 
\begin{theorem}\label{thm:P1}
	Assume $D_s$ is a smooth degree $2\le k<n$ hypersurface in $\CP^n$ for $n\ge 3$ odd, then $(k-1)^{\SD}\le \Hcx(\partial D_s^{\com})\le (2k-2)^{\SD}$. When $n$ is even and $2\le k <\frac{n+1}{2}$, then  we have $\Hcx(\partial D_s^{\com})\le (2k-2)^{\SD}$.
\end{theorem}
\begin{proof}
	Let $p$ be the critical point in \eqref{curve1} of Proposition \ref{prop:curve}, Then we have $\eta_{D_s^{\com}}(q_{\gamma^k_p})\ne 0$ by the same argument of Corollary \ref{cor:planarity}. We can pick the Morse function on $D_s$ to be perfect and self-indexing, similar to \cite[Proposition 3.1]{zhou2020mathbb}, we can choose the perturbation Morse function such that if
	\begin{equation}\label{eqn:action}
	\int \alpha^*\gamma^d_p-\sum_{i=1}^{j}\int \alpha^*\gamma^{d_i}_{p_i}\ge 0,
	\end{equation}
	for $d\le k, \sum d_i=d$,  then for every $i$ we have $\ind(p_i)\ge \ind(p)$\footnote{Note that in our setup here higher $f(p)/\ind(p)$ means smaller contact action, since we apply the perturbation in the cap of the positive prequantization bundle instead of the filling of the negative prequantization bundle. In particular, the order is reversed compared to \cite[\S 2.1]{zhou2020mathbb} and the proof of Theorem \ref{thm:quotient} below. }. This energy constraint will help us exclude certain configurations. 
	
	\begin{claim*} $q_{\gamma^k_p}$ is closed in $(\overline{B}^1V_{\partial D_s^{\com}},\ell^1_{\epsilon})$ for any augmentation $\epsilon$ for $2\le k<\frac{n+1}{2}$ or $n$ odd with $2\le k < n$.
    \end{claim*}
    \begin{proof}
    	Since the parity of the SFT grading of $\gamma^{d}_q$ is the same as the parity of $\ind(q)$. As a consequence, we only need to consider $\langle\ell^1_{\epsilon}(q_{\gamma^k_p}),q_{\gamma^d_q}\rangle$ for $\ind(q)=n-1$ when $n$ is even. In other word, we need to consider $\cM:=\overline{\cM}_{\partial D_s^{\com}}(\{\gamma_p^k\},\{\gamma_q^{d}\}\cup \{\gamma_{q_i}^{d_i}\}_{1\le i \le r})$ for $d+\sum d_i=k$. By \eqref{eqn:action}, to have $\cM\ne \emptyset$, we must have $\ind(q_i),\ind(q)\ge \ind(p)=2n-2k$. By $k<\frac{n+1}{2}$, we have $\ind(p)=2n-2k>\ind(q)=n-1$ and $\cM$ is empty. Hence the claim follows.
    \end{proof}
    \begin{claim*}
    	 $\ell^1_{\bullet,\epsilon}(q_{\gamma^k_p})$ is independent of $\epsilon$.
    \end{claim*}
    \begin{proof}
    		It is sufficient to show that $\overline{\cM}_{\partial D_s^{\com},o}(\{\gamma^k_p\},\Gamma^-)$ is empty for $\Gamma^-\ne \emptyset$. If $\Gamma^-\ne \emptyset$, then $\Gamma^-=\{\gamma^{d_i}_{p_i}\}_{1\le i \le r}$ with $\sum d_i=k$ and $\ind(p_i)\ge \ind(p)$ for every $i$. Then claim follows from the same argument of Proposition \ref{prop:max} and $\ind(p)>0$.
    \end{proof}
	
    \begin{claim*}
    	 We have $\Hcx(\partial D_s^{\com})\le (2k-2)^{\SD}$.
    \end{claim*}
    \begin{proof}
    	 We need to show that $U^{2k-1}(q_{\gamma^k_p})=0$ for any augmentation. Note that the $U$ map decreases contact action. By homology reason and  contact action property \eqref{eqn:action}, for any augmentation and  $d\le k$,  $U(q_{\gamma^d_q})$ can only have nontrivial coefficient for $q_{\gamma^{d'}_{q'}}$ for $d'<d$ and $\ind(q')\ge \ind (q)$ and $\ind(q')=\ind(q) \mod 2$ , or for $q_{\gamma^d_{q'}}$ for $\ind(q')=\ind (q)+2$. Therefore, we have $U^{2k-1}(q_{\gamma^k_p})=0$ for any augmentation. 
    \end{proof}

    \begin{claim*}
    	When $n$ is odd, We have $\Hcx(\partial D_s^{\com})\ge (k-1)^{\SD}$.
    \end{claim*}
    \begin{proof}
        The linearized contact homology/positive $S^1$-equivariany symplectic cohomology has a contact action filtration, such that the filtered theory around period $\int (\gamma^k_p)^*\alpha =k\int \gamma_p^*\alpha$is generated by the $k$th covered orbits. The $U$ map on this filtered theory is represented by multiplying $c_1(\cO(k)|_D)$ to the cochain represented by the critical point, i.e.\ the Poincar\'e dual of the unstable manifold. Using \eqref{iso} of Claim \ref{claim:u}, we can consider the $U$ map on positive $S^1$-equivariant symplectic cohomology. In fact, one can prove the filtered linearized contact homology (up to $k$-th multiple of simple orbits) is isomorphic to the filtered $S^1$ positive symplectic cohomology following \cite{bourgeois2009exact} as there is no room for the influence of augmentations and transversality can be achieved. Then the $U$-map on the filtered $S^1$ positive symplectic cohomology is given by multiplying the first Chern class by a standard Morse-Bott argument, e.g.\ \cite[Proposition 5.9]{zhou2019symplectic}. Therefore, by the argument in the second claim, we have $U^{k-1}(q_{\gamma^k_p})=k^{k-1}q_{\gamma^k_{p_{\max}}}+$terms with lower multiplicities for the unique maximum $p_{\max}$ with $\ind(p_{\max})=2n-2$. When $n$ is odd, all generators have even SFT degree and represent non-trivial classes in the linearized contact homology, hence $U^{k-1}(q_{\gamma^k_p})\ne 0$ in homology.
    \end{proof}
\end{proof}

\begin{remark}\label{rmk:n=3}
	The $n$ being odd condition in Theorem \ref{thm:CP},  \ref{thm:infinity}, \ref{thm:P1}, as well as  \ref{thm:generalization} below is not necessary, as one can show $q_{\gamma^n_{p_{\min}}}q^{k-n}_{\gamma_{p_{\min}}}$ is always closed. This is because a differential from $q_{\gamma^n_{p_{\min}}}q^{k-n}_{\gamma_{{p}_{\min}}}$ involves counting $\overline{\cM}_{Y}(\Gamma^+,\Gamma^-)$ with $\Gamma^+$ is a subset of $\{\gamma^n_{p_{\min}},\gamma_{p_{\min}},\ldots,\gamma_{p_{\min}}\}$ and $\Gamma^-=\{\gamma^{d_i}_{p_i}\}_{1\le i \le r}$ for $\sum d_i=k$. Therefore if we use the Morse-Bott contact form and cascades construction, the relevant holomorphic curve must be covers of trivial cylinders. Then the moduli space $\overline{\cM}_{Y}(\Gamma^+,\Gamma^-)$ is the fiber product of $D\times \cM$ with unstable/stable manifolds of $p_{\min},p_i$, where $\cM$ before compactification is the space of meromorphic functions on $\CP^1$ with a pole of order $n$ and $k-n$ simple poles and a zero with order $d_i$ for all $1\le i \le r$ modulo the $\R$ rescaling on meromorphic functions and the automorphism of the punctured Riemann surface. Then by the nontrivial $S^1$ action on meromorphic functions, we expect to have $\#\overline{\cM}_{Y}(\Gamma^+,\Gamma^-)=0$  unless $|\Gamma^+|=|\Gamma^-|=1$. If  $|\Gamma^+|=|\Gamma^-|=1$, $\overline{\cM}_{Y}(\Gamma^+,\Gamma^-)$ is identified with Morse trajectories (here the $S^1$ action on meromorphic function is identical with the $S^1$ action in the automorphism group of surface, hence is trivial on the quotient), whose algebraic count is zero, as we assume the Morse function is perfect. To make this precise, one can follow a Morse perturbation of the contact form as before. And we use a $J$ that is $S^1$-invariant under the rotation in the fiber direction, then apply the $S^1$-equivariant transversality for quotients from \cite{zhou2020quotient}, we can argue that $\#\overline{\cM}_{Y}(\Gamma^+,\Gamma^-)=0$  unless $|\Gamma^+|=|\Gamma^-|=1$ similar to Floer's proof of the isomorphism between Hamiltonian Floer cohomology and Morse cohomology. This argument requires building our functors using polyfolds as in \cite{SFT}. 
\end{remark}

\subsubsection{Generalization to Fano hypersurfaces} In the following, we will generalize Theorem \ref{thm:CP} to some affine varities contained in a Fano hypersurface in $\CP^{n+1}$.

\begin{theorem}\label{thm:generalization}
	Let $X$ be a smooth degree $m$ hypersurface in $\CP^{n+1}$ for $2\le m< \frac{n+1}{2}\le n$ and $D$ be $k\ge n$ generic hyperplanes, i.e.\ $D=(H_1\cup \ldots \cup H_k)\cap X$ for $H_i$ is a hyperplane in $\CP^{n+1}$ in generic position with each other and $X$, then $\Pl(\partial D^{\com})=k+m-n$ for $n$ odd and $k+m<\frac{3n+1}{2}$.
\end{theorem}
\begin{proof}
	We separate the proof into several steps. The Reeb dynamics on $\partial D^{\com}$ has the same property with the $\CP^n$ case, with the only difference that the minimal Chern number of $X$ is $n+2-m$, which will enter into the computation of virtual dimensions.
	\begin{claim*}
		For any Reeb orbits set $\Gamma:=\{\gamma_1,\ldots,\gamma_r\}$ for $r<k+m-n$ with $\sum[\gamma_i]=0\in H_1(D^{\com})$, the virtual dimension of the moduli space $\overline{\cM}_{D^{\com},A,o}(\Gamma,\emptyset)$ is negative for any $A$.
	\end{claim*}
    \begin{proof}
    	This follows from the same argument in Proposition \ref{prop:unirule}, with the difference that $c_1(u\#_{i=1}^r u_i)=2N(n+2-m)$. Therefore we have
    	\begin{eqnarray*}
    		 \ind(u) & \le & 2N(n+2-m)-4-\sum_{i=1}^r(2\sum t_i-2+\ind(p_i)) \\
    		         & \le & 2N(n+2-m)-4-2kN-4+2r\\
    		         & = & 2(N-1)(n+2-m-k)+2(r+n-m-k) < 0,
    	\end{eqnarray*}
    since $r<m+k-n$ and $k\ge n,m\ge 2$. This computation also implies the lower bound of $\Pl(\partial D^{\com})$ is $k+m-n$.
    \end{proof}
    \begin{claim*}
       Assume $D_s$ in the generic intersection of a degree $k$ hypersurface in $\CP^{n+1}$ with $X$. Then we have $$\eta_{D_s^{\com}}(q_{\gamma^{n+1-m}_{p_{\min}}}q^{k+m-n-1}_{\gamma_{p_{\min}}})\ne 0,$$ and $q_{\gamma^{n+1-m}_{p_{\min}}}q^{k+m-n-1}_{\gamma_{p_{\min}}}$ is closed in $(\overline{B}^{k+m-n}V_{\partial D_s^{\com}},\widehat{\ell}_{\epsilon_{D_s^{\com}}})$.
    \end{claim*} 
     \begin{proof}
     	That $\eta_{D_s^{\com}}(q_{\gamma_p}^{k+m-n-1} q_{\gamma^n_p})\ne 0$ follows from the non-vanishing of $$\GW^{X,D}_{0,1,(n+1-m,1,\ldots,1),A}([pt],\underbrace{[D_s],\ldots,[D_s]}_{k+m-n})$$ from \cite{gathmann2002absolute} for $A$ is the positive generator of $H_2(X)$ that is mapped to the generator of $H_2(\CP^{n+1})$ and the same argument in Proposition \ref{prop:curve}. The remaining of the argument is exactly same as Corollary \ref{cor:planarity}. To see the non-vanishing of the Gromov-Witten invariants, we have
        $$\GW^{X,D}_{0,1,(n+1-m,1,\ldots,1),A}([pt],\underbrace{[D_s],\ldots,[D_s]}_{k+m-n}) = (k)^{k+m-n-1}\GW^{X,D}_{0,1,(n+1-m),A}([pt],[D_s])$$ 
        by the divisor axiom. Then by \cite[Theorem 2.6]{gathmann2002absolute}, we have
        $$\GW^{X,D}_{0,1,(n+1-m),A}([pt],[D_s]) = \int_{[\overline{\cM}_{0,2}(X,A,pt)]^{\mathrm{vir}}} \prod_{i=0}^{n+1-m}(ev_2^*\mathrm{PD}([D_s])+i\psi).$$
        Following the strategy in \cite[Corollary 5.7]{gathmann2002absolute}, by taking $\alpha=(1,m)$ in \cite[(8)]{gathmann2002absolute}
        we have 
        $$((m-1)\psi+ev_2^*\mathrm{PD}([X]))\cdot [\overline{\cM}_{0,(1,m-1)}(\CP^{n+1},A,pt)]^{\mathrm{vir}}=[\overline{\cM}_{0,2}(X,A,pt)]^{\mathrm{vir}}$$
        where the point constraint is in $X\subset \CP^{n+1}$. Then we apply  \cite[Theorem 2.6]{gathmann2002absolute} for another $m-2$ times, we get
        \begin{eqnarray*}
            \GW^{X,D}_{0,1,(n+1-m),A}([pt],[D_s]) & = &  \int_{[\overline{\cM}_{0,2}(\CP^{n+1},A,pt)]^{\mathrm{vir}}} \prod_{i=0}^{n+1-m}(ev_2^*\mathrm{PD}(k[H])+i\psi) \prod_{i=1}^{m-1}(ev_2^*\mathrm{PD}([X])+i\psi) \\
            & = & \frac{k!m!}{(k+m-n-1)!}
        \end{eqnarray*}
      where $H$ is the hyperplane class in $\CP^{n+1}$. 
     \end{proof}
     Then by the same neck-stretching argument in Proposition \ref{prop:neck}, we have 
     \begin{eqnarray*}
     \widehat{\ell}_{\bullet,\epsilon_{D^{\com}_s}}(q_{\gamma^{n+1-m}_{p_{\min}}}q^{k+m-n-1}_{\gamma_{p_{\min}}})& = &\widehat{\ell}_{\bullet,\epsilon_{D^{\com}}\circ\phi}(q_{\gamma^{n+1-m}_{p_{\min}}}q^{k+m-n-1}_{\gamma_{p_{\min}}}) \\ & = & \widehat{\ell}_{\bullet,\epsilon_{D^{\com}}}\circ \widehat{\phi}^1_{\epsilon_{D^{\com}}}(q_{\gamma^{n+1-m}_{p_{\min}}}q^{k+m-n-1}_{\gamma_{p_{\min}}}) \ne 0.
     \end{eqnarray*}
     Next Proposition \ref{prop:max} and Proposition \ref{prop:ind} also holds, as the dimension computation there is essentially for trivial homology class, which does not depend on $m$. It is important to note that in the proof of Proposition \ref{prop:ind}, we use that $\check{D}_I$ is Weinstein is obtain an upper bound of Morse indices. Such property also holds here as we assume $k\ge n$. Then the remaining of the proof is the same as Theorem \ref{thm:CP}.
\end{proof}	
From the proof above, the source of holomorphic curves is supplied by the degree $1$ holomorphic curves in $X$ for $m\le n$. For $m=n+1$, the degree $1$ curve does not unirule $X$ anymore, but a degree $2$ curve unirules $X$. In the proof Theorem \ref{thm:CP} and Theorem \ref{thm:generalization}, being degree $1$ is used in several places to obtain somewhere injectivity (the capping argument). Indeed, for $m=n+1$, the situation is different, we will prove $\Pl(\partial D^c)\ge 2$ for $D$ is a generic intersection of $X$ with a hyperplane in $\CP^{n+1}$. For $m\ge n+2$, then $X$ is not uniruled, which implies $D^c$ is not $k$ uniruled for any $k$ by \cite{mclean2014symplectic}, therefore $\Pl(\partial D^{\com})=\infty$ by Corollary \ref{cor:lower}.

In view of Theorem \ref{thm:unirule}, Theorem \ref{thm:CP}, Theorem \ref{thm:embedding}, and Theorem \ref{thm:generalization}, we make the following conjecture.
\begin{conjecture}
	$V$ is a $k$-uniruled affine variety then $\Pl(V)<\infty$ and $\Pl(V)=\U(V)=\AU(V)$.
\end{conjecture}

On the other hand, by Theorem \ref{thm:infinity}, it is not true that any uniruled affine variety has a contact boundary with finite planarity. It is subtle question to determine which affine variety with a $\CP^n$ compactification has a finite planarity boundary.
\begin{question}
Let $D$ be $k$ generic hyperplanes in $\CP^n$, is $\Pl(\partial D^{\com})$ always finite?
\end{question}
Theorem \ref{thm:CP} and Theorem \ref{thm:generalization} along with Lemma \ref{lemma:perturb} and Lemma \ref{lemma:multi} imply that there are many sequences of contact manifolds where exact cobordisms only exist in one direction. On the other hand, exact embedding problems in the flavor of Theorem \ref{thm:embedding} are studied in \cite{LC}. It is an interesting question to determine whether those embedding obstructions can lift to cobordism obstructions.

\subsection{Links of singularities}
Another natural source of contact manifolds is links of isolated singularities. In the following, we will consider the Brieskorn singularities and quotient singularities from diagonal cyclic actions on $\C^n$.
\subsubsection{Brieskorn singularities}
A Brieskorn singularity is of the following form
$$x_0^{a_0}+\ldots+x_n^{a_n}=0,$$
for $2\le a_0\le \ldots \le a_n$.  We use $\overline{a}$ to denote the sequence, the link $LB(\overline{a})$ is defined to be the intersection $LB(\overline{a}):=\{(x_0,\ldots,x_n)\in \C^{n+1}|x_0^{a_0}+\ldots+x_n^{a_n}=0\}\cap S^{2n+1}$, which is a $(2n-1)$-dimensional contact manifold. Moreover, $LB(\overline{a})$ is exactly fillable by the smooth affine variety $x_0^{a_0}+\ldots+x_n^{a_n}=1$, which is called the Brieskorn variety. We refer readers to \cite{MR3483060} for more details on the contact topology of Brieskorn manifolds. We have the following fact about embedding relations for Brieskorn varieties.

\begin{proposition}[{\cite[Lemma 9.9]{MR3217627}}]\label{prop:embed}
	We say $\overline{a}\le \overline{b}$ iff $a_i\le b_i$ for all $i$. Then if $\overline{a}\le \overline{b}$, the Brieskorn variety of $\overline{a}$ embeds exactly into the Brieskorn variety of $\overline{b}$. In particular $LB(\overline{a})\le LB(\overline{b})$ in $\overline{\cont}_{\le}$.
\end{proposition}

Many Brieskorn varieties were showed in \cite[Theorem A]{zhou2019symplectic} to support $k$-dilations. Some of the computation in \cite{zhou2019symplectic} can be improved to be the computation of $\Hcx$, i.e.\ independent of $BL_\infty$ augmentations. In particular, we will have either a computation or an estimate of $\Hcx(LB(\overline{a}))$ for any $\overline{a}$ from the theorem below and Proposition \ref{prop:embed}.

\begin{theorem}\label{thm:brieskorn}
	We use $LB(k,n)$ to denote the contact link of the Brieskorn singularity $x_0^{k}+\ldots+x_n^k=0$,  then $\Hcx(LB(k,n))$ is
	\begin{enumerate}
		\item $(k-1)^{\SD}$ if $k<n$ and is $\ge (k-1)^{\SD}$ if $k=n$;
		\item $>1^{\Pl}$ if $k=n+1$;
		\item $\infty^{\Pl}$, if $k>n+1$.
	\end{enumerate}
\end{theorem}
\begin{proof}
     If $k=n+1$, since the log-Kodaira dimension of the corresponding Brieskorn variety $V$ is $0$, we know that $V$ is not algebraically $1$-uniruled. Hence the planarity is greater than $1$ by Corollary \ref{cor:lower}. When $k>n+1$, the Brieskorn variety admits a compactification that is not uniruled, hence the planarity is infinity by Corollary \ref{cor:lower}.
     
     When $k\le n$, the associated Brieskorn variety $V(k,n)$ carries a $k-1$ (semi)-dilation by \cite[Theorem A]{zhou2019symplectic} using the definition with $S^1$ equivariant symplectic cohomology. Moreover, the $k-1$ semi-dilation is observed by the truncated $S^1$ equivariant symplectic cohomology generated by simple Reeb orbits. By Theorem \ref{thm:BO-iso}, this means the order of semi-dilation using a augmentation from the Brieskorn variety is $k-1$. Then we have at least $\SD\ge k-1$. 
     
     Note that the  Brieskorn variety is an affine variety $X\backslash D$, where $X$ is the smooth projective variety $x_0^k+\ldots+x_n^k=x_{n+1}^k$ in $\CP^{n+1}$ and $D$ is the smooth divisor $X\cap \{x_{n+1}=0\}$. We will adopt the same notation for Reeb orbits on smooth divisor complement as before. The semi-dilation is provided by a simple Reeb orbit ${\gamma_p}$ with $\ind(p)=2n-2k$ by \cite[Theorem A]{zhou2019symplectic}. 
     
    Next we will argue that the semi-dilation supplied by ${\gamma_p}$ with $\ind(p)=2n-2k$ is independent of the augmentation when $k<n$. To see that, we first claim that $q_{\gamma_p}$ contribute to $\Pl=1$ is independent of augmentation. If not, we have a non-empty moduli space $\overline{\cM}_{LB(k,n),o}(\{\gamma_p\},\{\gamma_q\})$, whose expected dimension is $\ind(q)-\ind(p)+2-2n=\ind(q)+2k+2-4n<0$ when $k<n$ since $\ind(q)\le 2n-2$. Therefore planarity of $LB(k,n)$ is always $1$ if $k<n$. Moreover, $U^i(q_{\gamma_p})$ is independent of augmentation as we are at the minimal period, there is no room for $U$ to depend on augmentations. Therefore we have $\Hcx(LB(k,n))\le (k-1)^{\SD}$ if $k<n$. Hence the claim follows.
\end{proof}

\begin{remark}
	If Conjecture \ref{conj} was proven, one can get better estimate for $\Hcx(LB(\overline{a}))$ by writing $LB(\overline{a})$ as an open book with a Brieskorn variety page. In the context of symplectic cohomology, computation in such spirit can be found in \cite[\S 5]{zhou2019symplectic}. 
\end{remark}

\begin{proof}[Proof of Theorem \ref{thm:RSFT}]
	This theorem is a combination of Theorem \ref{thm:planar_torsion_1}, Theorem \ref{thm:planar_torsion_2}, Theorem \ref{thm:planar_torsion_3}, Corollary \ref{cor:lower}, Theorem \ref{thm:upper}, Corollary \ref{cor:product} and Theorem \ref{thm:brieskorn}. 
\end{proof}	

\subsubsection{Quotient singularities by cyclic groups.}
Let $\Z_k$ acts on $\C^n$ by the diagonal action by multiplying $e^{\frac{2\pi i}{k}}$, then the link of the quotient singularity $\C^n/\Z_k$ is the quotient contact manifold $(S^{2n-1}/\Z_k,\xi_{std})$. Such contact manifolds provide many examples of strongly fillable but not exactly fillable contact manifolds \cite{zhou2020mathbb}. In fact, the symplectic part of \cite{zhou2020mathbb} is a computation of the hierarchy functor $\Hcx$ in the context of symplectic cohomology, which will be rephrased as follows.
\begin{theorem}\label{thm:quotient}
	Let $Y$ be the quotient $(S^{2n-1}/\Z_k,\xi_{std})$ by the diagonal action by $e^{\frac{2\pi i}{k}}$ for $n\ge 2$.
	\begin{enumerate}
		\item If $n>k$, we have $\Hcx(Y)=0^{\SD}$.
		\item If $n\le k$, we have $0^{\SD} \le \Hcx(Y) \le (n-1)^{\SD}$. When $n=k$, we have $\Hcx(Y)\ge 1^{\SD}$.
	\end{enumerate}
\end{theorem}
\begin{proof}
	We follow the same setup as in \cite[Proposition 3.1]{zhou2020mathbb}. We have a non-degenerate contact form on $\xi_{std}$  by perturbing with a $C^2$-small perfect Morse function  $f$ on $\CP^{n-1}$, such that Reeb orbits are the following.
	\begin{enumerate}
	\item Reeb orbits of period smaller than $k+1$ are $\gamma_i^j$ for $0\le i \le n-1, 1 \le j \le k$, where $\gamma_i^j$ is the $j$-multiple cover of $\gamma_i$ and $\gamma_i$ projects to the $i$th critical point $q_i$ of $f$ with $\ind(q_i)=2i$.\footnote{It is important to note that now we perturb contact form using $f$ is the prequantization filling in $\cO(-k)$ following the convention in \cite{zhou2020mathbb}, therefore higher Morse index means lager period, which is reverse to Theorem \ref{thm:CP}.}
    \item The period of $\gamma_j$ is $1+\epsilon_j$.
    \item $\epsilon_j<\frac{\epsilon_{j+1}}{k},\epsilon_j\ll 1$.
    \item The Conley-Zehnder index of  $\gamma^j_i$ with the natural disk in $\cO(-k)$ satisfies $\mu_{CZ}(\gamma_i^j)+n-3=2i+2j-2$.
	\end{enumerate}    
    \begin{claim*}
    	We have $\Pl(Y)=1$ for $n\ge 2,\forall k$.
    \end{claim*}
    \begin{proof}
    	By the same argument as \cite[Step 3 of Proposition 3.1]{zhou2020mathbb}, we have $\#\overline{\cM}_{Y,o}(\{\gamma^k_0\},\emptyset)=k$ for $n\ge 2$, which is induced from the holomorphic curve in the symplectization of the standard sphere. When $\Gamma^-\ne \emptyset$, we have $\#\overline{\cM}_{Y,o}(\{\gamma^k_0\},\Gamma^-)=\emptyset$ by action and homology reasons, unless $\Gamma^-=\{\gamma^{d_i}_0\}_{1\le i \le r}$ for $\sum d_i = k$. In this case, a curve in $\overline{\cM}_{Y,o}(\{\gamma^k_0\},\Gamma^-)$ is necessarily a branched cover over a trivial cylinder. In particular, $\overline{\cM}_{Y,o}(\{\gamma^k_0\},\Gamma^-)=\emptyset$ for generic $o$.  Since all Reeb orbits have even SFT degree, we have $q_{\gamma^k_0}$ is closed in any linearized contact homology, and the planarity is $1$ for any augmentation (which exists in abundance as all SFT gradings are even) by $q_{\gamma^k_0}$.
    \end{proof}
    \begin{claim*}
    	If $k<n$, then we have $\Hcx(Y)=0^{\SD}$.
    \end{claim*}
    \begin{proof}
    	 By action reasons, $U(q_{\gamma^k_0})$ can only have nontrivial coefficients in $q_{\gamma^d_0}$ for $d<k$. Note that the filtered linearized contact homology/$S^1$-equivariant symplectic cohomology with action supported around $d$ is generated by $q_{\gamma^d_{r}}$ for $0\le r \le n-1$. In particular, the homology is $H^*(\CP^{n-1})$ with the $U$ map is the multiplication by $c_1(\cO(k))$. As a consequence, we have $U(q_{\gamma^d_r})=kq_{\gamma^d_{r-1}}+\sum_{i=1}^{d-1} \sum_{j=0}^r  a_{ij} q_{\gamma^{i}_j}$ by action reasons. Therefore for any augmentation, there exist $c_{ij}$ such that  $U(q_{\gamma^k_0}+\sum_{i=1}^{k-1}\sum_{j=1}^{k-i}c_{ij}q_{\gamma^i_j})=0$ by the same argument as \cite[(3.2)]{zhou2020mathbb}. In order to finish the proof, it is sufficient to prove $\overline{\cM}_{Y,o}(\gamma^i_j,\Gamma^-)=\emptyset$ for $i+j\le k$ and $j>0$. This follows from the same dimension computation in \cite[Step 7 of Proposition 3.1]{zhou2020mathbb} and is the place where $n>k$ is essential. Then $q_{\gamma^k_0}+ \sum_{i=1}^{k-1}\sum_{j=1}^{k-i}c_{ij}q_{\gamma^i_j}$ also contributes $\Pl=1$ and is killed by $U$. In particular, $\Hcx(Y)=0^{\SD}$. 
    \end{proof}
    \begin{claim*}
     	If $k\ge n$, then we have $0^{\SD}\le \Hcx(Y)\le (n-1)^{\SD}$.
    \end{claim*}
    \begin{proof}
    	First note that all generators have even degree, hence any maps $\{\epsilon^k\}_{k\ge 1}$ form an augmentation. By the argument in Remark \ref{rmk:max}, we have that $\#\overline{\cM}_{Y,o}(\{\gamma^d_{n-1}\},\{\gamma^d_{0}\})=1$ for $1\le d \le k-1$. On the other hand, following the argument to obtain \cite[$\langle \partial(\check{p}_{k_+}),\hat{p}_{k_-} \rangle $ in Theorem 9.1, Lemma 9.4]{diogo2019symplectic}, we know that $\langle U(q_{\gamma^{k_+}_i}), q_{\gamma^{k_-}_i}\rangle=(k_+-k_-)\epsilon^1(q_{\gamma^{k_+-k_-}_0})$ for augmentation $\{\epsilon^k\}_{k\ge 1}$\footnote{Although such structure originally appears as part of differential in symplectic cochain complex, it contributes to the $U$-map in the $S^1$-equivariant symplectic cohomology, see \cite[\S 5]{zhou2020minimal} for discussion.}. If for every $1\le i \le k-n$, we have $\epsilon^1(q_{\gamma^i_0})=0$. Then for $d\le k$, we have  $U(q_{\gamma^d_0})=\sum_{j=1}^{d-1+n-k} (d-j) \epsilon^1(q_{\gamma^{d-j}_0}) q_{\gamma_0^j}$. Therefore we have $U^n(q_{\gamma^{k}_0})=0$. Otherwise, we assume $i$ is the minimum among $\{1,\ldots,k-n\}$ such that  $\epsilon^1(q_{\gamma^i_0})\ne 0$. As a consequence, we have planarity $1$ contributed by $q_{\gamma^i_{2n-2}}$ by $\#\overline{\cM}_{Y,o}(\{\gamma^i_{n-1}\},\{\gamma^i_{0}\})=1$. Since $i$ the minimal one with nontrivial augmentation, we know that $U^n(q_{\gamma^i_{n-1}})=0$. Hence we have $\Hcx(Y)\le (n-1)^{\SD}$. 
    \end{proof}
    
    \begin{claim*}
    	If $k=n$, then we have $\Hcx(Y)\ge 1^{\SD}$
    \end{claim*}
    \begin{proof}
    	It suffices to find one augmentation such that the order of semi-dilation is $1$. We choose our augmentation to be $\epsilon^1(q_{\gamma_0})=-n$ and $\epsilon^k=0$ in all other cases. We first list the following expected dimensions of various moduli spaces.
    	\begin{enumerate}
    		\item\label{dim1} $\vdim \cM_{Y,o}(\{\gamma^{nm}_{p_i}\},\emptyset)=2i+2n(m-1)\ge 0$, which is positive, unless $i=0,m=1$, where we know $\#\overline{\cM}_{Y,o}(\{\gamma^n_0\},\emptyset)=n$.
    		\item\label{dim2} For $l>0$, $$\vdim \cM_{Y,o}(\{\gamma^{mn+l}_{p_i}\},\{\underbrace{\gamma_{p_0},\ldots,\gamma_{p_0}}_{l}\})=2i+2l+2n(m-1).$$ Then it is zero iff 
    		\begin{enumerate}[(i)]
    		    \item $i=n-1,l=1,m=0$, then $\#\overline{\cM}_{Y,o}(\{\gamma_{2n-2}\},\{\gamma_{0}\})=1$,
    		    \item $i<n-1, l=n-i>1, m=0$, then $\overline{\cM}_{Y,o}(\{\gamma^{l}_{p_{n-l}}\},\{\gamma_{p_0},\ldots,\gamma_{p_0}\})=\emptyset$ by the argument of Proposition \ref{prop:max}.
    		\end{enumerate}
    		\item\label{dim3}  For $l>0$, $$\vdim \cM^1_Y(\gamma^{nm+l}_{p_i},\gamma^{l}_{p_j},\emptyset)=\vdim  \cM^2_Y(\gamma^{nm+l}_{p_i},\gamma^*_*,\gamma^{l}_{p_j},\emptyset,\emptyset)= 2mn+2i-2j-2.$$ 
    		Then it is zero iff
    		\begin{enumerate}[(i)]
    		    \item $m=0, i=j+1$, which corresponding to $\langle U(q_{\gamma^l_{i+1}}), q_{\gamma^l_{i}} \rangle=n$,
    		    \item $m=1,i=0,j=n-1$, we assume $\langle U(q_{\gamma^{n+l}_{0}}), q_{\gamma^l_{n-1}} \rangle=a_{l}$.
    		\end{enumerate}
    		\item\label{dim4}  For $l,s>0$, $$\vdim \cM^1_Y(\gamma^{nm+l+s}_{p_i},\gamma^{l}_{p_j},\{\underbrace{\gamma_{p_0},\ldots,\gamma_{p_0}}_{s}\})=2mn+2s+2i-2j-2.$$ 
    		Then it is zero iff $m=0,j=s+i-1$. In this case we have $\cM^1$ as well as the corresponding $\cM^2$ are possibly non-empty iff $s=1$, for otherwise, we have $s+i-1>i$. Then we can use the compactness argument in Proposition \ref{prop:max} and that the stable manifold of $q_i$ does not intersect with unstable manifold of $q_{s+i-1}$. When $m=0,s=1,j=i$, this corresponds to $\langle U(q_{\gamma^{l+1}_i}), q_{\gamma^{l}_i}\rangle=\epsilon^1(q_{\gamma_0})=-n$.
    	\end{enumerate}
    	By \eqref{dim1} and \eqref{dim2}, to supply for planarity $1$, we must have $a q_{\gamma^n_0}+bq_{\gamma_{n-1}}$ with $a-b\ne 0$. Note that $U(a q_{\gamma^n_0}+bq_{\gamma_{n-1}})=-anq_{\gamma^{n-1}_0}+bnq_{\gamma_{n-2}}\ne 0$. Therefore if the order of semi-dilation of this augmentation is smaller than $1$, then there exists $A$ generated by generators other than $q_{\gamma^n_0},q_{\gamma_{n-1}}$, such that $U(A)=anq_{\gamma^{n-1}_0}-bnq_{\gamma_{n-2}}$. The only way to eliminate $-bnq_{\gamma_{n-2}}$ is having $bq_{\gamma^2_{n-2}}$ in $A$, which adds $bnq_{\gamma^2_{n-3}}$  to $U(A)$. The only way to compensate such term is add a $b q_{\gamma^3_{n-3}}$ to $A$. We can keep the argument going, and claim that $A$ must be $b \sum_{i=2}^{n-1} q^i_{\gamma_{n-i}}$, then $U(A)=bnq_{\gamma^{n-1}_0}-bnq_{\gamma_{n-2}}\ne anq_{\gamma^{n-1}_0}-bnq_{\gamma_{n-2}} $, since $a-b\ne 0$. The claim follows.
    \end{proof}      
\end{proof}

When $n\le k$, there are augmentations with zero order of semi-dilation. For example, one can use the augmentation from natural prequantization bundle filling, then the order of semi-dilation is $0$ since the symplectic cohomology vanishes \cite{ritter2014floer}. However, we conjecture that $\Hcx(Y)\ge 1^{\SD}$ whenever $n\le k$. It is possible that there are $BL_\infty$ augmentations that are not from (even singular) fillings. Note that $n>k$ is the region where the quotient singularity is terminal. Hence we ask whether there is relation between this algebro-geometric property with contact property of the link via the hierarchy functor $\Hcx$.
\begin{conjecture}
	For discrete $G\subset U(n)$, if $\C^n/G$ is an isolated singularity, then $\Hcx(S^{2n-1}/G,\xi_{std})=0^{\SD}$ if the singularity is terminal.
\end{conjecture}
Combining with Theorem \ref{thm:brieskorn}, we can also ask the following question.
\begin{question}
	Is the planarity of an isolated terminal singularity always $1$? Is it true for terminal hypersurface singularities?
\end{question}

Similarly to Theorem \ref{thm:quotient}, \cite[Theorem 1.1 (2)]{zhou2020filings} can be rephrased as follows.
\begin{theorem}\label{thm:product}
	Let $V$ be an exact domain, then $\Hcx(\partial(V\times \D))=0^{\SD}$.
\end{theorem}

\subsection{An obstruction to IP}
In dimension $3$, obstructions to planar open book decomposition were studied from many different perspectives in \cite{etnyre2004planar,ozsvath2005planar}. In higher dimensions, obstructions to supporting an iterated planar structure were found in \cite{acu2018planarity}. By Corollary \ref{cor:lower} and Theorem \ref{thm:upper}, we the following easy to check obstruction to iterated planar structure.
\begin{corollary}\label{cor:ob}
	If contact manifold $Y$ admits an exact filling that is not $k$-uniruled for any $k$, then $Y$ is not iterated planar.
\end{corollary}
As an application of this corollary, we have the following.
\begin{corollary}\label{cor:cotangent}
	Let $Q$ be a hyperbolic manifold of dimension $\ge 3$, then $S^*Q$ is not iterated planar.
\end{corollary} 
\begin{proof}
	The claim follows from a result of Viterbo \cite[Theorem 1.7.5]{eliashberg2000introduction} that $T^*Q$ is not $k$-uniruled for any $k$. 
\end{proof}

For other classes of cosphere bundles, by Theorem \ref{thm:brieskorn}, $\Hcx(S^*S^n)=1^{\SD}$ for $n\ge 2$. By Corollary \ref{cor:product}, $\Hcx(S^*T^n)=2^{\Pl}$ for $n\ge 2$. By Theorem \ref{thm:BO-iso}, since $SH^*(T^*Q)\ne 0$ for any $Q$, we know $\Hcx(S^*Q)>0^{\SD}$ (assuming Claim \eqref{claim:u}). As a consequence, there is no exact cobordism from $S^*Q$ to $\partial(V\times \D)$ for any Liouville domain $V$, which is a generalization of a result of Gromov \cite{gromov1985pseudo}. By \cite[Proposition 5.1]{zhou2019symplectic}, $T^*Q$ admits a $k$-dilation for some $k\ge 1$ for rationally-inessential $n$-manifold $Q$, i.e. if $H_n(Q;\Q)\to H_n(B\pi_1(Q);\Q)$ vanishes, then we can update the estimate $\Hcx(S^*Q)$ by figuring out $k$. For Lagrangian $Q$ that is a $K(\pi,1)$ space, we have $\Hcx(S^*Q)\ge 2^{\Pl}$, since $T^*Q$ carries no $k$-semi-dilation for any $k$. 

\begin{corollary}\label{cor:tight}
	For every $n\ge 3$, there exists a tight $S^{2n-1}$ with the standard almost contact structure that is not iterated planar.	
\end{corollary}
\begin{proof}
	Note that the contact boundary of the Brieskorn variety $x_0^{n+2}+\ldots+x_{n}^{n+2}=1$ has planarity order $\infty$ by Theorem \ref{thm:brieskorn}. By \cite[Proposition 3.6]{MR3483060}, there are $a_i\ge n+2$, such that $Y:=LB(a_0,\ldots,a_n)$ is an exotic sphere.  Proposition \ref{prop:embed} implies that $\Pl(Y)=\infty$. Then there exists $k$ such that $\#^kY$ is the standard smooth sphere,  where $\#$ is the contact connected sum. However, the almost contact structure, which can be computed from \cite[(19), (20)]{MR3483060}, may not be standard. By \cite[Theorem 1.2]{MR2104476}, there exists a Weinstein fillable contact sphere $Y'$, such that $\#^k Y\# Y'$ is the standard almost contact sphere. Then $\Pl(\#^k Y\# Y')\ge \Pl(Y)\otimes \Pl (Y)\otimes \ldots \otimes \Pl(Y')=\infty$ as $\Pl(Y')\ge 1$ and the claim follows.
\end{proof}

\begin{corollary}\label{cor:noIP}
	In all dimension $\ge 5$, if $(Y,J)$ is an almost contact manifold which has an exactly fillable contact representation $(Y,\xi)$. Then there is a contact structure $\xi'$ in the homotopy class of $J$, such that $(Y,\xi')$ is not iterated planar. In particular, any almost contact simply connected $5$-manifolds admits a contact representation which is not iterated planar.
\end{corollary}
\begin{proof}
	Let $Y'$ be the tight sphere from corollary \ref{cor:tight}.  Since $\Pl(Y,\xi)>0$ as $(Y,\xi)$ has an exact filling, then $\Hcx(Y\#Y')=\infty^{\Pl}$. By Corollary \ref{cor:ob}, $Y\#Y'$ is not iterated planar. The last claim follows from any almost contact simply connected $5$-manifold is almost Weinstein fillable \cite{geiges1991contact}, in particular, there is a contact representation that is Weinstein fillable by \cite{cieliebak2012stein}.
\end{proof}

\bibliographystyle{plain} 
\bibliography{ref}
\Addresses
\end{document}